\documentclass[11pt,a4paper]{article}
\usepackage[utf8]{inputenc}
\usepackage[top=0.99in, bottom=0.99in, left=0.98in, right=0.98in]{geometry}
\usepackage{amsmath,amsthm,amssymb,amscd,dsfont,bm,epic,accents}
\usepackage{graphicx}
\usepackage[all]{xy}
\usepackage{tikz}
\usetikzlibrary{intersections}
\usetikzlibrary{shapes.multipart, positioning}

\usepackage{multicol}

\usepackage{mathrsfs}
\usepackage{enumerate}
\usepackage{caption}
\usepackage{subcaption}
\usepackage{graphicx}
\usepackage{color}
\usepackage{tikz}
\usepackage{tkz-euclide}
\usetikzlibrary{quotes}
\tikzset{every edge quotes/.style =
          { fill = white,
            sloped,
            execute at begin node = $,
            execute at end node   = $  }}

\makeatletter
\newcommand{\proofpart}[2]{%
  \par
  \addvspace{\medskipamount}%
  \noindent\emph{Part #1: #2}\par\nobreak
  \addvspace{\smallskipamount}%
  \@afterheading
}
\makeatother

\usepackage[driverfallback=hypertex]{hyperref}
\usepackage{nameref,zref-xr}                 

\usepackage{cancel}
\usepackage[normalem]{ulem}
\usepackage{enumitem}

\newtheorem{dummy}{}[section]
\newtheorem{thm}[dummy]{Theorem}
\newtheorem{prop}[dummy]{Proposition}
\newtheorem{pr}[dummy]{Proposition}
\newtheorem{lem}[dummy]{Lemma}

\newtheorem{lemma}[dummy]{Lemma}
\newtheorem{cor}[dummy]{Corollary}

\theoremstyle{definition}
\newtheorem{definition}[dummy]{Definition}
\newtheorem{dfn}[dummy]{Definition}
\newtheorem{ntt}[dummy]{Notation}

\theoremstyle{remark}
\newtheorem{rmk}[dummy]{Remark}
\newtheorem{ex}[dummy]{Example}
\newtheorem{obs}[dummy]{Observation}

\newcommand{\CM}{{\mathcal{M}}}

\newcommand{\oCM}{{\overline{\mathcal{M}}}}

\newcommand{\tw}{\text{tw}}

\newcommand{\Ass}{{{\maltese}}}

\newcommand{\oPM}{{\overline{\mathcal{PM}}}^{1/r}}
\newcommand{\oPMb}{{\overline{\mathcal{PM}}}}
\newcommand{\oPMh}{{\overline{\mathcal{PM}}}^{1/r,\h}}

\DeclareMathOperator{\rk}{rank}

\numberwithin{equation}{section}

\usepackage{booktabs,multirow,tabularx}

\def\Mbar{\overline{\mathcal M}}
\def\Mbarstar{\overline{\mathcal M^*}}

\def\GPI{\text{GPI}}
\def\sGPI{\text{sGPI}}
\def\PI{\text{PI}}
\def\ev{\text{ev}}
\def\h{\mathfrak h}
\def\m{\mathfrak m}
\def\b{\mathfrak b}
\def\c{\mathfrak c}
\def\N{\mathfrak N}

\usepackage{scalerel}

\DeclareMathOperator*{\bboxplus}{\scalerel*{\boxplus}{\sum}}

\title{The point insertion technique and open $r$-spin theories II: intersection theories in genus-zero}
\date{}
\author{Ran J. Tessler \and Yizhen Zhao}

\AtEndDocument{\bigskip{\footnotesize%
    \textsc{R. J. Tessler:}\par
  \textsc{Department of Mathematics, Weizmann Institute of Science, POB 26, Rehovot 7610001, Israel} \par  
  \textit{E-mail address}: \texttt{ran.tessler@weizmann.ac.il} \par
  \addvspace{\medskipamount}
  \textsc{Y. Zhao:}\par
  \textsc{Department of Mathematics, Weizmann Institute of Science, POB 26, Rehovot 7610001, Israel} \par  
  \textit{E-mail address}: \texttt{yizhen.zhao@weizmann.ac.il}
}}

\begin{document}

\maketitle

\begin{abstract}
    The papers \cite{BCT1,BCT2,BCT3,GKT,GKT2} initiated the study of open $r$-spin and open FJRW intersection theories, and related them to integrable hierarchies and mirror symmetry.
    This paper uses a new technique, the \emph{point insertion technique}, developed in the prequel \cite{TZ1}, to define new open $r$-spin and open FJRW intersection theories.
    These new constructions provide potential candidates for theories whose existence was conjectured before:
    \begin{enumerate}
    \item K. Hori \cite{Horiprivate} predicted the existence of open $r$-spin theory with $\lfloor\frac{r}{2}\rfloor$ types of boundary states. The one constructed in \cite{BCT1, BCT2} has only one type of boundary state.
    In this work we describe $\lfloor\frac{r}{2}\rfloor$ open $r$-spin theories, labelled by $\h\in\{0,\ldots,\lfloor\frac{r}{2}\rfloor-1\},$ where the $\h$th one has $\h+1$ boundary states. We prove that the $\h=0$ theory is equivalent to the \cite{BCT1,BCT2} construction, and calculate all intersection numbers for all these theories.
    \item In \cite{Melissa} K. Aleshkin and C.C.M. Liu conjectured the existence of a quintic Fermat FJRW theory. We construct such an FJRW theory, and provide evidence that this is the conjectured theory.
    \end{enumerate}
    We also explain how the point insertion technique can be used for constructing other open enumerative theories, satisfying the same universal recursions. 
\end{abstract}

\section{Introduction}
Witten's KdV conjecture \cite{Witten2DGravity} was one of the most influential results in the study of the moduli space of curves, as well as in the study of integrable hierarchies. This conjecture, which was proven a year later by Kontsevich~\cite{Kontsevich}, motivated Gromov--Witten theory and related enumerative geometry theories. To describe this theory, write $\psi_1, \ldots, \psi_n \in H^2(\Mbar_{g,n})$ for the first Chern classes of the relative cotangent line bundles at the $n$ marked points. Let $\{t_i\}_{i\geq 0}$ and~$u$ be formal variables and define the generating function $F^c(t_0,t_1,\ldots,u)$ ($c$ stands for `closed') by
\[
F^c(t_0,t_1,\ldots,u) = \!\!\!\sum_{\substack{g \geq 0, n\geq 1\\2g-2+n>0}} \sum_{d_1,\ldots,d_n\geq 0} \frac{u^{2g-2}}{n!}\left(\int_{\Mbar_{g,n}} \psi_1^{d_1} \cdots \psi_n^{d_n} \right)t_{d_1} \cdots t_{d_n}.
\]
Witten conjectured that $\exp(F^c)$ is the unique $\tau$-function of the KdV hierarchy satisfying the string equation. This conjecture uniquely determined $F^c,$ hence all $\psi$-integrals on $\Mbar_{g,n}$.

In 93' Witten proposed a surprising generalization of the above conjecture, Witten's $r$-spin conjecture \cite{Wi93b}. A (smooth) $r$-spin curve is a smooth marked curve $(C;z_1, \ldots, z_n)$, together with an \emph{$r$-spin structure}: a line bundle~$S$ which satisfies
\[S^{\otimes r} \cong \omega_{C}\left(-\sum_{i=1}^n a_i[z_i]\right),\]
where each \emph{twist} $a_i$ belongs to $\{0,1,\ldots, r-1\},$ and $\omega_C$ is the canonical line bundle. The moduli space of $r$-spin curves of a given genus and twists has a natural compactification to the space $\Mbar_{g,\{a_1, \ldots, a_n\}}^{1/r}.$ Witten's $r$-spin theory relies on the existence of a distinguished class $c_W$ in the cohomology of this space, the \emph{Witten class}. In the so-called concave cases, which include all genus $0$ cases,
$E=(R^1\pi_*\mathcal{S})^{\vee},$ where $\pi: \mathcal{C} \rightarrow \Mbar_{0,\{a_1, \ldots, a_n\}}^{1/r}$ is the universal curve and $\mathcal{S}$ is the universal $r$-spin structure, is an (orbifold) vector bundle. In this case
Witten's class is defined to be the Euler class of $E,$
\begin{equation}\label{eq Witten's class}
c_W = e(E) = e((R^1\pi_*\mathcal{S})^{\vee}).
\end{equation}
In positive genus $R^1\pi_*\mathcal{S}$ needs not to be a vector bundle, and the definition of Witten's class is more involved (see \cite{PV,ChiodoWitten,Moc06,FJR,CLL}).

The \emph{(closed) $r$-spin intersection numbers} are defined by
\begin{gather}\label{eq closed r-spin}
\left\langle\tau^{a_1}_{d_1}\cdots\tau^{a_n}_{d_n}\right\rangle^{\frac{1}{r},c}_g:=r^{1-g}\int_{\Mbar^{1/r}_{g,\{a_1, \ldots, a_n\}}} \hspace{-1cm} c_W \cup \psi_1^{d_1} \cdots \psi_n^{d_n}.
\end{gather}
Witten conjectured \cite{Witten93} that the exponential of the \emph{$r$-spin potential}
\[
F^{\frac{1}{r},c}(t^*_*,u):=\sum_{\substack{g \geq 0, n \geq 1\\2g-2+n>0}} \sum_{\substack{0 \leq a_1, \ldots, a_n \leq r-1\\ d_1, \ldots, d_n \geq 0}} \frac{u^{2g-2}}{n!}\left\langle\tau^{a_1}_{d_1}\cdots\tau^{a_n}_{d_n}\right\rangle^{\frac{1}{r},c}_g t^{a_1}_{d_1} \cdots t^{a_n}_{d_n},
\]
where $t_d^a$ are formal variables indexed by $d\geq 0$ and $0\leq a\leq r-1$, is, after a simple change of variables, a $\tau$-function of the $r$-KdV (or $r$-th Gelfand--Dikii) hierarchy.  This conjecture was proven by Faber, Shadrin and Zvonkine \cite{FSZ10}.

In \cite{FJR2,FJR,FJR3} Fan, Jarvis and Ruan vastly generalized Witten's $r$-spin theory to the FJRW theory of quantum singularities. This theory turned out to give rise to a cohomological field theory and satisfy interesting mirror symmetry statements \cite{FJR2}. It also has surprising relations with other enumerative geometric theories; a notable example is the LG/CY correspondence \cite{ChiodoRuan,CIR} which relates the Gromov--Witten theory of the Calabi--Yau quintic, to the FJRW theory of the Fermat quintic.

In \cite{PST14} Pandharipande, Solomon, and the first named author initiated the study of intersection theory on the moduli space of Riemann surfaces with boundary. Denote by $\Mbar_{0,k,l}$ the moduli space of stable marked disks $(\Sigma;x_1,\ldots,x_k;z_1\ldots,z_l)$, where $x_i \in \partial\Sigma$ are the boundary markings and $z_j \in \Sigma^\circ$ are the internal markings. The paper  \cite{PST14} constructed intersection numbers on $\Mbar_{0,k,l}$, which are the analogue of $\psi$-integrals over the moduli of stable marked disks. Let $F^o_0(t_0,t_1,\ldots,s)$ denote their potential, where the formal variable $s$ tracks the number of boundary marked points. The paper \cite{PST14} proved that the open disk potential satisfies the \emph{open KdV} equations, which are related to the (genus $0$) KdV wave function \cite{Bur16}. The all-genus construction was found in \cite{ST1,Tes15}, and the all-genus analogue of the result, also conjectured in \cite{PST14}, was proven in \cite{Tes15,BT17}, establishing an open analogue of the Witten--Kontsevich theorem.

More recently, Buryak, Clader and the first named author \cite{BCT1,BCT2,BCT3} have found and studied an open analogue of Witten's $r$-spin conjecture. 
The paper \cite{BCT1} constructed the moduli space of \emph{graded $r$-spin disks}, which is the disk analogue of $r$-spin curves, and its \emph{open Witten bundle}. The paper \cite{BCT2} then defined intersection numbers on this moduli space, which involved open $\psi$ classes and the \emph{open Witten class}, calculated them, and showed that their potential, after a simple transformation, becomes the genus-$0$ part of the $r$-KdV wave function. The paper \cite{BCT3} made an all-genus conjecture, and provided evidence for it. The paper \cite{GKT} showed that the open $r$-spin theory satisfies open mirror symmetry with Saito's theory of $A_{r-1}$ singularity.

A first step towards more general open FJRW theory, was taken in \cite{GKT2}, where it was also shown that the open FJRW theory of a rank $2$ Fermat polynomial satisfies mirror symmetry with the corresponding Saito theory. \cite{Melissa} conjectured the existence of a Fermat quintic FJRW theory, and moreover suggested that such a theory should satisfy an open version of the LG/CY correspondence (see Section \ref{sec:open_fjrw} for more details).

The construction of \cite{BCT1,BCT2,BCT3} allows the internal twists $a_i$ to vary in $\{0,\ldots,r-1\},$ but the only allowed boundary twist is $r-2$. It is natural to ask whether one can construct more general intersection numbers, allowing a larger variety of boundary twists. K. Hori (\cite{Horiprivate}, and implicitly also in \cite{Hori}) suggested the existence of an open intersection theory with $\lfloor\frac{r}{2}\rfloor$ possible boundary states.

In this work we justify this expectation, by showing the existence of $\lfloor\frac{r}{2}\rfloor$ open $r$-spin theories, indexed by $\h=0,\ldots,\lfloor\frac{r-2}{2}\rfloor$, where the $\h$-th theory has $\h+1$ boundary states, ranging in $\{r-2,r-4,\ldots,r-2-2\h\}.$ 
We also suggest a construction of an open FJRW quintic theory, and provide evidence that this is the sought-after theory by \cite{Melissa}. 

In the prequel \cite{TZ1} we described a \emph{point insertion technique}, which allowed the identification of certain pairs of boundary strata of different moduli spaces, which we call \emph{spurious boundaries}. We showed that the identification of moduli points lifts to identifications of the fibers of the Witten bundles and relative cotangent line at these points, and analyzed the behaviours of orientations under this identification. Roughly speaking, the outcome of \cite{TZ1} was that one can glue the moduli spaces, Witten bundles and lines $\mathbb{L}_i$ along the corresponding pairs of boundaries, to obtain an orbifold vector bundles $\widetilde{\mathcal{W}},~\widetilde{\mathbb{L}}_1,\ldots,\widetilde{\mathbb{L}}_l$ over a compact orbifold with corners $\widetilde{\mathcal M}^{\frac{1}{r},\h}_{0,\{b_1,\ldots,b_k\},\{a_1,\ldots,a_l\}}
    $, where $b_j$ are the boundary twists and $a_i$ are the internal twists, and \[{\widetilde{\mathcal{W}}}^{\oplus m} \oplus \bigoplus_{i=1}^l \widetilde{\mathbb{L}}_i^{\oplus d_i}\to\widetilde{\mathcal M}^{\frac{1}{r},\h}_{0,\{b_1,\ldots,b_k\},\{a_1,\ldots,a_l\}}\] is canonically relatively oriented, for all choices of $d_1,\ldots,d_n\geq 0$ and of odd $m$.

In this work we define \emph{canonical boundary conditions} to the bundles $\widetilde{\mathcal{W}},\widetilde{\mathbb{L}}_i$ over the boundary of $\widetilde{\mathcal M}^{\frac{1}{r},\h}_{0,\{b_1,\ldots,b_k\},\{a_1,\ldots,a_l\}},$ and we use them to define open FJRW  intersection numbers by setting
\[\left\langle\tau_{d_1}^{a_1}\cdots\tau^{a_l}_{d_l}\sigma^{b_1}\cdots\sigma^{b_k}\right\rangle^{\frac{1}{r},\text{o},\h,m }_0:=\int_{\widetilde{\mathcal M}^{\frac{1}{r},\h}_{0,\{b_1,\ldots,b_k\},\{a_1,\ldots,a_l\}}} e\left( {\widetilde{\mathcal{W}}}^{\oplus m} \oplus \bigoplus_{i=1}^l \widetilde{\mathbb{L}}_i^{\oplus d_i}, s_{\text{canonical}}\right),
\]
where the notation $e(E,s)$ stands for the Euler class of $E$ relative to the boundary conditions $s$ (see Section \ref{sec sections} for precise definitions), and $m$ is again assumed to be odd.
The case $m=1,$ on which we mainly concentrate, is called \emph{open $r$-spin}, and in this case we omit the superscript $m$ from the notations.

\subsection{Main results}
Our first main result is (see Theorem \ref{thm intersection numbers well-defined} for an accurate statement):
\begin{thm}
Let 
\[e=\frac{\sum_{i=1}^l a_i+\sum_{j=1}^k b_j-(r-2)}{r}.\] Suppose that $e$ is a non-negative integer and $e\equiv 1+k\mod 2$, let $d_1,\ldots, d_l$ be non-negative integers and $m$ an odd natural number, such that
\[me+2\sum_{i=1}^ld_i=k+2l-3.\]
Then the open intersection numbers $\left\langle
		\tau^{a_1}_{d_1}\tau^{a_2}_{d_2}\dots\tau^{a_l}_{d_l}\sigma^{b_1}\sigma^{b_2}\dots\sigma^{b_k}
		\right\rangle_0^{\frac{1}{r},\text{o},\h,m}$ are well-defined, independent of all choices.
\end{thm}
Even though the intersection numbers are defined in a complicated way, the following topological recursion relations (TRR) hold.
\begin{thm}[Theorem \ref{thm TRR}]\label{intro_thm TRR}
\begin{itemize}
\item[(a)] If $l,k\ge 1$, then
\begin{equation*}
		\begin{split}
		&\left\langle
		\tau^{a_1}_{d_1+1}\tau^{a_2}_{d_2}\dots\tau^{a_l}_{d_l}\sigma^{b_1}\sigma^{b_2}\dots\sigma^{b_k}
		\right\rangle_0^{\frac{1}{r},\text{o},\h,m}
		\\=&\sum_{\substack{s\ge 0\\-1\le a \le r-2}}\sum_{\substack{0 \le t_i \le \h\\\coprod_{j=-1}^{s}R_j=\{2,3,\dots,l\}\\\coprod_{j=0}^{s}T_j=\{2,3,\dots,k\}\\ \{(R_j,T_j,t_j)\}_{1\le j \le s}\text{ unordered}}}(-1)^s\left\langle
		\tau^{a}_{0}\tau^{a_1}_{d_1}\prod_{i\in R_{-1}}\tau^{a_i}_{d_i}\prod_{j=1}^{s}\tau^{t_j}_{0}
		\right\rangle_0^{\frac{1}{r},\text{ext},m}\\
		&\cdot\left\langle
		\tau^{r-2-a}_{0}\sigma^{b_1}\prod_{i\in R_0}\tau^{a_i}_{d_i}\prod_{i\in T_0}\sigma^{b_i}
		\right\rangle_0^{\frac{1}{r},\text{o},\h,m}
		\prod_{j=1}^{s}\left\langle
		\sigma^{r-2-2t_j}\prod_{i\in R_j}\tau^{a_i}_{d_i}\prod_{i\in T_j}\sigma^{b_i}
		\right\rangle_0^{\frac{1}{r},\text{o},\h,m}.
		\end{split}
		\end{equation*}

\item[(b)] If $l\ge 2$, then
\begin{equation*}
		\begin{split}
		&\left\langle
		\tau^{a_1}_{d_1+1}\tau^{a_2}_{d_2}\dots\tau^{a_l}_{d_l}\sigma^{b_1}\sigma^{b_2}\dots\sigma^{b_k}
		\right\rangle_0^{\frac{1}{r},\text{o},\h,m}
		\\=&\sum_{\substack{s\ge 0\\-1\le a \le r-2}}\sum_{\substack{0 \le t_i \le \h\\ \coprod_{j=-1}^{s}R_j=\{3,4,\dots,l\}\\\coprod_{j=0}^{s}T_j=\{1,2,\dots,k\}\\ \{(R_j,T_j,t_j)\}_{1\le j \le s}\text{ unordered}}}(-1)^s\left\langle
		\tau^{a}_{0}\tau^{a_1}_{d_1}\prod_{i\in R_{-1}}\tau^{a_i}_{d_i}\prod_{j=1}^{s}\tau^{t_j}_{0}
		\right\rangle_0^{\frac{1}{r},\text{ext},m}\\
		&\cdot\left\langle
		\tau^{r-2-a}_{0}\tau^{a_2}_{d_2}\prod_{i\in R_0}\tau^{a_i}_{d_i}\prod_{i\in T_0}\sigma^{b_i}
		\right\rangle_0^{\frac{1}{r},\text{o},\h,m}
		\prod_{j=1}^{s}\left\langle
		\sigma^{r-2-2t_j}\prod_{i\in R_j}\tau^{a_i}_{d_i}\prod_{i\in T_j}\sigma^{b_i}
		\right\rangle_0^{\frac{1}{r},\text{o},\h,m}.
		\end{split}
		\end{equation*}
\end{itemize}
\end{thm}
This form of this TRR is very different from TRRs that appeared before in OGW theory, notably Solomon's Open WDVV \cite{horev2012open,solomon2023relative} and the TRRs of \cite{PST14,BCT2}. But as we will argue in Section \ref{sec:generalization}, this is the \emph{universal form} of TRRs for theories based on the point insertion technique.

Turning to consider $m=1,$ the open $r$-spin case, we first show in Section \ref{section:comparison} that the $\h=0$ theory is equivalent to \cite{BCT1,BCT2} open $r$-spin theory in a non-trivial way, by writing an explicit transformation relating the potentials of the two theories. 
Even though the theories for $\h>0$ have more states, hence more possible intersection numbers than the theory \cite{BCT1,BCT2}, we prove that many of them vanish:
 \begin{prop}[see Proposition \ref{prop:vanish small internal} for a more general statement]
If $l\ge 1$, $2l+k\ge 4$ and $a_1\le \h$ , we have
    $$
\left\langle
		\tau^{a_1}_{0}\tau^{a_2}_{d_2}\tau^{a_3}_{d_3}\dots\tau^{a_l}_{d_l}\sigma^{b_1}\sigma^{b_2}\dots\sigma^{b_k}
		\right\rangle_0^{\frac{1}{r},\text{o},\h}=0.
$$
\end{prop}
In Theorem \ref{thm calculate} we show that the above TRRs and vanishing result are enough to calculate all open intersection numbers for all $r,\h$.

An important example for $m>1$ is the open Fermat quintic FJRW theory (see Section \ref{sec:open_fjrw} for details). Its intersection numbers are of the form 
$\langle\sigma_{1}^{5d-2}\rangle^{\frac{1}{r}=\frac{1}{5},\text{o},\h=1,m=5}_0$, for $5$-spin disks with $5d-2$ boundary markings of twist $1.$ We prove:
\begin{prop}[Proposition \ref{prop:even_vanishing_quintic}]
$\langle\sigma_{1}^{5d-2}\rangle^{\frac{1}{r}=\frac{1}{5},\text{o},\h=1,m=5}_0$ vanishes for even $d.$
\end{prop}
We conjecture that the odd $d$ numbers do not vanish, and that their potential satisfies the open LG/CY correspondence conjectured in \cite{Melissa}.

\subsection{The point insertion technique}
The open $r$-spin construction of \cite{BCT1,BCT2} was based on boundary conditions of two types: positivity boundary conditions and forgetful boundary conditions. The former uses a hidden intrinsic notion of positivity for sections of the Witten bundle, and the latter can be applied in certain cases where, after normalizing a node, one of its branches can be forgotten.
The forgetful boundary conditions appeared in \cite{PST14,ST1,BPTZ,GKT2} and a homotopical generalization of them also appeared in \cite{solomon2023relative,ZernikFP}. Many theories which allow the forgetful boundary conditions, also allow a definition of open intersection numbers which are independent of choices and satisfy the Solomon's \emph{Open WDVV} \cite{solomon2023relative,horev2012open}.

In the case of open $r$-spin, forgetful boundary conditions can be applied only with respect to a half-node of twist $0.$ The main novel geometric idea presented here is a new type of boundary treatment which strictly generalizes the forgetful boundary conditions, and allows defining well-defined open intersection theories in many cases where the forgetful scheme fails. In this paper we concentrate on applying this idea to open $r$-spin and open FJRW theories. Other examples are sketched in Section \ref{sec:generalization}, including open Gromov--Witten theories and open Hodge theories.
Interestingly, the open intersection theories which are defined via the point insertion technique all satisfy again universal topological recursion relations which generalize Theorem \ref{intro_thm TRR}. We leave the study of these generalizations to future work. 
\subsection{Structure of the paper}
This paper is constructed as follows.  Section~\ref{sec mod and bundle} reviews definitions and results from \cite{TZ1} concerning the graded $r$-spin disks, their moduli, associated bundles and orientations and the point insertion technique. In Section \ref{sec sections}, we define what are canonical multisections, and use them to define open $r$-spin and certain open FJRW intersection numbers.  In Section \ref{sec:trr} we prove the topological recursion relations of Theorem \ref{intro_thm TRR}. Section \ref{section:comparison} relates the $(r,\h=0)$ theory to the open $r$-spin theory of \cite{BCT1,BCT2}, while Section \ref{sec:computations} uses the topological recursion relations to calculate all open $r-$spin intersection numbers. 
We shortly discuss certain open FJRW theories, which include the Fermat quintic theory, in Section \ref{sec:open_fjrw}.
Finally, in Section \ref{sec:generalization} we explain how the point insertion idea can be used in many other enumerative theories of open topological string flavour.

\subsection{Acknowledgements}
The authors would like to thank A. Buryak, M.~Gross, Y.~Groman, T.~Kelly, K. Hori, J.~Solomon and E.~Witten for interesting discussions related to this work.  
R.T. (incumbent of the Lillian and George Lyttle Career Development Chair) and Y. Zhao were supported by the ISF grants No. 335/19 and 1729/23 later.

\section{Review of graded $r$-spin disks and point insertion}
\label{sec mod and bundle}
In this section, we review the definition of graded $r$-spin disks, their moduli space, the relevant bundles and the point insertion technique.  More details can be found in \cite{BCT1,TZ1}.

Throughout the paper we will use the notation $[N],$ where $N$ is a natural number, to denote the set $\{1,2,\dots,N\}.$

\subsection{The moduli space of graded $r$-spin disks}

The main objects of study in this paper are genus-zero marked Riemann surfaces with boundary; we view them as arising from closed spheres with an orientation-reversing involution.  More precisely, a \textit{nodal marked disk} is defined as a tuple
$$(C, \phi, \Sigma, \{z_i\}_{i \in I}, \{x_j\}_{j \in B}, m^I, m^B),$$
in which
\begin{itemize}
\item  $C$ is a nodal orbifold Riemann surface, which may be composed of disconnected components; each component is genus-zero and has isotropy only at special points;
\item $\phi: C \to C$ is an anti-holomorphic involution which, from a topological perspective, realizes the coarse underlying Riemann surface $|C|$  as the union of two Riemann surfaces, $\Sigma$ and $\overline{\Sigma}=\phi(\Sigma),$ glued along their common subset $\text{Fix}(|\phi|)$;
\item $\{z_i\}_{i \in I}\subset  C$ consists of distinct \textit{internal marked points} (or \textit{internal tails}) whose images in $|C|$ lie in $\Sigma\setminus\text{Fix}(|\phi|)$. We denote by $\overline{z_i}:= \phi(z_i)$ their \textit{conjugate marked points};
\item $\{x_j\}_{j \in B} \subset \text{Fix}(\phi)$ consists of distinct \textit{boundary marked points} (or \textit{boundary tails}) whose images in $|C|$ lie in ${\partial} \Sigma$;
\item $m^I$ (respectively $m^B$) is a marking of $I$ and (respectively $B$), \textit{i.e.} a one-to-one correspondence between  $I$ (respectively $m^B$) and $\left[\lvert I \rvert\right]$ (respectively $[\lvert B \rvert]$).
\end{itemize}
We say that a nodal marked disk is \textit{stable} if each irreducible component has at least three special points.

A node of a nodal marked disk can be internal, boundary or contracted boundary, as illustrated in Figure~\ref{fig three types of nodes} by shading $\Sigma \subseteq |C|$ in each case.  Note that ${\partial}\Sigma\subset\text{Fix}(|\phi|)$ is a union of cycles, and $\text{Fix}(|\phi|)\setminus{\partial}\Sigma$ is the union of the contracted boundary nodes.
\begin{figure}[h]
\centering
\begin{subfigure}{.3\textwidth}
  \centering

\begin{tikzpicture}[scale=0.3]
  \fill[color = gray, opacity = 0.3] (0,4) circle (2);
  \draw (0,4) circle (2);
  \draw (-2,4) arc (180:360:2 and 0.6);
  \draw[dashed] (2,4) arc (0:180:2 and 0.6);

  \draw (0,0) circle (2cm);
  \draw (-2,0) arc (180:360:2 and 0.6);
  \draw[dashed] (2,0) arc (0:180:2 and 0.6);
   \fill[color = gray, opacity = 0.3] (-2,0) arc (180:360:2 and 0.6) arc (0:180:2);

  \draw (0,-4) circle (2cm);
  \draw (-2,-4) arc (180:360:2 and 0.6);
  \draw[dashed] (2,-4) arc (0:180:2 and 0.6);
\end{tikzpicture}

  \caption{Internal node}
\end{subfigure}
\begin{subfigure}{.3\textwidth}
  \centering
\begin{tikzpicture}[scale=0.4]
 
  \draw (-2,0) circle (2cm);
  \draw (-4,0) arc (180:360:2 and 0.6);
  \draw[dashed] (0,0) arc (0:180:2 and 0.6);
  \fill[color=gray, opacity=0.3] (-4,0) arc (180:360:2 and 0.6) arc (0:180:2);

  \draw (2,0) circle (2);
  \draw (0,0) arc (180:360:2 and 0.6);
  \draw[dashed] (4,0) arc (0:180:2 and 0.6);
  \fill[color=gray, opacity=0.3] (0,0) arc (180:360:2 and 0.6) arc (0:180:2);
\end{tikzpicture}
\vspace{0.9cm}

  \caption{Boundary node}
\end{subfigure}
\begin{subfigure}{.3\textwidth}
  \centering

\begin{tikzpicture}[scale=0.4]
\vspace{0.1cm}
  \fill[color=gray, opacity=0.3] (0,2) circle (2);
  \draw (0,2) circle (2);
  \draw (-2,2) arc (180:360:2 and 0.6);
  \draw[dashed] (2,2) arc (0:180:2 and 0.6);

  \draw (0,-2) circle (2);
  \draw (-2,-2) arc (180:360:2 and 0.6);
  \draw[dashed] (2,-2) arc (0:180:2 and 0.6);
\end{tikzpicture}
\vspace{0.1cm}

  \caption{Contracted boundary node}
\end{subfigure}
\caption{The three types of nodes in a nodal marked disk.}
\label{fig three types of nodes}
\end{figure}
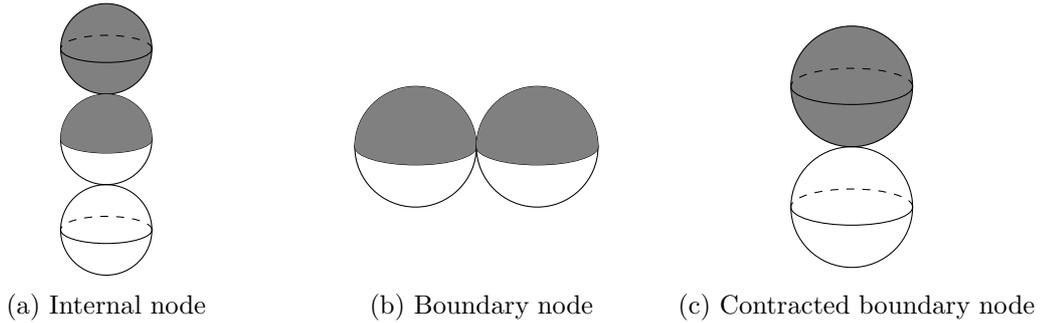

An \textit{anchored nodal marked disk} is a nodal marked disk together with a $\phi$-invariant choice of a distinguished internal marked point (called the \textit{anchor}) on each connected component $C'$ of $C$ that is disjoint from the set $\text{Fix}(\phi)$. We denote by $Anc\subseteq I$ the set of indexes of anchors lying on $\Sigma$. 
\begin{rmk}
We focus mainly on disks with boundaries and their degenerations, \textit{i.e.} connected nodal marked disks with non-empty $\text{Fix}(\phi)$. Disconnected disks will usually appear as the result of normalizations at nodes (and later also from the point insertion method we shall define). The anchors help us keep track of these normalizations: Consider a connected component $C'$ of a nodal marked disk $C$, that does not intersect with the set $\text{Fix}(\phi)$, and assume that $C'$ is obtained by normalizing internal or contracted boundary nodes. 
The anchor $z_i$ of $C'$ is the half-node corresponding to the internal node \emph{closest to $\text{Fix}(\phi)$} that we had normalized in order to obtain $C'$. By `closest to $\text{Fix}(\phi)$' we mean either the normalized contracted boundary node, if there is such a node. Otherwise it is the unique internal node whose normalization makes the connected component containing $C'$ disconnected from $\text{Fix}(\phi)$.
\end{rmk}
\subsubsection{$r$-spin structures}
Let $C$ be an anchored nodal marked disk with
order-$r$ cyclic isotopy groups at markings and nodes, a \textit{$r$-spin structure} on $C$ is 
\begin{itemize}
    \item 
 an orbifold complex line bundle $L$ on $C$,

\item an isomorphism 
$$\kappa:L^{\otimes r} \cong \omega_{C,log}:=\omega_{C} \otimes {\mathcal{O}}\left(-\sum_{i \in I}  [z_i] - \sum_{i \in I}  [\overline{z_i}] - \sum_{j \in B} [x_j]\right),$$

\item an involution $\widetilde{\phi}: L \to L$ lifting $\phi$.
\end{itemize}
The local isotopy of $L$ at a point $p$ is characterized by an integer $\operatorname{mult}_p(L)\in \{0,1,\dots,r-1\}$ in the following way: the local structure of the total space of $L$ near $p$ is $[\mathbb C^2/(\mathbb Z/r\mathbb Z)]$, where the canonical generator $\xi \in \mathbb Z/r\mathbb Z$ acts by $\xi\cdot (x,y)=(\xi x, \xi^{\operatorname{mult}_p(L)}y)$.

We denote by $RI\subseteq I$ and $RB\subseteq B$ the subsets of internal and boundary marked points $p$ satisfying $\operatorname{mult}_p(L)=0$. An \textit{associated twisted r-spin structure} $S$ is defined by
$$
S:=L\otimes \mathcal O\left( -\sum_{i \in \widetilde{RI}} r [z_i] - \sum_{i \in \widetilde{RI}} r [\overline{z_i}] - \sum_{j \in RB} r [x_j]\right),
$$
where $\widetilde{RI}\subseteq RI$ is a subset satisfying $ RI\setminus \widetilde{RI}\subseteq Anc$.

For an internal marked point $z_i$, we define the  \textit{internal twist} at $z_i$ to be $a_i:=\operatorname{mult}_{z_i}(L)-1$ if $i\in I\setminus \widetilde{RI}$
and $a_i:=r-1$ if $i\in \widetilde{RI}$. For a boundary marked point $x_i$, we define the  \textit{boundary twist} at $x_j$ as $b_j:=\operatorname{mult}_{x_j}(L)-1$ if $i\in B\setminus {RB}$ and as $b_j:=r-1$ if $j\in {RB}$. Note that all the marked points with twist $-1$ are indexed by $RI\setminus \widetilde{RI}\subseteq Anc$. When the disk $C$ is smooth, the coarse underlying bundle $|S|$ over the coarse underlying sphere $|C|$ satisfies
$$|S|^{\otimes r} \cong \omega_{|C|} \otimes {\mathcal{O}}\left(-\sum_{i \in I} a_i [z_i] - \sum_{i \in I} a_i [\overline{z_i}] - \sum_{j \in B} b_j[x_j]\right).$$

\begin{obs}
\label{obs rank open}
A connected nodal marked disk admits a twisted $r$-spin structure with twists $a_i$ and $b_j$ if and only if
\begin{equation}\label{eq rank open}
\frac{2\sum_{i \in I} a_i + \sum_{j\in B} b_j -r+2}{r}\in \mathbb Z.
\end{equation}
This formula is the specialization to our setting of the more well-known fact \cite{Witten93}: a (closed) connected nodal marked genus-zero curve admits a twisted $r$-spin structure twists $a_i$ if and only if 
\begin{equation}\label{eq rank close}
\frac{\sum_{i\in I} a_i -r+2}{r}\in {\mathbb{Z}}.
\end{equation}
\end{obs}

We can extend the definition of twists to \textit{half-nodes}. Let $n: \widehat{C} \to C$ be the normalization morphism. For a half-node $q\in \widehat{C}$, we denote by $\sigma_0(q)$ the other half-node corresponding to the same node $n(q)$ as $q$. The isotopies of $n^*L$ at $q$ and $\sigma_0(q)$ satisfy $$\operatorname{mult}_{q}(n^*L)+\operatorname{mult}_{\sigma_0(q)}(n^*L)\equiv0 \mod r.$$ It is important to note that $n^*S$ may not be a twisted $r$-spin structure (associated with $n^*L$), because its connected components could potentially contain too many marked points with twist $-1$ (note that marked points with twist $-1$ are anchors).  Nevertheless, there is a canonical way to choose a minimal subset $\mathcal{R}$ of the half-nodes making
\begin{equation}
\label{eq normalize S}
\widehat{S}:= n^* S \otimes {\mathcal{O}}\left(-\sum_{q \in \mathcal{R}} r [q]\right)
\end{equation}
a twisted $r$-spin structure: denoting by $\mathcal T$ the set of half-nodes $q$ of $C$ where $\operatorname{mult}_{q}(n^*L)=0$, we define
\begin{equation*}
\mathcal A:= \left\{q \in \mathcal T\colon 
\begin{array}{ccc}
    \text{$n(q)$ is an internal node; after normalizing the }\\
    \text{node $n(q)$, the half-node $\sigma_0(q)$ belongs to a connected }\\
    \text{component meeting $\text{Fix}(\phi)$ or containing an anchor.}
\end{array} \right\}
\end{equation*}
and set $\mathcal R := \mathcal T \setminus \mathcal A$. See \cite[Section 2.3]{BCT1} for more details. 
We define $c_t$, \textit{the twist of $S$ at a half-node} $h_t$, as $c_t:=\operatorname{mult}_{h_t}(n^*L)-1$ if $h_t \notin \mathcal R$ and as $c_t:=r-1$ if $h_t\in \mathcal R$.

For each irreducible component $C_l$ of $\widehat{C}$ with half-nodes $\{h_t\}_{t \in N_l}$, we have
$$\bigg(|\widehat{S}|\big|_{|C_l|}\bigg)^{\otimes r} \cong \omega_{|C_l|} \otimes {\mathcal{O}}\left(-\sum_{\substack{i \in I\\z_i\in C_l}} a_i[z_i] - \sum_{\substack{i \in I\\ \overline{z_i}\in C_l}} a_i [\overline{z_i}] - \sum_{\substack{j \in B\\ x_j\in C_l}} b_j[x_j] - \sum_{t \in N_l} c_t [h_t]\right);$$
note that in the case where $C_l$ intersects with $\partial \Sigma$, the set $\{h_t\}_{t \in N_l}$ is invariant under $\phi$, and the half-nodes conjugated by $\phi$ have the same twist.

Note that if $h_{t_1}=\sigma_0(h_{t_2})$, then
\begin{equation}\label{eq sum of twist at node}c_{t_1} + c_{t_2} \equiv -2 \mod r.\end{equation}
We say a node is \textit{Ramond} if one (hence both) of its half-nodes $h_t$ satisfy $c_t \equiv -1 \mod r$,  and it is said to be \textit{Neveu--Schwarz (NS)} otherwise. Note that if a node is Ramond, then both of its half-nodes lie in the set $\mathcal T$; moreover, a half-node has twist $-1$ if and only if it lies in $\mathcal A$. The set $\mathcal{R}$ in equation~\eqref{eq normalize S} is chosen in a way that each internal Ramond node has precisely
one half-edge in $\mathcal{R}.$

Associated to each twisted $r$-spin structure $S$, we define a Serre-dual bundle \begin{equation}\label{def of J}
    J:=S^{\vee} \otimes \omega_{C}.
\end{equation}
Note that the involutions on $C$ and $L$ induce involutions on $S$ and $J$; by an abuse of notation, we denote the involutions on $S$ and $J$ also by $\widetilde{\phi}$.

\subsubsection{Gradings}
For a nodal marked disk, the boundary ${\partial}\Sigma$ of $\Sigma$ is endowed with a well-defined orientation, determined by the complex orientation on the preferred half $\Sigma \subseteq |C|$. This orientation induces the notion of positivity for $\phi$-invariant sections of $\omega_{|C|}$ over ${\partial} \Sigma$: let $p$ be a point of $\partial \Sigma$ which is not a node, we say a section $s$ is \textit{positive} at a $p$ if, for any tangent vector $v \in T_p({\partial} \Sigma)$ in the orientation direction, we have $\langle s(p), v \rangle > 0$, where $\langle -, - \rangle$ is the natural pairing between cotangent and tangent vectors.

Let $C$ be an anchored nodal marked disk, and let $A$ be the complement of the special points in $\partial \Sigma$.  We say a twisted $r$-spin structure on such $C$ is \textit{compatible on the boundary components} if there exists a $\widetilde{\phi}$-invariant section $v \in C^0\left(A, |S|^{\widetilde{\phi}}\right)$ (called a \textit{lifting} of $S$  on boundary components) such that the image of $v^{\otimes r}$ under the map on sections induced by the inclusion $|S|^{\otimes r} \to  \omega_{|C|}$ is positive.  We say $w \in C^0\left(A, |J|^{\widetilde{\phi}}\right)$ is a \textit{Serre-dual lifting of $J$ on the boundary components with respect to $v$} if $\langle w, v \rangle \in C^0(A, \omega_{|C|})$ is positive, where $\langle -, - \rangle$ is the natural pairing between $|S|^{\vee}$ and $|S|$.  This $w$ is uniquely determined by $v$ up to positive scaling.

 Equivalence classes of liftings of $J$ (or equivalently, $S$) on the boundary components up to positive scaling are equivalent to continuous sections of $S^0(|J|^{\widetilde\phi}),$ the $S^0$-bundle $\left(|J|^{\widetilde{\phi}}\setminus |J|_0\right)\big/ \mathbb R_+$ over $A$, where $|J|_0$ denotes the zero section of $|J|^{\widetilde{\phi}}$. Given an equivalence class $[w]$  of liftings, we say a boundary marked point or boundary half-node $x_j$ is \textit{legal}, or that $[w]$ \textit{alternates} at $x_j$, if $[w]$, as a section of $S^0(|J|^{\widetilde\phi})$, cannot be continuously extended to $x_j$.
We say an equivalence class $[w]$ of liftings of $J$ on boundaries is a \textit{grading of a twisted $r$-spin structure on boundary components} if, for every Neveu--Schwarz boundary node, one of the two half-nodes is legal and the other is illegal. 
\begin{rmk}
    The requirement that every NS boundary node has one legal and one illegal half-node arises from the behaviour of a grading on boundary components at degenerations. This condition, together with \eqref{eq sum of twist at node} allows smoothing the boundary node. See \cite{BCT1} for more details.
\end{rmk}

Let $q$ be a contracted boundary node of $C$, we say a twisted $r$-spin structure on $C$ is \textit{compatible} at $q$ if $q$ is Ramond and there exists a $\widetilde{\phi}$-invariant element $v \in |S|\big|_q$  (called a \textit{lifting} of $S$ at $q$) such that the image of $v^{\otimes r}$ under the map $|S|^{\otimes r}\big|_q \to \omega_{|C|}\big|_q$ is positive imaginary under the canonical identification of $\omega_{|C|}\big|_q$ with ${\mathbb{C}}$ given by the residue. See \cite[Definition 2.8]{BCT1} for more details.  Such a $v$ also admits a Serre-dual lifting, \textit{i.e.}  a $\widetilde{\phi}$-invariant $w \in |J|\big\vert_q$ such that $\langle v, w \rangle$ is positive imaginary.  We refer to the equivalence classes $[w]$ of such $w$ up to positively scaling as a \textit{grading at contracted boundary node $q$}.

We say a twisted $r$-spin structure is \textit{compatible} if it is compatible on boundary components and at all contracted boundary nodes.  A (total) \textit{grading} is the collection of grading on boundary components together with a grading at each contracted boundary node. We say a grading is \textit{legal} if every boundary marked point is legal.

The grading is crucial in determining a canonical relative orientation for the Witten bundle (in \cite{TZ1}) and defining canonical boundary conditions (in Section \ref{sec sections}), which are key ingredients in defining open $r$-spin intersection numbers.

The relation between the twists and legality, and the obstructions to having a grading, are summarized in the following proposition.
\begin{prop}{\cite[Proposition 2.3]{BCT2}}
\label{prop lifting}
\begin{enumerate}
\item\label{it lift odd exist} When $r$ is odd, any twisted $r$-spin structure is compatible, and there is a unique grading.
\item\label{it lift odd legal} When $r$ is odd, a boundary marked point, or boundary half-node, $x_j$ in a twisted $r$-spin structure with a grading is legal if and only if its twist is odd.
\item\label{it lift even compatible} When $r$ is even, the boundary twists $b_j$ in a compatible twisted $r$-spin structure must be even.  

\item\label{it NS nodes}
In a graded $r$-spin structure, any Neveu-Schwarz boundary node has one legal half-node and one illegal half-node.\footnote{This item is part of the definition of grading, we put it here since it is an important property of the grading.}

\item\label{it Ramond boundary node}
Ramond boundary nodes can appear in a graded $r$-spin structure only when $r$ is odd, and in this case, their half-nodes are illegal with twists $r-1.$

\item\label{it lift compatible parity}
There exists grading that alternates precisely at a subset $D \subset \{x_j\}_{j \in B}$ if and only if
\begin{equation}\label{eq parity}
   \frac{2\sum a_i + \sum b_j+ 2}{r}\equiv |D| \mod 2.
\end{equation}

\end{enumerate}
\end{prop}

When a Ramond contracted boundary node is normalized, the grading at this boundary node induces an additional structure at the corresponding half-node. We call such half-nodes by \textit{normalized contracted boundary marked point}. Because such an additional structure is not necessary in this paper, we refer the readers to \cite[Definition 2.5]{TZ1} for the precise definition.

We can now define the primary objects of interest in this paper:

\begin{definition}
\label{def graded rspin disk}
A \textit{stable graded $r$-spin disk} (legal stable graded $r$-spin disk respectively) is a stable anchored nodal marked disk, together with
\begin{enumerate}
\item a compatible twisted $r$-spin structure $S$ in which all contracted boundary nodes are Ramond;
\item a choice of grading (legal gradings respectively);

\item an additional structure of normalized contracted boundary marked point at each anchor with twist $r-1$.

\end{enumerate}
For an integer $0\le\h\le \lfloor \frac{r-2}{2}\rfloor$, we say a stable graded $r$-spin disk is of level-$\h$ if every legal boundary marked point has twist greater than or equal to $r-2-2\h$, and every illegal boundary marked point has twist smaller than or equal to $2\h$. We will omit the term ``level-$\h$" when $\h$ is chosen as a fixed integer or is clear from context.
\end{definition}

\begin{rmk}
    Note that in \cite{BCT2}, the term ``stable graded $r$-spin disk" refers to a legal stable graded level-$0$ $r$-spin disk.
\end{rmk}

We denote by $\Mbarstar_{0,k,l}^{1/r}$ ($\Mbar_{0,k,l}^{1/r}$ respectively) the moduli space of connected stable graded $r$-spin disks (legal connected stable graded $r$-spin disks respectively) with $k$ boundary and $l$ internal marked points. In \cite[Theorem 3.4]{BCT1} and \cite[Theorem 2.8]{TZ1}, $\Mbar_{0,k,l}^{1/r}$ and $\Mbarstar_{0,k,l}^{1/r}$ are shown to be a compact smooth orientable orbifold with corners of real dimension $k+2l-3$.

\begin{ntt}
    Assuming that the internal marked points $\{z_i\}_{i\in I}$ have twists $\{a_i\}_{i\in I}$, by an abuse of notation, we also denote the multiset\footnote{We use the word ``multiset" here because the set $I$ may contain multiple $a_i$ with a same value, but we view them as different elements.} $\{a_i\}_{i\in I}$ by $I$.
    
Similarly, we denote by $B$ the multiset $\{b_j\}_{j\in B}$ equipped with a preselected legality for each of its elements. Furthermore, if $B$ is equipped with a cyclic order on $\sigma\colon B\to B$, we denote it by $\bar B$ and write $\bar{B}=\overline{\{b_1,b_2,\dots,b_{\lvert B \rvert}\}}$ to make the cyclic order manifest, where $\sigma_2(b_i)=b_{i+1}$ for $1\le i \le \lvert B \rvert$ and $\sigma_2(b_{\lvert B \rvert})=b_{1}$.
\end{ntt}

We denote by $\Mbarstar_{0,B,I}^{1/r}$ the moduli space of graded $r$-spin disks with boundary points marked by $B$, and internal points marked by $I$. Note that for a graded $r$-spin disk, the canonical orientation on $\partial\Sigma$ induces a cyclic order on $B$; given a cyclic order $\sigma_2\colon B \to B$, we denote by $\Mbarstar_{0,\bar{B},I}^{1/r}\subseteq \Mbarstar_{0,{B},I}^{1/r}$ the connected component parameterizing the $r$-spin disks such that the induced cyclic order on $B$ coincides with $\sigma_2$.

\begin{rmk}
   In most parts of this article, we primarily consider moduli of legal graded $r$-spin disks. We use the notation with a superscript $*$ only when we want to emphasize that there might be illegal boundary markings. 
\end{rmk}

\subsubsection{Stable graded $r$-spin graphs}
Each connected anchored marked disk $\Sigma$ can be characterized by a decorated dual graph $\Gamma(\Sigma)$ as follows.
\begin{itemize}
\item The set of vertices of $\Gamma(\Sigma)$ is the set of irreducible components of $\Sigma$, which is decomposed into \textit{open} and \textit{closed} vertices $V = V^O \sqcup V^C$ depending on whether the corresponding irreducible component meets $\partial \Sigma$.
\item The set of half-edges $H(v)$ emanating from a vertex $v\in V$ is the set of the special points (\textit{i.e.} half-nodes and marked points) on the irreducible component corresponding to $v$. The set $H(v)$ is decomposed into \textit{boundary} and \textit{internal} half-edges $H(v) = H^B(v) \sqcup H^I(v)$ depending on whether the corresponding special point lies on $\partial \Sigma$. We write $H:=\sqcup_{v\in V}H(v)$ and $H = H^B \sqcup H^I$ the correspond decomposition. Two half-edges correspond to an (internal or boundary) edge $e$ in the set of edges $E=E^I\sqcup E^B$ if their corresponding special points are two half-nodes of the same (internal or boundary) node. 
\item The set of \textit{tails} $T$ is the set of all marked points together with the contracted boundary nodes. We write $T^B:=T\cap H^B$ and $T^I:=T\cap H^I$. The set of \textit{contracted boundary tails} $H^{CB}\subseteq T^I$ corresponds to the contracted boundary node and the set $T^{anc}\subseteq T^I\setminus H^{CB}$ corresponds to the anchor.
\item 
The canonical orientation on $\partial\Sigma$ induced a cyclic order $\sigma_2\colon H^B(v) \to H^B(v)$ for each $v\in V^O$.
\end{itemize}
We say $\Gamma(\Sigma)$ is \textit{smooth} if $E = H^{CB} = \emptyset$, or equivalently, $\Sigma$ is smooth.
If $\Sigma$ is endowed with a graded $r$-spin structure $S$, we have the additional decorations:
\begin{itemize}
    \item a map $\text{tw}: H \rightarrow \{-1,0,1,\ldots, r-1\}$ encoding the twist of $S$ at each special points;
    \item a map
$\text{alt}: H^B \rightarrow \mathbb{Z}/2\mathbb{Z}$
given by $\text{alt}(h) = 1$ if the special point corresponding to $h$ is legal and $\text{alt}(h) = 0$ otherwise.
\end{itemize}

A \textit{genus-zero stable graded $r$-spin graph} is a decorated graph for which each connected component is the dual graph of a connected stable graded $r$-spin disk as above; an intrinsic definition can be found in \cite[Section 3.2]{TZ1}.

In this paper, since we exclusively focus on the genus-zero case, whenever we refer to a stable graded $r$-spin graph, we always mean a genus-zero stable graded $r$-spin graph. 

For an edge $e$ of a stable graded $r$-spin graph $\Gamma=\Gamma(\Sigma)$,  the \textit{smoothing} of $\Gamma$ along $e$ is the stable graded $r$-spin graph $d_e\Gamma$ that is dual to the stable graded $r$-spin disk obtained by smoothing the node in $\Sigma$ corresponding to $e$. The \textit{detaching} of $\Gamma$ along $e$ is the stable graded $r$-spin graph $\text{Detach}_e\Gamma$ that is dual to the stable graded $r$-spin disk obtained by normalizing the node in $\Sigma$ corresponding to $e$. See \cite[Subesection 3.2]{TZ1} for intrinsic definitions.

We say a stable graded $r$-spin graph is \textit{legal} if every boundary tail is legal, \textit{i.e.} ${\text{alt}}(t)=1~\forall t \in T^B$. We say a stable graded $r$-spin graph is \textit{level-$\h$} if every legal boundary tail has twist greater than or equal to $r-2-2\h$, and every illegal boundary tail has twist smaller than or equal to $2\h$.

 If $\Gamma$ is a connected graded $r$-spin graph, we denote by ${\mathcal M^*}_{\Gamma}^{1/r}\subseteq {\Mbarstar}_{0,  T^B, T^I\setminus H^{CB},}^{1/r}$ the (open) submoduli consisting of $r$-spin disks whose dual graph is exactly $\Gamma$, and by $\overline{{\mathcal M^*}}_{\Gamma}^{1/r}$ its closure. If $\Gamma$ is disconnected, we define $\Mbarstar_{\Gamma}^{1/r}$ as the product of the moduli spaces associated to its connected components. When there is no room for confusion (which is always the case, except in Subsection \ref{subsec decomp}), we omit the superscript $1/r$ in $\overline{{\mathcal M^*}}_{\Gamma}^{1/r}$ and $\Mbarstar_{\Gamma}^{1/r}$.
 If all boundary tails are legal, we also omit the superscript $*$ in the notation.

\subsection{The Witten bundle and the relative cotangent line bundles}
\label{subsec Witten bundle}

The \textit{Witten bundle} is the protagonist of the $r$-spin theory. Roughly speaking, let $\pi: \mathcal{C} \to \Mbarstar_{0,k,l}^{1/r}$ be the universal curve and $\mathcal{S} \to \mathcal{C}$ be the twisted universal spin bundle with the universal Serre-dual bundle 
\begin{equation}\label{eq universal serre dual bundle}
\mathcal{J}:= \mathcal{S}^{\vee} \otimes \omega_{\pi},
\end{equation}
then we define
\begin{equation}
\label{eq Witten bundle def}
{\mathcal{W}}:= (R^0\pi_*\mathcal{J})_+ = (R^1\pi_*\mathcal{S})^{\vee}_-,
\end{equation}
where the subscripts $+$ and $-$ denote invariant or anti-invariant sections under the universal involution $\widetilde{\phi}: \mathcal{J} \to \mathcal{J}$ or $\widetilde{\phi}: \mathcal{S} \to \mathcal{S}$.  To be more precise, defining~${\mathcal{W}}$ by \eqref{eq Witten bundle def} would require dealing with derived pushforward in the category of orbifold-with-corners. To avoid this technicality, we define ${\mathcal{W}}$ as the pullback of the analogous bundle from a subset of the closed moduli space $\Mbar_{0,k+2l}^{1/r}$; see \cite[Section 4.1]{BCT1}.

 On a non-empty component of the moduli space which parameterizes $r$-spin disks with internal twists $\{a_i\}$ and boundary twists $\{b_j\}$, the (real) rank of the Witten bundle is
\begin{equation}\label{eq rank of witten bundle}\frac{2 \sum_{i\in I} a_i + \sum_{j\in B} b_j - (r-2)}{r}.\end{equation}

In \cite[Definition 3.5]{TZ1}, a canonical relative orientation $o_\Gamma$ of the Witten bundle $\mathcal W_\Gamma \to \Mbar_\Gamma$ is defined for every connected legal stable graded $r$-spin graph $\Gamma$.

Other important line bundles in open $r$-spin theory are the \emph{relative cotangent line bundles} or \emph{tautological line bundles} at internal marked points. These line bundles have already been defined on the moduli space $\Mbar_{0,k,l}$ of stable marked disks (without spin structure) in \cite{PST14}, as the line bundles with fiber $T^*_{z_i}\Sigma$. Equivalently, these line bundles are the pullback of the usual relative cotangent line bundles ${\mathbb{L}}_i\to\Mbar_{0,k+2l}$ under the doubling map $\Mbar_{0,k,l} \to \Mbar_{0,k+2l}$ that sends $\Sigma$ to $C=\Sigma\sqcup_{\partial\Sigma}\overline{\Sigma}$.  The bundle $\mathbb{L}_i\to\Mbarstar_{0,k,l}^{1/r}$ is the pullback of this relative cotangent line bundle on $\Mbar_{0,k,l}$ under the morphism $\text{For}_{\text{spin}}$ that forgets the spin structure.  Note that $\mathbb{L}_i$ is a complex line bundle, hence it carries a canonical orientation.

\subsubsection{Decomposition properties of the Witten bundle}\label{subsec decomp}
In \cite{BCT1} the Witten bundle is proven to satisfy certain decomposition properties along nodes.  We state these properties here, further details and proofs can be found in \cite[Section 4.2]{BCT1}.

Given a genus-zero stable graded $r$-spin graph $\Gamma$, let $\widehat{\Gamma}$ be obtained by detaching either an edge or a contracted boundary tail of $\Gamma$.  In order to state the decomposition properties of the Witten bundle, we need the morphisms
\begin{equation}
\label{eq Witten decomp sequence}
\Mbarstar_{\widehat{\Gamma}}^{1/r} \xleftarrow{q} \Mbar_{\widehat{\Gamma}} \times_{\Mbar_{\Gamma}} \Mbarstar_{\Gamma}^{1/r} \xrightarrow{\mu} \Mbarstar_{\Gamma}^{1/r} \xrightarrow{i_{\Gamma}} \Mbarstar_{0,k,l}^{1/r},
\end{equation}
where $\Mbar_{\Gamma} \subseteq \Mbar_{0,k,l}$ is the moduli space of marked disks (without $r$-spin structure) corresponding to the dual graph $\Gamma$.  The morphism $q$ is defined by sending the $r$-spin structure $S$ to the $r$-spin structure $\widehat{S}$ defined by \eqref{eq normalize S}; it has degree one but is not an isomorphism because it does not induce an isomorphism between isotropy groups.  The morphism $\mu$ is the projection to the second factor in the fiber product; it is an isomorphism, but we distinguish between its domain and target because they have different universal objects.  The morphism $i_{\Gamma}$ is the inclusion.

We denote by ${\mathcal{W}}$ and $\widehat{{\mathcal{W}}}$ the Witten bundles on $\Mbarstar_{0,k,l}^{1/r}$ and $\Mbarstar_{\widehat{\Gamma}}^{1/r}$, the decomposition properties below show how these bundles are related under pullback via the morphisms \eqref{eq Witten decomp sequence}.

\begin{pr}{\cite[Proposition 4.7]{BCT1}}
\label{prop decomposition}
Let $\Gamma$ be a genus-zero stable graded $r$-spin graph with a single edge $e$, and let $\widehat{\Gamma}$ be the detaching of $\Gamma$ along $e$.  Then the Witten bundle decomposes as follows:
\begin{enumerate}
\item\label{it NS} If $e$ is a Neveu--Schwarz edge, then \begin{equation}\label{eq NSdecompses}
\mu^*i_{\Gamma}^*{\mathcal{W}} = q^*\widehat{{\mathcal{W}}}.
\end{equation}

\item\label{it decompose Ramond boundary edge} If $e$ is a Ramond boundary edge, then there is an exact sequence
\begin{equation}
\label{eq decompose}0 \to \mu^*i_{\Gamma}^*{\mathcal{W}} \to q^*\widehat{{\mathcal{W}}} \to \underline{\mathbb{R}_+} \to 0,
\end{equation}
where $\underline{\mathbb{R}_+}$ is a trivial real line bundle.

\item If $e$ is a Ramond internal edge connecting two closed vertices, write $q^*\widehat{\mathcal{W}} = \widehat{\mathcal{W}}_1 \boxplus \widehat{\mathcal{W}}_2$, where $\widehat{\mathcal{W}}_1$ is the Witten bundle on the component containing a contracted boundary tail or the anchor of $\Gamma$, and $\widehat{\mathcal{W}}_2$ is the Witten bundle on the other component.  Then there is an exact sequence
\begin{equation}
\label{eq decompose2}
0 \to \widehat{\mathcal{W}}_2 \to \mu^*i_{\Gamma}^*{\mathcal{W}} \to \widehat{\mathcal{W}}_1 \to 0.
\end{equation}
Furthermore, if $\widehat\Gamma'$ is defined to agree with $\widehat\Gamma$ except that the twist at each Ramond tail is $r-1$, and $q': \Mbar_{\widehat{\Gamma}} \times_{\Mbar_{\Gamma}} \Mbarstar_{\Gamma}^{1/r} \to \Mbarstar_{\widehat\Gamma'}^{1/r}$ is defined analogously to $q$, then there is an exact sequence
\begin{equation}
\label{eq decompose3}
0 \to \mu^*i_{\Gamma}^*{\mathcal{W}} \to (q')^*\widehat{{\mathcal{W}}}' \to {\underline{\mathbb{C}}}^{1/r} \to 0,
\end{equation}
where $\widehat{{\mathcal{W}}}'$ is the Witten bundle on $\Mbarstar_{\widehat\Gamma'}^{1/r}$ and ${\underline{\mathbb{C}}}^{1/r}$ is a line bundle whose $r$-th power is trivial.

\item If $e$ is a Ramond internal edge connecting an open vertex to a closed vertex, write $q^*\widehat{\mathcal{W}} = \widehat{\mathcal{W}}_1 \boxplus \widehat{\mathcal{W}}_2$, where $\widehat{\mathcal{W}}_1$ is the Witten bundle on the open component (defined via $\widehat{\mathcal{S}}|_{\mathcal{C}_1}$) and $\widehat{\mathcal{W}}_2$ is the Witten bundle on the closed component.  Then the exact sequences \eqref{eq decompose2} and \eqref{eq decompose3} both hold.
\end{enumerate}
Analogously, if $\Gamma$ has a single vertex, no edges, and a contracted boundary tail~$t$, and $\widehat{\Gamma}$ is the detaching of $\Gamma$ along~$t$, then there is a decomposition property:
\begin{enumerate}
\setcounter{enumi}{4}
\item\label{it decompose cb tail} If ${\mathcal{W}}$ and $\widehat{\mathcal{W}}$ denote the Witten bundles on $\Mbarstar_{0,k,l}^{1/r}$ and $\Mbarstar_{\widehat\Gamma}^{1/r}$, respectively, then the sequence \eqref{eq decompose} holds.
\end{enumerate}
\end{pr}

\begin{rmk}\label{rmk decompose NS boundary node}
If the edge $e$ is a Neveu--Schwarz boundary edge, then the map $q$ is an isomorphism, and in this case, the proposition implies that the Witten bundle pulls back under the gluing morphism $\Mbarstar_{\widehat\Gamma}^{1/r} \to \Mbarstar_{0,k,l}^{1/r}$.

\end{rmk}

\subsubsection{Coherent sections and the assembling operator}\label{subsec coherent} 

   With the same notation as in Proposition \ref{prop decomposition}, let $\Gamma$ be a connected graded $r$-spin graph and $e$ is a edge of $\Gamma$, let $\Mbar_{\widehat{\Gamma}_1}$ and $\Mbar_{\widehat{\Gamma}_2}$ be the two components of $\widehat\Gamma:={\text{detach} }_e{\Gamma}$. We write $q^*\widehat{\mathcal{W}} = \widehat{\mathcal{W}}_1 \boxplus \widehat{\mathcal{W}}_2$, where $\widehat{\mathcal{W}}_1$ and $\widehat{\mathcal{W}}_2$ are the (pullback of) Witten bundles on $\Mbar_{\widehat{\Gamma}_1}$ and $\Mbar_{\widehat{\Gamma}_2}$. 
   Given sections $s_1$ and $s_2$ of $\widehat{\mathcal{W}}_1$ and $\widehat{\mathcal{W}}_2$, we want to construct a section of $\mathcal W \to \Mbarstar_\Gamma$. In the case where $e$ is an NS boundary edge, according to Remark \ref{rmk decompose NS boundary node}, we can be glue $s_1,s_2$ 
 to a section of $\mathcal W \to \Mbarstar_\Gamma$. However, when $e$ is an internal edge, we cannot glue $s_1$ and $s_2$ directly. In the case of an internal NS edge, this is because the automorphism groups of ${\mathcal{W}}$ and of the direct sum are not the same. Ramond internal edges introduce a more fundamental problem, since ${\mathcal{W}}$ does not decompose naturally as a direct sum of the $\widehat{\mathcal{W}}_1,\widehat{\mathcal{W}}_2,$ by Proposition \ref{prop decomposition} above. To construct a section of $\mathcal W \to \Mbarstar_\Gamma$, we need the \textit{assembling operator} introduced in \cite{BCT2}, which is based on the following technical notion of \textit{coherent multisections}. 
  
  Let $ \Gamma_c$ be a connected stable $r$-spin dual graph with an anchor $t\in T^I(\Gamma_c)$, \textit{i.e.} $ \Gamma_c$ has no open vertices or contracted boundary tails.
  \begin{itemize}
      \item If $\tw(t)=-1$, let $\mathcal J\to \Mbar_{ \Gamma_c}$ be the universal Serre-dual bundle as in \eqref{eq universal serre dual bundle}, we set $\mathcal{J}' := \mathcal{J} \otimes \mathcal O \left( r\Delta_{z_t}\right)$, where $\Delta_{z_t}$ is the divisor in the universal curve corresponding to the anchor $t$.
    We define an orbifold bundle ${\overline{\mathcal{R}}}_{\Gamma_c}$ on $\Mbar_{ \Gamma_c}$ by
    ${\overline{\mathcal{R}}}_{ \Gamma_c}:= \sigma_{z_t}^*\mathcal{J}'$,
     where $\sigma_{z_t}$ is the section corresponding to the  anchor $t$ in the universal curve. We also abuse notations to denote by ${\overline{\mathcal{R}}}_{ \Gamma_c}$ the total space of this bundle. We write
$\wp: {\overline{\mathcal{R}}}_{ \Gamma_c} \to \Mbar_{ \Gamma_c}$ for the projection.     We denote by ${\mathcal{W}}$ the  bundle $R^0\pi_*\mathcal{J}'$ on ${\overline{\mathcal{R}}}_{\Gamma_c}$.
\item If $\tw(t)\ne 1$, we set ${\overline{\mathcal{R}}}_{ \Gamma_c}:=\oCM_{\Gamma_c}$ and set $\mathcal W\to {\overline{\mathcal{R}}}_{ \Gamma_c}$ the same as $\mathcal W\to \oCM_{\Gamma_c}$.
  \end{itemize}

\begin{definition}{\cite[Definition 4.2]{BCT2}}
\label{def coherent}
 Let $ \Gamma_c$ be a connected stable $r$-spin dual graph with an anchor $t\in T^I(\Gamma_c)$, and let $s$ be a section of $\mathcal{W}$ over a subset $U\subset{\overline{\mathcal{R}}}_{\Gamma_c}$. We say $s$ is \textit{coherent} if either the twist $\tw(t)\ne -1$, or, in the case where $\tw(t)= -1$, for any point $\zeta=(C, u_t) \in U$  corresponding to a graded $r$-spin disk $C$ and an element $u_t$ in the fiber over $z_t\in C$ of the line bundle $J':=J\otimes \mathcal O(r[z_t])\to C$, the element $s(\zeta) \in H^0(J')$ satisfies
$$\ev_{z_t}s(\zeta) = u_t.$$ A coherent multisection $s$ is defined as a multisection  (see \cite[Appendix A]{BCT2}) whose local branches are coherent. We say that a multisection of $\mathcal W^{\oplus}$ is coherent if it can be written as a direct sum of coherent multisections of $\mathcal W$.
\end{definition}

Let $s$ be a coherent multisection of $\mathcal{W} \rightarrow {\overline{\mathcal{R}}}_{\Gamma_c}$. Note that, in the case $\tw(t)=-1$, for any $\zeta\in\Mbar_{\Gamma_c} \hookrightarrow {\overline{\mathcal{R}}}_{\Gamma_c}$, the evaluation $\ev_{z_t}(s(\zeta))$ is equal to zero, thus $s(\zeta)$ is induced by a multisection of $J$.  In other words, the restriction of a coherent multisection $s$ to $\Mbar_{\Gamma_c}$ is canonically identified with a multisection of ${\mathcal{W}} \rightarrow \Mbar_{\Gamma_c}$; we denote this induced multisection by $\overline{s}$.  In case $\tw(t)\ne -1$, we write $\overline{s} = s$.
\begin{lem}
    Let $\Gamma_c$ be a connected graded $r$-spin graph with an anchor $t$, for any multisection $\hat s$ of $\mathcal W\to \Mbar_{\Gamma_c}$,
    there exist a coherent multisection $s$ of $\mathcal W\to {\overline{\mathcal{R}}}_{\Gamma_c}$ satisfying $\overline{s}=\hat s$.
\end{lem}
\begin{proof}

Let $\hat s$ be a section of ${\mathcal{W}}\to\Mbar_{\Gamma_c}$, we can identify $\wp^*\hat s$ as a section of ${\mathcal{W}} \rightarrow {\overline{\mathcal{R}}}_{\Gamma_c}$. Note that, under this identification, the  evaluation $\ev_{z_t}\wp^*\hat s\left(\zeta\right)$ always vanishes, hence $\wp^*\hat s$ is not coherent if $\tw(t)=-1$. On the other hand, if $s_0$ is a coherent multisection of ${\mathcal{W}} \rightarrow {\overline{\mathcal{R}}}_{\Gamma_c}$, then $s_0+\wp^*\hat s$ is also a coherent multisection. Moreover we have $\overline{s_0+\wp^*\hat s}=\overline{s_0}+\hat s$.

For any connected graded $r$-spin graph $\Gamma_c$ with an anchor $t$, the bundle ${\mathcal{W}} \rightarrow {\overline{\mathcal{R}}}_{\Gamma_c}$ admits (at least) one coherent multisection $s_{\Gamma_c}$. Actually if $\tw(t)\ne -1$, any multisection is coherent; 
if $\tw(t)=-1$, a coherent  multisection of ${\mathcal{W}} \rightarrow {\overline{\mathcal{R}}}_{\Gamma_c}$ is constructed in \cite[Section 4.1.2]{BCT2}. 
As a consequence, for any multisection $\hat s$ of ${\mathcal{W}}\to\Mbar_{\Gamma_c}$, we can define a coherent multisection
$s:=s_{\Gamma_c}+\wp^*(\hat s-\overline{s_{\Gamma_c}})$ and it satisfies $\overline{s}=\hat s$ as required. 
\end{proof}
Given a connected graded $r$-spin graph $\Gamma$ with an open vertex or a contracted boundary, let $e$ be an internal edge $e\in E^I(\Gamma)$. We denote by $\Gamma_o$ and $\Gamma_c$ the connected components of ${\text{detach} }_e \Gamma$, where $\Gamma_o$ has an open vertex or a contracted boundary (hence $\Gamma_c$ has an anchor). Using the assembling operator $\Ass$ defined in \cite[Section 4.1.3]{BCT2}, we can glue a multisection $s_o$ of $\mathcal W^{\oplus m}\to \Mbar_{\Gamma_o}$ and a multisection $s_c$ of $\mathcal W^{\oplus m}\to {\overline{\mathcal{R}}}_{\Gamma_c}$ to obtain a multisection $\Ass_{\Gamma,e}(s_c\boxplus s_o)$ of $\mathcal W^{\oplus m}\to \Mbar_{\Gamma}$. 
We refer the reader to \cite[Section 4.1.3]{BCT2} for further details and exact definitions.

\subsection{The point insertion technique}\label{subsec PI}
Just like the closed $r$-spin theory considers an intersection theory over the moduli spaces of $r$-spin curves,  the genus-zero open $r$-spin theory considers the intersection theory over the moduli of the $r$-spin disks $\Mbar^{1/r}_{0,B,I}$. However, since $\Mbar^{1/r}_{0,B,I}$ is an orbifold with corners, the intersection theory is not well-defined. The grading structure allows us to deal with certain types of boundaries using a notion of positivity (see Section \ref{sec sections}). The procedure of \textit{point insertion}, which also relies on the grading, was developed in \cite{TZ1} in order to treat the remaining boundaries: we can glue another moduli spaces to $\Mbar^{1/r}_{0,B,I}$ along those strata and by that cancel them. 
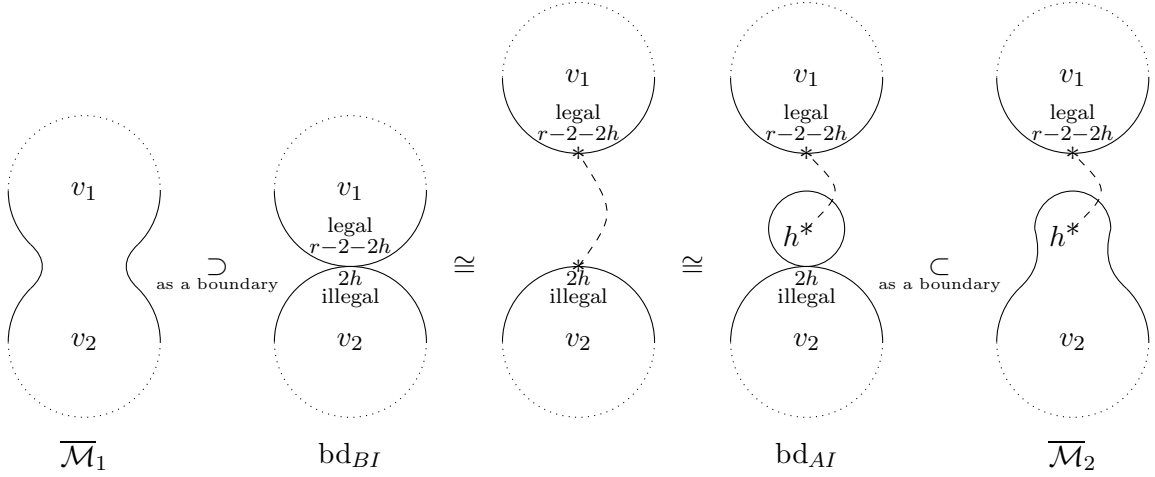
\begin{figure}[h]
    \centering
    \begin{tikzpicture}
        \draw (0,0) arc (0:180:1);
        \draw[dotted] (0,0) arc (0:-180:1);
        \node at (-1,1) {$*$};
        \node at (-1,0.7) {$\substack{2h\\ \text{illegal}}$};
        \node at (-1,0) {$v_2$};
                
        \draw[dotted] (0,3.5) arc (0:180:1);
        \draw (0,3.5) arc (0:-180:1);
        \node at (-1,2.5) {$*$};
        \node at (-1,2.9) {$\substack{\text{legal}\\r-2-2h}$};
        \node at (-1,3.5) {$v_1$};

        \draw[dashed] (-1,1) .. controls (-0.5,1.7) .. (-1,2.5);

        \draw (-3,0) arc (0:180:1);
        \draw[dotted] (-3,0) arc (0:-180:1);
       
        \node at (-4,0.7) {$\substack{2h\\ \text{illegal}}$};
        \node at (-4,0) {$v_2$};
        \node at (-4,-1.5) {$\text{bd}_{BI}$};

        \draw[dotted] (-3,2) arc (0:180:1);
        \draw (-3,2) arc (0:-180:1);
       
        \node at (-4,1.4) {$\substack{\text{legal}\\r-2-2h}$};
        \node at (-4,2) {$v_1$};

        \node at (-2.5,1) {$\cong$};

        \draw (3,0) arc (0:180:1);
        \draw[dotted] (3,0) arc (0:-180:1);
        \node at (2,1.5) {$*$};
        \node at (1.8,1.4) {$h$};
        \node at (2,0.7) {$\substack{2h\\ \text{illegal}}$};
        \node at (2,0) {$v_2$};
        \node at (2,-1.5) {$\text{bd}_{AI}$};
                
        \draw[dotted] (3,3.5) arc (0:180:1);
        \draw (3,3.5) arc (0:-180:1);
        \node at (2,2.5) {$*$};
        \node at (2,2.9) {$\substack{\text{legal}\\r-2-2h}$};
        \node at (2,3.5) {$v_1$};
        
        \draw (2.5,1.5) arc (0:360:0.5);
        \draw[dashed] (2,1.5) .. controls (2.5,2) .. (2,2.5);

        \node at (0.5,1) {$\cong$};

         \draw (-6.5,0) arc (0:45:1);
         \draw (-8.5,0) arc (180:135:1);
        \draw[dotted] (-6.5,0) arc (0:-180:1);
        
        \node at (-7.5,0) {$v_2$};
        \node at (-7.5,-1.5) {$\Mbar_1$};

        \draw[dotted] (-6.5,2) arc (0:180:1);
        \draw (-6.5,2) arc (0:-45:1);
        \draw (-8.5,2) arc (-180:-135:1);
        
        \node at (-7.5,2) {$v_1$};
        \draw (-6.793,1.293) .. controls (-7,1.1) and (-7,0.9).. (-6.793,0.707);
        \draw (-8.207,1.293) .. controls (-8,1.1) and (-8,0.9).. (-8.207,0.707);

        \node at (-5.75,1) {$\supset$};
        \node at (-5.75,0.75) {\tiny{as a boundary}};

        \draw (6.5,0) arc (0:45:1);
        \draw (4.5,0) arc (180:135:1);
        \draw[dotted] (6.5,0) arc (0:-180:1);
        \node at (5.5,1.5) {$*$};
        \node at (5.3,1.4) {$h$};
       
        \node at (5.5,0) {$v_2$};
        \node at (5.5,-1.5) {$\Mbar_2$};
                
        \draw[dotted] (6.5,3.5) arc (0:180:1);
        \draw (6.5,3.5) arc (0:-180:1);
        \node at (5.5,2.5) {$*$};
        \node at (5.5,2.9) {$\substack{\text{legal}\\r-2-2h}$};
        \node at (5.5,3.5) {$v_1$};
        
        \draw (6,1.5) arc (0:180:0.5);
        
        \draw[dashed] (5.5,1.5) .. controls (6,2) .. (5.5,2.5);
        
        \draw (6.207,0.707) .. controls (6,0.9) and (5.9,1.1).. (6,1.5);
        \draw (4.793,0.707) .. controls (5,0.9) and (5.1,1.1).. (5,1.5);

        \node at (3.75,1) {$\subset$};
        \node at (3.75,0.75) {\tiny{as a boundary}};
        
    \end{tikzpicture}
    \caption{In point insertion procedure we glue $\Mbar_1$ and $\Mbar_2$ together along their isomorphic boundaries $\text{bd}_{BI}$ and $\text{bd}_{AI}$. The first isomorphism follows from the decomposition property for boundary NS nodes; the second isomorphism holds because the moduli $\Mbar^{1/r}_{0.\{r-2-2h\},\{h\}}$ (represented by the smallest bubble in the figure) is a single point. The new markings coming from the point insertion procedure are represented by $*$; the dashed line between the new markings indicates that they come from the same node.} 
    \label{fig point insertion demonstration}
\end{figure}

More precisely, as shown in Figure \ref{fig point insertion demonstration}, let $\Mbar_1$ be a moduli of $r$-spin disks, and $\text{bd}_{BI}\subset \Mbar_1$ be a boundary corresponding to an NS boundary node with twist $2h$ at the illegal half-node. We can glue to $\Mbar_1$, along the boundary $\text{bd}_{BI}$, another moduli $\Mbar_2$ which has a boundary $\text{bd}_{AI}$ diffeomorphic to $\text{bd}_{BI}$. Note that $\Mbar_2$ is a moduli of two disconnected $r$-spin disks, obtained by first detaching the boundary node, then ``inserting" the illegal twist-$2h$ boundary marked point to the interior as a twist-$h$ internal marked point. The boundary strata along which we glue the moduli spaces are called \emph{spurious boundaries}.

By applying this procedure repeatedly we get a glued moduli (see \cite[Section 4.6]{TZ1}) with only real (non-spurious) boundaries which can be dealt with positivity. A point in the pre-glued space represents a disjoint union of graded $r$-spin disks, together with the combinatorial data of dashed lines connecting each pair of boundary marking and internal marking which appear together in the point insertion procedure. A point in the glued space represents an equivalence class of such objects under the equivalence relation induced by point insertion procedure.

We will define $\lfloor \frac{r}{2}\rfloor$ different point insertion theories indexed by an integer $\h\in \{0,1,\dots,\lfloor \frac{r-2}{2}\rfloor\}$. For a chosen $\h$, we do point insertion at an NS boundary node $n$ if and only if the twist of the illegal half-node of $n$ is less than or equal to $2\h$.

\subsubsection{$(r,\h)$-disks, $(r,\h)$-graphs and moduli}
We now more formally describe the objects of the pre-glued moduli space.
\begin{dfn}{\cite[Definition 4.3]{TZ1}}\label{dfn rh disk}
    An $(r,\h)$-disk is a collection of legal connected level-$\h$ stable graded $r$-spin disks (the components) with non-empty $\text{Fix}(\phi)$, together with
    \begin{enumerate}
        \item a bijection (denoted by dashed lines) between a subset $I^p$ of the tails and a subset $B^p$ of the tails,  where the twist $a$ and $b$ for paired internal tail and boundary tail satisfies $a+2b=r-2$ and $0\le a\le \h$;
        \item markings on the set of unpaired internal tails $I^{up}$ and boundary tails $B^{up}$, \textit{i.e.} identifications $I^{up}=\{1,2,\dots,\lvert I^{up}\rvert\}$ and $B^{up}=\{1,2,\dots,\lvert B^{up}\rvert\}$.
    \end{enumerate}
    We require that, in the collection of graphs, there exists no genus-zero stable graded $r$-spin disk with only one internal tail in $I^p$, one boundary tail in $B^p$ and no tails in $I^{up}$ or $B^{up}$.
    We also require that the graph $\hat{\mathbf{G}}$, whose vertices are connected disks in the collection and there is an edge between two vertices if they contain a pair of points corresponding to a dashed line, is a connected genus-zero graph.
    
\end{dfn}

The graph $\hat{\mathbf{G}}$ mentioned above characterizes the topological type of an $(r,\h)$-disk. Since each graded $r$-spin disk in the collection of an $(r,\h)$-disk is associated with a graded $r$-spin graph, by assigning each vertex of $\hat{\mathbf{G}}$ the corresponding graded $r$-spin graph, and specifying the corresponding pair of tails for each edge of $\hat{\mathbf{G}}$, we obtain the following combinatorial object, which can be viewed as a refined version of $\hat{\mathbf{G}}$.
\begin{dfn}{\cite[Definition 4.7]{TZ1}}\label{def rh graphs}
For $0\le \h \le \lfloor\frac{r}{2}\rfloor-1$, a \textit{genus-zero $(r,\h)$-graph}  $\mathbf{G}$ consists of 
\begin{itemize} 
    \item a set $V(\mathbf{G})$ of connected genus-zero legal level-$\h$ stable graded $r$-spin graphs with at least one open vertex or contracted boundary tail;
    \item two partitions of sets
    $$
    \bigsqcup_{\Gamma\in V(\mathbf{G})}\left(T^I(\Gamma)\backslash H^{CB}(\Gamma)\right)=I(\mathbf{G})\sqcup I'(\mathbf{G})
    $$
    and
    $$
    \bigsqcup_{\Gamma\in V(\mathbf{G})}T^B(\Gamma)=B(\mathbf{G})\sqcup B'(\mathbf{G});
    $$
    \item a set of edges (the \textit{dashed lines}) $$E(\mathbf{G})\subseteq \{(a,b)\colon a\in I'(\mathbf{G}),b\in B'(\mathbf{G}), 2a+b=r-2\}$$ which induces an one-to-one correspondence $\delta$ between $I'(\mathbf{G})$ and $B'(\mathbf{G})$;
    \item a labelling of the set $I(\mathbf{G})$ by $\{1,2,\dots,l(\mathbf{G}):=\lvert I(\mathbf{G})\rvert\}$ and a labelling of the set $B(\mathbf{G})$ by $\{1,2,\dots,k(\mathbf{G}):=\lvert B(\mathbf{G})\rvert\}$.
\end{itemize}
    We require that 
    \begin{enumerate}
        \item there exists no $\Gamma\in V(\mathbf{G})$ satisfying 
    $
    H^I(\Gamma)\subseteq I'(\mathbf{G}), H^B(\Gamma)\subseteq B'(\mathbf{G})$ 
    and $
         \lvert H^I(\Gamma)\rvert= \lvert H^B(\Gamma)\rvert=1;
    $
    \end{enumerate}
    We define an auxiliary graph (in the normal sense) $\hat{\mathbf{G}}$ in the following way: the set of vertices of $\hat{\mathbf{G}}$ is $V(\mathbf{G})$, the set of edges of $\hat{\mathbf{G}}$ is $E(\mathbf{G})$; an element $(a,b)\in E(\mathbf{G})$ corresponds to an edge between the vertices $\Gamma_a$ and $\Gamma_b$, where  $a\in T^I(\Gamma_a)$ and $b\in T^B(\Gamma_b)$. We also require that
\begin{enumerate}[resume*]
    \item  the graph $\hat{\mathbf{G}}$ is a connected and genus-zero.
\end{enumerate}

\end{dfn}

In this paper, since we exclusively focus on the genus-zero case, when we refer to an $(r,\h)$-graph, we mean a genus-zero $(r,\h)$-graph.

\begin{dfn}
Let $\mathbf{G}$ be an $(r,\h)$-graph. Let $e$ be an edge or a contracted boundary tail of some $\Gamma\in V(\mathbf{G})$. Since $T^I(\Gamma)\backslash H^{CB}(\Gamma)=T^I(d_e \Gamma)\backslash H^{CB}(d_e \Gamma)$ and $T^B(\Gamma)=T^B(d_e \Gamma)$, we define the \textit{smoothing} of $\mathbf{G}$ along $e$ to be the $(r,\h)$-graph $d_e \mathbf{G}$ obtained by replacing $\Gamma$ with $d_e \Gamma$.

We say $\mathbf{G}$ is smooth if all $\Gamma\in V(\mathbf{G})$ are smooth stable graded $r$-spin graphs.
We denote by $\GPI^{r,\h}_0$ the set of all genus-zero $(r,\h)$-graphs, by $\GPI^{r,\h}_{0,B,I}$ the set of all genus-zero $(r,\h)$-graphs $\mathbf{G}$ satisfying $I(\mathbf{G})=I$ and $B(\mathbf{G})=B$, by $\sGPI^{r,\h}_{0,B,I}$ the set 
$$ 
\sGPI^{r,\h}_{0,B,I}:=\{\mathbf{G}\in \GPI^{r,\h}_{0,B,I}\colon \mathbf{G}\text{ smooth}\}.
$$
\end{dfn}

\begin{dfn}
    An isomorphism between two $(r,\h)$-graphs $\mathbf{G_1}$ and $\mathbf{G_2}$ consists of a collection of isomorphism of stable graded $r$-spin graphs between elements of $V(\mathbf{G_1})$ and $V(\mathbf{G_2})$, which induces a bijection between $V(\mathbf{G_1})$ and $V(\mathbf{G_2})$, and preserves the partitions, dashed lines, and labellings.
\end{dfn}

For each $\mathbf{G}\in \GPI^{r,\h}_0$, let $\operatorname{Aut} \mathbf{G}$ be the group of automorphisms of $\mathbf{G}\in \GPI^{r,\h}_0$, then there is a natural action of $\operatorname{Aut} \mathbf{G}$ over the product $\prod_{\Gamma\in V(\mathbf{G})} \Mbar_\Gamma$. We define 
$$
\Mbar_\mathbf{G}:=\left(\prod_{\Gamma\in V(\mathbf{G})} \Mbar_\Gamma\right)\bigg\slash \operatorname{Aut} \mathbf{G}.
$$
Let $\mathcal W_\Gamma$ be the Witten bundle over $\Mbar_\Gamma$, we define the Witten bundle $\mathcal W_\mathbf{G}$ over $\Mbar_\mathbf{G}$ to be 
$$
 \mathcal W_\mathbf{G}:= \left(\bboxplus_{\Gamma\in V(\mathbf{G})}\mathcal W_\Gamma\right)\bigg\slash \operatorname{Aut} \mathbf{G}.
$$
\begin{rmk}\label{rmk aut trivial}
    In this paper, when we only consider genus-zero  $(r,\h)$-graphs $\mathbf{G}$, the automorphism groups $\operatorname{Aut}\mathbf{G}$ are always trivial. In this case we denote by $\pi_{\mathbf{G},\Gamma}$ the projection maps 
$$
\pi_{\mathbf{G},\Gamma}\colon \Mbar_{\mathbf{G}}\to \Mbar_\Gamma. 
$$
\end{rmk}
Let $o_\Gamma$ be the canonical relative orientation of $\mathcal W_\Gamma$ over $\Mbar_\Gamma$, we define the canonical relative orientation $o_\mathbf{G}$ of $\mathcal W_\mathbf{G}$ over $\Mbar_\mathbf{G}$ by
\begin{equation}\label{eq orientation point insertion}
    o_\mathbf{G}:=(-1)^{\lvert E(\mathbf{G})\rvert}\bigwedge_{\Gamma\in V(\mathbf{G})}o_\Gamma.
\end{equation}
Observe that $o_\mathbf{G}$ is independent of the order of the wedge product, since for each $\Gamma\in V(\mathbf{G})$ we have 
$$
\dim \Mbar_\Gamma \equiv \operatorname{rank} \mathcal W_\Gamma \mod 2.
$$
\begin{dfn} Given an integer $\h\in \{0,1,\dots,\lfloor\frac{r-2}{2}\rfloor\}$, a finite set $I$ of internal markings with twist in $\{0,1,\dots,r-1\}$ and a finite set $B$ of boundary markings with twist in $\{r-2-2\h,r-2\h,\dots,r-4,r-2\}$, we define the moduli space $\Mbar^{\frac{1}{r},\h}_{0,B,I}$ of $(r,\h)$-disks with markings $B,I$ to be
\begin{equation}\label{eq def moduli point insertion}
    \Mbar^{\frac{1}{r},\h}_{0,B,I}:=\bigsqcup_{\mathbf{G}\in \sGPI^{r,\h}_{0,B,I}}\Mbar_\mathbf{G}.
\end{equation}

The Witten bundles with relative orientations over the connected components $\Mbar_\mathbf{G}$ of $\Mbar^{\frac{1}{r},\h}_{0,B,I}$ induce the Witten bundle $\mathcal W^{\frac{1}{r},\h}_{0,B,I}$ over $\Mbar^{\frac{1}{r},\h}_{0,B,I}$ with relative orientation.
\end{dfn}

\subsubsection{Boundary strata and point insertion}
For an $(r,\h)$-graph $\bm{G}$, we write
$$
E(\bm{G}):=\bigsqcup_{\Gamma\in V(\bm{G})} E(\Gamma)
$$
and 
$$
H^{CB}(\bm{G}):= \bigsqcup_{\Gamma\in V(\bm{G})}  H^{CB}(\Gamma).
$$
For a set 
$S\subseteq E(\bm{G})\sqcup H^{CB}(\bm{G})$, 
we can perform a sequence of smoothings (in any order, since they lead to the same result) and obtain an $(r,\h)$-graph $d_S \bm{G}$.  We set
\begin{align*}
&\partial^!\bm{G} = \{\bm{H} \; | \; \bm{G} = d_S\bm{H} \text{ for some } S\},\\
&\partial \bm{G} = \partial^!\bm{G} \setminus \{\bm{G}\},\\
&\partial^B \bm{G} = \{\bm{H} \in \partial \bm{G} \; | \; E^B(\bm{H}) \cup H^{CB}(\bm{H}) \neq \emptyset\}.
\end{align*}
For a graded $r$-spin graph $\Gamma$, we can define $\partial^!\Gamma, \partial \Gamma$ and $\partial^B \Gamma$ in a similar way.

For an $(r,\h)$-graph $\mathbf{G}$, a boundary stratum of $\Mbar_\mathbf{G}$ corresponds to a graph in $\partial^!\mathbf{G}$, or more precisely, a choice of $\Delta_i\in \partial^!\Gamma_{i}$ for each $\Gamma_i\in V(\mathbf{G})$. In particular, a codimension-1 boundary of $\Mbar_\mathbf{G}$ for smooth $\mathbf{G}$ is determined by a choice of $\Gamma\in V(\mathbf{G})$ and a graph $\Delta \in \partial \Gamma$, where $\Delta$ has either one contracted boundary tail and no edges, or exactly one edge which is a boundary edge. There are five different types of codimension-1 boundaries of $\Mbar_\mathbf{G}$ depending on the type of the (half-)edge of $\Delta$ (or equivalently, the corresponding node of a curve $C\in \mathcal M_\Delta$): 
\begin{enumerate}
    \item[{CB}] contracted boundary tails;
    \item[{R}] Ramond boundary edges;
    \item[{NS+}] NS boundary edges whose twist on the illegal side is greater than $2\h$;
    \item[AI] NS boundary edges whose twist on the illegal side is less than or equal to $2\h$, and the vertex containing the legal half-node only contains this half-edge and an internal tail $a\in T^I(\Delta)\cap I'(\mathbf{G})$;
    \item[BI] the remaining NS boundary edges whose twist on the illegal side is less than or equal to $2\h$.
\end{enumerate}
Therefore, the codimension-1 boundary of $\Mbar^{\frac{1}{r},\h}_{0,B,I}$ is a union of five different types of boundaries. 
\begin{rmk}
    The abbreviation ``BI" stands for ``before-insertion", while the abbreviation ``AI" stands for ``after-insertion". 
\end{rmk}

We claim that there is a one-to-one correspondence between the type-AI boundaries and the type-BI boundaries.

    \begin{thm}{\cite[Theorem 4.12]{TZ1}}\label{thm  PI boundaries paried}
 For fixed $I$ and $B$, there is a one-to-one correspondence $\PI$ between the type-BI boundaries and the type-AI boundaries of $\Mbar^{\frac{1}{r},\h}_{0,B,I}$. Two boundaries paired by the correspondence $\PI$ are canonically diffeomorphic, and this diffeomorphism can be lifted to the Witten bundles and the relative cotangent line bundles restricted to them. Moreover, the canonical relative orientations of the Witten bundles on the paired (spurious) boundaries induced by the canonical relative orientations are opposite to each other. 
 \end{thm}

 \begin{figure}[h]
         \centering

         \begin{subfigure}{.45\textwidth}
  \centering
\begin{tikzpicture}[scale=0.45]
\draw (0,-0.5) circle (1);
\draw (-1,-0.5) arc (180:360:1 and 0.333);
\draw[dashed](1,-0.5) arc (0:180:1 and 0.333);

\draw (-1.5,-3) arc (180:360:1.5 and 0.5);
\draw[dashed](1.5,-3) arc (0:180:1.5 and 0.5);
\draw(1.5,-3) arc (0:180:1.5);

\draw (1.5,-3) arc (180:360:1 and 0.333);
\draw[dashed](3.5,-3) arc (0:180:1 and 0.333);
\draw(3.5,-3) arc (0:180:1);

\node at (1.5,-3) [circle,fill,inner sep=1pt]{};

\draw (4.5,-3) arc (180:360:2 and 0.667);
\draw[dashed](8.5,-3) arc (0:180:2 and 0.667);
\draw(8.5,-3) arc (0:180:2);

\node at (3.5,-3) [circle,fill,inner sep=1pt]{};
\node at (6.5,-2){*};
\draw[dashed] (3.5,-3) .. controls (5.5,-1.5) ..(6.5,-2);

\node at (1.5,-4.3){$C_1^{BI}$};

\end{tikzpicture} 
\end{subfigure}
\begin{subfigure}{.45\textwidth}
  \centering
\begin{tikzpicture}[scale=0.45]
\draw (0,-0.5) circle (1);
\draw (-1,-0.5) arc (180:360:1 and 0.333);
\draw[dashed](1,-0.5) arc (0:180:1 and 0.333);

\draw (-1.5,-3) arc (180:360:1.5 and 0.5);
\draw[dashed](1.5,-3) arc (0:180:1.5 and 0.5);
\draw(1.5,-3) arc (0:180:1.5);

\draw (1.5,-3) arc (180:360:0.5 and 0.167);
\draw[dashed](2.5,-3) arc (0:180:0.5 and 0.167);
\draw(2.5,-3) arc (0:180:0.5);

\fill[color = gray, opacity = 0.5] (1.5,-3) arc (180:360:0.5 and 0.167) arc (0:180:0.5);

\draw (4,-3) arc (180:360:1 and 0.333);
\draw[dashed](6,-3) arc (0:180:1 and 0.333);
\draw(6,-3) arc (0:180:1);

\draw (7,-3) arc (180:360:2 and 0.667);
\draw[dashed](11,-3) arc (0:180:2 and 0.667);
\draw(11,-3) arc (0:180:2);

\node at (6,-3) [circle,fill,inner sep=1pt]{};
\node at (9,-2){*};
\draw[dashed] (6,-3) .. controls (8,-1.5) ..(9,-2);

\node at (4,-3) [circle,fill,inner sep=1pt]{};

\node at (2,-2.8){*};

\draw[dashed] (4,-3) .. controls (3,-2) ..(2,-2.8);

\node at (1,-4.3){$ D_1^{AI}$};
\node at (5,-4.3){$ D_2$};

\end{tikzpicture} 
\end{subfigure}

        \caption{An example of two $(r,\h)$-disks lying on two boundaries $\text{bd}_{BI}$ and $\text{bd}_{AI}$ paired by $PI$. The component $C_1^{BI}$ of the $(r,\h)$-disk on the left has a type-BI node, while the component $D_1^{AI}$ of the $(r,\h)$-disk on the left has a type-AI node. The shaded irreducible component only contains a legal half-node (corresponding to a type-AI node) and an internal tail in $I^p$. }
        \label{fig rh surface}
    \end{figure}
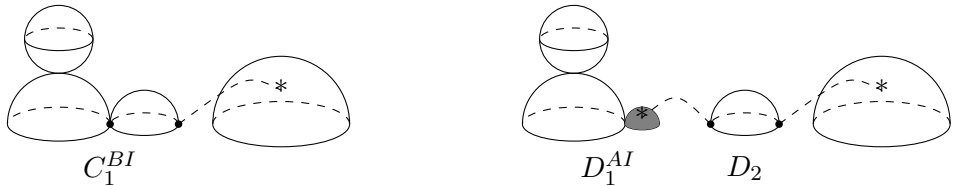

  Let $\sim_{PI}$ be the equivalent relation induced by the correspondence $\PI$ on the boundaries of $\Mbar^{\frac{1}{r},\h}_{0,B,I}$. Theorem \ref{thm  PI boundaries paried} shows that we can glue $\Mbar^{\frac{1}{r},\h}_{0,B,I}$ along the paired boundaries and obtain a piecewise smooth glued moduli space
    $$
    \widetilde{\mathcal M}^{\frac{1}{r},\h}_{0,B,I}:=\Mbar^{\frac{1}{r},\h}_{0,B,I}\big/\sim_{PI}.
    $$
    The objects parameterized by $\widetilde{\mathcal M}^{\frac{1}{r},\h}_{0,B,I}$  are called \textit{reduced $(r,\h)$-disks} in \cite{TZ1}, they are equivalence classes of $(r,\h)$-disks under the relation induced by $\sim_{PI}$. Note that $\widetilde{\mathcal M}^{\frac{1}{r},\h}_{0,B,I}$ has only boundaries of type CB, R, and NS+.
    
    The Witten bundles and the relative cotangent line bundles over the different connected components of $\Mbar^{\frac{1}{r},\h}_{0,B,I}$  can also be glued along the same boundaries into a glued Witten bundle $\widetilde{\mathcal W}\to\widetilde{\mathcal M}^{\frac{1}{r},\h}_{0,B,I}$ and glued relative cotangent line bundles $\widetilde{\mathbb L}_i\to\widetilde{\mathcal M}^{\frac{1}{r},\h}_{0,B,I}$.
\begin{rmk}
In the case $r=2,\h=0$ and only NS insertions the Witten bundle is a trivial zero rank bundle. In this case the idea of gluing different moduli spaces to obtain an orbifold without boundary is due to Jake Solomon and the first named author \cite{ST_unpublished}. \cite{ST_unpublished} defined the relative cotangent lines slightly differently than here (this different definition appears in Remark \ref{rmk:geometric_comparison}) and showed that in $g=0$ the resulting theory produces the same intersection numbers as those defined in \cite{PST14}.
\end{rmk}
     By Theorem \ref{thm  PI boundaries paried}, the fact that the glued relative cotangent line bundles over $\widetilde{\mathcal M}^{\frac{1}{r},\h}_{0,B,I}$ carry canonical complex orientations, and the fact that direct sums of even number of copies of vector bundles also carry canonical orientations, we obtain the following theorem:
    \begin{thm}\cite[Theorem 4.13]{TZ1}
      All bundles of the form
    \[(\widetilde{{\mathcal{W}}})^{2d+1}\oplus\bigoplus_{i=1}^l\widetilde{\mathbb{L}}_i^{\oplus d_i}\to \widetilde{\mathcal M}^{\frac{1}{r},\h}_{0,B,I}\]are canonically relatively oriented.   
    \end{thm}

\section{Canonical boundary conditions and correlators}\label{sec sections}
In this section we define the canonical multisections, and use them to define open $r$-spin and certain open FJRW correlators.

We refer to a boundary edge $e,$ or the corresponding node, as \textit{positive} if one half-edge $h_1$ of $e$ satisfies ${\text{alt}}(h_1) = 0$ and $\text{tw}(h_1) > 2\h$, or, in other words, if $e$ is of type R or NS+.  For a graded $r$-spin graph $\Gamma$ we write $$H^+(\Gamma):=\left\{h\in H^B(\Gamma)\colon \text{either $h$ or $\sigma_1(h)$ is positive}\right\};$$ 
for $(r,\h)$-graph $\mathbf{G}$ we write
$$
H^+(\mathbf{G}):=\bigsqcup_{\Gamma\in V(\mathbf{G})}H^+(\Gamma)
$$
and
$$\partial^{+} \mathbf{G}:=\{ \mathbf{\Delta} \in \partial^{!} \mathbf{G} \colon H^+(\mathbf{\Delta})\ne \emptyset \}.$$ We define $\partial^+\Gamma$ for a graded $r$-spin graph $\Gamma$ in the same way.

For a graded $r$-spin graph $\Gamma$, let $\oPMb_\Gamma$ and $\partial^+\oCM_\Gamma$ be the orbifolds with corners defined by
\[\oPMb_\Gamma:= \oCM_\Gamma\setminus
\left(\bigsqcup_{\Lambda \in \partial^{+}\Gamma}\CM_\Lambda\right),\qquad \partial^+\oCM_\Gamma:=\bigsqcup_{\Lambda\in\partial^+\Gamma}\CM_\Lambda.\]
For an $(r,\h)$-graph $\mathbf{G}$, we write
$$
\oPMb_\mathbf{G}:=\prod_{\Gamma\in V(\mathbf{G})}\oPMb_\Gamma=\oCM_\mathbf{G}\setminus
\left(\bigsqcup_{\Lambda \in \partial^{+}\mathbf{G}}\CM_\Lambda\right), \quad \partial^+\oCM_\mathbf{G}:=\bigsqcup_{\Lambda\in\partial^+\mathbf{G}}\CM_\Lambda.
$$

We also define
\begin{equation}\label{eq def positive moduli point insertion}
    \oPMb^{\frac{1}{r},\h}_{0,B,I}:=\bigsqcup_{\mathbf{G}\in \sGPI^{r,\h}_{0,B,I}}\oPMb_\mathbf{G},\quad \partial^+\Mbar^{\frac{1}{r},\h}_{0,B,I}:=\bigsqcup_{\mathbf{G}\in \sGPI^{r,\h}_{0,B,I}}\partial^+\oCM_\mathbf{G}.
\end{equation}
Note that the boundary $\partial \oPMb_\mathbf{G}$ of $\oPMb_\mathbf{G}$  contains only strata corresponding to graphs without positive half-edges, \textit{i.e.} they contain only boundaries of type CB, AI, and BI. We have a decomposition $$\partial \oPMb_\mathbf{G}= \partial^{CB} \oPMb_\mathbf{G}\cup \partial^{PI} \oPMb_\mathbf{G},$$ where $\partial^{CB} \oPMb_\mathbf{G}$ consists of type-CB boundaries and 
$\partial^{PI} \oPMb_\mathbf{G}$ consists of type-AI and type-BI boundaries. We define  $\partial^{CB}\oPMh_{0,B,I}$ and $\partial^{PI}\oPMh_{0,B,I}$ in the same way.

\begin{definition}
Let $C$ be a graded $r$-spin disk, $q\in C$ a point, and $v\in{\mathcal{W}}_C$. The \textit{evaluation} of $v$ at $q$ is $\text{ev}_q(v) := v(q) \in J_q$. In particular, if $q$ is a contracted boundary node, or a point on $\Sigma$ which is not a legal special point, we say $v$ \textit{evaluates positively} at $q$ if $\text{ev}_q(v)$ is positive with respect to the grading.
\end{definition}

Roughly speaking, a canonical multisection of $\mathcal W\to \Mbar^{1/r,\h}_{0,B,I}$ is a multisection that can be glued to a multisection over $ \widetilde{\mathcal M}^{1/r,\h}_{0,B,I}$ and satisfies certain positivity constraints at $ \partial\widetilde{\mathcal M}^{1/r,\h}_{0,B,I}$. Note that $ \partial\widetilde{\mathcal M}^{1/r,\h}_{0,B,I}$ consists of boundaries corresponding to contracted boundary edges (type-CB) and boundaries corresponding to positive edges (type-R or type-NS+). The positivity constraint at a type-CB boundary for a multisection $s$ is chosen to be that the evaluation of $s$ at the contracted boundary node is positive. A naive definition of the positivity constraint at a type-R or type-NS+ boundary for a multisection $s$ is that the evaluation of $s$ at the illegal half-node is positive; however, this constraint is too strong and the canonical multisection may not exist (see \cite[Example 3.23]{BCT2}).

As in \cite{BCT2}, instead of imposing 
positivity constraints at type-R or type-NS+ boundaries, we impose positive constraints at their neighbourhoods (therefore we work on $\oPMh_{0,B,I}$) and the positive evaluations will be required on certain ``intervals" in the boundary of the disk, which we now recall; they should be viewed as a smoothly-varying family of intervals that approximate neighbourhoods of boundary nodes in a nodal disk. 

\begin{rmk}
    An alternative, equivalent, way for defining the boundary conditions but working on $\oCM_{0,k,l}^{1/r}$ rather than $\oPM_{0,k,l}$ is to allow vanishing of sections of the Witten bundle on the boundary, but requiring a Neumann-like boundary condition on the derivative, following \cite{BCT2}
 we chose to work on $\oPM_{0,k,l}.$\end{rmk}

\begin{definition}[\cite{BCT2}, Definition 3.4]\label{def positive neighbourhoods}
Let $\Gamma$ be a graded $r$-spin graph and let $\Lambda \in \partial\Gamma$.  
We say an open set $U\subseteq \oCM_\Gamma$ is a \textit{$\Lambda$-set with respect to $\Gamma$} if $U$ does not intersect with the strata $\CM_\Xi$  for any $\Xi \in\partial^! \Gamma $ satisfying $\Lambda\notin \partial^!\Xi$.
We say a neighbourhood $U\subseteq\oCM_\Gamma$ of $u\in\oCM_\Gamma$ is a \textit{$\Lambda$-neighbourhood of $u$ with respect to $\Gamma$} if it is a $\Lambda$-set with respect to $\Gamma$.  Since there is a unique smooth graph $\Gamma$ with $\Lambda\in\partial^!\Gamma$, we refer to a $\Lambda$-neighbourhood with respect to a smooth $\Gamma$ simply as a $\Lambda$-neighbourhood.

Let $\Sigma$ be the preferred half of a graded marked disk. We write $\partial \Sigma = S/\sim$ for a space $S$ homeomorphic to $S^1$ and denote by $q\colon S \to \partial \Sigma$ the quotient map. An \emph{interval} $I$ is the image of a connected open set of $S$ under $q$. Note that the preimage of $I$ under the quotient map $q$ is the union of an open set with a finite number of isolated points.

Suppose $C\in\CM_\Lambda$, and let $n_h$ be a boundary half-node corresponding to a half-edge $h\in H^B(\Lambda)$. We denote by $N: \widehat{C} \to C$ the normalization map and write $\widehat{\Sigma} = N^{-1}(\Sigma)$, We say that $n_h$ \textit{belongs} to the interval $I$ if the following two conditions hold:
\begin{itemize}
    \item the corresponding node $N(n_{h})$ lies in $I$;
    \item  $N^{-1}(I)$ contains a half-open interval with starting point $n_h$ and a half-open interval with endpoint $n_{\sigma_1(h)}$, where the starting and end points are determined by the canonical orientation of $\partial\widehat{\Sigma}$.
\end{itemize}

Let $U$ be a $\Lambda$-set. We denote by $\pi:\lvert\mathcal C\rvert\to U$ the coarse universal curve and by $\Sigma_u \subseteq \pi^{-1}(u)$ the preferred half in the fiber. A \textit{$\Lambda$-family of intervals} $\{I_{h}(u)\}_{h\in{H^+}(\Lambda),u\in U}$ for $U$ is a choice of an interval $I_{h}(u)$ for each $h\in{H^+}(\Lambda)$ and $u\in U$, such that:
\begin{enumerate}
\item
The endpoints of each $I_h(u)$ vary smoothly with respect to the smooth structure of the universal curve restricted to $U$.
\end{enumerate}
We say that $\Xi$ is \textit{a smoothing of $\Lambda$ away from $h\in H^+(\Lambda)$} if $\Lambda\in\partial^!\Xi$ and $h$ lies in the image of the injection $\iota:H^+(\Xi)\to H^+(\Lambda)$. We further require that:
\begin{enumerate}
\setcounter{enumi}{1}
\item
If $h \notin\{ h', \sigma_1(h')\}$, we have $I_h(u')\cap I_{h'}(u')=\emptyset$.  If $h=\sigma_1(h')$, we have $I_h(u')\cap I_{h'}(u')\neq\emptyset$ if and only if $u'\in{\mathcal M}_\Xi$ for some smoothing $\Xi$ of $\Lambda$ away from $h$.
  In this case, $I_h(u')\cap I_{\sigma_1(h)}(u')$ consists exactly of the node $N(n_{h})$.
\item There are no marked points in $I_h(u')$. the interval $I_h(u')$ contains at most one half-node; it contains one half-node if and only if $u'\in{\mathcal M}_\Xi$ for some smoothing $\Xi$ of $\Lambda$ away from $h$. The half-node that belongs to $I_h(u')$ is $n_{h}$ in this case.
\end{enumerate}

If the moduli point $u$ represents the stable graded disk $C$ we write $I_h(C)$ for $I_h(u).$ Figure \ref{fig intervals} shows the local picture of the intervals in the nodal and smooth disks.
\end{definition}

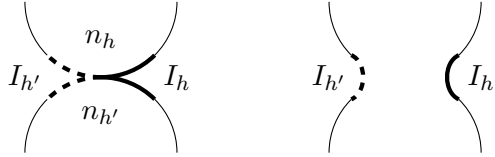
\begin{figure}[h]
\centering

\begin{tikzpicture}[scale=1]

\draw (1,1) arc (0:-45:1);
\draw[ultra thick] (0,0) arc (-90:-45:1);
\draw[ultra thick, dashed] (0,0) arc (-90:-135:1);
\draw (-1,1) arc (-180:-135:1);

\draw (1,-1) arc (0:45:1);
\draw[ultra thick] (0,0) arc (90:45:1);
\draw[ultra thick, dashed] (0,0) arc (90:135:1);
\draw (-1,-1) arc (180:135:1);

\node at (0,-0.5) {$n_{h'}$};
\node at (0,0.5) {$n_{h}$};
\node at (1,0) {$I_h$};
\node at (-1,0) {$I_{h'}$};

\draw (5,1) arc (0:-45:1);
\draw[ultra thick] (4.707,0.293) .. controls (4.5,0.2) and (4.5,-0.2) .. (4.707,-0.293);
\draw (3,1) arc (-180:-135:1);
\draw (5,-1) arc (0:45:1);
\draw[ultra thick,dashed] (3.293,0.293) .. controls (3.5,0.2) and (3.5,-0.2) .. (3.293,-0.293);
\draw (3,-1) arc (180:135:1);

\node at (5,0) {$I_h$};
\node at (3,0) {$I_{h'}$};

\end{tikzpicture}

\caption{The thicker lines representing the intervals $I_h$ and $I_{h'}$ associated to the half-nodes $n_h$ and $n_{h'}$ are drawn over the thinner boundary lines. The image on the right represents a point in the moduli space that is close to the image on the left, where the node is smoothed.}

\label{fig intervals}
\end{figure}

\begin{definition}\label{def positive section}
Let $C$ be a graded $r$-spin disk, and let $A \subseteq \partial\Sigma$ be a subset without legal special points.  Then an element $w\in{\mathcal{W}}_\Sigma$ \emph{evaluates positively at~$A$} if $\ev_x(w)$ is positive for every $x\in A$ with respect to the grading.

Let $\Gamma$ be a graded graph, $U$ a $\Lambda$-set with respect to $\Gamma$, and $\{I_h\}$ a $\Lambda$-family of intervals for $U$.  Given a multisection $s$ of ${\mathcal{W}}$ defined in a subset of $\oPMb_\Gamma$ containing $U\cap\oPMb_\Gamma$, we say $s$ is \emph{$(U,I)$-positive} (with respect to $\Gamma$) if for any $u'\in U\cap \oPMb_\Gamma$, any local branch $s_i(u')$ evaluates positively at each $I_h(u')$.

We say a multisection $s$ defined in $W\cap\oPMb_\Gamma$, where $W$ is a neighbourhood of $u\in\oCM_\Lambda$, is \emph{positive near} $u$ (with respect to $\Gamma$) if there exists a $\Lambda$-neighbourhood $U\subseteq W$ of $u$ and a $\Lambda$-family of intervals $I_*(-)$ for $U$ such that $s$ is $(U,I)$-positive.

If $W$ is a neighbourhood of $\partial^+\Mbar_\Gamma$, then a multisection $s$ defined in a set
\begin{equation}\label{eq U_+}
U_{+,\Gamma}=\left(W\cap\oPMb_\Gamma\right)\cup \bigcup_{H^{CB}(\Lambda)\ne \emptyset}{\oPMb_\Lambda}
\end{equation}
is \emph{positive} (with respect to $\Gamma$) if it is positive near each point of $\partial^+\Mbar_\Gamma$ and evaluates positively at the contracted boundary nodes.  As above, we omit the phrase ``with respect to $\Gamma$" if $\Gamma$ is smooth.
\end{definition}

\begin{definition}\label{def positive for Witten}
Let $\mathbf{G}$ be a smooth $(r,\h)$-graph. For each $\Gamma \in V(\mathbf{G})$, let $U_{+,\Gamma}$ be as in \eqref{eq U_+}, and let $U\subseteq \oPMb_{\mathbf{G}}$ be a set containing $ \bigcup_{\Gamma\in V(\mathbf{G})}\left(U_{+,\Gamma}\times\prod_{\Delta\in V(\mathbf{G})\setminus\{\Gamma\}}\oPMb_\Delta\right)$.  Then a smooth multisection $s$ of ${\mathcal{W}}$ over $U$ is \emph{positive} if, for each $\Gamma \in V(\mathbf{G})$ and $p\in \prod_{\Delta\in V(\mathbf{G})\setminus\{\Gamma\}}\oPMb_\Delta$, the restriction of $s$ to $U\cap \left(\oPMb_\Gamma \times \{p\} \right)$ is positive under the identification $\oPMb_\Gamma \times \{p\}\cong \oPMb_\Gamma$.

\end{definition}

\begin{definition}\label{def canonical for Witten}
Let $U\subseteq \oPMh_{0,B,I}$ be a set containing $$\partial\oPMh_{0,B,I}\cup \bigsqcup_{\mathbf{G}\in \sGPI^{r,\h}_{0,B,I}}\bigcup_{\Gamma\in V(\mathbf{G})}\left(U_{+,\Gamma}\times\prod_{\Delta\in V(\mathbf{G})\setminus\{\Gamma\}}\oPMb_\Delta\right).$$  Then a smooth multisection $s$ of ${\mathcal{W}}$ over $U$ is \emph{canonical} if
\begin{enumerate}
    \item the restriction of $s$ each component  $\oPMb_{\mathbf{G}}\subseteq \oPMh_{0,B,I}$ is positive;
    \item  for every pairs $(\text{bd}_{BI},\text{bd}_{AI})$  of boundaries paired by $PI$ (see Theorem \ref{thm  PI boundaries paried}), we have 
$$s\vert_{\text{bd}_{BI}}=s\vert_{\text{bd}_{AI}}$$
under the isomorphism $PI$.
\end{enumerate}

\end{definition}

\begin{definition}\label{def canonical for L_i}
Let $U\subseteq \oPMh_{0,B,I}$ be a set containing $\partial^{PI}\oPMh_{0,B,I}.$
A smooth multisection $s$ of ${\mathbb L}_i$  over $U$ is called \emph{canonical} if  for every pairs $(\text{bd}_{BI},\text{bd}_{AI})$  of boundaries paired by $PI$, we have 
$$s\vert_{\text{bd}_{BI}}=s\vert_{\text{bd}_{AI}}$$
under the isomorphism $PI$. A multisection $\bm{s}$ of the direct sum of copies of Witten bundle and relative cotangent bundles is canonical if $\bm{s}$ is a direct sum of canonical multisections of the corresponding bundles.
\end{definition}

\begin{rmk}\label{rmk glued section canonical}
    The canonical multisections defined above can be regarded as continuous piecewise-smooth multisections over (a subset of) the glued moduli $\widetilde{\mathcal M}^{\frac{1}{r},\h}_{0,B,I}$.
    The canonicity condition for multisections of $\widetilde{{\mathbb L}}_i,$ and the second condition in the definition of canonicity for multisections of $\widetilde{{\mathcal{W}}},$ are just equivalent to saying one can glue the multisections over different connected components using the isomorphism $PI.$
\end{rmk}

\begin{thm}\label{thm intersection numbers well-defined}
Let $E\to\Mbar^{1/r,\h}_{0,B,I}$ be the bundle
$$E := {\mathcal{W}}^{\oplus m}\oplus \bigoplus_{i=1}^{\lvert I\rvert}{\mathbb L}_i^{\oplus d_i},$$ and assume that $\text{rank}(E)=\dim\Mbar^{1/r,\h}_{0,B,I}.$
Then there exists $U_+=U\cap\oPMh_{0,B,I}$, where $U$ is a neighbourhood of $\partial^+\Mbar^{1/r,\h}_{0,B,I}$ such that $\oPMh_{0,B,I}\setminus U$ is a compact orbifold with corners, and a nowhere-vanishing canonical multisection $\mathbf{s}\in C_m^\infty(U_+\cup\partial\oPMh_{0,B,I},E)$.

Moreover, for odd $m$, one can define, using the canonical relative orientation of ${\mathcal{W}}$, the Euler number (see \cite[Appendix]{BCT2})
\begin{equation}\label{eq  Euler number}
    \int_{\oPMh_{0,B,I}}e(E ; \mathbf{s})\in\mathbb{Q}.
\end{equation}
The result is independent of the choice of $U$ and $\mathbf{s}$.
\end{thm}

\begin{dfn}\label{dfn correlator point insertion}
Fix $r\geq 2,~0\leq\h\leq \lfloor r/2\rfloor-1$ and odd $m.$ With the above notations, if $\text{rank}(E)=\dim\Mbar_{0,B,I}^{1/r,\h}$, we refer to the integral 
\[
\left\langle \prod_{i=1}^{\lvert I\rvert}\tau^{a_i}_{d_i}\prod_{i=1}^{\lvert B\rvert}\sigma^{b_i}\right\rangle^{1/r,o,\h,m}_0 := \int_{\oPMh_{0,B,I}}e(E ; \mathbf{s})\in\mathbb{Q}, 
\]
where $\mathbf{s}$ is any canonical multisection $\mathbf{s}$ that does not vanish at $U_+\cup\partial\oPMh_{0,B,I}$, and~$U_+$ is as above,  as an \emph{open FJRW intersection number for the pair $(x_1^r+\ldots x_m^r,\mathbb{Z}/r\mathbb{Z})$}. For shortness we sometimes call the above number an \emph{open FJRW intersection number}, or a \emph{correlator}. When $\text{rank}(E)\neq\dim\Mbar_{0,B,I}^{1/r,\h}$, the integral is defined to be zero. We also write the correlators in the two-row form
$$\left\langle
	\begin{array}{c}
		\hfill  a_1\psi^{d_1} \quad\hfill   a_2\psi^{d_2}\hfill\dots \hfill a_{\lvert I\rvert}\psi^{d_{\lvert I\rvert}}\hfill\null\\
		\hfill b_1\quad \hfill b_2\quad\hfill b_3\hfill \dots \hfill b_{\lvert B\rvert}\hfill\null\\
	\end{array}
	\right\rangle_0^{1/r,\text{o},\h,m}=\left\langle \prod_{i=1}^{\lvert I\rvert}\tau^{a_i}_{d_i}\prod_{i=1}^{\lvert B\rvert}\sigma^{b_i}\right\rangle^{1/r,o,\h,m}_0. $$
 In the case $m=1$ we refer to this number also as an \emph{open $r$-spin intersection number}, or a \textit{$(r,\h)$-spin intersection number} to emphasize the choice of $\h$.
  In this case we omit $m$ from the notation.
\end{dfn}

\begin{rmk}\label{rmk intersection numbers on glued mod}
    The definition of open FJRW intersection numbers can be explained in terms of characteristic classes. We denote by $\widetilde{\mathcal{PM}}^{\frac{1}{r},\h}_{0,B,I} \subseteq\widetilde{\mathcal M}^{\frac{1}{r},\h}_{0,B,I}$ the moduli glued from $\oPMb_\mathbf{G}\subseteq \Mbar_\mathbf{G}$ for all $\Mbar_\mathbf{G}\subseteq {\Mbar}^{\frac{1}{r},\h}_{0,B,I}$. Note that $\partial \widetilde{\mathcal M}^{\frac{1}{r},\h}_{0,B,I}=\partial^{CB} \widetilde{\mathcal M}^{\frac{1}{r},\h}_{0,B,I}\cup \partial^{+} \widetilde{\mathcal M}^{\frac{1}{r},\h}_{0,B,I}$, where $\partial^{CB} \widetilde{\mathcal M}^{\frac{1}{r},\h}_{0,B,I}$ is the union of type-CB boundaries and $\partial^{+} \widetilde{\mathcal M}^{\frac{1}{r},\h}_{0,B,I}$ is the union of type-R and type-NS+ boundaries. Moreover, we have  $\partial^+ \widetilde{\mathcal M}^{\frac{1}{r},\h}_{0,B,I}=\widetilde{\mathcal M}^{\frac{1}{r},\h}_{0,B,I}\setminus \widetilde{\mathcal{PM}}^{\frac{1}{r},\h}_{0,B,I}$. 
    
   Let $\widetilde{\mathcal W}\to \widetilde{\mathcal M}^{\frac{1}{r},\h}_{0,B,I}$ be the glued Witten bundle over the glued moduli space. Let $\bm{s}$ be a canonical multisection of $\mathcal W^{\oplus m}\to\overline{\mathcal {PM}}^{\frac{1}{r},\h}_{0,B,I}$, we denote by $\widetilde{\bm{s}}$ the corresponding glued multisection of $\widetilde{\mathcal W}^{\oplus m}\to \widetilde{\mathcal {PM}}^{\frac{1}{r},\h}_{0,B,I}$. Let $\widetilde{U}\subset \widetilde{\mathcal M}^{\frac{1}{r},\h}_{0,B,I}$ be a tubular neighbourhood of $\partial^+ \widetilde{\mathcal M}^{\frac{1}{r},\h}_{0,B,I}$ such that each one of the $m$ components of $\widetilde{\bm{s}}$ is positive over $\widetilde{U}_+:=U\cap \widetilde{\mathcal{PM}}^{\frac{1}{r},\h}_{0,B,I}$.

    Since $\widetilde{\bm{s}}$ vanishes nowhere on $\widetilde{U}_+\cup\partial^{CB} \widetilde{\mathcal M}^{\frac{1}{r},\h}_{0,B,I}$, a relative Euler class $$e(\widetilde{\mathcal W}^{\oplus m},\widetilde{\bm{s}})\in H^{\rk \widetilde{\mathcal W}}\left(\widetilde{\mathcal{PM}}^{\frac{1}{r},\h}_{0,B,I},\widetilde{U}_+\cup\partial^{CB} \widetilde{\mathcal M}^{\frac{1}{r},\h}_{0,B,I};\mathbb Q\right)$$ is defined in \cite{RelativeEulerClass} for odd $m$. 
    Note that the pair $\left(\widetilde{\mathcal{PM}}^{\frac{1}{r},\h}_{0,B,I},\widetilde{U}_+\cup\partial^{CB} \widetilde{\mathcal M}^{\frac{1}{r},\h}_{0,B,I}\right)$ is homotopy equivalent to $\left(\widetilde{\mathcal{M}}^{\frac{1}{r},\h}_{0,B,I},\partial \widetilde{\mathcal{M}}^{\frac{1}{r},\h}_{0,B,I}\right)$, we regard $e(\widetilde{\mathcal W}^{\oplus m},\widetilde{\bm{s}})$ as an element in $H^{\rk \widetilde{\mathcal W}}\left(\widetilde{\mathcal{M}}^{\frac{1}{r},\h}_{0,B,I},\partial \widetilde{\mathcal{M}}^{\frac{1}{r},\h}_{0,B,I};\mathbb Q\right)$ and denote by $$[\widetilde{\mathcal W}^{\oplus m}]\in H_{\dim \widetilde{\mathcal M} - \rk \widetilde{\mathcal W}} \left(\widetilde{\mathcal{M}}^{\frac{1}{r},\h}_{0,B,I},\mathbb Q\right)$$ its Lefschetz dual.

    On the other hand, we have the Euler class of the glued relative cotangent line bundles $$e(\widetilde{\mathbb L}_i)\in H^2\left(\widetilde{\mathcal{M}}^{\frac{1}{r},\h}_{0,B,I};\mathbb Q\right).$$ 
    We write $\widetilde{E}:=\widetilde{\mathcal{W}}^{\oplus m}\oplus \bigoplus_{i=1}^{\lvert I\rvert}\widetilde{\mathbb L}_i^{\oplus d_i}$ the bundle glued from $E$ in Theorem \ref{thm intersection numbers well-defined} and $\widetilde{\bm{s}}_{total}$ be a glued canonical section of $\widetilde{E}$, then the class
    $$
        [\widetilde{\mathcal W}^{\oplus m}]\cap \bigcup e(\widetilde{\mathbb L}_i)^{\cup d_i}\in H_{\dim \widetilde{\mathcal M} - \rk \widetilde{l E}}\left(\widetilde{\mathcal{M}}^{\frac{1}{r},\h}_{0,B,I};\mathbb Q\right)
    $$
    is the Lefschetz dual of the relative Euler class $e(\widetilde E,\widetilde{\bm{s}}_{total})$. In particular, in the case where
    $  \rk \widetilde{E}=\dim \widetilde{\mathcal{M}}^{\frac{1}{r},\h}_{0,B,I}$,
     the open $r$-spin correlator is the  product
    $$
        [\widetilde{\mathcal W}^{\oplus m}]\cap \bigcup e(\widetilde{\mathbb L}_i)^{\cup d_i}\in H_0\left(\widetilde{\mathcal{M}}^{\frac{1}{r},\h}_{0,B,I};\mathbb Q\right)\simeq\mathbb Q.
    $$
\end{rmk}

\subsection{Proof of Theorem \ref{thm intersection numbers well-defined}}
The following proposition guarantees the local existence of positive sections of $\mathcal W$. It is a modification of \cite[Proposition 3.20]{BCT2}.
\begin{prop}\label{prop positivity pointwise}
The following positivity claims hold:
\begin{enumerate}
\item
Let $C$ be a stable $r$-spin disk with a contracted boundary node. Then there exists $w\in{\mathcal{W}}_{C}$ that evaluates positively at the contracted boundary node.
\item
Let $U$ be a contractible $\Lambda$-neighbourhood of $\Sigma\in\partial^+\mathcal M_\Gamma$, where $\Gamma$ is a connected smooth genus-zero legal level-$\h$ stable graded $r$-spin graph,  all edges of the graph $\Lambda\in\partial^!\Gamma$ are boundary positive edges, \textit{i.e.} $\Lambda$ has no internal edges nor illegal boundary half-edges of twist less than or equal to $2\h$.  Let $(I_h)_{h\in{H^+}(\Lambda)}$ be a $\Lambda$-family of intervals for $U.$ Then there exists a $(U,I)$-positive section $s$.
\end{enumerate}

\end{prop}

Our basic tool for constructing positive sections in the proof of Proposition~\ref{prop positivity pointwise} is the following lemma:

\begin{lemma}{\cite[Lemma 3.21]{BCT2} }\label{lem surjection of total evaluation map}
Suppose $C$ is a smooth connected graded $r$-spin disk.  Let $h=\deg(|J|)$, and let $\alpha$ and $\beta$ be non-negative integers with $\alpha+ 2\beta \leq h+1$.  Then for any distinct boundary points $p_1,\ldots, p_\alpha,$ and internal points $p_{\alpha+1},\ldots, p_{\alpha+\beta}$, the total evaluation map
\[\bigoplus_i \ev_{p_i}:{\mathcal{W}}_{C}\to\bigoplus_{i=1}^\alpha |J|_{p_i}^{\tilde\phi}\oplus\bigoplus_{i=\alpha+1}^{\alpha+\beta} |J|_{p_i}\] is surjective; here we use the canonical identification $H^0(J) = H^0(|J|)$ to identify the fiber of ${\mathcal{W}}$ with $H^0(|J|)$. In particular, for any $\alpha$ distinct boundary points and $\beta$ distinct internal points, with $\alpha+2\beta=h$, there exists a unique (up to a real scalar) non-zero section of ${\mathcal{W}}$ vanishing at all of them. 

If $C$ has a Ramond marking $p$ and no point with $-1$ twist, in particular if it has a contracted boundary, then the evaluation map at $p$ is surjective.
\end{lemma}
\begin{proof}[Proof of Proposition \ref{prop positivity pointwise}]
To prove the first item, let $C'\subseteq \Sigma \subset C$ be the irreducible component (in the preferred half $\Sigma$) of $C$ containing the contracted boundary node. Since the contracted boundary node has twist $r-1$, the restriction of $\lvert J\rvert$ to $C'$ (which is a line bundle over a rational curve) has a non-negative degree, therefore it has a section which evaluates positive at the contracted boundary node.

To prove the second item, we assume  that $\rk({\mathcal{W}})>0$, since if 
$\rk({\mathcal{W}})=0$ then the statement is trivially true. 

Let $[C]$ be a moduli point in $U$ representing an open, smooth $C$. Let $w\in \mathcal W_C$ be a non-zero vector (which can be regarded as a section of $\lvert J\rvert\to \lvert C\rvert $) whose zeros are simple, we denote by $Z$ the set of zeros on boundary $\partial \Sigma$. We use the following observation to construct sections with the desired positivity constraints. 
\begin{obs}\label{obs zero on boundary}
    If $w\in{\mathcal{W}}_C$ is a non-zero vector, which is positive at one boundary point $p$ and negative at another boundary point $q$, then the number of zeros for $w$ in the boundary arc from $p$ to $q$ (or from $q$ to $p$), plus the number of legal boundary markings in this arc, is an odd number.
\end{obs}

Therefore, if the set of boundary zeros $Z$ of $w$ satisfies the following \textit{local parity conditions}:
\begin{itemize}
    \item[--] $Z\subset A:=\partial\Sigma\setminus\left(\sqcup_{h\in{H^+}(\Lambda)} I_{h}(C)\right)$;
    \item[--] for each connected component $K$ of $A$ (which is an interval), the total number of boundary markings in $K$ plus $|Z\cap K|$ is an even number;
\end{itemize}
then the vector $w$ or $-w$ evaluates positively at $I(C)_h$ for all $h\in{H^+}(\Lambda)$. Our strategy is to construct $w$ with such boundary zero set $Z$; in fact, this can be done by Lemma \ref{lem surjection of total evaluation map} if \begin{equation}\label{eq bound of Z}
\lvert Z \rvert\le \rk \mathcal{W}-1
\end{equation}
and (note that the number of internal zeros is even since they are invariant under $\phi$) \begin{equation}\label{eq parity of Z}\lvert Z \rvert\equiv\rk({\mathcal{W}})-1\mod 2.
\end{equation}

Moreover, because $U$ is contractible, we can construct a diffeomorphism $\phi:\mathcal C(U')\to U'\times C$ between the universal curve on $U'=U\cap\oPMb_{\Gamma}$ and $U'\times C$. This diffeomorphism maps the fiber over $[C']$ to $([C'],C)$ and satisfies $\phi(I(C')_h)=([C'],I(C)_h)$. Then we can define a section $s$ of $\mathcal W \to U'$, such that $s([C'])$ is the unique (up to 
 real scaling) section of $\lvert J\rvert\to \lvert C'\rvert$  having (under the identification between $C'$ and $C$ induced by $\phi$) the same zeros as $w$ constructed above by Lemma \ref{lem surjection of total evaluation map}. The real scaling factor is chosen to ensure that the vector $s([C])$ evaluates positively at $I(C)_h$ for all $h\in{H^+}(\Lambda)$. By construction, the section $s$ is $(U,I)$-positive.

It remains to construct a set $Z$ satisfying the local parity conditions, \eqref{eq bound of Z} and \eqref{eq parity of Z}. We introduce the following combinatorial lemma and postpone its proof to the end of the proposition

\begin{lem}
\label{lem degv}
For a vertex $v$ of $\Lambda$, write $\deg(v) = \rk({\mathcal{W}}_v)+m_v-1$, where $m_v$ is the number of legal half-edges of $v$ with twist less than $r-2-2\h.$ We also denote by $k_v$ the number of legal half-edges of $v$ with twist greater than or equal to $r-2-2h$, which by the assumptions on $\Lambda$ must be tails.
\begin{enumerate}
\item\label{it lem lem1 } For every vertex $v \in V(\Lambda)$, we have $\deg(v) \equiv k_v  \mod 2$.
\item\label{it lem lem2} We have $
\rk({\mathcal{W}})-1=\sum_{v\in V(\Lambda)}\deg(v).
$
\end{enumerate}
\end{lem}

We first construct a set $Z$ that satisfies the local parity conditions in a minimal way: for each connected component $K$ of $A$, we choose (arbitrary) $k_K$ points in $K$ as elements of $Z$, where $k_K=1$ if the number of boundary markings on $K$ is odd, and $k_K=0$ if the number of boundary markings on $K$ is even. Now we prove that \eqref{eq bound of Z} and \eqref{eq parity of Z} hold for the set $Z$ constructed in this way. Note that, by definition, we have 
\begin{equation}\label{eq Z as sum k}
\lvert Z \rvert=\sum_{\substack{\text{$K$ a connected}\\\text{component of $A$}}}k_K.
\end{equation}

Any connected component $K$ of $A$ is associated with a vertex of $v\in V(\Lambda)$: assuming that the endpoints of $K$ are endpoints of $I_{h_1}(C)$ and $I_{h_2}(C)$, for a nodal curve $C'$ such that $[C']\in U \cap \mathcal M_{\Lambda}$, the corresponding endpoints of $I_{h_1}(C')$ and $I_{h_2}(C')$ determine an interval $K'\subset \partial \Sigma'\subset \lvert C'\rvert$, then $K$ is associated with the vertex $v$ that corresponds to the irreducible component $C'_v$ of $C'$ containing $K'$; we denote this relation by $K\to v$.  

Note that \eqref{eq Z as sum k} and the second item of Lemma \ref{lem degv} indicate that \eqref{eq bound of Z} will follow from
\begin{equation}\label{eq main for positivity}
\sum_{K\to v}k_K \leq \deg(v),\quad\forall v\in V(\Lambda)
\end{equation}
and \eqref{eq parity of Z} will follow from
\begin{equation}\label{eq main for positivity2}
\sum_{K\to v}k_K = \deg(v) \mod 2,\quad\forall v\in V(\Lambda).
\end{equation}
Actually, by the definition of $k_K$ and the first item of Lemma~\ref{lem degv}, we have
\begin{equation}
\label{eq mod 2 for positivity}
\sum_{K\to v}k_K  \equiv k_v \equiv \deg(v) \mod 2,
\end{equation}
which implies \eqref{eq main for positivity2}.  

We now prove \eqref{eq main for positivity}. In fact, due to \eqref{eq main for positivity2} and integrality, we only need to prove the weaker inequation
\begin{equation}\label{eq main for positivity weaker}
    \sum_{K\to v}k_K \leq \deg(v)+1,\quad \forall v\in V(\Lambda).
\end{equation}

We check \eqref{eq main for positivity weaker} in two cases. We denote by $n_v$ the number of illegal half-edges of $v$, and as above, by $k_v$ and $m_v$ the numbers of legal half-edges of $v$ with twist greater than or equal to $r-2-2\h$, and legal half-edges of $v$ with twist less than $r-2-2h$ respectively. Note that, by the assumption on $\Lambda$, the twists of the illegal half-edges are at least $2\h+2$.
\begin{itemize}
    \item The case $n_v\geq k_v$. In this case we have 
    \begin{equation}
\label{eq main est for positivity}
\begin{split}
\deg(v) &= \rk({\mathcal{W}}_{v})+m_v-1\\
&\ge \left\lceil\frac{(2\h+2)n_v+(r-2-2\h)k_v-(r-2)}{r}\right\rceil+m_v-1\\
&\ge \left\lceil\frac{(2\h+2)k_v+(r-2-2\h)k_v-(r-2)}{r}\right\rceil+m_v-1\\
&= \left\lceil\frac{rk_v-(r-2)}{r}\right\rceil+m_v-1\\
&= k_v+m_v-1\geq k_v-1\geq \sum_{K\to v}k_K-1,
\end{split}
\end{equation}
where we have used the integrality of $\deg(v)$, and the fact that the number of intervals that may contain a boundary marked point of twist greater than or equal to $r-2-2\h$ is no more than $k_v$.
   
\item The case $k_v\geq n_v+1.$ In this case, similarly to \eqref{eq main est for positivity}, we have
\begin{equation*}
    \begin{split}
        \deg(v) &\ge \left\lceil\frac{(2\h+2)n_v+(r-2-2\h)(n_v+1)-(r-2)}{r}\right\rceil+m_v-1\\
    &= n_v+m_v-1,
    \end{split}
\end{equation*}
where we have used the integrality of $\deg(v)$ again.  Then \eqref{eq main for positivity weaker} holds in this case because $n_v+m_v=|\{K \; |\; K\to v\}|\ge \sum_{K\to v}k_K$.
\end{itemize}
Therefore \eqref{eq main for positivity weaker} is proven in both cases, and the proposition follows.
\end{proof}

\begin{proof}[Proof of Lemma~\ref{lem degv}]
To prove the first item, 
note that  the number of legal boundary half-edges at vertex $v$ is $k_v+m_v$, according to \eqref{eq parity} we have 
$$ \rk({\mathcal{W}}_v)+1=\frac{2 \sum a_i + \sum b_j - (r-2)}{r}+1 \equiv k_v+m_v\mod 2.$$

To prove the second item, according to Proposition \ref{prop decomposition} we have
$$\rk({\mathcal{W}})-1 = \sum_{v\in V(\Lambda)}(\rk({\mathcal{W}}_v)-1) + | E^{\text{NS}}|=\sum_{v\in V(\Lambda)}\deg(v),$$
where $E^{\text{NS}}$ is the set of NS boundary edges. The right equality is a consequence of the fact that any NS boundary edge has exactly one legal half-edge of twist smaller than $r-2-2\h$, by the assumption on $\Lambda$.   
\end{proof}
 
\begin{cor}[Corollary of Proposition \ref{prop positivity pointwise}]\label{cor local positivity more general}
    Let $U$ be a contractible $\Lambda$-neighbourhood (with respect to $\Gamma$) of $\Sigma\in\partial^+\mathcal M_\Gamma$, where $\Gamma$ is a connected graded $r$-spin graph without positive edges, all the boundary tails of  $\Gamma$ have twist greater or equal to $r-2-2\h$,  all edges in $E(\Lambda)\setminus E(\Gamma)$ for the graph $\Lambda\in\partial^!\Gamma$ are positive boundary edges.  Let $(I_h)_{h\in{H^+}(\Lambda)}$ be a $\Lambda$-family of intervals for $U.$ Then there exists a $(U,I)$-positive (with respect to $\Gamma$) section $s$.
\end{cor}
\begin{proof}
      The case where $\Gamma$ is smooth is treated in the second item of Proposition \ref{prop positivity pointwise}. 
      
      We first assume that $\Gamma$ consists of two open vertices $v_1$ and $v_2$ connected by a boundary edge $e$, where the half-edge $h_1$ on the $v_1$ side is illegal and has twist $2t\le 2\h$, the half-edge $h_2$ on the $v_2$ side is legal. Note that $e$ is an NS boundary edge, a section $s$ of $\mathcal W\to \Mbar_{\Gamma}$ is given by a section $s_1$ of $\mathcal W\to \Mbarstar_{v_1}$ and a section $s_2$ of $\mathcal W\to \Mbar_{v_2}$. Since $U$ is contractible, as in the proof of Proposition \ref{prop positivity pointwise}, it is enough to find vectors $w_1\in \mathcal W_{C_1}$ and $w_2\in \mathcal W_{C_2}$ for moduli points $C_1\in {\mathcal M^*}_{v_1}$ and  $C_2\in \mathcal M_{v_2}$ which evaluate positively on the corresponding intervals.

      We denote by $\Lambda_1$ and  $\Lambda_2$ the graphs obtained by detaching the edge $e\in E(\Lambda)$ (note that $\Lambda_1$ has an illegal half-edge $h_1$). The vector $w_2$ can be constructed by applying Proposition \ref{prop positivity pointwise} to $\Lambda_2\in \partial^! v_2$. We denote by $\hat v_1$ and $\hat\Lambda_1$ the graphs obtained by replacing the illegal boundary half-edge $h_1$ (in the corresponding graph) with an internal tail $\hat h_1$ with twist $t$. By applying Proposition \ref{prop positivity pointwise} to $\hat\Lambda_1\in \partial^! \hat v_1$ we obtain a vector $\hat w_1\in \mathcal W_{\hat C_1}$ for a moduli point $\hat C_1\in \mathcal M_{\hat v_1}$ which evaluates positively on the corresponding intervals. We can choose $\hat C_1$ to be a moduli point matching $C_1$ in the sense that, after forgetting the $r$-spin structures and the markings corresponding to $h_1$ or $\hat h_1$, we get the same moduli point $\Sigma_1$ in the moduli space of marked disks from $C_1$ and $\hat C_1$ (we can always assume $\Sigma_1$ is stable, since otherwise there will be no intervals on $v_1$ side and hence no positivity constraints needed to be satisfied). The set of zeros of $\hat w_1\in H^0\left(\hat C_1,\lvert J_{\hat C_1}\rvert\right)$ induces via $\Sigma_1$ a set $Z_{C_1}\subset C_1$. We take the vector $ w_1\in H^0\left( C_1,\lvert J_{ C_1}\rvert\right)$ to be the one that vanishes exactly at the set $Z_{C_1}$, then up to a sign flip,  $w_1$ evaluates positively on all the intervals on $v_1$ side as required.

      By using the above argument repeatedly, we can prove the corollary for all $\Gamma$ without internal edges. In the case where $\Gamma$ has internal edges, denoting by $\Gamma^o$ the open part of $\Gamma$ after detaching all the internal edges, we first find a vector $w_{C^o}\in \mathcal W_{C^o}$ for a moduli point $C^o\in \mathcal M_{\Gamma^o}$ which evaluates positively at all intervals, then we extend it to $w_{C}\in \mathcal W_{C}$ for a moduli point $C\in \mathcal M_{\Gamma}$ using the assembling operator.
\end{proof}

The following homotopy claims will be used in the proof of Theorem \ref{thm intersection numbers well-defined}.
\begin{thm}{\cite[Theorem 6.8]{BCT2}}\label{thm hirsch}
Let $E\to M$ be an orbifold vector bundle over a smooth orbifold with corners, and let $W = {\mathbb R}_{+}^n.$ Fix smooth multisections $s_0,\ldots,s_n\in C_m^\infty(M,E)$, and let $F: W \to C_m^\infty(M,E)$ be the map
\[
(\lambda_{i})_{i\in [n]}\to F_\Lambda = s_0 + \sum \lambda_i s_i.
\]
Denote by $p_M:W\times M\to M$ the projection. If the multisection
\[
F^{ev}\in C^\infty_m( W \times M, p_M^*E), \qquad
F^{ev}\left(\lambda,x\right)= F_\lambda\left(x\right),
\]
is transverse to zero, then the set $\{w\in W \colon F_w \pitchfork 0\}$ is residual.
\end{thm}

\begin{lemma}{\cite[Lemma 4.12]{BCT2}}\label{lem zero difference as homotopy}
Let $E \to M$ be an orbifold vector bundle over an orbifold with corners, with $\text{rank}(E) = \dim(M)$, let  $U \subseteq M$ be an open set with $M \setminus U$ compact, and let $s_0$ and $s_1$ be nowhere-vanishing smooth multisections of $E$ over $U\cup \partial M$. Denote by $p :   M\times[0,1]\to M$ the projection, consider the homotopy
\[
H \in C_m^\infty((U\cup \partial M)\times[0,1],p^*E),~~H(x,i) = s_i(x),~i=1,2.
\]  
Suppose that $H \pitchfork 0$ and that $H(u,t) \neq 0$ for all $t \in [0,1]$ and $u \in U$.  Then
\[
\int_M e(E ; s_1) - \int_M e(E ; s_0) = \# Z(H)=\# Z(H|_{\partial M\times[0,1]}).
\]
\end{lemma}

The following lemma is a tool to construct positive multisections of 
$\mathcal W$ inductively.
\begin{lem}\label{lem construct positive section from boundary}
    Let $\Gamma$ be a graded $r$-spin graph without positive edges. Let $\{s_\Delta\}_{\Delta \in \partial^B \Gamma\setminus \partial^+ \Gamma}$  be a collection of multisections, where each $s_\Delta$ is a positive multisection (with respect to $\Delta$) of $\mathcal W$ over $\oPMb_{\Delta}$, and such that $s_{\Delta_1}\big\vert_{\oPMb_\Theta}=s_{\Delta_2}\big\vert_{\oPMb_\Theta}$ for $\Theta \in \partial \Delta_1\cap \partial \Delta_2$. Then there exist a multisection $s_\Gamma$ of $\mathcal W$ over $\oPMb_{\Gamma}$ which is positive (with respect to $\Gamma$), and satisfies $s_\Gamma\big\vert_{\oPMb_{\Delta}}=s_\Delta$ for all $\Delta \in \partial^B \Gamma\setminus \partial^+ \Gamma$.
\end{lem}
\begin{proof}
    If $H^{CB}(\Gamma)\ne\emptyset$, then $\partial^B\Gamma=\emptyset$. In this case we only need to show that there exists a multisection $s_\Gamma$ of $\mathcal W$ over $\oPMb_\Gamma=\Mbar_{\Gamma}$ which evaluates positively at the contracted boundary node at every moduli point. The local existence of such multisection is guaranteed by the first item of Proposition \ref{prop positivity pointwise}, and we can glue them to a global multisection by a partition of unity (note that $\Mbar_{\Gamma}$ is compact in this case).

    Now we assume $H^{CB}(\Gamma)=\emptyset$. For each $\Lambda\in  \partial^B \Gamma$, we denote by $\mathcal S\Lambda$ the graph obtained by smoothing all the boundary positive of $\Lambda$. We first observe that, for each $u\in \mathcal M_\Lambda\subseteq \partial \Mbar_{\Gamma}$, there exists a $\Lambda$-neighbourhood  $U_u\subseteq \Mbar_{\Gamma}$ of $u$ and a section $s_u$ of $\mathcal W\to U_u\cap \oPMb_{\Gamma}$ satisfying:
    \begin{itemize}
        \item[--] there exist a $\Lambda$-family of intervals $\{I_{u,h}\}$ such that $s_u$ is $(U_u,I_u)$-positive (with respect to $\Gamma$) near $u$ (note that in the case $\Lambda\notin \partial^+\Gamma$  there is no interval);
        \item[--] the restriction of $s_u$ to $\oPMb_{\mathcal S \Lambda}\cap U_u$ coincides with $s_{\mathcal S \Lambda}$ (unless $\mathcal S \Lambda=\Gamma$).
    \end{itemize}
    Actually, we can take $s_u$ to be an extension of $s_{\mathcal S \Lambda}$ in a sufficiently small neighbourhood $U_u$ in the case where $\mathcal S \Lambda\ne\Gamma$; in the case where $\mathcal S \Lambda=\Gamma$, we can apply Corollary \ref{cor local positivity more general}.

    We define 
    $$
    \partial_n  \Mbar_{\Gamma}:=\bigcup_{\substack{\Lambda \in \partial \Gamma\\\lvert E^B(\Lambda)\setminus E^B(\Gamma)\rvert+\lvert H^{CB}(\Lambda)\rvert\ge \dim \Mbar_\Gamma-n}}\Mbar_\Lambda.
    $$
    Using the above sections $\{s_u\}$ and partitions of unity, we can inductively construct sections $s_n$ of $\mathcal W \to V_n$ for $0\le n \le \dim \Mbar_\Gamma-1$, where $V_n$ is a neighbourhood of $\partial_n  \Mbar_{\Gamma}$ in $\Mbar_{\Gamma}$, such that $s_n$ is positive (with respect to $\Gamma$) near $u$ for all $u\in V_n\cap \partial \Mbar_{\Gamma}$, and $s_n$ coincides with $s_\Delta$ for all $\Delta \in \partial^B \Gamma\setminus \partial^+ \Gamma$ on their common domain. 

    For $n=0$, since $\partial_0 \Mbar_\Gamma$ is a finite set of points, we take $V_0:=\bigsqcup_{u\in \partial_0 \Mbar_{\Gamma}} U_u$ (we choose $U_u$ small enough to avoid intersection), and set $s_0\vert_{U_u}:=s_u$ for each $u\in \partial_0 \Mbar_{\Gamma}$.

    Assuming we have constructed $V_{n-1}$ and $s_{n-1}$, we now construct $V_n$ and $s_n$. Note that $ \partial_{n} \Mbar_{\Gamma}\setminus V_{n-1}$ is compact, there exists $u_1^n,u_2^n,\dots,u^n_{k_n}\in \partial_{n} \Mbar_{\Gamma}\setminus V_{n-1}$ such that $$V_{n-1}\cap \partial_{n} \Mbar_{\Gamma},~ U_{u_1^n}\cap \partial_{n} \Mbar_{\Gamma},~\dots,~ U_{u_{k_n}^n}\cap \partial_{n} \Mbar_{\Gamma}$$ is an open cover of $\partial_{n} \Mbar_{\Gamma}$. We take $\rho_0,~\rho_1,~\dots,~\rho_{n_k}$ to be a partition of unity subordinated to this open cover, and extend them smoothly to non-negative functions supported in $V_{n-1},~U_{u_1^n},~\dots,~U_{u_{k_n}}^n$. Then we can take 
    $$
    V_n:=V_{n-1}\cup U_{u_1^n}\cup \dots \cup U_{u_{k_n}}^n
    $$
    and 
    $$
    s_n:=\rho_0 s_{n-1}+\rho_1 s_{u^n_1}+\dots+\rho_{{k_n}} s_{u^n_{k_n}}.
    $$
    The section $s_n$ satisfies the desired properties because it is glued from the sections which satisfy the properties locally.

    Because $V_{\dim \Mbar_\Gamma-1}$ is a neighbourhood of $\partial \Mbar_\Gamma$,  we can take $s_\Gamma$ to be a smooth extension of $s_{\dim \Mbar_\Gamma-1}$ to the entire $\oPMb_\Gamma$. 
\end{proof}

\begin{proof}[Proof of Theorem \ref{thm intersection numbers well-defined}]
We split the proof into three parts.
\proofpart{1}{Construction of a canonical multisection $\mathbf{s}_0$}
We consider the set
$$
\mathfrak P:=\{\Gamma\colon  \Gamma\in V(\mathbf{G}), E^I(\Gamma)=\emptyset \text{ for an $(r,\h)$-graph $\mathbf{G}$ such that }\emptyset\ne \oPMb_\mathbf{G}\subseteq \oPMh_{0,B,I} \},
$$
note that we identify $\Gamma_1\in V(\mathbf{G_1})$ and $\Gamma_2\in V(\mathbf{G_2})$ in the following two situations:
\begin{enumerate}
    \item  $\Gamma_1$ is naturally identified with $\Gamma_2$ as unchanged vertices when $\mathbf{G_1}$ is a smoothing of $\mathbf{G_2}$;
    \item 
    $\Gamma_1$ is naturally identified with $\Gamma_2$ as unchanged vertices when $\mathbf{G_1}$ is paired with $\mathbf{G_2}$ by $PI$.
\end{enumerate} 

All graphs $\Gamma\in V(\mathbf{G})$ in $\mathfrak P$ have no positive boundary edge since $\oPMb_\mathbf{G}\ne \emptyset$. Thus, all boundary edges of $\Gamma$ are of type AI or BI.

Let $\Gamma\in \mathfrak P$ be a graded $r$-spin graph with $s-1$ type-BI boundary edges, we denote by $\Gamma_1,\dots,\Gamma_s$ the graphs obtained by detaching all these type-BI boundary edges, and by $\{(h_j^{legal},h_j^{ill})\}_{1\le j \le s-1}$ the corresponding pairs of legal and illegal half-edges. Note that we have $\tw(h_j^{legal})\ge r-2-2\h$ and $\tw(h_j^{ill})\le 2\h$ because they correspond to type-BI edges. We modify the graphs $\Gamma_i$ to legal level-$\h$ graded $r$-spin graphs $\hat \Gamma_i$ in the following way: 
to each illegal half-edge $h_j^{ill}$ we attach a new vertex $v_{\hat{h}_j}$, where $v_{\hat{h}_j}$ has only one boundary half-edge with twist $\tw(h_j^{legal})$ and only one internal tail $\hat{h}_j$ with twist $\frac{\tw (h_j^{ill})}{2}$. We define the collection of graded $r$-spin graphs
$$
\mathfrak S \Gamma:=\{\hat \Gamma_1,\hat \Gamma_2,\dots,\hat \Gamma_s\}.
$$
Note that if $\Gamma\in V(\mathbf{G}),$ for an $(r,\h)$-graph $\mathbf{G}$, then $\Mbar_\mathbf{G}$ is contained in the intersection of $s-1$ type-BI codimension-1 boundaries $\text{bd}_{BI}^1,\text{bd}_{BI}^2,\dots,\text{bd}_{BI}^{s-1}$ of $\oPMh_{0,B,I}$. 
If we denote by $\mathbf{\hat G}$ the $(r,\h)$-graph obtained by replacing $\Gamma\in V(\mathbf{G})$ with graded $r$-spin graphs in $\mathfrak S \Gamma$ and adding a dashed line between $h_j^{legal}$ and $\hat{h}_j$ for each $1\le j \le s-1$, then $\Mbar_{\mathbf{\hat G}}$ is contained in the intersection of $s-1$ type-AI codimension-1 boundaries $\text{bd}_{AI}^1,\text{bd}_{AI}^2,\dots,\text{bd}_{AI}^{s-1}$  of $\oPMh_{0,B,I}$, where $\text{bd}_{AI}^i$ is paired with $\text{bd}_{BI}^i$ by $PI$. Moreover, $\Mbar_{\mathbf{\hat G}}$ is identified with $\Mbar_{\mathbf{ G}}$ under a sequence of $PI$s. 

We claim that there exists a multisection $s_\Gamma$ of $\mathcal W\to \oPMb_\Gamma$ for each $\Gamma \in \mathfrak P$, satisfying:
\begin{itemize}
    \item[--] $s_\Gamma$ is positive (with respect to $\Gamma$);
    \item[--] by identifying $\Mbar_\Gamma$ with $\prod_{\hat\Gamma_i\in \mathfrak S\Gamma}\Mbar_{\hat \Gamma_i}$ using Remark \ref{rmk decompose NS boundary node}, we have $s_\Gamma=\bboxplus_{\hat\Gamma_i\in \mathfrak S\Gamma} s_{\hat \Gamma_i}$;
    \item[--] for $\Delta\in (\partial^B \Gamma\setminus\partial^+ \Gamma)\cap \mathfrak P$, we have $s_\Gamma\big\vert_{\oPMb_\Delta}=s_\Delta$.
\end{itemize}

We construct $s_\Gamma$ by an induction on $\dim \oPMb_\Gamma$. In the case where $\dim \oPMb_\Gamma=0$, the moduli space $\oPMb_\Gamma=\Mbar_{\Gamma}$ is a single point. If $\Gamma$ has a contracted boundary edge,  we take $s_\Gamma$ to be a vector in $\mathcal W_{\Mbar_{\Gamma}}$ evaluating positively at the contracted boundary node; if $\Gamma$ has no contracted boundary edge, we take $s_\Gamma$ to be an arbitrary vector in $\mathcal W_{\Mbar_{\Gamma}}$ (note that there is no positivity constraint in this case).  

Assuming we have constructed $s_\Lambda$ for all $\Lambda\in \mathfrak P$ where $\dim \oPMb_\Lambda \le n-1$, We now construct  $s_\Gamma$ for $\Gamma$ where $\dim \oPMb_\Gamma = n$. 

We first construct $s_\Gamma$  for  $\Gamma$ with no type-BI boundary edges. For each $\Delta\in \partial^B \Gamma\setminus \partial^+ \Gamma$, if $E^I(\Delta)=\emptyset$, then a multisection $s_\Delta$  of $\mathcal W\to \oPMb_\Delta$ is constructed by the previous step; if $E^I(\Delta)\ne\emptyset$, let $\check{\Delta}\in  \partial^B \Gamma\setminus \partial^+ \Gamma$ be the graph obtained by smoothing all internal edges of $\Delta$, we define $s_\Delta:=s_{\check{\Delta}}\big\vert_{\oPMb_\Delta}$. We define $s_\Gamma$ to be the multisection constructed from $\{s_\Delta\}_{\Delta\in \partial^B \Gamma\setminus \partial^+ \Gamma}$ as in Lemma \ref{lem construct positive section from boundary}.

For $\Gamma$ with at least one type-BI boundary edge, we define $s_\Gamma:=\bboxplus_{\hat\Gamma_i\in \mathfrak S\Gamma} s_{\hat \Gamma_i}$ under the identification between $\Mbar_\Gamma$ and $\prod_{\hat\Gamma_i\in \mathfrak S\Gamma}\Mbar_{\hat \Gamma_i}$.

Now we check that the multisection $s_\Gamma$ we constructed in this way satisfies the three requirements. The first requirement is satisfied by definition for $\Gamma$ without type-BI boundary edges; for $\Gamma$ with type-BI boundary edges, it follows from the fact that $PI$ preserves the positivity of multisections. The second and third requirements are satisfied by definition. 

For each smooth $(r,\h)$-graph $\mathbf{G}$ such that $\oPMb_\mathbf{G}\subseteq \oPMh_{0,B,I}$, we define a multisection $s_\mathbf{G}$ of $\mathcal W\to \oPMb_\mathbf{G}$ by 
$$
s_\mathbf{G}:=\bboxplus_{\Gamma\in V(G)}s_\Gamma.
$$
All these $s_G$ together give a multisection $s_{\mathcal W}$ of $\mathcal W\to \oPMh_{0,B,I}$. Actually, $s_{\mathcal W}$ is a canonical multisection, this can be seen from the three requirements satisfied by $s_\Gamma$.

We define a canonical multisection $\mathbf{s}_0$ of $E = {\mathcal{W}}\oplus \bigoplus_{i=1}^{\lvert I\rvert}{\mathbb L}_i^{\oplus d_i}$ by  $$\mathbf{s}_0:=s_{\mathcal W}^{\oplus m}\oplus0\oplus\dots\oplus 0.$$

\proofpart{2}{Perturbation of $\mathbf{s}_0$}
 By the positivity and continuity of $s_{\mathcal W}$, the multisection $s_0$ does not vanish on a neighbourhood of $\partial^+\Mbar^{1/r,\h}_{0,B,I}$ and a neighbourhood of $\partial^{CB}\oPMb^{1/r,\h}_{0,B,I}$.  We denote by $U_+$ the union of all these neighbourhoods. We also choose a smaller neighbourhood $V_+\subset U_+$ such that the closure of $V_+$ is contained in $U_+$.
 
To find a multisection $\mathbf{s}$ satisfying the requirements of Theorem \ref{thm intersection numbers well-defined}, we need to show that we can perturb $\mathbf{s}_0$ away from $V_+$ to make it also does not vanish on  $\partial^{PI}\oPMb^{1/r,\h}_{0,B,I}$, \textit{i.e.} boundaries of type BI or AI. In fact, by using Theorem \ref{thm hirsch}, we can perturb $\mathbf{s}_0$ to a canonical multisection $\mathbf{s}$ which is transverse to zero when restricted to $\partial^{PI}\oPMb^{1/r,\h}_{0,B,I}$, then such a perturbation $\mathbf{s}$ will not vanish on $\partial^{PI}\oPMb^{1/r,\h}_{0,B,I}$ for dimensional reasons. We will construct such a perturbation $\mathbf{s}$ which is transverse to zero also on the entire $\oPMb^{1/r,\h}_{0,B,I}$.

We choose open sets $U_u,V_u\subset \oPMb^{1/r,\h}_{0,B,I}\setminus \overline{V_+}$ and a set of sections $\{\bm{s}^i_{u}\}_{1\le i \le \rk E}$ of $E\to \oPMb^{1/r,\h}_{0,B,I}$ for each $u\in  \oPMb^{1/r,\h}_{0,B,I}\setminus U_+$, such that
\begin{itemize}
    \item[--] $u\in V_u\subset \overline{V_u}\subset U_u\setminus V_+$;
    \item[--] for each $u'\in U_u$, if $u\in \mathcal M_G$ and $u'\in \mathcal M_{G'}$ for  $(r,\h)$-graphs $G$ and $G'$, we have $G\in \partial^! G'$;
    \item[--] for each $u'\in V_u$, the vector space $E_{u'}$ is generated by the set of vectors $\{\bm{s}^i_{u}(u')\}_{1\le i \le \rk E}$;
    \item[--] the sections $\{\bm{s}^i_{u}\}_{1\le i \le \rk E}$ are supported in $U_u$;
    \item[--]  each section $\bm{s}^i_{u}$ can be written as a direct sum of sections of the Witten bundles and the relative cotangent line bundles;
    \item[--] if $u\in \text{bd}_{AI}$ and $u'\in \text{bd}_{BI}$ are identified via the identification $PI$ between the codimension-1  boundaries $\text{bd}_{AI}$ and $\text{bd}_{BI}$, we have $\bm{s}^i_{u}\vert_{\text{bd}_{AI}}=\bm{s}^i_{u'}\vert_{\text{bd}_{BI}}$ for all $1\le i \le \rk E$.
\end{itemize}
Such a choice is possible because we can extend basis vectors of $E_{u}$ smoothly to sections of $E$ supported in a neighbourhood of $u$. The sixth requirement can be satisfied if we construct $\bm{s}^i_u$ inductively as in the construction of $s_{\mathcal W}$ above, \textit{i.e.} we construct the sections $\bm{s}^i_u$ on the lower dimensional boundaries in a $PI$-compatible way first, then extend it to the higher dimensional boundaries.

For each $u\in  \oPMb^{1/r,\h}_{0,B,I}\setminus U_+$ , we denote by $U_{[u]}$ (respectively $V_{[u]}$) the disjoint union of all $U_{u'}$ (respectively $V_{u'}$) such that $u'$ is identified with $u$ via a sequence of $PI$s; moreover, for all $1\le i \le \rk E$, we denote by $\bm{s}^i_{[u]}$ the section of $E$ supported in $U_{[u]}$ given by the sections $s^i_{u'}$ on each connected components. Note that if $\hat{\bm{s}}$ is a canonical multisection of $E\to \oPMh_{0,B,I}$, then for any real number $\epsilon$, the multisection $\hat{\bm{s}}+\epsilon \bm{s}^i_{[u]}$ is also canonical.

Since $\oPMb^{1/r,\h}_{0,B,I}\setminus U_+$ is compact, there exists a finite set of moduli points $\{u_1,\dots,u_N\}\subset \oPMb^{1/r,\h}_{0,B,I}\setminus U_+$ such that  $\oPMb^{1/r,\h}_{0,B,I}\setminus U_+\subset \bigcup_{i=1}^N V_{[u_i]}$. We denote by $W:=\mathbb R_+^{N \times \rk E}$ the space with coordinators $\lambda=\left(\lambda_i^j\right)_{\substack{1\le i \le N\\1\le j \le \rk E}}$. We pull back the vector bundle $E$ from $\oPMb^{1/r,\h}_{0,B,I}$ to $W\times \oPMb^{1/r,\h}_{0,B,I}$ via the projection and consider the multisection
$$
\bm{t}(\lambda,u):=\bm{s}_0(u) + \sum_{\substack{1\le i \le N\\1\le j \le \rk E}} \lambda_i^j \bm{s}^j_{[u_i]}(u)
$$
of $E\to W\times \oPMb^{1/r,\h}_{0,B,I}$. We claim that $\bm{t}$ is transverse to zero on $W\times \oPMb^{1/r,\h}_{0,B,I}$. In fact, the multisection $\bm{t}$ does not vanish on $p=(\lambda,u)\in W\times V_+$ since $\bm{s}_0(u)$ does not vanish and all $\bm{s}^j_{[u_i]}(u)$ vanishes; if $\bm{t}(p)=0$ for some $p\in W\times V_{[u_i]}$, the vectors $\left\{\frac{\partial \bm{t}}{\partial \lambda_i^j}(p)=\bm{s}^j_{[u_i]}(u)\right\}_{ 1\le j \le \rk E}$ generate the fiber $E_p$. 

Note that we can write $\partial^{PI}\oPMb^{1/r,\h}_{0,B,I}=\bigcup_{\bm{H_i}\in X} \oPMb_{\bm{H_i}}$, where $X$ is a finite set of $(r,\h)$-graphs with a single boundary edge. The restriction of $\bm{t}$ to each $\oPMb_{\bm{H_i}}$ is also transverse to zero by the same argument as above.

According to Theorem \ref{thm hirsch}, the sets
$W_0:=\{\lambda\in W\colon \bm{s}(u):=\bm{t}( \lambda,u) \pitchfork 0\}$ and $W_i:=\{\lambda\in W\colon \bm{s}(u):=\bm{t}( \lambda,u)\big\vert_{\oPMb_{\mathbf{H}_i}} \pitchfork 0\}$ are residual. Thus we can find $\hat \lambda \in \bigcap W_i$ such that the canonical multisection of $E\to \oPMb^{1/r,\h}_{0,B,I}$ defined by $\bm{s}(u):=\bm{t}(\hat \lambda,u)$ is transverse to zero on the entire $\oPMb^{1/r,\h}_{0,B,I}$, and its restriction to $\partial^{PI}\oPMb^{1/r,\h}_{0,B,I}$ is also transverse to zero (and hence vanishes nowhere on $\partial^{PI}\oPMb^{1/r,\h}_{0,B,I}$).

\proofpart{3}{Independence of choices}
The last step is to prove, when $m$ is odd, the Euler number \eqref{eq  Euler number} is independent of the choice of the canonical multisection $\mathbf{s}$.

Let $\mathbf{s_1}$ and $\mathbf{s_2}$ be two nowhere-vanishing canonical multisections of $E$ over  $U_+\cup\partial\oPMh_{0,B,I}$, we need to show that
\begin{equation}\label{eq independence}
\int_{\oPMh_{0,B,I}}e(E ; \mathbf{s_1})=\int_{\oPMh_{0,B,I}}e(E ; \mathbf{s_2}).\end{equation}

The strategy is to use Lemma \ref{lem zero difference as homotopy}. We denote by $p :  \oPMh_{0,B,I}\times[0,1] \to \oPMh_{0,B,I}$ the projection. Assuming we can find a homotopy 
$$H \in C_m^\infty((U_+\cup\partial\oPMh_{0,B,I})\times[0,1],p^*E),~~H(x,t) = \mathbf{s}_t(x),~t=1,2$$ 
between $\mathbf{s_1}$ and $\mathbf{s_2}$ such that 
\begin{itemize}
    \item[--] $H$ is transverse to zero;
    \item[--] $H$ vanishes nowhere on $(U_+\cup\partial^{CB}\oPMh_{0,B,I})\times  [0,1]$;
    \item[--] $H\vert_{\text{bd}_{AI}\times [0,1]}=H\vert_{\text{bd}_{BI} \times [0,1]}$ for each pair of boundaries $\text{bd}_{AI}$ and $\text{bd}_{BI}$ identified by $PI$.
\end{itemize} 
Then by Lemma \ref{lem zero difference as homotopy} we have 
\begin{equation*}
    \begin{split}
    &\int_{\oPMh_{0,B,I}}e(E ; \mathbf{s_1})-\int_{\oPMh_{0,B,I}}e(E ; \mathbf{s_2})\\=&\# Z\left(H\big\vert_{\partial^{PI} \oPMh_{0,B,I}\times[0,1]}\right)\\
    =&\sum_{\substack{\text{$\text{bd}_{AI}$ and $\text{bd}_{BI}$ are}\\ \text{identified by $PI$}}} \left(\# Z\left(H|_{\text{bd}_{AI}\times[0,1]}\right)+\# Z\left(H|_{\text{bd}_{BI}\times[0,1]}\right)\right).
    \end{split}
\end{equation*}
The zeros of $H|_{\text{bd}_{AI}\times[0,1]}$ and $H|_{\text{bd}_{BI}\times[0,1]}$ appear in pairs, they contribute opposite numbers in the weighted count since, according to Theorem \ref{thm  PI boundaries paried}, the induced relative orientation of $E$ over $\text{bd}_{AI}$ and $\text{bd}_{BI}$ are opposite. 
Thus \eqref{eq independence} is proven. 

It remains to construct such a homotopy $H$. We shall imitate the analogous argument from \cite{PST14}. Start from the linear homotopy $H_0$ between $\mathbf{s_1}$ and $\mathbf{s_2}$, which satisfies the second and the third requirements. As in the previous part, we use Theorem \ref{thm hirsch} to perturb $H_0$ to transverse to zero. Actually, let $\bm{s}^j_{[u_i]}(u)$ be the multisections of $E\to \oPMh_{0,B,I}$ which we used to perturb $\mathbf{s}_0$ in the previous part, we perturb $H_0$ by the multisections
$$
\tilde{\bm{s}}^j_{[u_i]}(u,t):=t(1-t)\bm{s}^j_{[u_i]}(u)
$$
of $p^*E\to \oPMh_{0,B,I}\times [0,1]$ to get a multisection $H$ which is transverse to zero. The properties satisfied by $\bm{s}^j_{[u_i]}(u)$ guarantee the second and the third requirements are satisfied by $H$, while the factor $t(1-t)$ guarantees that $H$ is still a homotopy between $\mathbf{s_1}$ and $\mathbf{s_2}$.

\end{proof}

\begin{rmk}\label{rmk compare construction of section}
    When every boundary marking has twist $r-2$ (which happens in $\h=0$ theory we constructed), a different intersection theory on $\Mbar^{1/r}_{0,B,I}$ is defined in \cite{BCT2}, which we refer to as BCT theory.
    
    Instead of gluing other spaces to the type-BI boundaries of $\Mbar^{1/r}_{0,B,I}$ as in $\h=0$ theory, the canonical section $s^{\text{BCT}}$ in BCT theory is required to be the pullback of a canonical multisection over the moduli obtained by detaching the node and \emph{forgetting} the illegal half-node with twist $0$. The construction of $s^{\text{BCT}}$ in \cite{BCT2} focuses on the compatibility with the forgetful morphisms, while the construction of the canonical section in the above proof of Theorem \ref{thm intersection numbers well-defined} focuses on the compatibility with the identification of boundaries via $PI$.

    We will compare the BCT theory and the $\h=0$ theory in Section \ref{section:comparison}.
\end{rmk}

\subsection{Closed extended FJRW theories}\label{subsec:closed_ext_r_m}
In \cite{BCT_Closed_Extended,BCT2} it was observed that a certain extension of Witten's $r$-spin theory plays an important role in the open $r$-spin theory. This \emph{closed extended $r$-spin theory} is defined just like the usual genus $0$ closed $r$-spin theory, only that we allow one marking to have a negative twist of $-1.$ It was first observed in \cite{JKV2} that, as in the usual closed setting with only non-negative twists, also here $R^1\pi_*\mathcal{S}$ is an orbifold vector bundle in genus zero, so Witten's class $c_W$ can be defined by the formula~\eqref{eq Witten's class}. The \textit{closed extended $r$-spin correlators} are
\begin{equation}\label{eq:define closed ext}
\left\langle\tau^{a_1}_{d_1}\cdots\tau^{a_n}_{d_n}\right\rangle^{\frac{1}{r},\text{ext}}_0:=r\int_{\Mbar^{1/r}_{0,\{a_1, \ldots, a_n\}}} \hspace{-1cm} c_W \cap \psi_1^{d_1} \cdots \psi_n^{d_n}.
\end{equation}
 All these numbers are calculated in \cite{BCT_Closed_Extended}.

This definition easily extends to \emph{Fermat FJRW theories}, see \cite{FJR2,GKT2} and also Section \ref{sec:open_fjrw}.
In the context of this paper we will be interested in the following \emph{closed-extended FJRW correlators} for pairs
\[(x_1^r+\cdots x_m^r,\mathbb{Z}/r\mathbb{Z}).\]
We shall refer to them as \emph{closed extended $(r,m)$ theories.}
The \emph{closed extended $(r,m)$-correlators} are defined by
\begin{equation}\label{eq:closed ext_r_m}
\left\langle\tau^{a_1}_{d_1}\cdots\tau^{a_n}_{d_n}\right\rangle^{\frac{1}{r},\text{ext},m}_0:=r\int_{\Mbar^{1/r}_{0,\{a_1, \ldots, a_n\}}} \hspace{-1cm} c_W^m \cap \psi_1^{d_1} \cdots \psi_n^{d_n}. 
\end{equation}

\begin{rmk}
In fact, there is no difficulty in defining closed extended FJRW theories for Fermat polynomials 
\[x_1^{r_1}+\cdots x_m^{r_m},\] and an associated symmetry group, by considering the moduli spaces of $g=0$ curves with $r_i$-spin bundles for $i=1,\ldots, m$ allowing, for each $i=1,\ldots,m,$ at most one marking with a negative $i$th twist, and that twist must be $-1.$
Then the closed-extended FJRW class will just be the product of Witten classes for the $r_i$-spin bundles.
One can easily generalize, with the same arguments, the topological recursion relations found in \cite{BCT_Closed_Extended} to this more general setting.
\end{rmk}

\section{Open topological recursion relations}\label{sec:trr}
The open correlators satisfy topological recursion relations (TRRs).

The general form of these recursion relations allows writing certain open correlators as polynomial combinations of simpler open correlators, as well as the closed-extended correlators of Section \ref{subsec:closed_ext_r_m}.

For open $r$-spin theories, as we shall see in Section \ref{sec:computations}, these relations and a few additional vanishing results, allow all correlators to be calculated.  

\begin{thm}\label{thm TRR} \begin{itemize} The open correlators defined in Definition \ref{dfn correlator point insertion} satisfy:
\item[(a)] (TRR with respect to the boundary marking $b_1$) Suppose $l,k\ge 1$, then
\begin{equation}\label{eqtrr1}
		\begin{split}
		&\left\langle
		\tau^{a_1}_{d_1+1}\tau^{a_2}_{d_2}\dots\tau^{a_l}_{d_l}\sigma^{b_1}\sigma^{b_2}\dots\sigma^{b_k}
		\right\rangle_0^{\frac{1}{r},\text{o},\h,m}
		\\=&\sum_{\substack{s\ge 0\\-1\le a \le r-2}}\sum_{\substack{0 \le t_i \le \h\\\coprod_{j=-1}^{s}R_j=\{2,3,\dots,l\}\\\coprod_{j=0}^{s}T_j=\{2,3,\dots,k\}\\ \{(R_j,T_j,t_j)\}_{1\le j \le s}\text{ unordered}}}(-1)^s\left\langle
		\tau^{a}_{0}\tau^{a_1}_{d_1}\prod_{i\in R_{-1}}\tau^{a_i}_{d_i}\prod_{j=1}^{s}\tau^{t_j}_{0}
		\right\rangle_0^{\frac{1}{r},\text{ext},m}\\
		&\cdot\left\langle
		\tau^{r-2-a}_{0}\sigma^{b_1}\prod_{i\in R_0}\tau^{a_i}_{d_i}\prod_{i\in T_0}\sigma^{b_i}
		\right\rangle_0^{\frac{1}{r},\text{o},\h,m}
		\prod_{j=1}^{s}\left\langle
		\sigma^{r-2-2t_j}\prod_{i\in R_j}\tau^{a_i}_{d_i}\prod_{i\in T_j}\sigma^{b_i}
		\right\rangle_0^{\frac{1}{r},\text{o},\h,m}.
		\end{split}
		\end{equation}

\item[(b)] (TRR with respect to the internal marking $a_2$) Suppose $l\ge 2$, then

\begin{equation}\label{eqtrr2}
		\begin{split}
		&\left\langle
		\tau^{a_1}_{d_1+1}\tau^{a_2}_{d_2}\dots\tau^{a_l}_{d_l}\sigma^{b_1}\sigma^{b_2}\dots\sigma^{b_k}
		\right\rangle_0^{\frac{1}{r},\text{o},\h,m}
		\\=&\sum_{\substack{s\ge 0\\-1\le a \le r-2}}\sum_{\substack{0 \le t_i \le \h\\ \coprod_{j=-1}^{s}R_j=\{3,4,\dots,l\}\\\coprod_{j=0}^{s}T_j=\{1,2,\dots,k\}\\ \{(R_j,T_j,t_j)\}_{1\le j \le s}\text{ unordered}}}(-1)^s\left\langle
		\tau^{a}_{0}\tau^{a_1}_{d_1}\prod_{i\in R_{-1}}\tau^{a_i}_{d_i}\prod_{j=1}^{s}\tau^{t_j}_{0}
		\right\rangle_0^{\frac{1}{r},\text{ext},m}\\
		&\cdot\left\langle
		\tau^{r-2-a}_{0}\tau^{a_2}_{d_2}\prod_{i\in R_0}\tau^{a_i}_{d_i}\prod_{i\in T_0}\sigma^{b_i}
		\right\rangle_0^{\frac{1}{r},\text{o},\h,m}
		\prod_{j=1}^{s}\left\langle
		\sigma^{r-2-2t_j}\prod_{i\in R_j}\tau^{a_i}_{d_i}\prod_{i\in T_j}\sigma^{b_i}
		\right\rangle_0^{\frac{1}{r},\text{o},\h,m}.
		\end{split}
		\end{equation}

\end{itemize}
\end{thm}

\begin{rmk}\label{rmk central side contribution}
	    Each term in the TRRs is a product of a closed-extended contribution and several (at least one) open contributions, as shown in Figure \ref{fig trr}.  The second bracket in each term is special among the open contributions: it comes from the moduli containing the chosen marking, \textit{i.e.} $a_2$ in \eqref{eqtrr1} and $b_1$ in \eqref{eqtrr2}, and it corresponds to the central irreducible component in Figure \ref{fig trr}. In what follows, we refer to this special open contribution by \textit{central contribution}, and refer to the other open contributions by \textit{side contributions}.
	\end{rmk}

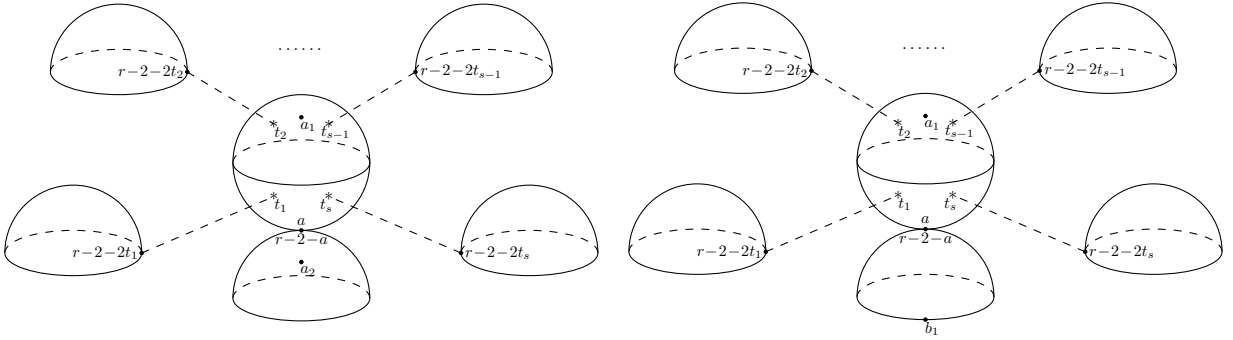
\begin{figure}[h]
    \centering

\begin{tikzpicture}[scale=0.6, every node/.style={scale=0.6}]
\draw (0,0) circle (1.5);
\draw (-1.5,0) arc (180:360:1.5 and 0.5);
\draw[dashed](1.5,0) arc (0:180:1.5 and 0.5);

\draw (-1.5,-3) arc (180:360:1.5 and 0.5);
\draw[dashed](1.5,-3) arc (0:180:1.5 and 0.5);
\draw(1.5,-3) arc (0:180:1.5);

\node at (0,-1.5) [circle,fill,inner sep=1pt]{};
\node at (0,-1.3) {\small$a$};
\node at (0,-1.7) {\small${r\!-\!2\!-\!a}$};

\draw (-6.5,-2) arc (180:360:1.5 and 0.5);
\draw[dashed](-3.5,-2) arc (0:180:1.5 and 0.5);
\draw(-3.5,-2) arc (0:180:1.5);

\node at (-3.5,-2) [circle,fill,inner sep=1pt]{};
\node at (-4.3,-2) {\small $r\!-\!2\!-\!2t_{1}$};

\draw (-5.5,2) arc (180:360:1.5 and 0.5);
\draw[dashed](-2.5,2) arc (0:180:1.5 and 0.5);
\draw(-2.5,2) arc (0:180:1.5);

\node at (-2.5,2) [circle,fill,inner sep=1pt]{};
\node at (-3.3,2) {\small $r\!-\!2\!-\!2t_2$};

\draw (2.5,2) arc (180:360:1.5 and 0.5);
\draw[dashed](5.5,2) arc (0:180:1.5 and 0.5);
\draw(5.5,2) arc (0:180:1.5);

\node at (2.5,2) [circle,fill,inner sep=1pt]{};
\node at (3.5,2) {\small $r\!-\!2\!-\!2t_{s-1}$};

\draw (3.5,-2) arc (180:360:1.5 and 0.5);
\draw[dashed](6.5,-2) arc (0:180:1.5 and 0.5);
\draw(6.5,-2) arc (0:180:1.5);

\node at (3.5,-2) [circle,fill,inner sep=1pt]{};
\node at (4.3,-2) {\small $r\!-\!2\!-\!2t_{s}$};

\node at (-0.6,0.8){*};
\node at (-0.45,0.65) {\small $t_{2}$};
\draw[dashed] (-2.5,2) -- (-0.6,0.85);

\node at (0.6,0.8){*};
\node at (0.75,0.65) {\small $t_{s-1}$};
\draw[dashed] (2.5,2) -- (0.6,0.85);

\node at (-0.6,-0.8){*};
\node at (-0.45,-0.95) {\small $t_{1}$};
\draw[dashed] (-3.5,-2) -- (-0.6,-0.75);

\node at (0.6,-0.8){*};
\node at (0.55,-0.95) {\small $t_{s}$};
\draw[dashed] (3.5,-2) -- (0.6,-0.75);

\node at (0,1)[circle,fill,inner sep=1pt]{};
\node at (0.15,0.8) {\small $a_1$};

\node at (0,-2.2)[circle,fill,inner sep=1pt]{};
\node at (0.15,-2.4) {\small $a_2$};

\node at (0,2.5) {$\dots\dots$};

\node at (0.15,-3.82) {};
\end{tikzpicture} 
\hfill
\begin{tikzpicture}[scale=0.6, every node/.style={scale=0.6}]
\draw (0,0) circle (1.5);
\draw (-1.5,0) arc (180:360:1.5 and 0.5);
\draw[dashed](1.5,0) arc (0:180:1.5 and 0.5);

\draw (-1.5,-3) arc (180:360:1.5 and 0.5);
\draw[dashed](1.5,-3) arc (0:180:1.5 and 0.5);
\draw(1.5,-3) arc (0:180:1.5);

\node at (0,-1.5) [circle,fill,inner sep=1pt]{};
\node at (0,-1.3) {\small$a$};
\node at (0,-1.7) {\small${r\!-\!2\!-\!a}$};

\draw (-6.5,-2) arc (180:360:1.5 and 0.5);
\draw[dashed](-3.5,-2) arc (0:180:1.5 and 0.5);
\draw(-3.5,-2) arc (0:180:1.5);

\node at (-3.5,-2) [circle,fill,inner sep=1pt]{};
\node at (-4.3,-2) {\small $r\!-\!2\!-\!2t_1$};

\draw (-5.5,2) arc (180:360:1.5 and 0.5);
\draw[dashed](-2.5,2) arc (0:180:1.5 and 0.5);
\draw(-2.5,2) arc (0:180:1.5);

\node at (-2.5,2) [circle,fill,inner sep=1pt]{};
\node at (-3.3,2) {\small $r\!-\!2\!-\!2t_2$};

\draw (2.5,2) arc (180:360:1.5 and 0.5);
\draw[dashed](5.5,2) arc (0:180:1.5 and 0.5);
\draw(5.5,2) arc (0:180:1.5);

\node at (2.5,2) [circle,fill,inner sep=1pt]{};
\node at (3.5,2) {\small $r\!-\!2\!-\!2t_{s-1}$};

\draw (3.5,-2) arc (180:360:1.5 and 0.5);
\draw[dashed](6.5,-2) arc (0:180:1.5 and 0.5);
\draw(6.5,-2) arc (0:180:1.5);

\node at (3.5,-2) [circle,fill,inner sep=1pt]{};
\node at (4.3,-2) {\small $r\!-\!2\!-\!2t_{s}$};

\node at (-0.6,0.8){*};
\node at (-0.45,0.65) {\small $t_{2}$};
\draw[dashed] (-2.5,2) -- (-0.6,0.85);

\node at (0.6,0.8){*};
\node at (0.75,0.65) {\small $t_{s-1}$};
\draw[dashed] (2.5,2) -- (0.6,0.85);

\node at (-0.6,-0.8){*};
\node at (-0.45,-0.95) {\small $t_{1}$};
\draw[dashed] (-3.5,-2) -- (-0.6,-0.75);

\node at (0.6,-0.8){*};
\node at (0.55,-0.95) {\small $t_{s}$};
\draw[dashed] (3.5,-2) -- (0.6,-0.75);

\node at (0,1)[circle,fill,inner sep=1pt]{};
\node at (0.15,0.8) {\small $a_1$};

\node at (0,-3.5)[circle,fill,inner sep=1pt]{};
\node at (0.15,-3.7) {\small $b_1$};

\node at (0,2.5) {$\dots\dots$};
\end{tikzpicture} 

    \caption{The figure on the left corresponds to a term in TRR \eqref{eqtrr1} with respect to a boundary marking $b_1$; The figure on the right corresponds to a term in TRR  \eqref{eqtrr2} with respect to an internal marking $a_2$.}
    \label{fig trr}
\end{figure}

The correlators also satisfy the following vanishing relation:
\begin{prop}\label{prop:vanish small internal}
If $l\ge 1$, $2l+k\ge 4$ and $a_1\le \h$ , we have
    $$
\left\langle
\begin{array}{c}
\hfill  a_1 \quad\hfill   a_2 \psi^{d_2}\hfill \quad a_3 \psi^{d_3}\hfill\quad\dots\quad\hfill a_l \psi^{d_l}\hfill\null\\
\hfill b_1\quad \hfill b_2\quad\hfill b_3\hfill\quad\dots\quad \hfill b_k\hfill\null\\
\end{array}
\right\rangle_0^{\frac{1}{r},\text{o},\h,m}=0.
$$
 
In particular, if $l\ge 1$, $2l+k\ge 4$ and $\min \{a_i\}_{i=1,\dots,l}\le \h$ , we have
    $$
\left\langle
\begin{array}{c}
\hfill  a_1 \quad\hfill   a_2\hfill\quad\dots\quad\hfill a_l\hfill\null\\
\hfill b_1\quad \hfill b_2\quad\hfill b_3\hfill\quad\dots\quad \hfill b_k\hfill\null\\
\end{array}
\right\rangle_0^{\frac{1}{r},\text{o},\h,m}=0.
$$
\end{prop}
The case $a_1=0$ can be thought of as the \emph{string equation}, and it takes a much simpler form than the string equation of \cite{PST14,BCT2} for example. 

The rest of this section is devoted to proving the above claims.
\subsection{Proof of Proposition \ref{prop:vanish small internal}}

\begin{proof}[Proof of Proposition \ref{prop:vanish small internal}]
   Let $\mathbf{G}\in \sGPI^{r,\h}_{0,B,I}$ be an $(r,\h)$-graph (note that $\Mbar_\mathbf{G}\subseteq \Mbar^{1/r,\h}_{0,B,I}$ as a connected component), we assume $a_1\in T^I(\Gamma)$ for $\Gamma\in V(\mathbf{G})$. There are two possibilities:
    \begin{enumerate}
        \item The graded $r$-spin graph $\Gamma$ consists of a single open vertex with a single internal tail $a_1$ and a single boundary tail $\hat b \in B'(\mathbf{G})$. In this case we say $\mathbf{G}$ is \emph{reflected}.
        \item The graded $r$-spin graph $\Gamma$ is not as in the previous item. In this case we say $\mathbf{G}$ is \emph{original}.
    \end{enumerate}
    We will construct a bijection 
    $$\nu\colon \{\mathbf{G}\in \sGPI^{r,\h}_{0,B,I}\colon \mathbf{G}\text{ original}\}\to\{\mathbf{G}\in \sGPI^{r,\h}_{0,B,I}\colon\mathbf{G}\text{ reflected}\}.$$ 
    For an original graph $\mathbf{G}_1$, we can construct the reflected graph $\mathbf{G}_2=\nu(\mathbf{G}_1)$ in the following way. We set $$V(\mathbf{G}_2):=\left(V(\mathbf{G}_1)\setminus\{\Gamma\}\right)\sqcup\{\hat \Gamma,\Gamma_a\},$$
    where 
    \begin{itemize}
        \item $\Gamma_a$ consists of a single open vertex with a single internal tail $a_1$ and a single boundary tail $\hat b \in B'(\mathbf{G}_2)$;
        \item $\hat \Gamma$ is the graph obtained by replacing the internal tail $a_1$ of $\Gamma$ with a new internal tail $\hat a \in I'(\mathbf{G}_2)$ with the same twist as $a_1$;
        \item we add a dashed line $( \hat a, \hat b )\in E(\mathbf{G}_2)$ (note that $a_1\le \h$ by assumption).
    \end{itemize}
    By inverting this construction we can associate, for a reflected graph $\mathbf{G}_2$, the original graph $\mathbf{G}_1=\nu^{-1}(\mathbf{G}_2).$ Since $\Mbar_{\Gamma_a}$ is a point, we have a natural isomorphism  \[\operatorname{ref}_{\mathbf{G}_1}:\Mbar_{\mathbf{G}_1}\to\Mbar_{\nu(\mathbf{G}_1)}\]
    which can be lifted to isomorphisms between the corresponding vector bundles $\mathcal W$ and $\mathbb L_i$.
    This isomorphism reverses the canonical relative orientation \eqref{eq orientation point insertion} of $\mathcal W$ since $E(\mathbf{G}_2)=E(\mathbf{G}_2)+1$, and it preserves the complex orientation of $\mathbb L_i$, so it reverses (recall $m$ is always odd) the canonical relative orientation of
    $$
    E:=\mathcal W^{\oplus m} \oplus \bigoplus_{i=2}^{l} \mathbb L_i^{\oplus  d_i}.
    $$
    The collection of these isomorphisms and their inverses (note that each of them is defined on a connected component of $\Mbar^{1/r,\h}_{0,B,I}$) give an orientation reversing involution
    $$
        \operatorname{ref}\colon \Mbar^{1/r,\h}_{0,B,I}\to \Mbar^{1/r,\h}_{0,B,I}.
    $$
     Let $\mathbf{s}$ be a transverse canonical multisection of $E$, then we can write the correlator as $e(E;\mathbf{s}).$
     We have
     \begin{align*}&\left\langle
\begin{array}{c}
\hfill  a_1 \quad\hfill   a_2 \psi^{d_2}\hfill \quad a_3 \psi^{d_3}\hfill\quad\dots\quad\hfill a_l \psi^{d_.}\hfill\null\\
\hfill b_1\quad \hfill b_2\quad\hfill b_3\hfill\quad\dots\quad \hfill b_k\hfill\null\\
\end{array}
\right\rangle_0^{\frac{1}{r},\text{o},\h,m}=
e(E;\mathbf{s})\\&\qquad\quad=-e(E;\operatorname{mir}^*\mathbf{s})=-
\left\langle
\begin{array}{c}
\hfill  a_1 \quad\hfill   a_2 \psi^{d_2}\hfill \quad a_3 \psi^{d_3}\hfill\quad\dots\quad\hfill a_l \psi^{d_.}\hfill\null\\
\hfill b_1\quad \hfill b_2\quad\hfill b_3\hfill\quad\dots\quad \hfill b_k\hfill\null\\
\end{array}
\right\rangle_0^{\frac{1}{r},\text{o},\h,m}.
\end{align*}
The second equality follows from the fact that $\operatorname{ref}$ reverses orientation, and the last follows from the observation that the pullback of a canonical multisection by $\operatorname{ref}$ is still canonical.
\end{proof}

\subsection{Proof of Theorem \ref{thm TRR}}\label{subsec proof trr}
We shall construct a canonical section $t$ of ${\mathbb L}_1$ and show that its zero locus consists of internal strata ${\oPMb}^{1/r,\h}_\Gamma$ which parameterizes surfaces with at least one internal node. We will then show that the number of zeros of the intersection problem we consider is the sum of products of open and closed contributions. 

To construct the section $t$, we need the following tracking map.
\begin{dfn}\label{dfn track point} 
For genus-zero $(r,\h)$-graph $\mathbf{G}$, we define maps which track markings in the point insertion procedure.  We first define a map
$$
\theta\colon \bigsqcup_{\Gamma\in V(\mathbf{G})}\left(T^I(\Gamma)\backslash H^{CB}(\Gamma)\right)\sqcup \bigsqcup_{\Gamma\in V(\mathbf{G})}T^B(\Gamma)\to 2^{I(\mathbf{G})\sqcup B(\mathbf{G})}
$$
in the following way.
\begin{itemize}
    \item For $x\in I(\mathbf{G})\sqcup B(\mathbf{G})$ we define $$\theta(x):=\{x\}.$$
    \item For $x\in I'(\mathbf{G})\sqcup B'(\mathbf{G})$, let $\delta(x)$ be the element paired with $x$ by a dashed line, and $\Gamma_{\delta(x)}$ be the stable graded $r$-spin graph such that
    $$
    \delta(x)\in T^I(\Gamma_{\delta(x)})\sqcup T^B(\Gamma_{\delta(x)}).
    $$
    Let $\mathbf{\hat{G}}_x$ and $\mathbf{\hat{G}}_{\delta(x)}$ be the two connected subgraphs obtained by removing the edge connecting $x$ and $\delta(x)$ in the connected genus-zero graph $\mathbf{\hat{G}}$ (see Definition \ref{def rh graphs}), where $\mathbf{\hat{G}}_{\delta(x)}$ is the subgraph containing the vertex $\Gamma_{\delta(x)}$. We define $$\theta(x):=\bigsqcup_{\Gamma\in V(\mathbf{\hat{G}}_{\delta(x)})}\left(T^I(\Gamma)\cap I(\mathbf{G})\right)\sqcup \left(T^B(\Gamma)\cap B(\mathbf{G})\right).$$
\end{itemize}
Since $\mathbf{\hat{G}}$ is connected, for any $\Gamma\in V(\mathbf{G})$ we have
$$
 \bigsqcup_{x\in \left(T^I(\Gamma)\backslash H^{CB}(\Gamma)\right)\sqcup T^B(\Gamma)}\theta(x)=I(\mathbf{G})\sqcup B(\mathbf{G}).
$$
We define 
$$
\zeta_\Gamma\colon I(\mathbf{G})\sqcup B(\mathbf{G}) \to \left(T^I(\Gamma)\backslash H^{CB}(\Gamma)\right)\sqcup T^B(\Gamma)
$$
to be the unique map satisfying
$$
y\in \theta\left(\zeta_\Gamma(y)\right),\quad \forall y\in I(\mathbf{G})\sqcup B(\mathbf{G}).
$$
\end{dfn}

\begin{proof}[Proof of Theorem \ref{thm TRR}]
We assume $\frac{2\sum_{i=1}^l a_l +\sum_{i=1}^k b_k-r+2}{r}=2l-k+3$, otherwise both sides of the equation vanish. The prove \eqref{eqtrr1} and \eqref{eqtrr2} by the same argument: we denote by $c$ the marking that the TRR respects to, \textit{i.e.} $c$ is either $a_2$ in \eqref{eqtrr1} or $b_1$ in \eqref{eqtrr2}.

We define a canonical section $t \in C^\infty(\Mbar^{\frac{1}{r},\h}_{0,B,I},{\mathbb L}_1)$ as follows.  For $\mathbf{G}\in \sGPI^{r,\h}_{0,B,I}$, we assume $a_1\in T^I(\Gamma)$ for some $\Gamma \in V(\mathbf{G})$.
For a smooth $r$-spin disk $C\in \mathcal M^{}_\Gamma$, we denote by $z_1$ the marking on $C$ corresponding to $a_1$, and by $z_{\zeta,C}$ the marking on $C$ corresponding to $\zeta_{\Gamma}(c)$ (see Definition \ref{dfn track point}). We identify the preferred half $\Sigma$ with the upper half-plane and set
\begin{equation}\label{eq trr section}
t\left(C\right) = dz\left.\left(\frac{1}{z - z_{\zeta,C}}-\frac{1}{z-\bar{z}_1}\right)\right|_{z = z_1} \in T_{z_1}^*\Sigma = T_{z_1}^*C.
\end{equation}
This fiber-by-fiber definition glues to a section of $\mathbb{L}_1\to\mathcal M^{}_\Gamma$, which is the pullback of a (constant) non-zero section from the zero-dimensional moduli space $\oCM_{0,1,1}$. This section is easily seen to extend to a smooth global section $t$ over $\Mbar^{}_\Gamma,$ which is again pulled back from $\oCM_{0,1,1}$: for a graded $r$-spin disk $C$ that is not necessarily smooth,  
let $\varphi_C$ be the unique meromorphic differential on $C$ with simple poles at $\bar{z}_1$ and $z_{\zeta,C}$ and at no other marked points or smooth points, such that the residue equals $1$ at $z_{\zeta,C}$ and equals $-1$ at $\bar{z}_1$, and the residues at every pair of half-nodes sum to $0$.
Then $t(C)$ is the evaluation of $\varphi_C$ at $z_1\in C.$ 

Since ${\mathbb L}_1\to\Mbar^{}_\mathbf{G}$ is canonically identified with the pullback of ${\mathbb L}_1\to\Mbar^{}_\Gamma$ by the projection map $\pi_{\mathbf{G},\Gamma}$ defined in Remark \ref{rmk aut trivial} we get a section $t \in C^\infty(\Mbar^{}_\mathbf{G},{\mathbb L}_1)$. We can define such sections for all $\mathbf{G}\in \sGPI^{r,\h}_{0,B,I}$, thus we have defined a global section $t\in C^\infty(\Mbar^{1/r,\h}_{0,B,I},{\mathbb L}_1)$.

\begin{lem}
    The restriction of $t$ to $\oPMh_{0,B,I}$ is a canonical section of ${\mathbb L}_1\to \oPMh_{0,B,I}$.
\end{lem}
\begin{proof}
    According to Definition \ref{def canonical for L_i}, we need to check that, for every pair type-AI and type-BI boundaries $\text{bd}_{AI}$ and $\text{bd}_{BI}$ identified by $PI$, we have $t\vert_{\text{bd}_{AI}}=t\vert_{\text{bd}_{BI}}$. 
    
    For a point $p$ on a type-BI boundary $\text{bd}_{BI}$, we can represent it by an $(r,\h)$-disk (see Definition \ref{dfn rh disk}) collecting the graded $r$-spin disks $C_0,C_1,\dots,C_i$, where $C_0$ has a type-BI node $n$. By normalizing $n$ we get two components $N_1$ and $N_2$. We assume the half-node $n_2$ on the $N_2$ is legal. Let $\text{bd}_{AI}$ be the type-AI boundary that paired with $\text{bd}_{BI}$, and let $p'\in\text{bd}_{AI}$ be the point that is identified with $p$ under $\sim_{PI}$, we show that $t(p)=t(p')$. By the construction of $PI$ (see \cite[Theorem 4.12]{TZ1} for details), the point $p'$ can be represented by the $(r,\h)$-disk collection graded $r$-spin disks $D_1,D_2,C'_1,C'_2,\dots,C'_i$, where $D_2=N_2$, $C'_i=C_1$ for $j=1,2,\dots,i$, and $D_1$ has a type-AI node $\hat n$ such the normalization of $\hat n$ gives two components $D'_1=N_1$ and $D''_1$, the later one has only one internal marking $\hat n_1$ and one boundary half-node $\hat n_2$. 
    
    If $z_1 \in C_j$ for some $1\le j \le i$, since $z_{\zeta,C_j}$ and $z_{\zeta,C'_j}$ represent the same point on $C_j$ and $C'_j$, we have $t(C_i)=t(C'_i)$, hence $t(p)=t(p')$.

    If $z_1\in C_0$, $z_1$ and $z_{\zeta,C_0}$ are on the same component $N_1$ (respectively, $N_2$), then $z_{\zeta,C_0}$ and $z_{\zeta,D_1}$ (respectively, $z_{\zeta,D_2}$) represent the same point on $N_1$ (respectively, $N_2$) and $D'_1$ (respectively, $D_2$). Then $t(p)=t(p')$ since the unique meromorphic differentials coincide.

    In the case $z_1\in C_0$ but $z_1$ and $z_{\zeta,C_0}$ are on different components, we assume $z_1\in N_2$ and $z_{\zeta,C_0}\in N_1$ (the other case is similar). Then by Definition \ref{dfn track point} $z_{\zeta,D_2}$ represents the half-node $n_2$ on $D_2=N_2$. Note that when restricted to $N_2$, the unique meromorphic differential $\varphi_{C_0}$ on $C_0$ has a pole at $n_2$ with residue $1$, thus $\varphi_{C_0}\vert_{N_2}$ coincides with $\varphi_{D_2}$, and hence $t(p)=t(p')$.
\end{proof}
\begin{rmk}As in Remark \ref{rmk glued section canonical} the section $t$ descends to a global section of $\widetilde{\mathbb{L}}_i\to\widetilde{\mathcal M}^{\frac{1}{r},\h}_{0,B,I},$ and it is useful to think of it as a section over the glued moduli.
\end{rmk}

For each $(r,\h)$-graph $\mathbf{G}\in \sGPI^{r,\h}_{0,B,I}$, let $U_\mathbf{G} \subseteq \partial \mathbf{G}$ be the collection of $(r,\h)$-graphs $\mathbf{F}$ satisfying:
\begin{itemize}
    \item only one graded $r$-spin graph $\Delta\in V(\mathbf{F})$ is non-smooth;
    \item $\Delta$ consists of one open vertex $v_\Delta^o$ and one closed vertex $v_\Delta^c$ joined by a unique edge $e$;
    \item $a_1$ is a internal tail of $v_\Delta^c$ and $\zeta_{\Delta}(c)$ is a tail of $v_\Delta^o$.
\end{itemize}   We write 
\begin{equation}\label{eq:all open vertex}
    \oPMb_{\mathbf{F}^o}:=\oPMb_{v^o_\Delta}\times \prod_{\Sigma\in V(F)\backslash \{\Delta\}}\oPMb_\Sigma.
\end{equation}

The morphism 
$\text{Detach}_e:\oPMb_\Delta \to \oPMb_{v^o_\Delta} \times \oCM_{v^c_\Delta}$ induces a
morphism 
$
\text{Detach}_e:\oPMb_\mathbf{F} \to \oPMb_{\mathbf{F}^o} \times \oCM_{v^c_\Delta}
$, which is generically one-to-one.

\begin{lemma}\label{lem zero locus trr section local}
The zero locus of $t$ in $\oPMb_\mathbf{G}$ is $\bigcup_{\mathbf{F} \in {U_{\mathbf{G}}}} \oPMb_{\mathbf{F}}.$ The multiplicity of $\oPMb_\mathbf{F}$ in the zero locus is $r$. Moreover, we denote by $o_\mathbf{G}$ and $o_{v^c_\Delta}$ the canonical relative orientation of the Witten bundles ${\mathcal{W}}$ over $\Mbar_\mathbf{G}$ and $\Mbar_{v^c_\Delta}$. Let  $o_{\mathbf{F}^o}$ be the relative orientation  of  ${\mathcal{W}}$ over $\oCM_{\mathbf{F}^o}$ defined by  
$$o_{\mathbf{F}^o}:=(-1)^{\lvert V(\mathbf{G})\rvert-1}o_{v_{\Delta}^o}\wedge\bigwedge_{\Sigma\in V(\mathbf{F})\backslash \{\Delta\}}o_\Sigma.$$ 
We denote by $o_\mathbf{F}$ the relative orientation  for ${\mathcal{W}}\to\oCM_\mathbf{F}$ induced by the section $t$, then we have $${o}_\mathbf{F}=\text{Detach}_e^*({o}_{\mathbf{F}^o}\boxtimes {o}_{v^c_\Delta}).$$
\end{lemma}
\begin{proof}
    The case where $\lvert V(\mathbf{G})\rvert=1$ is \cite[Lemma 4.13]{BCT2}. 
    Suppose $\lvert V(\mathbf{G})\rvert>1$, and assume $a_1\in T^I(\Gamma)$ for some $\Gamma \in V(\mathbf{G})$.   Recall that the section $t$ is pulled back from a section of $\mathbb L_1\to\oPMb_\Gamma$ via the projection map $\pi_{\mathbf{G},\Gamma}\colon \Mbar_\mathbf{G}\to \Mbar_\Gamma$, hence the zero locus of $t$ in $\oPMb_\mathbf{G}$  is the product of $\prod_{\Sigma\in V(\mathbf{F})\backslash \{\Gamma\}}\oPMb_\Sigma$ and the zero locus of $t$ in $\oPMb_\Gamma$. The lemma now follows again from \cite[Lemma 4.13]{BCT2}, applied to $\oPMb_\Gamma.$
\end{proof}
 For two multisets of internal twists $I^o,I^c$ and a multiset of boundary twists $B^o$, we define 
 $
 U_{B^o,I^o}^{I^c}
 $ to be the collection of genus-zero $(r,\h)$-graphs $\mathbf{H}$ such that 
 \begin{itemize}
     \item only one graded $r$-spin graph $\Delta\in V(\mathbf{H})$ is non-smooth, $\Delta$ consists of one open vertex $v^o_{\Delta,\mathbf{H}}$ and one closed vertex $v^c_{\Delta,\mathbf{H}}$ joined by a unique edge $e$;
     \item $T^I(v^c_{\Delta,\mathbf{H}})=I^c$, $I(\mathbf{H})=I^o\sqcup I^c$, $B(\mathbf{H})=B^o$.
 \end{itemize}
 Write 
\begin{equation}\label{eq:closed plus centre}
    \hat{\mathcal M}_{B^o,I^o}^{I^c}:=\bigsqcup_{\mathbf{H}\in U_{B^o,I^o}^{I^c}} \oPMb_\mathbf{H}.
\end{equation}
The closed vertices $v^c_{\Delta,\mathbf{H}}$ are the same for all $\mathbf{H}\in U_{B^o,I^o}^{I^c}$. Denote by $a$ the twist of $e$ on the $v^c_{\Delta,\mathbf{H}}$ side. Write $$ \oPMb_{\mathbf{H}^o}:=\oPMb_{v^o_{\Delta,\mathbf{H}}}\times \prod_{\Sigma\in V(\mathbf{H})\backslash \{\Delta\}}\oPMb_\Sigma.$$
 Note that $\oPMb_{\mathbf{H}^o}$ is a connected component in $\oPMb_{0,B^o,I^o\sqcup \{r-2-a\}}$ for each $\mathbf{H}\in U_{B^o,I^o}^{I^c}$; moreover, every connected component in $\oPMb_{0,B^o,I^o\sqcup \{r-2-a\}}$ can be realized as $\oPMb_{\mathbf{H}^o}$ for an unique $\mathbf{H}\in U_{B^o,I^o}^{I^c}$.
 Therefore the detaching morphisms on  $\oPMb_\mathbf{H}\subseteq \hat{\mathcal M}_{B^o,I^o}^{I^c}$ induce  generically one-to-one morphisms $$\text{Detach}_e:\hat{\mathcal M}_{B^o,I^o}^{I^c} \to \oPMb_{0,B^o,I^o\sqcup \{r-2-a\}} \times \oCM_{v^c_{\Delta,\mathbf{H}}}.$$
\begin{lemma}\label{lem zero locus trr section global}
    The zero locus of $t$ in $\oPMh_{0,B,I}$ can be written as
    \begin{equation}\label{eq:total zero locus}
        \bigcup_{\substack{s\ge 0,~ 0\le t_i\le \h,\\ \bigsqcup_{j=-1}^s I_j=I,~\bigsqcup_{j=0}^s B_j=B,\\ a_1\in I_{-1},~c\in I_0\sqcup B_0,\\ \{(t_j,I_j,B_j)\}_{1\le j \le s}\text{ unordered}}}\left( \hat{\mathcal M}_{B_0,I_0}^{I_{-1}\cup\{t_1,\dots,t_s\}}\times\prod_{i=1}^s \oPMb_{0,B_i\sqcup \{r-2-2t_i\},I_i}\right).
    \end{equation}
    The multiplicity of each $\hat{\mathcal M}_{B_0,I_0}^{I_{-1}\cup\{t_1,\dots,t_s\}}\times\prod_{i=1}^s \oPMb_{0,B_i\sqcup \{r-2-2t_i\},I_i}$ in the zero locus is $r$. 
    
    Moreover,  we have an induced detach morphism  \begin{equation*}
        \begin{split}
            \text{Detach}_e:&\hat{\mathcal M}_{B_0,I_0}^{I_{-1}\cup\{t_1,\dots,t_s\}}\times\prod_{i=1}^s \oPMb_{0,B_i\sqcup \{r-2-2t_i\},I_i} \\ \to& \oPMb_{0,B_0,I_0\sqcup \{r-2-a\}} \times\prod_{i=1}^s \oPMb_{0,B_i\sqcup \{r-2-2t_i\},I_i}\times \oCM_{v^c_{}}.
        \end{split}
\end{equation*}
     We denote by $o_{0,B,I}$, $o_{v^c}$ the canonical relative orientation of the Witten bundles ${\mathcal{W}}$ over $\oCM_{0,B,I}$, $\oCM_{v^c}$. Let  $o_{open}$ be the relative orientation  of  ${\mathcal{W}}$ over $$\oPMb_{0,B_0,I_0\sqcup \{r-2-a\}}\times \prod_{i=1}^s \oPMb_{0,B_i\sqcup \{r-2-2t_i\},I_i}$$ defined by  
$$o_{open}:=(-1)^{s}o_{0,B_0,I_0\sqcup \{r-2-a\}}\wedge\bigwedge_{i=1}^s o_{0,B_i\sqcup \{r-2-2t_i\},I_i}.$$ 
We denote by $o_{ind}$ the relative orientation  for ${\mathcal{W}}\to \hat{\mathcal M}_{B_0,I_0}^{I_{-1}\cup\{t_1,\dots,t_s\}}\times\prod_{i=1}^s \oPMb_{0,B_i\sqcup \{r-2-2t_i\},I_i}^{1/r,\h}$ induced by the section $t$, then we have $${o}_{ind}=\text{Detach}_e^*({o}_{open}\boxtimes {o}_{v^c}).$$

\end{lemma}
\begin{proof}
    We expand \eqref{eq:total zero locus} by \eqref{eq:closed plus centre} and \eqref{eq def moduli point insertion}. Each term in this expansion is of the form $\oPMb_\mathbf{F}$ for some $\mathbf{F}\in U_\mathbf{G}$, where $\mathbf{G}\in \sGPI^{r,\h}_{0,B,I}$. On the other hand, for each $\mathbf{G}\in \sGPI^{r,\h}_{0,B,I}$ and $\mathbf{F}\in U_\mathbf{G}$,  the term $\oPMb_\mathbf{F}$ appears in this expansion exactly once. Thus, all claims of the lemma follow from Lemma \ref{lem zero locus trr section local}.
\end{proof}
To simplify the notation, we rewrite \eqref{eq:total zero locus} as $\bigcup_{\alpha\in S_{\text{ind}}}\mathcal Z_\alpha$, 
where $S_{\text{ind}}$ is the index set over which the union in \eqref{eq:total zero locus} is taken, \textit{i.e.} $$
S_{\text{ind}}:=\left\{(s,\{t_i\},\{I_i\},\{B_i\})\left\vert \substack{s\ge 0,~ 0\le t_i\le \h,\\ \bigsqcup_{j=-1}^s I_j=I,~\bigsqcup_{j=0}^s B_j=B,\\ a_1\in I_{-1},~c\in I_0\sqcup B_0,\\ \{(t_j,I_j,B_j)\}_{1\le j \le s}\text{ unordered}} \right.\right\}.
$$ Note that for each $\alpha \in S_{\text{ind}}$, the real codimension of $\mathcal Z_\alpha\subset \oPMh_{0,B,I}$ is two. Moreover, for $\alpha\ne \beta$, the real codimension of $\mathcal Z_\alpha \cap \mathcal Z_\beta \subset \oPMh_{0,B,I}$ is at least four.

We write $E_1=\mathbb L_1$ and $E_2=\mathcal W^{\oplus m}\oplus \bigoplus_{i=1}^{\lvert I\rvert}{\mathbb L}_i^{\oplus d_i}$. Let $\bm{\rho}$ be a canonical multisection of $E_2\to \oPMh_{0,B,I}$ such that
\begin{enumerate}
    \item $\bm{\rho}\vert_{\mathcal Z_\alpha}\pitchfork 0$ for all $\alpha\in S_{\text{ind}}$;
    \item $\bm{\rho}$ does not vanish on $\mathcal Z_\alpha\cap \mathcal Z_\beta$ for any $\alpha\ne \beta,~\alpha,\beta \in S_{\text{ind}}$.
\end{enumerate}
The construction of such a multisection $\bm{\rho}$ is similar to the construction of the transverse canonical multisection $\bm{s}$ in the second part of the proof of Theorem \ref{thm intersection numbers well-defined}: starting from an arbitrary canonical multisection $\bm{\rho}_0$ of $E_2\to \oPMh_{0,B,I}$, we perturb it to a multisection which is transverse to zero on each $\mathcal Z_\alpha$ and $\mathcal Z_\alpha\cap \mathcal Z_\beta$. In details, denote by $\bm{\rho}^j_{[u_i]}$ the sections of $E_2\to \oPMh_{0,B,I}$ constructed like the section $\bm{s}^j_{[u_i]}$ in the second part of the proof of Theorem \ref{thm intersection numbers well-defined}. For each $\lambda=\left(\lambda_i^j\right)_{\substack{1\le i \le N\\1\le j \le \rk E_2}}\in W= \mathbb R_+^{N\times \rk E_2}$, $\bm{\rho}_\lambda$ is the multisection of $E_2\to \oPMh_{0,B,I}$ defined by
$$
\bm{\rho}_\lambda=\bm{\rho}_0 + \sum_{\substack{1\le i \le N\\1\le j \le \rk E_2}} \lambda_i^j \bm{\rho}^j_{[u_i]}.
$$
By Theorem \ref{thm hirsch}, for each $\mathcal Z=\mathcal Z_\alpha$ or $\mathcal Z_\alpha\cap \mathcal Z_\beta$, the set  $W_{\mathcal Z}:=\{\lambda \in W \colon \bm{\rho}_\lambda\vert_{\mathcal Z} \pitchfork 0\}$ is residual. Since $S_{\text{ind}}$ is a finite set, setting $\bm{\rho}=\bm{\rho}_{\lambda}$ for an arbitrary $\lambda$ in the non-empty set $\bigcap_{\mathcal Z=\mathcal Z_\alpha} W_{\mathcal Z} \cap \bigcap_{\mathcal Z=\mathcal Z_\alpha\cap \mathcal Z_\beta} W_{\mathcal Z}$, has the required properties.

Note that $\hat {\bm{\rho}}:=t\oplus \bm{\rho}$ is a canonical multisection of $E_1\oplus E_2=\mathcal W^{\oplus m}\oplus \mathbb L_1^{\oplus d_1+1}\oplus\bigoplus_{i=2}^{\lvert I\rvert}{\mathbb L}_i^{\oplus d_i}$. Moreover, $\hat {\bm{\rho}}$ is transverse to zero since, by construction, the restriction of $\bm{\rho}$ to the zero locus of $t$ is transverse to zero. By Definition \ref{dfn correlator point insertion}, and according to the interpretation of relative Euler classes as weighted count of zeros in \cite[Appendix A]{BCT2}, we can write 
$$\left\langle
		\tau^{a_1}_{d_1+1}\tau^{a_2}_{d_2}\dots\tau^{a_l}_{d_l}\sigma^{b_1}\sigma^{b_2}\dots\sigma^{b_k}
		\right\rangle_0^{\frac{1}{r},\text{o},\h,m}=\# Z(\hat {\bm{\rho}}),$$
  where by $\# Z(-)$ we mean the weighted count of zeros of a multisection. Moreover, by Lemma \ref{lem zero locus trr section global}, we have
$$
\# Z(\hat {\bm{\rho}})=\sum_{\alpha\in S_{\text{ind}}} r\# Z(\bm{\rho}\vert_{\mathcal Z_\alpha}),
$$
To prove \eqref{eqtrr1} and \eqref{eqtrr2}, we need to show that, for each $\alpha\in S_{\text{ind}}$ with non-empty $$\mathcal Z_\alpha= \hat{\mathcal M}_{B_0,I_0}^{I_{-1}\cup\{t_1,\dots,t_s\}}\times\prod_{j=1}^s \oPMb_{0,B_j\sqcup \{r-2-2t_j\},I_j},$$ we have 
\begin{equation} \label{eq: equiv to trrs}
    \begin{split}
        r\# Z(\bm{\rho}\vert_{\mathcal Z_\alpha})=&(-1)^s\left\langle
		\tau^{a}_{0}\prod_{a_i\in I_{-1}}\tau^{a_i}_{d_i}\prod_{j=1}^{s}\tau^{t_j}_{0}
		\right\rangle_0^{\frac{1}{r},\text{ext},m}\cdot \left\langle
		\tau^{r-2-a}_{0}\prod_{a_i\in I_0}\tau^{a_i}_{d_i}\prod_{b_i\in B_0}\sigma^{b_i}
		\right\rangle_0^{\frac{1}{r},\text{o},\h,m}\\
		&\qquad\quad\qquad\quad\qquad\quad\qquad\quad\qquad\cdot\prod_{j=1}^{s}\left\langle
		\sigma^{r-2-2t_j}\prod_{a_i\in I_j}\tau^{a_i}_{d_i}\prod_{b_i\in B_j}\sigma^{b_i}
		\right\rangle_0^{\frac{1}{r},\text{o},\h,m},
    \end{split}
\end{equation}
where $a\in\{-1,0,\dots,r-2\}$ and $a+\sum_{a_i\in I_{-1}}a_i+\sum_{k=1}^s t_k+2\equiv 0 \mod r$.

We consider the detaching morphism
$$\text{Detach}_e:\mathcal Z_\alpha \to \oCM_{v^c_{}} \times \oPMb_{0,B_0,I_0\sqcup \{r-2-a\}} \times\prod_{i=1}^s \oPMb_{0,B_i\sqcup \{r-2-2t_i\},I_i}.$$
For $i=1,2,\dots,s$, we choose a transverse canonical multisection $\sigma_i$ of the vector bundle 
$$F_i:=\mathcal W^{\oplus m}_{0,B_i\sqcup \{r-2-2t_i\},I_i}\oplus \bigoplus_{a_j\in I_i} \mathbb L_j^{\oplus d_j}\to \oPMb_{0,B_i\sqcup \{r-2-2t_i\},I_i}.$$
We choose a canonical multisection $\sigma_o^{\mathcal W}$ of 
$
\mathcal W^{\oplus m}_{0,B_0,I_0\sqcup \{r-2-a\}}\to \oPMb_{0,B_0,I_0\sqcup \{r-2-a\}},
$
and a canonical multisection $\sigma_o^{\mathbb L}$  of 
$
\bigoplus_{a_j\in I_{0}} \mathbb L_j^{\oplus d_j}\to \oPMb_{0,B_0,I_0\sqcup \{r-2-a\}}$, 
such that \[{\sigma_o}:={\sigma_o^{\mathcal W}}\oplus \sigma_o^{\mathbb L}\pitchfork 0\] as a multisection of the vector bundle 
$$
F_o:=\mathcal W^{\oplus m}_{0,B_0,I_0\sqcup \{r-2-a\}}\oplus \bigoplus_{a_j\in I_0} \mathbb L_j^{\oplus d_j}\to \oPMb_{0,B_0,I_0\sqcup \{r-2-a\}}.
$$
We also choose a coherent multisection $\sigma_c^{\mathcal W}$  (see Section \ref{subsec coherent}) of 
$
\mathcal W^{\oplus m}_{v^c}\to {\overline{\mathcal{R}}}_{v^c}
$
and a multisection $\sigma_c^{\mathbb L}$  of 
$
\bigoplus_{a_j\in I_{-1}} \mathbb L_j^{\oplus d_j}\to \Mbar_{v^c}$, 
such that \[\overline{\sigma_c}:=\overline{\sigma_c^{\mathcal W}}\oplus \sigma_c^{\mathbb L}\pitchfork0\]as a multisection of the vector bundle 
$$
F_c:=\mathcal W^{\oplus m}_{v^c}\oplus \bigoplus_{a_j\in I_{-1}} \mathbb L_j^{\oplus d_j}\to \Mbar_{v^c}.
$$

For each $\mathbf{H}\in U_{B^o,I^o}^{I^c}$, let the graded 
$r$-spin graph $\Delta\in V(\mathbf{H})$ be the unique non-smooth vertex with the unique edge $e$. The assembling operator (see Section \ref{subsec coherent}  and \cite[Section 4.1]{BCT2} for more details) $\Ass_{\Delta,e}$ induces an assembling operator $\Ass_{\mathbf{H},e}$, which glues $\sigma^{\mathcal W}_c$ and $\sigma^{\mathcal W}_o\vert_{\oPMb_{\mathbf{H}^o}}$ into a positive multisection $\sigma^{\mathcal W}_\mathbf{H}:=\Ass_{\mathbf{H},e}\left(\sigma^{\mathcal W}_c\boxplus \sigma^{\mathcal W}_o\vert_{\oPMb_{\mathbf{H}^o}}\right)$ of the vector bundle 
$
\mathcal W^{\oplus m}_\mathbf{H}\to\oPMb_\mathbf{H}\subseteq \hat{\mathcal M}_{B^o,I^o}^{I^c}.
$
We denote by $\Ass_{e}\left(\sigma_c\boxplus \sigma_o\right)$ the canonical multisection of 
$$
\mathcal W^{\oplus m}\oplus  \bigoplus_{a_j\in I_0\cup I_{-1}} \mathbb L_j^{\oplus d_j}\to\hat{\mathcal M}_{B^o,I^o}^{I^c}.
$$
given by the collection of $\sigma^{\mathcal W}_\mathbf{H}\oplus\left(\sigma_o^{\mathbb L}\vert_{\oPMb_{\mathbf{H}^o}}\boxplus \sigma_c^{\mathbb L}\right)$ for all $\mathbf{H}\in U_{B^o,I^o}^{I^c}$. We define a canonical multisection $\mathbf{\sigma}$ of $E_2\vert_{\mathcal Z_\alpha}$ by
$$
\bm{\sigma}:=\Ass_{e}\left(\sigma_c\boxplus \sigma_o\right)\boxplus \bboxplus_{i=1}^s \sigma_i.
$$
The number $\# Z(\bm{\sigma})$ is zero unless $\rk F_i=\dim \oPMb_{0,B_i\sqcup \{r-2-2t_i\},I_i}$ for each $i=1,2,\dots,s$; when these equations hold we have 
\begin{equation}\label{eq:split side contribution}
    \begin{split}
        \# Z(\bm{\sigma})&=(-1)^s \cdot\# Z\left(\Ass_{e}\left(\sigma_c\boxplus \sigma_o\right)\right)\cdot \prod_{j=1}^{s} \# Z(\sigma_j)\\
        &=(-1)^s \cdot\# Z\left(\Ass_{e}\left(\sigma_c\boxplus \sigma_o\right)\right)\cdot \prod_{j=1}^{s}\left\langle
		\sigma^{r-2-2t_j}\prod_{a_i\in I_j}\tau^{a_i}_{d_i}\prod_{b_i\in B_j}\sigma^{b_i}
		\right\rangle_0^{\frac{1}{r},\text{o},\h,m}.
    \end{split}
\end{equation}
Moreover, by the construction of the assembling operator, the number $\# Z\left(\Ass_{e}\left(\sigma_c\boxplus \sigma_o\right)\right)$ is zero unless $\rk F_c=\dim \Mbar_{v^c}$ and $\rk F_o= \dim \oPMb_{0,B_0,I_0\sqcup \{r-2-a\}}$, when these equations hold we have (see the proof of \cite[Lemma 4.14]{BCT2} for more details)
\begin{equation}\label{eq:assembling split}
    \begin{split}
        \# Z\left(\Ass_{e}\left(\sigma_c\boxplus \sigma_o\right)\right)&=Z(\overline{\sigma_c})\cdot Z(\sigma_o)\\
        &=\frac{1}{r}\left\langle
		\tau^{a}_{0}\prod_{a_i\in I_{-1}}\tau^{a_i}_{d_i}\prod_{j=1}^{s}\tau^{t_j}_{0}
		\right\rangle_0^{\frac{1}{r},\text{ext}}\cdot \left\langle
		\tau^{r-2-a}_{0}\prod_{a_i\in I_0}\tau^{a_i}_{d_i}\prod_{b_i\in B_0}\sigma^{b_i}
		\right\rangle_0^{\frac{1}{r},\text{o},\h,m},
    \end{split}
\end{equation}
the factor $\frac{1}{r}$ comes from the defining equation \eqref{eq:define closed ext} of the closed extended $r$-spin correlators. Combining \eqref{eq:split side contribution} and \eqref{eq:assembling split} we have
\begin{equation}
    \begin{split}
        r\# Z(\bm{\sigma
        })=&(-1)^s\left\langle
		\tau^{a}_{0}\prod_{a_i\in I_{-1}}\tau^{a_i}_{d_i}\prod_{j=1}^{s}\tau^{t_j}_{0}
		\right\rangle_0^{\frac{1}{r},\text{ext}}\cdot \left\langle
		\tau^{r-2-a}_{0}\prod_{a_i\in I_0}\tau^{a_i}_{d_i}\prod_{b_i\in B_0}\sigma^{b_i}
		\right\rangle_0^{\frac{1}{r},\text{o},\h,m}\\
		&\qquad\quad\qquad\quad\qquad\quad\qquad\quad\qquad\quad\cdot\prod_{j=1}^{s}\left\langle
		\sigma^{r-2-2t_j}\prod_{a_i\in I_j}\tau^{a_i}_{d_i}\prod_{b_i\in B_j}\sigma^{b_i}
		\right\rangle_0^{\frac{1}{r},\text{o},\h,m}.
    \end{split}
\end{equation}
Thus, we can prove \eqref{eq: equiv to trrs}, and consequently validate \eqref{eqtrr1} and \eqref{eqtrr2}, by showing 
\begin{equation}\label{eq: need homotopy in proof of trr}
     \# Z(\bm{\sigma})=\# Z(\bm{\rho}\vert_{\mathcal Z_\alpha}).
\end{equation} 
In fact, \eqref{eq: need homotopy in proof of trr} follows from the exact same homotopy argument that was used in the proof of Theorem \ref{thm intersection numbers well-defined}, the part dealing with independence of choices. We can find a homotopy between $\bm{\sigma}$ and $\bm{\rho}$ which is a perturbation of the linear homotopy and transverse to zero. This homotopy cannot have zeros on boundaries of type CB, R, and NS+ due to positivity; when a zero moves in (out resp.) through a type-BI boundary $\text{bd}_{BI}$, another zero moves out (in resp.) through the type-AI boundary $\text{bd}_{AI}$ which is paired with $\text{bd}_{BI}$ via $\text{PI}$.
   
\end{proof}

\begin{rmk}
In \cite[Theorem 1.5]{PST14} and \cite[Theorem 4.1]{BCT2} different looking topological recursion relations were written. Their proof also involved studying the zero locus of a closely related section. In these works the section was not canonical, and a new type of contributions appeared from homotoping it to a canonical one. Here the section can be made globally defined and is automatically canonical, but since the moduli space is more complicated than the ones of \cite{PST14,BCT2},
 the final expression has a very different form.
 \end{rmk}

\section{Comparison between BCT $r$-spin and the $(r,\h=0)$ point insertion theories}\label{section:comparison}

In \cite{BCT1,BCT2}, an open $r$-spin theory, to which we refer as the BCT theory, was constructed. In that theory boundary markings may only have twist $r-2$. In the $(r,\h=0)$ point insertion theory, we also allow only boundary markings with twist $r-2$. In this section we compare these two theories.  It turns out that the two theories are equivalent, in the sense that there is a transformation relating the correlators of one theory to those of the other, but the equivalence is non-trivial, as Theorems \ref{thm: compare primary},~\ref{thm:compare_descendants} below show.

    Let $I=\{a_1,\dots,a_l\}$ and $B=\{\underbrace{r-2,r-2,\dots,r-2}_k\}$. The BCT moduli space is $\Mbar^{1/r}_{0,B,I}$, which is a subset of $\Mbar^{1/r,\h=0}_{0,B,I}$. The BCT correlators $\left\langle\tau^{a_1}_{d_1}\cdots\tau^{a_n}_{d_n}\smash{\underbrace{\sigma^{r-2}\sigma^{r-2}\dots\sigma^{r-2}}_k}\right\rangle^{\frac{1}{r},o,\text{BCT}}_0\vphantom{\left\langle\tau^{a_1}_{d_1}\cdots\tau^{a_n}_{d_n}{\underbrace{\sigma^{r-2}\sigma^{r-2}\dots\sigma^{r-2}}_k}\right\rangle^{\frac{1}{r},o,\text{BCT}}_0}$ are defined as the Euler number of $\mathcal W\oplus \bigoplus_{i=1}^l \mathbb L_i^{\oplus d_i}$ with respect to a BCT-canonical multisection $s^{\text{BCT}}$. 
    
 $\Mbar^{1/r}_{0,B,I}$ is different from $\Mbar^{1/r,\h=0}_{0,B,I}$ since $\Mbar^{1/r}_{0,B,I}$ has no type-AI boundaries, because no point insertion procedure is applied. Also the BCT-canonical multisection $s^{\text{BCT}}$ is defined differently: near boundaries of type CB, R and NS+, the BCT-canonical multisection $s^{\text{BCT}}$ satisfies the same positivity boundary condition as in the $\h=0$ point insertion theory. A type-BI boundary corresponds to an NS node, whose illegal half-node has twist 0. Instead of gluing another space to this type-BI boundary, $s^{\text{BCT}}$ is required to be the pullback of a canonical multisection over the moduli obtained by detaching the node and forgetting the illegal half-node with twist $0$. 
    See \cite[Section 3]{BCT2} for more details.
        
    \begin{rmk}\label{rmk: orientation issue}
        In this paper we define the BCT correlators with our orientation conventions, but still using the BCT boundary conditions.  
        Thus, every BCT correlator we consider here equals the corresponding correlator defined in \cite{BCT2}, up to sign. With this choice of orientation the TRRs for BCT correlators  \textit{i.e. \cite[Theorem 4.1]{BCT2}} still holds. 
    \end{rmk}

    \begin{thm}\label{thm: compare primary}
        The primary correlators in the BCT theory and the $(r,\h=0)$ point insertion theory coincide, \textit{i.e.} when $$\frac{2\sum_{i=1}^l a_i+ (k-1)(r-2)}{r}=2l+k-3,$$ we have
        $$
            \left\langle\tau^{a_1}_{0}\dots\tau^{a_l}_{0}\smash{\underbrace{\sigma^{r-2}\sigma^{r-2}\dots\sigma^{r-2}}_k}\right\rangle^{\frac{1}{r},o,\text{BCT}}_0\vphantom{\left\langle\tau^{a_1}_{0}\dots\tau^{a_n}_{0}{\underbrace{\sigma^{r-2}\sigma^{r-2}\dots\sigma^{r-2}}_k}\right\rangle^{\frac{1}{r},o,\text{BCT}}_0}=\left\langle\tau^{a_1}_{0}\dots\tau^{a_l}_{0}\smash{\underbrace{\sigma^{r-2}\sigma^{r-2}\dots\sigma^{r-2}}_k}\right\rangle^{\frac{1}{r},o,\h=0}_0\vphantom{\left\langle\tau^{a_1}_{0}\dots\tau^{a_n}_{0}{\underbrace{\sigma^{r-2}\sigma^{r-2}\dots\sigma^{r-2}}_k}\right\rangle^{\frac{1}{r},o,\h=0}_0}.
        $$
    \end{thm}
    \begin{proof}
        For a graded $r$-spin graph $\Gamma,$ we denote by $\mathcal B\Gamma$ the graph obtained by detaching all the NS edges of $\Gamma$ whose twist at the illegal half-edge is $0$, then forgetting all the twist-$0$ illegal half-edges. Note that forgetting tails may produce unstable components $\mathcal B \Gamma$; we view the moduli space corresponding to an unstable component, which is a direct summand of $\Mbar_{\mathcal B \Gamma}$, as a single point.  Let $F_\Gamma\colon \Mbar_\Gamma \to \Mbar_{\mathcal B\Gamma}$ be the induced map in the level of moduli spaces.
     \cite[Proposition 6.2]{BCT2} constructed the \emph{special BCT-canonical multisections}, which are collections of transverse multisections $\hat s^{\text{BCT}}_\Gamma$ of $\mathcal W \to \oPMb_\Gamma$  for every graded $r$-spin graph $\Gamma$, which satisfy, in addition to the same positivity as in $(r,\h=0)$ theory near the boundaries of type CB, R and NS+, that for every $\Gamma\in \partial^! \Delta$ 
        $$
         \hat s^{\text{BCT}}_\Delta\big\vert_{ \Mbar_{\Gamma_{}}}=F_\Gamma^* \hat s^{\text{BCT}}_{\mathcal B \Gamma}.
        $$ 
        \cite[Theorem 3.17, Definition 3.18]{BCT2} showed that 
 the BCT correlator can be written as
        \[\int_{\oPM_{0,B,I}}
        e\left(\mathcal W ; {\hat s^{\text{BCT}}}\right)
        =\#Z\left({\hat s^{\text{BCT}}}\right),\]
        where the equality is explained in \cite[Appendix A]{BCT2}.

        We can construct a canonical multisection $s^{\h=0}$ for $(r,\h=0)$ point insertion theory from the transverse special BCT-canonical multisections $\hat s^{\text{BCT}}_\bullet$. For each $(r,\h)$-graph $\mathbf{G}\in \GPI^{r,\h=0}_{0,B,I}$, we define $\mathcal B \mathbf{G}$ to be the graph obtained by forgetting all the tails in $I'(\mathbf{G})$. Note that $\mathcal B \mathbf{G}$ is no longer an $(r,\h)$-graph because we forgot tails in $I'(\mathbf{G})$ but kept tails in $B'(\mathbf{G})$. Since $\h=0$, all the tails in $I'(\mathbf{G})$ have twist 0, therefore  $\mathcal B \mathbf{G}$ is still a graded $r$-spin graph. 
        We denote by $F_\mathbf{G}\colon \Mbar_\mathbf{G} \to \Mbar_{\mathcal B\mathbf{G}}$ be the corresponding map between moduli spaces. We define a multisection $s^{\h=0}$ of $\mathcal W\to \oPMb^{1/r,\h=0}_{0,B,I}$ by
        $$
            s^{\h=0}\big\vert_{\Mbar_\mathbf{G}}:=F_\mathbf{G}^*\hat s^{\text{BCT}}_{\mathcal B\mathbf{G}}.
        $$
        The multisection $s^{\h=0}$ defined in this way is a canonical multisection in the sense of point insertion (Definition \ref{def canonical for Witten}).
        The positivity boundary conditions (the first item in Definition \ref{def canonical for Witten}) are satisfied since the BCT-canonical multisections $\hat s^{\text{BCT}}_{\mathcal BG}$  satisfy the same positivity boundary conditions. 
        The second item in Definition \ref{def canonical for Witten} is also satisfied: for a pair of type-AI boundary $\text{bd}_{AI}\subset\Mbar_{\mathbf{G_1}}$ and type-BI boundary $\text{bd}_{BI}\subset\Mbar_{\mathbf{G_2}}$ paired by $PI$, we have $\text{bd}_{BI}=\Mbar_{\mathbf{G}_{BI}}\cong\Mbarstar_{v_0}\times \prod_i \Mbar_{v_i}$ and $\text{bd}_{AI}=\Mbar_{\mathbf{G}_{AI}}\cong\{pt\}\times\Mbarstar_{v_0}\times \prod_i \Mbar_{v_i}$, where $v_0$ has an illegal boundary tail with twist $0$ and $\{pt\}$ is a zero-dimensional moduli with rank-zero Witten bundle on it (see the proof of \cite[Theorem 4.12]{TZ1} for more details); denoting by $\hat v_0$ the vertex obtained by forgetting the twist-0 illegal tail of $v_0$, we have
        \[
        \Mbar_{\mathcal B\mathbf{G}_{BI}}= \Mbar_{\mathcal B\hat v_0}\times \prod_i \Mbar_{\mathcal B v_i}\cong\{pt\}\times\Mbar_{\mathcal B \hat v_0}\times \prod_i \Mbar_{\mathcal B v_i}=\Mbar_{\mathcal B\mathbf{G}_{AI}}
        \]
        and $\hat s^{\text{BCT}}_{\mathcal B\mathbf{G}_{BI}}=\hat s^{\text{BCT}}_{\mathcal B\mathbf{G}_{AI}}$ under this identification; on the other hand, we have        
        \begin{equation*}
        \begin{split}
        s^{\h=0}\vert_{\oPMb_{\mathbf{G}_{BI}}}&= F_\mathbf{G_1}^*\left(\hat s^{\text{BCT}}_{\mathcal B\mathbf{G_1}}\big\vert_{\oPMb_{\mathcal B\mathbf{G}_{BI}}}\right)=F_\mathbf{G_1}^*F_{\mathbf{G}_{BI}}^*\hat s^{\text{BCT}}_{\mathcal B\mathbf{G}_{BI}}
            \end{split}
        \end{equation*}
        and 
        \begin{equation*}
        \begin{split}
        s^{\h=0}\vert_{\oPMb_{\mathbf{G}_{AI}}}&= F_\mathbf{G_2}^*\left(\hat s^{\text{BCT}}_{\mathcal B\mathbf{G_2}}\big\vert_{\oPMb_{\mathcal B\mathbf{G}_{AI}}}\right)=F_\mathbf{G_2}^*F_{\mathbf{G}_{AI}}^*\hat s^{\text{BCT}}_{\mathcal B\mathbf{G}_{AI}},
            \end{split}
        \end{equation*}
        they coincide under the identification between $\text{bd}_{BI}$ and $\text{bd}_{AI}$.

        Note that $\hat s^{\text{BCT}}_{\mathcal B\Gamma_{BI}}$ (hence also $\hat s^{\text{BCT}}_{\mathcal B\Gamma_{AI}}$) is a nowhere vanishing section of $\mathcal W_{\mathcal B\Gamma_{BI}}\to \oPMb_{\mathcal B\Gamma_{BI}}$ since \[\hat s^{\text{BCT}}_{\mathcal B\Gamma_{BI}}\pitchfork 0\] and
        $$\rk \mathcal W_{\mathcal B\Gamma_{BI}}=\rk \mathcal W_{\mathbf{G_1}}=\dim  \oPMb_{\mathbf{G_1}} >\dim \oPMb_{\Gamma_{BI}}\ge \dim \oPMb_{\mathcal B\Gamma_{BI}}.$$
        Thus $s^{\h=0}$ vanishes nowhere on boundaries of type AI or BI, and we can define $(r,\h=0)$ point insertion correlators as $\int_{\overline{\mathcal{PM}}^{1/r,\h=0}_{0,B,I}}e\left(\mathcal W ; {s^{\h=0}\big\vert_{U_+\cup\partial\oPMh_{0,B,I}}}\right)$. 

        Actually $s^{\h=0}$ is a transverse multisection. On a connected component $\Mbar_\mathbf{G}\subseteq \Mbar^{1/r}_{0,B,I}\subseteq \Mbar^{1/r,\h=0}_{0,B,I}$, we have $s^{\h=0}\big\vert_{\oPMb_\mathbf{G}}=\hat s^{\text{BCT}}_\mathbf{G}$ since $I'(\mathbf{G})=\emptyset$ and $\mathbf{G}=\mathcal B\mathbf{G}$; so $s^{\h=0}\big\vert_{\oPMb_\mathbf{G}}$ is transverse since $\hat s^{\text{BCT}}_\mathbf{G}$ is. 
        On a connected component $\Mbar_\mathbf{G}\subseteq  \Mbar^{1/r,\h=0}_{0,B,I}\setminus\Mbar^{1/r}_{0,B,I}$,  we claim that the section $s^{\h=0}\big\vert_{\oPMb_\mathbf{G}}$ is transverse because it vanishes nowhere: indeed, the section $s^{\h=0}\big\vert_{\oPMb_\mathbf{G}}$ is pulled back from a section $\hat s^{\text{BCT}}_{\mathcal B\mathbf{G}}$ of $\mathcal W_{\mathcal B\mathbf{G}}\to \oPMb_{\mathcal B\mathbf{G}}$; since $I'(\mathbf{G})\ne \emptyset$ in this case, we have $\rk \mathcal W_{\mathcal B\mathbf{G}}>\dim \oPMb_{\mathcal B\mathbf{G}}$, then $\hat s^{\text{BCT}}_{\mathcal B\mathbf{G}}$ vanishes nowhere since it is transversal to zero.

        Therefore, by \cite[Appendix A]{BCT2} again, we can write        $$\int_{\overline{\mathcal{PM}}^{1/r,\h=0}_{0,B,I}}e\left(\mathcal W ; {s^{\h=0}\big\vert_{U_+\cup\partial\oPMh_{0,B,I}}}\right)=\#Z\left({ s^{\h=0}}\right).$$
        The theorem is proven since we have
        \begin{equation*}
        \begin{split}
        \#Z\left({ s^{\h=0}}\right)&=\sum_{\Mbar_\mathbf{G}\subseteq \Mbar^{1/r}_{0,B,I}} \#Z\left({ s^{\h=0}\big\vert_{\oPMb_\mathbf{G}}}\right)+\sum_{\Mbar_\mathbf{G}\subseteq  \Mbar^{1/r,\h=0}_{0,B,I}\setminus\Mbar^{1/r}_{0,B,I}} \#Z\left({ s^{\h=0}\big\vert_{\oPMb_\mathbf{G}}}\right)\\
        &=\sum_{\Mbar_\mathbf{G}\subseteq \Mbar^{1/r}_{0,B,I}}\#Z\left({\hat s^{\text{BCT}}_{\mathbf{G}}}\right)+0=\#Z\left({\hat s^{\text{BCT}}}\right).
        \end{split}
        \end{equation*}

    \end{proof}
    To obtain the relation beyond primary correlators, we need the topological recursion relations.
     We recall the topological recursion relations for the BCT theory, proven in \cite{BCT2}.

    \begin{thm}[\cite{BCT2}, Theorem 4.1]\label{thm TRR BCT}
    The BCT correlators satisfy the following two types of topological recursion relations (note that all boundary twists $b_i$ are $r-2$):
\begin{itemize}
\item[(a)] (TRR with respect to boundary marking $b_1$) Suppose $l,k\ge 1$.  Then
\begin{equation}\label{eq: bct trr 1}
\begin{split}
&\left\langle \tau_{d_1+1}^{a_1}\prod_{i=2}^l\tau^{a_i}_{d_i}\prod_{i=1}^k\sigma^{b_i}\right\rangle^{\frac{1}{r},o,\text{BCT}}_0\hspace{-0.2cm}\\
=&\sum_{\substack{R_1 \sqcup R_2 = \{2,3,\dots,l\}\\ -1\le a \le r-2}}\hspace{-0.1cm}\left\langle \tau_0^{a}\tau_{d_1}^{a_1}\prod_{i \in R_1}\tau_{d_i}^{a_i}\right\rangle^{\frac{1}{r},\text{ext}}_0
\hspace{-0.1cm}\left\langle \tau_0^{r-2-a}\prod_{i\in R_2}\tau^{a_i}_{d_i}\prod_{i=1}^k\sigma^{b_i}\right\rangle^{\frac{1}{r},o,\text{BCT}}_0\\
&+\hspace{-0.1cm}\sum_{\substack{R_1 \sqcup R_2 =  \{2,3,\dots,l\} \\ T_1 \sqcup T_2 =  \{2,3,\dots,k\}}} \hspace{-0.1cm} \left\langle \tau^{a_1}_{d_1}\prod_{i \in R_1} \tau^{a_i}_{d_i}\prod_{i\in T_1}\sigma^{b_i}\right\rangle^{\frac{1}{r},o,\text{BCT}}_0 \hspace{-0.1cm}\left\langle \prod_{i \in R_2} \tau^{a_i}_{d_i} \sigma^{r-2}\sigma^{b_1}\prod_{i\in T_2}\sigma^{b_i}\right\rangle^{\frac{1}{r}, o,\text{BCT}}_0.
    \end{split}
\end{equation}
\item[(b)] (TRR with respect to internal marking $a_2$) Suppose $l\ge 2$.  Then
\begin{equation}\label{eq: bct trr 2}
\begin{split}
&\left\langle \tau_{d_1+1}^{a_1}\prod_{i=2}^l\tau^{a_i}_{d_i}\prod_{i=1}^k\sigma^{b_i}\right\rangle^{\frac{1}{r},o,\text{BCT}}_0\hspace{-0.2cm}\\=&
\sum_{\substack{R_1 \sqcup R_2 =  \{3,4,\dots,l\}\\ -1\le a \le r-2}}\hspace{-0.1cm}\left\langle \tau_0^{a}\tau_{d_1}^{a_1}\prod_{i \in R_1}\tau_{d_i}^{a_i}\right\rangle^{\frac{1}{r},\text{ext}}_0
\hspace{-0.1cm}\left\langle \tau_0^{r-2-a}\tau^{a_2}_{d_2}\prod_{i\in R_2}\tau^{a_i}_{d_i}\prod_{i=1}^k\sigma^{b_i}\right\rangle^{\frac{1}{r},o,\text{BCT}}_0\\
&+\hspace{-0.1cm}\sum_{\substack{R_1 \sqcup R_2 = \{3,4,\dots,l\} \\ T_1 \sqcup T_2 = \{1,2,\dots,k\}}} \hspace{-0.1cm} \left\langle \tau^{a_1}_{d_1}\prod_{i \in R_1} \tau^{a_i}_{d_i}\prod_{i\in T_1}\sigma^{b_i}\right\rangle^{\frac{1}{r},o,\text{BCT}}_0 \hspace{-0.1cm}\left\langle\tau^{a_2}_{d_2} \prod_{i \in R_2} \tau^{a_i}_{d_i} \sigma^{r-2}\prod_{i\in T_2}\sigma^{b_i}\right\rangle^{\frac{1}{r}, o,\text{BCT}}_0.
\end{split}
\end{equation}

\end{itemize}
\end{thm}
Observe that by Theorem \ref{thm: compare primary}, and the fact that the two correlators which involve descendants can be calculated from the primary ones and the closed extended correlators using the TRRs Theorem \ref{thm TRR BCT}, for the BCT theory, and Theorem \ref{thm TRR} for the $(r,\h=0)$ theory, we see that the correlators of each theory determine the correlators of the other. Theorem \ref{thm:compare_descendants} writes an explicit transformation realizing this observation.

To describe the relation between all correlators of the BCT theory and the $\h=0$ point insertion theory, we need the following combinatorial definition.
\begin{dfn}\label{def: I B D}
Let $\mathcal I=\{a_1,a_2,\dots,a_l\}$, $\mathcal B=\{b_1,b_2,\dots,b_k\}=\{r-2,\dots,r-2\}$ and $\mathcal D=\{d_1,d_2,\dots,d_l\}$ be three multisets. Note that both $\mathcal I$ and $\mathcal D$ are labelled by $\{1,2,\dots,l\}$, $\mathcal B$ is labelled by $\{1,2,\dots,k\}$.
    
    We denote by $\mathfrak G(\mathcal I, \mathcal B, \mathcal D)$ the collection of graphs ${G}=(V,E),$ where $V=V(G)$ is the vertex set and $E=E(G)$ is the edge set, together with 
    \begin{itemize}
    \item
    multisets $I(v),~B(v)$ associated to each vertex $v\in V(G)$, 
    \item
    a multiset $B'(a)$ associated to each $a\in \mathcal I$, 
    \end{itemize}
    such that
    \begin{enumerate}
        \item $\bigsqcup_{v\in V(G)}I(v)=\mathcal I$,  $\bigsqcup_{v\in V(G)}B(v)=\mathcal B\sqcup \bigsqcup_{{a\in \mathcal I}}B'(a)$;
        \item each edge $e\in E(G)$ connecting the vertices $v_1$ and $v_2,$ corresponds to a pair $(a,b)$ where $a\in I(v_1)$, $b\in B(v_2)\cap B'(a)$;
        \item $0\le \vert B'(a_i)\vert \le d_i$ for each $a_i\in I$, and $b=r-2$ for all $b\in B'(a_i)$;
        \item $G$ is connected, and of genus $0$, that is, $|V|-|E|+1=0.$ 
    \end{enumerate}
\end{dfn}
\begin{rmk}
    A graph $G\in \mathfrak G(\mathcal I, \mathcal B, \mathcal D)$ characterizes a disjoint union of smooth $r$-spin disks $C_v$ (with internal and boundary tails indexed by $I(v)$ and $B(v)$) indexed by $v\in V(G)$, together with dashed lines (corresponding to $E(G)$) connecting internal and boundary tails. Although they look very similar, we point out the differences between these objects and the $(r,\h)$-graphs:
    \begin{itemize}
        \item[--] The pairs induced by the dashed lines are not one-to-one. Each internal tail $a_i$ could be paired to multiple (at most $d_i$) boundary tails.
        \item[--] The twists $a$ and $b~(=r-2)$ of paired internal and boundary tail may not satisfy $2a+b=r-2$.
        \item[--] All internal tails are labelled by $\{1,2,\dots,\lvert \mathcal I\rvert\}$, whether they have been paired or not. (Still only unpaired boundary tails are labelled by $\{1,2,\dots,\lvert\mathcal B\rvert\}$ as in an $(r,\h)$-graph.)
    \end{itemize}
\end{rmk}
    \begin{thm}\label{thm:compare_descendants}
        The following explicit transformation relates the $(r,\h=0)$ point insertion correlators and the BCT correlators:
        \begin{equation}  \label{eq: compare all}       F^{\h=0}(\mathcal I, \mathcal B, \mathcal D)=\sum_{G\in \mathfrak G(\mathcal I, \mathcal B, {\mathcal D})}T(G),
        \end{equation}
        where 
        \begin{equation}\label{eq F h=0}F^{\h=0}(\mathcal I, \mathcal B, \mathcal D):=\left\langle\tau^{a_1}_{d_1}\dots\tau^{a_l}_{d_l}\sigma^{b_1}\dots\sigma^{b_k}\right\rangle^{\frac{1}{r},o,\h=0}_0
        \end{equation}
        and
        $$
        T(G):=(-1)^{\lvert E(G)\rvert}\prod_{v\in V(G)}\left\langle\prod_{a_i\in I(v)}\tau^{a_i}_{d_i-\lvert B'(a_i) \rvert}\prod_{b_i\in B(v)}\sigma^{b_i}\right\rangle^{\frac{1}{r},o,\text{BCT}}_0.
        $$

    \end{thm}
        \begin{proof}
        We prove \eqref{eq: compare all} by an induction on $S(\mathcal D):=\sum_{d\in \mathcal D}d$. Theorem \ref{thm: compare primary} shows that the primary correlators of the two theories coincide, therefore \eqref{eq: compare all} holds for $(\mathcal I,\mathcal B,\mathcal D) $ with $S(\mathcal D)=0$. Now we fix $\mathcal I_0=\{a_1,a_2,\dots,a_l\}$, $\mathcal B_0=\{b_1,b_2,\dots,b_k\}=\{r-2,\dots,r-2\}$, $\mathcal D_0=\{d_1,d_2,\dots,d_l\}$ and $\hat {\mathcal D}_0:=\{d_1+1,d_2,\dots,d_l\}$. Assuming that 
        \begin{equation}\label{eq inductive hypothesis comparision}
            \text{we have proven  that \eqref{eq: compare all} holds for all }(\mathcal I,\mathcal B,\mathcal D) \text{ with }(\mathcal D)\le S(\hat {\mathcal D}_0)-1,
        \end{equation}
         we will show that  \eqref{eq: compare all} also holds for $(\mathcal I_0,\mathcal B_0,\hat {\mathcal D}_0)$. In this way we can prove \eqref{eq: compare all} for all $(\mathcal I,\mathcal B,\mathcal D)$ by induction.

        We write \eqref{eq: compare all} for $(\mathcal I_0,\mathcal B_0,\hat {\mathcal D}_0)$ as
         \begin{equation}  \label{eq: compare all induction case}       \left\langle\tau^{a_1}_{d_1+1}\dots\tau^{a_l}_{d_l}\sigma^{b_1}\dots\sigma^{b_k}\right\rangle^{\frac{1}{r},o,\h=0}_0=\sum_{G\in \mathfrak G(\mathcal I_0, \mathcal B_0, \hat{\mathcal D}_0)}T(G).
        \end{equation}
        Note that the left-hand side of \eqref{eq: compare all induction case} is the same as the left-hand side of \eqref{eqtrr1}, we prove \eqref{eq: compare all induction case} by showing that the right-hand side $P_1$ of \eqref{eq: compare all induction case} coincides with the right-hand side $P_2$ of \eqref{eqtrr1}.

    Each term in the right-hand side $P_1$ of \eqref{eq: compare all induction case} is of the form $T(G)$ for a graph $G\in \mathfrak G(\mathcal I_0,\mathcal B_0,\hat {\mathcal D}_0)$. Assuming $a_1\in I(v_a)$ for $v_a\in V(G)$, we define the element $\mu^G_1(b_1)\in I(v_a)\sqcup B(v_a)$ as follows:
    \begin{itemize}
        \item[--] if $b_1\in B(v_a)$ we set $\mu^G_1(b_1)=b_1\in B(v_a)$;
        \item[--] if $b_1\in B(v_b)$ for some $v_b\ne v_a$,  we choose $\mu^G_1(b_1)$ to be the unique element (note that $G$ is genus-zero) in $I(v_a)\sqcup B(v_a)$ such that, the vertices $v_a$ containing $a_1$ and $v_b$ containing $b_1$ are disconnected if we remove all edges of $G$ corresponding to pairs of the form   $(\mu^G_1(b_1),c)$ or $(c,\mu^G_1(b_1))$ with  $c\in\bigsqcup_{v\in V(G)\setminus \{v_a\}}I(v)\sqcup B(v)$.
    \end{itemize}
    We divide the graphs $G\in \mathfrak G(\mathcal I_0,\mathcal B_0,\hat {\mathcal D}_0)$ and the corresponding terms $T(G)$ in $P_1$ into three types:
    \begin{itemize}
        \item $G$ satisfies $\mu^G_1(b_1)\ne a_1$ and $d_1\ge\lvert B'(a_1) \rvert$. We say such a graph $G$ and the corresponding term $T(G)$ is \textit{splittable}. 
        \item $G$ satisfies $\mu^G_1(b_1)=a_1$. We say such a graph $G$ and the corresponding term $T(G)$ is \textit{unsplittable}.
        \item $G$ satisfies $\mu^G_1(b_1)\ne a_1$ and $d_1+1=\lvert B'(a_1) \rvert$. We say such a graph $G$ and the corresponding term $T(G)$ is \textit{exceptional}. 
    \end{itemize}

    For a splittable term $T(G)$ of $P_1$, we can apply the BCT TRR \eqref{eq: bct trr 1} (if $\mu^G_1(b_1)\in B(v_a)$) or \eqref{eq: bct trr 2} (if $\mu^G_1(b_1)\in I(v_a)$) to the factor in $T(G)$ (corresponding to the vertex $v_1\in V(G)$) $$
        \left\langle\tau^{a_1}_{d_1-\lvert B'(a_1) \rvert+1}\prod_{a_i\in I(v_a)\setminus\{a_1\}}\tau^{a_i}_{d_i-\lvert B'(a_i) \rvert}\prod_{b_i\in B(v_a)}\sigma^{b_i}\right\rangle^{\frac{1}{r},o,\text{BCT}}_0
    $$
     with respect to $\mu^G_1(b_1)$, then this factor splits into two types of terms: open-open terms and closed-open terms. By multiplying other factors of $T(G)$, we can split the term $T(G)$ as 
    \begin{equation}\label{eq split splittable term}
        T(G)=\sum_{\text{partition }\mathfrak{p}} T^{oo}_\mathfrak{p}(G)+\sum_{\text{partition }\mathfrak{q} } T^{co}_{\mathfrak{q}}(G),
    \end{equation}
       where each $T^{oo}_{\mathfrak p}(G)$ has an open-open term as a factor, it corresponds to a partition $\mathfrak p$ 
       \begin{equation}\label{eq partition open open}
           I(v_a)\setminus\{a_1\}=I_1\sqcup I_2,~B(v_a)=B_1\sqcup B_2 \text{ such that } \mu^G_1(b_1)\in I_2\sqcup B_2
       \end{equation}
       and is of the form 
       \begin{equation}\label{eq open open explicit}
       \begin{split}
       T^{oo}_{\mathfrak p}(G)=&(-1)^{\lvert E(G)\rvert}\prod_{v\in V(G)\setminus\{v_a\}}\left\langle\prod_{a_i\in I(v)}\tau^{a_i}_{d_i-\lvert B'(a_i) \rvert}\prod_{b_i\in B(v)}\sigma^{b_i}\right\rangle^{\frac{1}{r},o,\text{BCT}}_0\\
           &\cdot \left\langle \tau^{a_1}_{d_1-\lvert B'(a_1)\rvert}\prod_{a_i \in I_1} \tau^{a_i}_{d_i-\lvert B'(a_i)\rvert}\prod_{b_i\in B_1}\sigma^{b_i}\right\rangle^{\frac{1}{r},o,\text{BCT}}_0 \cdot\left\langle \prod_{a_i \in I_2} \tau^{a_i}_{d_i-\lvert B'(a_i)\rvert} \sigma^{r-2}\prod_{b_i\in B_2}\sigma^{b_i}\right\rangle^{\frac{1}{r}, o,\text{BCT}}_0;
           \end{split}
       \end{equation}
       each $T^{co}_{\mathfrak q}(G)$ has a closed-open term as a factor, it corresponds to a partition $\mathfrak q$ (and an integer $-1\le a \le r-2$ which we omit, since at most one out of the $r$ choices of $a$ will give non-zero value due to dimensional reason)
        \begin{equation}\label{eq partition closed open}
           I(v_a)\setminus\{a_1\}=I_1\sqcup I_2 \text{ such that } \mu^G_1(b_1)\in I_2\sqcup B(v_a)
       \end{equation}
       and is of the form
       \begin{equation}\label{eq closed open explicit}
       \begin{split}
       T^{co}_{\mathfrak q}(G)=&(-1)^{\lvert E(G)\rvert}\prod_{v\in V(G)\setminus\{v_a\}}\left\langle\prod_{a_i\in I(v)}\tau^{a_i}_{d_i-\lvert B'(a_i) \rvert}\prod_{b_i\in B(v)}\sigma^{b_i}\right\rangle^{\frac{1}{r},o,\text{BCT}}_0\\
           &\cdot \left\langle \tau_0^{a}\tau_{d_1-\lvert B'(a_1)\rvert}^{a_1}\prod_{a_i \in I_1}\tau_{d_i-\lvert B'(a_i)\rvert}^{a_i}\right\rangle^{\frac{1}{r},\text{ext}}_0
\cdot\left\langle \tau_0^{r-2-a}\prod_{a_i\in I_2}\tau^{a_i}_{d_i-\lvert B'(a_i)\rvert}\prod_{b_i\in B(v_a)}\sigma^{b_i}\right\rangle^{\frac{1}{r},o,\text{BCT}}_0.
        \end{split}
        \end{equation}
      Thus we can write 
    \begin{equation}\label{eq splitted lhs}
        P_1=\sum_{\substack{G~\text{is}\\\text{ splittable}}}\sum_{\mathfrak p} T^{oo}_{\mathfrak p}(G)+\sum_{\substack{G~\text{is}\\\text{ splittable}}}\sum_{\mathfrak q} T^{co}_{\mathfrak q}(G)+\sum_{\substack{G~\text{is}\\\text{ unsplittable}}}T(G)+\sum_{\substack{G~\text{is}\\\text{ exceptional}}}T(G).
    \end{equation}
    We now describe a bijection $\Phi$ between the two sets
    $$
    \left\{(G,\mathfrak p)\left \vert \begin{aligned}G\in& \mathfrak G(\mathcal I_0,\mathcal B_0,\hat {\mathcal D}_0),~G \text{ splittable}\\ &\mathfrak p \text{ a partition as in \eqref{eq partition open open}}\end{aligned}\right.\right\}~\text{and}~\left\{G\left \vert \begin{aligned}G&\in \mathfrak G(\mathcal I_0,\mathcal B_0,\hat {\mathcal D}_0)\\&G \text{ unsplittable}\end{aligned}\right.\right\}.
    $$
    For a splittable $G\in \mathfrak G(\mathcal I_0,\mathcal B_0,\hat {\mathcal D}_0)$ and a partition $\mathfrak p$ as in \eqref{eq partition open open}, we define $\Phi(G,\mathfrak p)$ to be the graph $G'$ obtained by replacing the vertex $v_a\in V(G)$ by two new vertices $u_1,u_2$, and taking $I(u_1):=I_1\sqcup\{a_1\},~B(u_1):=B_1,~I(u_2):=I_2,~B(u_2):=B_2\sqcup\{b'=r-2\}$ and $B'_{G'}(a_1):=B'_G(a_1)\sqcup\{b'\}$. By construction we have $G'\in \mathfrak G(\mathcal I_0,\mathcal B_0,\hat {\mathcal D}_0)$. To show that $G'$ is unsplittable we argue as follows: by the definition of $\mu^G_1(b_1)$, it lies on the same connected component of $G$ as $b_1$ if no edges of $G$ corresponding to pairs of the form   $(\mu^G_1(b_1),c)$ or $(c,\mu^G_1(b_1))$  are removed, therefore $b_1$ and $\mu^G_1(b_1)$ lie on the same connected component of $G'$ after removing the edge of $G'$ corresponding to the pair $(a_1,b')$; on the other hand, $a_1$ is contained in $u_1$, and from \eqref{eq partition open open} we know that $\mu^G_1(b_1)$ is contained in $u_2$, which means that $a_1$ and $b_1$ lie on different connected components after removing the edge of $G'$ corresponding to the pair $(a_1,b')$, thus by definition we have $\mu^{G'}(b_1)=a_1$ and $G'$ is unsplittable.

    The above construction is easily seen to be reversible. Moreover, we have $$T_{\mathfrak p}^{oo}(G)=-T(\Phi(G,\mathfrak p))=-T(G')$$
    by definition, where the difference in sign comes from the fact that $G'$ has one more edge than $G$: the edge corresponding to the pair $(a_1,b')$. Therefore we have 
    $$
    \sum_{G\text{ splittable}}\sum_{\mathfrak p} T^{oo}_{\mathfrak p}(G)+\sum_{G\text{ unsplittable}}T(G)=0
    $$
    and \eqref{eq splitted lhs} can be simplified to
    \begin{equation}\label{eq decompose PL}
    P_1=\sum_{G\text{ splittable}}\sum_{\mathfrak q} T^{co}_{\mathfrak q}(G)+\sum_{G\text{ exceptional}}T(G).
    \end{equation}
    We now consider the right-hand side $P_2$ of \eqref{eqtrr1}. Each term of $P_2$ is of the form
    \begin{equation}\label{eq term in h=0 trr}
    (-1)^s F^{ext}\cdot F^{\h=0}(I_0\sqcup\{r-2-a\},B_0,D_0\sqcup\{0\})\cdot \prod_{j=1}^s F^{\h=0}(I_s,B_s\sqcup\{r-2-2t_j\},D_s),
    \end{equation}
   where    \begin{equation}\label{eq closed extend factor}
    F^{ext}:=\left\langle    \tau^{a}_{0}\tau^{a_1}_{d_1}\prod_{a_i\in I_{-1}}\tau^{a_i}_{d_i}\prod_{j=1}^{s}\tau^{t_j}_{0}\right\rangle_0^{\frac{1}{r},\text{ext}}\end{equation}
    is the closed extended factor, the factors of the form $F^{\h=0}$ (see \eqref{eq F h=0}) are the open factors, and we have  partitions
    \begin{equation}\label{eq partition p2}
    \bigsqcup_{j=-1}^s I_j=\mathcal I_0,~\bigsqcup_{j=-1}^s D_j=\mathcal D_0\text{ and }\bigsqcup_{j=0}^s B_j=\mathcal B_0\text{ such that }b_1\in I_0\sqcup B_0.
    \end{equation}
    Since $\h=0$, all the inserted internal markings have twist zero, so $t_j=0$ for all the $t_j$ in \eqref{eq closed extend factor}. We divide the terms of $P_2$ into two types according to their closed extended factor $F^{ext}$:
    \begin{itemize}
        \item terms with a closed extended factor \eqref{eq closed extend factor} such that $I_{-1}=\emptyset$ and $s=d_1+1$; we say such terms of $P_2$ are \textit{boundary terms};
        \item the remaining terms,  we say such terms of $P_2$ are \textit{non-boundary terms}.
    \end{itemize}
    Then we can decompose $P_2$ as 
    \begin{equation}\label{eq decompose PR}    P_2=\sum_{T^{nbd}_\alpha\text{ non-boundary}}T^{nbd}_\alpha+\sum_{T^{bd}_\beta\text{ boundary}}T^{bd}_\beta.
    \end{equation}

    For a non-boundary term $T^{nbd}_\alpha$ of $P_2$ of form \eqref{eq term in h=0 trr} corresponding to the partitions \eqref{eq partition p2}, we can rewrite its closed extended factor \eqref{eq closed extend factor} using the string equation \cite[Lemma 3.7]{BCT_Closed_Extended} as
    $$
    \left\langle    \tau^{a}_{0}\tau^{a_1}_{d_1}\prod_{a_i\in I_{-1}}\tau^{a_i}_{d_i}\prod_{j=1}^{s}\tau^{t_j}_{0}
		\right\rangle_0^{\frac{1}{r},\text{ext}}=\sum_{\bigsqcup_{a_i\in I_{-1}\sqcup\{a_1\}}S_i=\{1,2,\dots,s\}}\left\langle    \tau^{a}_{0}\tau^{a_1}_{d_1-\vert S_1\vert}\prod_{a_i\in I_{-1}}\tau^{a_i}_{d_i-\vert S_i\vert}
		\right\rangle_0^{\frac{1}{r},\text{ext}}.
    $$
    On the other hand, all the open factors of \eqref{eq term in h=0 trr} are of the form $F^{\h=0}(\mathcal I',\mathcal B',\mathcal D')$ for some $(\mathcal I',\mathcal B',\mathcal D')$ such that $S(\mathcal D')\le S(\mathcal D_0)=S(\hat{\mathcal D}_0-1)$. By our inductive hypothesis \eqref{eq inductive hypothesis comparision}, we can rewrite them as $\sum_{G\in \mathfrak G(\mathcal I',\mathcal B',\mathcal D')}T(G)$ using \eqref{eq: compare all}.

    Therefore, after rewrite every factors of terms in $\sum_{T^{nbd}_\alpha\text{ non-boundary}}T^{nbd}_\alpha$ 
 as above, we rewrite $\sum_{T^{nbd}_\alpha\text{ non-boundary}}T^{nbd}_\alpha$ as the sum of following terms
    \begin{equation}\label{eq:a closed open type summand on rhs}
    \begin{split}
          (-1)^s\left\langle    \tau^{a}_{0}\tau^{a_1}_{d_1-\vert S_1\vert}\prod_{a_i\in I_{-1}}\tau^{a_i}_{d_i-\vert S_i\vert}
		\right\rangle_0^{\frac{1}{r},\text{ext}}&\cdot T(G_0)
        \cdot \prod_{i=1}^s T(G_i),
    \end{split}   
    \end{equation}
    where the sum is taken over the set $NBD$ of all possible data \begin{equation}\label{eq data nbd}(s,a, \{I_j\},\{B_j\},\{S_j\},\{G_j\})\end{equation} satisfying:
     \begin{itemize}
        \item[--] $s\ge 0$, $-1\le a \le r-2$;
        \item[--]  $G_0\in \mathfrak G\left( I_0\sqcup\{\hat a=r-2-a\}, B_0, \{d_j\}_{a_j\in I_0}\sqcup\{0\}\right)$, $b_1\in  B_0$, $B'(\hat a)=\emptyset$; 
        \item[--]  
        $G_i\in \mathfrak G\left( I_i, B_i\sqcup \{b'_i=r-2\}, \{d_j\}_{a_j\in I_i}\right)$ for $1\le i \le s$;
        \item[--]  $I_{-1}\sqcup\bigsqcup_{i=0}^s  I_i=\mathcal I_0 \setminus\{a_1\}$, $\bigsqcup_{i=0}^s  B_i=\mathcal B_0$;
        \item[--]  $\bigsqcup_{a_i\in I_{-1}\sqcup\{a_1\}}S_i=\{1,2,\dots,s\}$;
        \item[--]$s\le d_1$ if $I_{-1}=\emptyset$. 
    \end{itemize}
    We now describe a bijection between the set $NBD$ of the above data and the set  $$\left\{(G,\mathfrak p)\left \vert \begin{aligned}G\in& \mathfrak G(\mathcal I_0,\mathcal B_0,\hat {\mathcal D}_0),~G \text{ splittable}\\ &\mathfrak q \text{ a partition as in \eqref{eq partition closed open}}\end{aligned}\right.\right\};$$
    moreover, the term \eqref{eq:a closed open type summand on rhs} coincides with the corresponding closed-open term $T^{co}_{\mathfrak{q}}(G)$ in $P_1$.

     Given an element in $NBD$ as above, we can construct $G$ as follows: 
     \begin{enumerate}
     \item 
     start with the disjoint union of $G_0,G_1,\dots,G_s$; 
     \item  let $\hat v\in V(G_0)$ be the vertex such that $\hat a \in I(\hat v)$, and modify $I_G(\hat v):= I_{-1}\sqcup\{a_1\}\sqcup I_{G_0}(\hat v)\setminus \{\hat a\}$; 
     \item 
     set $B'(a_i):=\{b'_j\}_{j\in S_i}$ for all $a_i\in I_{-1}\sqcup\{a_1\},$ and add the corresponding edges.
     \end{enumerate}
     By definition, $b_1$ is contained in a vertex $v_b$ of $G_0$, and $a_1$ is contained in the vertex $\hat v$ of $G_0$. Since the edges of $G$ corresponding to the pairs $\{(a_1,b'_j)\}_{j\in S_1}$ are new edges added to $G$, vertices of $G_0$ lie on the same connected component after removing them. Therefore $\mu^G_1(b_1)\ne a_1$ (note that $\mu^G_1(b_1)\notin I_{-1}$ for the same reason) and the graph $G$ constructed in this way is splittable.
    We take $\mathfrak q$ to be the partition
    $$I_G(\hat v)\setminus \{a_1\}=I_{-1}\sqcup \left(I_{G_0}(\hat v)\setminus \{\hat a\}\right);$$
    note that by construction $a_1\in \hat v$, and we have $\mu^G_1(b_1)\notin I_{-1}$ as above, so $\mathfrak q$ is actually a partition as in \eqref{eq partition closed open}. By construction the corresponding term $T^{co}_{\mathfrak{q}}(G)$ coincides with \eqref{eq:a closed open type summand on rhs}.
    
    The above construction is reversible, and thus we have an equality  \begin{equation}\label{eq closed open equal non-boundary} 
        \sum_{G\text{ splittable}}\sum_{\mathfrak q} T^{co}_{\mathfrak q}(G)=\sum_{T^{nbd}_\alpha\text{ non-boundary}}T^{nbd}_\alpha.
    \end{equation}

     Similarly, each boundary term of $P_2$ is of the form \begin{equation}\label{eq:bd term}
    \begin{split}
         T^{bd}_\beta=(-1)^{d_1+1}\left\langle    \tau^{a}_{0}\tau^{a_1}_{d_1}\prod_{k=1}^{d_1+1}\tau^{t_k}_{0}\right\rangle_0^{\frac{1}{r},\text{ext}}\cdot T(G_0)
        \cdot \prod_{i=1}^{d_1+1} T(G_i).
    \end{split}   
    \end{equation} 
    By using the string equation \cite[Lemma 3.7]{BCT_Closed_Extended}, the closed extended factor 
    \begin{equation}\label{eq closed extend factor boundary}
    F^{ext}=\left\langle    \tau^{a}_{0}\tau^{a_1}_{d_1}\prod_{j=1}^{d_1+1}\tau^{t_j}_{0}\right\rangle_0^{\frac{1}{r},\text{ext}}\end{equation}
    of $T^{bd}_\beta$ equals $1$  when $a=r-2-a_1$  and vanishes otherwise.

    Again there is a bijection between the set $BD$ of data (similar to \eqref{eq data nbd} in NBD) $$(s=d_1+1,a=r-2-a_1, \{I_j\},\{B_j\},\{S_j\},\{G_j\})$$ determining boundary terms $$T^{bd}_\beta=(-1)^{d_1+1}T(G_0)
        \cdot \prod_{i=1}^{d_1+1} T(G_i)$$ with non-zero closed extended factor, and the set
    $$\left\{G\left \vert \begin{aligned}G\in \mathfrak G(\mathcal I_0,\mathcal B_0,\hat {\mathcal D}_0),~G \text{ exceptional}\end{aligned}\right.\right\};$$
    and again, the boundary term $T^{bd}_\beta$ coincides with the corresponding exceptional term $T(G)$.
    
    $G$ can be constructed by taking the disjoint union of $G_0,G_1,\dots,G_s$,  setting  $B'_G(\hat a):=\{ b'_j\}_{1\le j \le d_1+1}$ and adding the corresponding $d_1+1$ edges.
        Note that $\hat a=a_1$ in this case, the two terms $T^{bd}_\beta$ and $T(G)$ coincide after we identifying $\hat a$ with $a_1$.
         Again this construction is reversible, and we therefore  have an equality\begin{equation}\label{eq exceptional boundary}
            \sum_{G\text{ exceptional}}T(G)=\sum_{T^{bd}_\beta\text{ boundary}}T^{bd}_\beta.
        \end{equation}

    By combining \eqref{eq decompose PL},\eqref{eq decompose PR},\eqref{eq closed open equal non-boundary} and \eqref{eq exceptional boundary}, we obtain $P_1=P_2$, therefore \eqref{eq: compare all induction case} is proven, and the theorem is proven by induction.

    \end{proof}

    \begin{rmk}\label{rmk:geometric_comparison}
        Although Theorem \ref{thm:compare_descendants} was proven in a completely algebraic way, we can also deduce it in a geometrical way.  
        The argument in the proof of Theorem \ref{thm: compare primary} fails for the correlators with $\psi$-class because the relative cotangent line bundles $\mathbb L_i\to \Mbar_\mathbf{G}$ are different from the pull-back of $\mathbb L_i\to \Mbar_{\mathcal B\mathbf{G}}$ via the forgetful morphism $F_\mathbf{G}$. We denote by $\check{\mathbb L}_i:=F^*_\mathbf{G} \mathbb L_i$ the pulled-back line bundle. Note that $\check{\mathbb L}_i$ on each $\Mbar_\mathbf{G}\subseteq \Mbar^{1/r,\h=0}_{0,B,I}$ also glue to a line bundle $\check{\mathbb L}_i\to \Mbar^{1/r,\h=0}_{0,B,I}$. If we replace $\mathbb L_i$ by $\check{\mathbb L}_i$ as direct summands of $E$ in Definition \ref{dfn correlator point insertion}, then the integral coincides with the
        BCT correlator, as an argument identical to that of Theorem \ref{thm: compare primary} shows. 
        
        One can therefore deduce Theorem \ref{thm:compare_descendants} from comparing $\mathbb L_1$ and $\check{\mathbb L}_i$. We will explain this comparison argument in the simplest non-trivial case, when $d_1=1,$ and all other $d_i=0.$ The difference of zero locus of generic sections of $\mathbb L_i$ and $\check{\mathbb L}_i$ can be represented (in the relative homology group $H_2\left(\Mbar^{1/r,\h=0}_{0,B,I},\partial \Mbar^{1/r,\h=0}_{0,B,I}\right)$) by the cycle $\bigcup_{\mathbf{G}\in U_i} \Mbar_\mathbf{G}\subset \Mbar^{1/r,\h=0}_{0,B,I}$, where $U_i$ is the collection of $(r,\h)$-graph $\mathbf{G}$ satisfying:
        \begin{itemize}
            \item[--] only one vertex $\Gamma\in V(\mathbf{G})$ is non-smooth;
            \item[--] the vertex $\Gamma$ is a graded $r$-spin graph consists of an open vertex $v^o$ and a closed vertex $v^c$ connected by an edge $e$;
            \item[--] $a_i\in H(v^c)$ and $H(v^c)\setminus\{a_i, e^c\}\subseteq I'(\mathbf{G})$, where $e^c$ is the half-edge of $e$ on $v^c$ side, and $H(v^c)$ is the set of all half-edges attached to $v^c$.
        \end{itemize}
        When $d_1=1$ and $d_i=0$ for all $2\le i\le l$,  \eqref{eq: compare all} can be written as 
        \begin{equation}\label{eq compare simple case}
        \begin{split}                 &\left\langle\tau^{a_1}_{1}\tau^{a_2}_{0}\dots\tau^{a_l}_{0}\sigma^{b_1}\dots\sigma^{b_k}\right\rangle^{\frac{1}{r},o,\text{BCT}}_0-\left\langle\tau^{a_1}_{1}\tau^{a_2}_{0}\dots\tau^{a_l}_{0}\sigma^{b_1}\dots\sigma^{b_k}\right\rangle^{\frac{1}{r},o,\h=0}_0\\
        =&\sum_{\substack{I_1\sqcup I_2=I\setminus\{a_1\}\\B_1\sqcup B_2=B}}\left\langle\tau^{a_1}_{0}\prod_{a_i\in I_1}\tau^{a_i}_{0}\prod_{b_i\in B_1}\sigma^{b_i}\right\rangle^{\frac{1}{r},o,\h=0}_0\cdot \left\langle\prod_{a_i\in I_2}\tau^{a_i}_{0}\sigma^{r-2}\prod_{b_i\in B_2}\sigma^{b_i}\right\rangle^{\frac{1}{r},o,\h=0}_0,
        \end{split}
        \end{equation}
        where we have also used Theorem \ref{thm: compare primary}. Now we justify \eqref{eq compare simple case} geometrically. The right-hand side of \eqref{eq compare simple case} equals the number of zeros of a canonical section of the Witten bundle on the locus $\bigcup_{\mathbf{G}\in U_1} \oPMb_\mathbf{G}\subset \oPMb^{1/r,\h=0}_{0,B,I}$. As in Lemma \ref{lem zero locus trr section global} (and with the same notation), the locus $\bigcup_{\mathbf{G}\in U_1} \oPMb_\mathbf{G}\subset \oPMb^{1/r,\h=0}_{0,B,I}$ can be written as        
        $$\bigcup_{\substack{s\ge 1,~ t_i=0,\\ \bigsqcup_{j=0}^s I_j=I\setminus\{a_1\},~\bigsqcup_{j=0}^s B_j=B,\\ \{(t_j,I_j,B_j)\}_{1\le j \le s}\text{ unordered}}}\left( \hat{\mathcal M}_{B_0,I_0}^{\{a_1,t_1,\dots,t_s\}}\times\prod_{i=1}^s \oPMb_{0,B_i\sqcup \{r-2-2t_i\},I_i}\right).$$
       As in the proof of Theorem \ref{thm TRR}, one can show that the zero count of a canonical section of the Witten bundle on this locus is 
        \begin{equation}\label{eq zero difference when change L}
        \begin{split}
            \sum_{\substack{s\ge 1,~ t_i=0,~-1\le a \le r-2\\ \bigsqcup_{j=0}^s I_j=I\setminus\{a_1\},~\bigsqcup_{j=0}^s B_j=B,\\ \{(t_j,I_j,B_j)\}_{1\le j \le s}\text{ unordered}}}&\left\langle    \tau^{a}_{0}\tau^{a_1}_{0}\prod_{i=1}^{s}\tau^{t_i}_{0}\right\rangle_0^{\frac{1}{r},\text{ext}}\cdot\left\langle\tau^{r-2-a}_{0}\prod_{a_i\in I_1}\tau^{a_i}_{0}\prod_{b_i\in B_1}\sigma^{b_i}\right\rangle^{\frac{1}{r},o,\h=0}_0\\ & \cdot\prod_{i=1}^s \left\langle\prod_{a_i\in I_2}\tau^{a_i}_{0}\sigma^{r-2-2t_i}\prod_{b_i\in B_2}\sigma^{b_i}\right\rangle^{\frac{1}{r},o,\h=0}_0 .  
         \end{split}
        \end{equation}
        By using the string equation \cite[Lemma 3.7]{BCT_Closed_Extended}, the factor $\left\langle    \tau^{a}_{0}\tau^{a_1}_{0}\prod_{i=1}^{s}\tau^{t_i}_{0}\right\rangle_0^{\frac{1}{r},\text{ext}}$ equals to $1$ when $s=1$ and $a_1=r-2-a$, and vanishes otherwise. Therefore \eqref{eq zero difference when change L} coincides with the right-hand side of \eqref{eq compare simple case}, this justifies Theorem \ref{thm:compare_descendants} in the simplest case. The general case is similar, and involves iterating this argument.

    \end{rmk}

\section{Computations in open $(r,\h)$-spin theories}\label{sec:computations}
In this section we calculate all open $(r,\h)$-spin correlators.

All $(r,\h)$-spin correlators can be calculated from the primary correlators (correlators without $\psi$-classes) by \eqref{eqtrr1} and \eqref{eqtrr2}. Proposition \ref{prop:vanish small internal} shows that we only need to calculate the primary correlators with internal twists greater than $\h$.

\begin{dfn}
Write $\m=r-2-2\h$. Let $\mathcal S$ be the collection of pairs $(B,I)$ of multisets of twists where
$$\forall a\in I,\quad a\in\{\h+1,\h+2,\dots,r-2,r-1\}$$
and 
$$\forall b\in B,\quad b\in\{\m,\m+2,\dots,r-4,r-2\}$$
such that
\begin{equation}\label{eq dim equals rank}
2\sum_{a\in I}a+\sum_{b\in B}b-r+2=r(2\lvert I\rvert+\lvert B\rvert-3),
\end{equation}
\textit{i.e.} the rank of the Witten bundle over $\overline{\mathcal M}_{0,B,I}^{\frac{1}{r},\h}$ equals the dimension of $\overline{\mathcal M}_{0,B,I}^{\frac{1}{r},\h}$. For $(B,I)\in \mathcal S$, we define 
$$
D(B,I):=2\lvert I \rvert+\lvert B\rvert-3,
$$
$$
M(B,I):=\max_{b\in B} b,
$$
$$
\Sigma(B,I):=\sum_{a\in I}a
$$
and
$$
F(B,I):=\left\langle
\prod_{a\in I} \tau^{a}_0\prod_{b\in B }\sigma^{b}
\right\rangle_0^{\frac{1}{r},\text{o},\h}.
$$
We define the \emph{exceptional subset} $\mathcal E\subseteq\mathcal S$ by setting
$$
\mathcal E:=\{(B,I)\in \mathcal S\colon \lvert\{b\in B\colon b>\m\}\rvert \le 1\}.
$$
We also define subsets $\mathcal E_1,\mathcal E_2$ of $\mathcal E$ by
$$
\mathcal E_1:=\{(B,I)\in \mathcal E\colon I=\emptyset\}
$$
and
$$
\mathcal E_2:=\{(B,I)\in \mathcal S\colon \lvert\{b\in B\colon b>\m\}\rvert =0\},
$$
\end{dfn}
In this section we calculate $F(B,I)$ for all $(B,I)\in \mathcal S$, then all open $r$-spin correlators can be calculated via TRRs.

\subsection{General strategy}\label{subsec general strategy}
To calculate $F(B,I)$ for all $(B,I)\in \mathcal S$, we perform several inductions. Most of lemmas and propositions in this section have the following form: based on the calculation of $F(B,I)$ for all $(B,I)\in \mathcal P\subset \mathcal S$ (which we sometimes refer to as \textit{given data}) for some subset $\mathcal P$ of $\mathcal{S}$, then we can calculate $F(B,I)$ for all $(B,I)\in \mathcal Q\subset \mathcal S$ for a larger subset $\mathcal Q$. 

The following partial order is required for stating Proposition \ref{prop general computation}.

\begin{dfn}
Define a partial order $\prec$ on $\mathcal S$ in the following way: let $(I_1,B_1),(I_2,B_2)\in S$, we write $$(I_1,B_1)\prec (I_2,B_2)$$ whenever one of the following holds:
\begin{itemize}
    \item $D(B_1,I_1)<D(B_2,I_2)$;
    \item $D(B_1,I_1)=D(B_2,I_2)$ and $\lvert I_1\rvert >\lvert I_2\rvert$;
    \item $D(B_1,I_1)=D(B_2,I_2)$, $\lvert I_1\rvert =\lvert I_2\rvert$ and $\Sigma(B_1,I_1)>\Sigma(B_2,I_2)$;
    \item $D(B_1,I_1)=D(B_2,I_2)$, $\lvert I_1\rvert =\lvert I_2\rvert$, $\Sigma(B_1,I_1)=\Sigma(B_2,I_2)$ and $M(B_1,I_1)>M(B_2,I_2)$.
\end{itemize}
\end{dfn}

\begin{prop}\label{prop general computation}
For all $s=(B,I)\in \mathcal S\backslash (\mathcal E_1 \cup \mathcal E_2)$, the open correlator $F(s)$ can be calculated from the following data: closed extended correlators and open correlators of the form $F(s')$ with $s'\prec s$.
\end{prop}
The proof of Proposition \ref{prop general computation} and all other calculations in this section (except the one in Proposition \ref{propfemptyset}) are based on the following methods.

A correlator with a single $\psi$-class can be calculated by applying TRR \eqref{eqtrr1} or \eqref{eqtrr2}, and the result is a polynomial combination of primary correlators (open and closed extended). Since applying TRRs with respect to different markings yields different ways to express the same number, comparing these expressions gives rise to relations between primary correlators.
In particular, there are two types of methods to calculate $\left\langle
\begin{array}{c}
\hfill  a_1 \quad\hfill   a_2 \hfill\quad\dots\quad\hfill a_l \hfill\null\\
\hfill b_1\quad \hfill b_2\hfill\quad\dots\quad \hfill b_k\hfill\null\\
\end{array}
\right\rangle_0^{\frac{1}{r},\text{o},\h}$.
\begin{itemize}
    \item\textbf{Type-I method.}
    \newline
     When $b_1\ge \m$, we can apply TRRs \eqref{eqtrr1} or \eqref{eqtrr2} to
\begin{equation}\label{eq method 1}
\left\langle
\begin{array}{c}
\hfill \frac{b_1-\m}{2}\psi \quad\hfill a_1 \quad\hfill   a_2\hfill\quad\dots\quad\hfill a_l\hfill\null\\
\hfill \m \quad \hfill b_2\quad\hfill b_3\hfill\quad\dots \quad\hfill b_{k}\hfill\null\\
\end{array}
\right\rangle_0^{\frac{1}{r},\text{o},\h};
\end{equation}
one term in the TRRs is
$$
\alpha:=\left\langle
		\begin{array}{c}
		\hfill \frac{b_1-\m}{2} \quad\hfill \frac{r-2-b_1}{2}\quad\hfill\frac{r-2+\m}{2} \hfill\null\\
		\end{array}
		\right\rangle_0^{\text{ext}}\cdot\left\langle
	\begin{array}{c}
	\hfill   \frac{r-2-\m}{2}\hfill\null\\
	\hfill \m\hfill\null\\
	\end{array}
	\right\rangle_0^{\frac{1}{r},\text{o},\h}\cdot\left\langle
\begin{array}{c}
\hfill  a_1 \quad\hfill   a_2 \hfill\quad\dots\quad\hfill a_l \hfill\null\\
\hfill b_1\quad \hfill b_2\hfill\quad\dots\quad \hfill b_k\hfill\null\\
\end{array}
\right\rangle_0^{\frac{1}{r},\text{o},\h},
$$
where $\left\langle
	\begin{array}{c}
	\hfill   \frac{r-2-\m}{2}\hfill\null\\
	\hfill \m\hfill\null\\
	\end{array}
	\right\rangle_0^{\frac{1}{r},\text{o},\h}$ is a central contribution and $\left\langle
\begin{array}{c}
\hfill  a_1 \quad\hfill   a_2 \hfill\quad\dots\quad\hfill a_l \hfill\null\\
\hfill b_1\quad \hfill b_2\hfill\quad\dots\quad \hfill b_k\hfill\null\\
\end{array}
\right\rangle_0^{\frac{1}{r},\text{o},\h}$ is a side contribution (see Remark \ref{rmk central side contribution}); we refer to this term as the \textit{main term}.

Let $c_1,c_2\in \{a_1,\dots,a_l,\}\sqcup\{\m,b_1,\dots,b_k\}$ be two different (internal or boundary) markings. Assume that the result of applying TRR to \eqref{eq method 1} with respect to $c_i$ has the form $$
A_i\alpha+\sum_j B_{j,i}\beta_j 
$$
for $i=1,2$, where $\{\beta_j\}_j\neq \alpha$ are other terms appearing in the TRRs and $A_i,B_{j,i}$ are coefficients (note that different partitions of sets in \eqref{eqtrr1} and \eqref{eqtrr2} may give the same term). If we can show that
\begin{itemize}
    \item the coefficients of $\alpha$ are different in two different TRRs, \textit{i.e.} $A_1\ne A_2$,
    \item all the factors in the terms $\beta_j$ are already calculated, \textit{i.e.} they are in the data given in the statement of the proposition (note that all the closed extended $r$-spin correlators are calculated in \cite[Proposition 5.1]{BCT2}),
\end{itemize}
then we can calculate $\alpha$ from the given data since
$$
\alpha=\frac{\sum_j(B_{j,2}-B_{j,1})\beta_j}{A_1-A_2},
$$
and we can calculate $\left\langle
\begin{array}{c}
\hfill  a_1 \quad\hfill   a_2 \hfill\quad\dots\quad\hfill a_l \hfill\null\\
\hfill b_1\quad \hfill b_2\hfill\quad\dots\quad \hfill b_k\hfill\null\\
\end{array}
\right\rangle_0^{\frac{1}{r},\text{o},\h}$ since both $\left\langle
		\begin{array}{c}
		\hfill \frac{b_1-\m}{2} \quad\hfill \frac{r-2-b_1}{2}\quad\hfill\frac{r-2+\m}{2} \hfill\null\\
		\end{array}
		\right\rangle_0^{\text{ext}}$ and $\left\langle
	\begin{array}{c}
	\hfill   \frac{r-2-\m}{2}\hfill\null\\
	\hfill \m\hfill\null\\
	\end{array}
	\right\rangle_0^{\frac{1}{r},\text{o},\h}$ are zero-dimensional correlators and equal $1$.

    \item\textbf{Type-II method.}
    \newline
    Let $\N$ and $\b$ be the unique pair integers satisfying 
    $$\left(\{\underbrace{ \m, \hfill \m,\hfill\dots\hfill, \m}_{\N},\b\},\emptyset\right)\in \mathcal S\text{ and } \m\le \b \le r-2,$$
    \textit{i.e.} $\N=\left\lfloor\frac{2r}{r-\m}\right\rfloor$ and $\b=(r-\m)\N-r-2$. We write $\c:=\frac{r-2-\b}{2}$; note that $\c$ is also an integer since $r-\m=2+2\h$ is even, so is $r-2-\b$.

    When $a_1\ge \c$, we can apply \eqref{eqtrr1} or \eqref{eqtrr2} to    
    $$
	\left\langle
	\begin{array}{c}
	\hfill  (a_1-\c)\psi\quad\hfill a_2\quad\hfill\dots\quad \hfill a_l\hfill\null\\
	\hfill\smash{\underbrace{ \m\quad \hfill \m\quad\hfill\dots\hfill\quad \m}_{\N}\quad\hfill b_1\quad\hfill b_2\hfill\quad\dots \quad\hfill b_{k}\hfill\null}\hfill\null\\
	\end{array}
	\right\rangle_0^{\frac{1}{r},\text{o},\h}\vphantom{\left\langle
		\begin{array}{c}
		\hfill  a_1\quad \hfill\dots\quad \hfill a_l\hfill\null\\
		\hfill\underbrace{ m\quad \hfill m\quad\hfill\dots\hfill\quad m}_{K(I)}\quad\hfill x(I)\hfill\null\\
		\end{array}
		\right\rangle_0^{\frac{1}{r},\text{o},\h}};
	$$
 this time the main term is
 $$
    \alpha:=\left\langle
		\begin{array}{c}
		\hfill \c \quad\hfill a_1-\c\quad\hfill r-2-a_1 \hfill\null
  \end{array}\right\rangle_0^{\text{ext}}\cdot\left\langle
\begin{array}{c}
\hfill  a_1 \quad\hfill   a_2 \hfill\quad\dots\quad\hfill a_l \hfill\null\\
\hfill b_1\quad \hfill b_2\hfill\quad\dots\quad \hfill b_k\hfill\null\\
\end{array}
\right\rangle_0^{\frac{1}{r},\text{o},\h}\cdot \left\langle
\begin{array}{c}
\hfill  \emptyset\hfill\null\\
\hfill \smash{\underbrace{\m\quad\hfill\dots\hfill\quad \m}_{\N}}\quad\hfill \b\hfill\null\\
\end{array}
\right\rangle_0^{\frac{1}{r},\text{o},\h}
\vphantom{\left\langle
		\begin{array}{c}
		\hfill  a_1\quad \hfill\dots\quad \hfill a_l\hfill\null\\
		\hfill\underbrace{ m\quad \hfill m\quad\hfill\dots\hfill\quad m}_{K(I)}\quad\hfill x(I)\hfill\null\\
		\end{array}
		\right\rangle_0^{\frac{1}{r},\text{o},\h}},
 $$
 where $\left\langle
\begin{array}{c}
\hfill  a_1 \quad\hfill   a_2 \hfill\quad\dots\quad\hfill a_l \hfill\null\\
\hfill b_1\quad \hfill b_2\hfill\quad\dots\quad \hfill b_k\hfill\null\\
\end{array}
\right\rangle_0^{\frac{1}{r},\text{o},\h}$ is a central contribution and  $\left\langle
\begin{array}{c}
\hfill  \emptyset\hfill\null\\
\hfill \smash{\underbrace{\m\quad\hfill\dots\hfill\quad \m}_{\N}}\quad\hfill \b\hfill\null\\
\end{array}
\right\rangle_0^{\frac{1}{r},\text{o},\h}
\vphantom{\left\langle
		\begin{array}{c}
		\hfill  a_1\quad \hfill\dots\quad \hfill a_l\hfill\null\\
		\hfill\underbrace{ m\quad \hfill m\quad\hfill\dots\hfill\quad m}_{K(I)}\quad\hfill x(I)\hfill\null\\
		\end{array}
		\right\rangle_0^{\frac{1}{r},\text{o},\h}}$ is a side contribution.

     Similarly to the type-I method, if the coefficients of $\alpha$ are different in two different TRRs, and all the factors in the terms $\beta_j$ are already calculated, then we can calculate $\alpha$ from the given data. Moreover, we will show in Section \ref{subsec computation minimal} that the factor $\left\langle
\begin{array}{c}
\hfill  \emptyset\hfill\null\\
\hfill \smash{\underbrace{\m\quad\hfill\dots\hfill\quad \m}_{\N}}\quad\hfill \b\hfill\null\\
\end{array}
\right\rangle_0^{\frac{1}{r},\text{o},\h}
\vphantom{\left\langle
		\begin{array}{c}
		\hfill  a_1\quad \hfill\dots\quad \hfill a_l\hfill\null\\
		\hfill\underbrace{ m\quad \hfill m\quad\hfill\dots\hfill\quad m}_{K(I)}\quad\hfill x(I)\hfill\null\\
		\end{array}
		\right\rangle_0^{\frac{1}{r},\text{o},\h}}$ of $\alpha$ is computable and non-zero, then we can calculate $\left\langle
\begin{array}{c}
\hfill  a_1 \quad\hfill   a_2 \hfill\quad\dots\quad\hfill a_l \hfill\null\\
\hfill b_1\quad \hfill b_2\hfill\quad\dots\quad \hfill b_k\hfill\null\\
\end{array}
\right\rangle_0^{\frac{1}{r},\text{o},\h}$ since the factor $\left\langle
		\begin{array}{c}
		\hfill \c \quad\hfill a_1-\c\quad\hfill r-2-a_1 \hfill\null
  \end{array}\right\rangle_0^{\text{ext}}$ is a zero-dimensional correlator and equals $1$.
\end{itemize}

\begin{rmk}\label{rmk determine central side contribution}
    We describe all the open correlators that appear as a factor in some terms after we apply TRRs \eqref{eqtrr1} or \eqref{eqtrr2} to
    $$\left\langle
\begin{array}{c}
\hfill a_0\psi \quad\hfill  a_1 \quad\hfill   a_2 \hfill\quad\dots\quad\hfill a_l \hfill\null\\
\hfill b_1\quad \hfill b_2\hfill\quad\dots\quad \hfill b_k\hfill\null\\
\end{array}
\right\rangle_0^{\frac{1}{r},\text{o},\h}.$$
\begin{itemize}
    \item All side contributions are of the form $F(B^{sd}\cup\{b'\},I^{sd})$, where $B^{sd}\subseteq \{b_1,b_2,\dots,b_k\}$, the boundary marking $b'$ comes from point insertion procedure, and $I^{sd}\subseteq\{a_1,a_2,\dots,a_l\}.$ 
    \item All central contributions are of the form $F(B^{ctr},I^{ctr}\cup\{a'\})$, where $B^{ctr}\subseteq \{b_1,b_2,\dots,b_k\}$,  $I^{ctr}\subseteq\{a_1,a_2,\dots,a_l\}$, and $a'$ corresponds to the internal half-node that connects the central open component to the close component in Figure \ref{fig trr}.
\end{itemize}
For a non-zero term of the form 
$$F^{ext}\cdot F(B^{ctr},I^{ctr}\cup\{a'\})\cdot \prod_{i=1}^s F(B_i^{sd}\cup\{b'_i\},I_i^{sd}),$$
where $F^{ext}$ denotes the closed extended correlator, we have 
$$
B^{ctr}\sqcup\bigsqcup_{i=1}^s B^{sd}_i=\{b_1,b_2,\dots,b_k\}~\text{and}~I^{ctr}\sqcup\bigsqcup_{i=1}^s I^{sd}_i\subseteq\{a_1,a_2,\dots,a_l\};
$$
denoting by $I^{ext}:=\{a_1,a_2,\dots,a_l\}\setminus \left(I^{ctr}\sqcup\bigsqcup_{i=1}^s I^{sd}_i\right)$, then
\begin{equation}\label{eq central side dimension}
    D(B^{ctr},I^{ctr}\cup\{a'\})+ \sum_{i=1}^s D(B_i^{sd}\cup\{b'_i\},I_i^{sd})+2(\lvert I^{ext}\rvert+s-1)=\frac{2\sum_{i=0}^l a_i+\sum_{j=1}^k b_j-r+2}{r}
\end{equation}
is a constant for all non-zero terms appearing in TRRs.

\end{rmk}

\begin{proof}[Proof of Proposition \ref{prop general computation}]
We denote $I=\{a_1,\dots,a_l\}$ and $B=\{b_1,\dots,b_k\}$. Without loss of generality, we assume $b_1\ge b_2\ge \dots \ge b_k$. Since $s\notin \mathcal E_2$ we have  $k\ge 1$ and $b_1>\m$. 

\par
To calculate 	$$
F(s)=\left\langle
\begin{array}{c}
\hfill  a_1 \quad\hfill   a_2\hfill\quad\dots\quad\hfill a_l\hfill\null\\
\hfill b_1\quad \hfill b_2\quad\hfill b_3\hfill\quad\dots\quad \hfill b_k\hfill\null\\
\end{array}
\right\rangle_0^{\frac{1}{r},\text{o},\h},
$$
we use the type-I method and apply TRRs \eqref{eqtrr1} or \eqref{eqtrr2} to
 $$
\left\langle
\begin{array}{c}
\hfill \frac{b_1-\m}{2}\psi \quad\hfill a_1 \quad\hfill   a_2\hfill\quad\dots\quad\hfill a_l\hfill\null\\
\hfill \m \quad \hfill b_2\quad\hfill b_3\hfill\quad\dots \quad\hfill b_{k}\hfill\null\\
\end{array}
\right\rangle_0^{\frac{1}{r},\text{o},\h}.
$$
Let $F(B',I')$ be an open contribution which appears in a possibly non-zero term satisfying $D(B',I')\ge D(B,I)$, there are several possibilities according to Remark \ref{rmk determine central side contribution}:
\begin{itemize}
    \item $F(B',I')$ is a side contribution with $(B',I')=(B,I)$. Then the term it contributed to is the main term $\alpha$ in the type-I method. 
    \item $F(B',I')$ is a side contribution with $I'=I$ and $B'=(B\backslash\{b_1,b_i\})\cup\{\m,b_1+b_i-\m\}$ for some $b_i > \m$. In this case we have $D(B',I')=D(B,I)$, $\lvert I'\rvert=\lvert I\rvert$ and $\Sigma(B',I')=\Sigma(B,I)$. Since 
    $$
     M(B',I')\ge b_1+b_i-\m>b_1=M(B,I),
    $$
    we have $(B',I')\prec (B,I)$.
    \item $F(B',I')$ is a central contribution, and the term it contributed to has no side contribution. In this case  $I'=(I\backslash\{a_i\})\cup\{a_i+\frac{b_1-\m}{2}\}$ and $B'=(B\backslash\{b_1\})\cup\{\m\}$. So we have $$D(B',I')=D(B,I), \lvert I'\rvert=\lvert I\rvert$$ and
    $$\Sigma(B',I')=\Sigma(B,I)+\frac{b_1-\m}{2}>\Sigma(B,I),$$
    which means $(B',I')\prec (B,I)$.
    \item $F(B',I')$ is a central contribution, and the term it contributed to has one side contribution $F(B'',I'')$ with $D(B'',I'')=0$.
    \begin{itemize}
        \item If $\lvert I''\rvert=\lvert B''\rvert=1$, we have $I''=\{a_i\}$ and $B''=\{r-2-2a_i\}$ for some $i\in\{1,2,\dots,l\}$. This is impossible since
        $$
        r-2-2a_i\le r-2-2(\h+1)=\m-2,
        $$
        which means $(B'',I'')\notin \mathcal S$.
        \item If $\lvert I''\rvert=0$ and $\lvert B''\rvert=3$, in this case we have $D(B',I')=D(B,I)$ and $\lvert I'\rvert=\lvert I\rvert+1$, which means $(B',I')\prec (B,I)$.
    \end{itemize}
\end{itemize}
Therefore all open contributions in the non-zero terms other than the main term $\alpha$ are of the form $F(B',I')$ for $(B',I')\prec (B,I)$.

To calculate $F(B,I)$ by the type-I method and prove the proposition, it remains to show that there exist two markings such that the coefficients of the main term $\alpha$ are different in the TRRs with respect to them. Actually, since $s\notin \mathcal E_1 \cup \mathcal E_2$, at least one of the following two cases holds:
\begin{enumerate}
    \item $l\ge 1, k\ge 1$.
    \par
    In this case, the coefficient of $\alpha$ in the TRR with respect to the internal marking $a_1$ is $0$, while the coefficient of $\alpha$ in the TRR with respect to the boundary marking $\m$ is $-1$.
    \item $l=0,k\ge 2, b_2\ge \m$.
    \par
    In this case, the coefficient of $\alpha$ in the TRR with respect to the boundary marking $b_2$ is $0$, while the coefficient of $\alpha$ in the TRR with respect to the boundary marking $\m$ is $-1$.
\end{enumerate}

\end{proof}

We calculate $F(s)$ for $s\in \mathcal{E}_1\cup \mathcal E_2$ in the following subsections. Note that $\mathcal S$ is a finite set, Proposition \ref{prop general computation} indicates that we can calculate $F(s)$ for all $s\in \mathcal S$ if we know $F(s)$ for all $s\in \mathcal{E}_1\cup \mathcal E_2$.

\subsection{Computation of $\left\langle
	\sigma^{\m}\sigma^{\m}\dots\sigma^{\m}\sigma^\b
	\right\rangle_0^{\frac{1}{r},\text{o},\h}$}\label{subsec computation minimal}
We begin by considering correlators of the form $\left\langle
\begin{array}{c}
\hfill  \emptyset\hfill\null\\
\hfill \m\quad \hfill \m\quad\hfill\dots\hfill\quad \m\quad\hfill \b\hfill\null\\
\end{array}
\right\rangle_0^{\frac{1}{r},\text{o},\h}$ which appear in the type-II method. Let $\N+1$
 be the number of boundary markings. For fixed  $r$ and $\m$, the correlators are zero unless  $$\N\m+\b-r+2=(\N-2)r\text{ and }\m\le \b\le r-2;$$ these equations uniquely determine the integers $\N=\left\lfloor\frac{2r}{r-\m}\right\rfloor$ and $\b=(r-\m)\N-r-2$. 
\begin{lem}\label{lemnopointinsertion}
	Let $B=\{\underbrace{ \m, \hfill \m,\hfill\dots\hfill, \m}_{\N},\b\}$. For any Neveu–-Schwarz boundary node in a boundary strata of $\overline{\mathcal M}_{0,B,\emptyset}^{\frac{1}{r}}$, the twist of the illegal side is strictly greater than $2\h$;
\end{lem}
\begin{proof}
	If $\m=0$, then $\N=2$ and  $\overline{\mathcal M}_{0,B,\emptyset}^{\frac{1}{r}}$ has dimension 0, in this case the lemma holds automatically. 
	\\We now assume $\m>0$.
	Notice first that since
	$$
	\m\le \b= (\N-2)r-\N\m+r-2\le r-2,
	$$
	then	\begin{equation}\label{eqroverm}	
	\frac{\N+1}{\N-1}+\frac{2}{(\N-1)\m}\le \frac{r}{\m}\le \frac{\N}{\N-2}.
	\end{equation}
	Consider the boundary node in Figure \ref{fig: boundary node no internal marking}, where $k\ge 1$, $l\ge 2$, $k+l=\N$ and $x+y=r-2$.
	
 \begin{figure}[h]
 \centering
	\begin{tikzpicture}
	\draw (0,0) arc (0:-80:1.5 and 0.5);
	\draw[dotted] (0,0) arc (0:-180:1.5 and 0.5);
	\draw (-3,0) arc (180:260:1.5 and 0.5);
	\draw[dashed](0,0) arc (0:180:1.5 and 0.5);
	\draw(0,0) arc (0:180:1.5);
	
	\node at (-2,-0.48) [circle,fill,inner sep=1pt]{};
	\node at (-2,-0.61) {\small$m$};
	
	\node at (-1,-0.48) [circle,fill,inner sep=1pt]{};
	\node at (-1,-0.61) {\small$m$};
	
	\node at (-0.6,-0.4) [circle,fill,inner sep=1pt]{};
	\node at (-0.6,-0.6) {\small$m$};
	
	\node at (-1.3,-0.8) {$\underbrace{\phantom{dddddddd}}_{l}$};
	
	\draw (0,0) arc (180:260:1.5 and 0.5);
	\draw (0,0)[dotted] arc (180:360:1.5 and 0.5);
	\draw (3,0) arc (0:-80:1.5 and 0.5);
	\draw[dashed](0,0) arc (180:0:1.5 and 0.5);
	\draw(0,0) arc (180:0:1.5);
	
	\node at (2,-0.48) [circle,fill,inner sep=1pt]{};
	\node at (2,-0.61) {\small$m$};
	
	\node at (1,-0.48) [circle,fill,inner sep=1pt]{};
	\node at (1,-0.61) {\small$m$};
	
	\node at (0.6,-0.4) [circle,fill,inner sep=1pt]{};
	\node at (0.6,-0.6) {\small$m$};
	
	\node at (2.4,-0.4) [circle,fill,inner sep=1pt]{};
	\node at (2.44,-0.55) {\small$\b$};
	
	\node at (1.3,-0.8) {$\underbrace{\phantom{dddddddd}}_{k}$};
	
	\node at (0,0) [circle,fill,inner sep=1pt]{};
	\node at (-0.2,0) {\small$x$};
	\node at (0.2,0) {\small$y$};
	
	\end{tikzpicture}
    \caption{}\label{fig: boundary node no internal marking}
 \end{figure}
	
	By \eqref{eq rank open} we have 
	$x\equiv k\m+\b$ and $y\equiv l\m$   mod  $r$. Let $s:=\frac{l\m-y}{r}$.
	There are two possibilities.
	\begin{enumerate}
		\item When the right half-edge is legal, we prove $y< \m$ by contradiction. Assume this is not the case, then
		$$
		\m\le y=l\m-sr\le r-2,
		$$
		which means
		$$
		\frac{sr}{\m}+1\le l\le \frac{(s+1)r}{\m}-\frac{2}{\m}.
		$$
		Thus by \eqref{eqroverm} 	\begin{equation}\label{eqestilegal}
		s+\frac{2s}{\N-1}+\frac{2s}{(\N-1)\m}+1\le l \le s+1+\frac{2(s+1)}{\N-2}-\frac{2}{\m}.
		\end{equation}
		Because the right half-edge is legal, by parity condition \eqref{eq parity}, 
        we can deduce that $s\not\equiv l \mod 2$. In fact, when $r$ is odd, we have $\m\equiv r \equiv 1$ and $1\equiv y=l\m-sr\equiv l-s$; when $r$ is even, we have 
		$$
		l+1\equiv \frac{l\m+x-r+2}{r}=s.
		$$
		
		For $s\not=0$, by the first inequality in \eqref{eqestilegal} we have $s+1<l$, hence $s+3\le l$; for $s=0$, since $l\ge 2$, we also have $s+3\le l$. Then, by the second inequality in \eqref{eqestilegal}, 
		$$
		 2\le \frac{2(s+1)}{\N-2}-\frac{2}{\m} <\frac{2(s+1)}{\N-2}\le\frac{2(l-2)}{\N-2}<2,
		$$
		which is a contraction.

		\item When the right half-edge is illegal, we prove $y\ge r-2-\m,$ again by contradiction.  Assume the opposite, that
		$$
		0\le y=l\m-sr\le r-2-\m,
		$$
		which means
		$$
		\frac{sr}{\m}\le l\le \frac{(s+1)r}{\m}-\frac{2}{\m}-1.
		$$
		Then by \eqref{eqroverm}		\begin{equation}\label{eqestiillegal}
		s+\frac{2s}{\N-1}+\frac{2s}{(\N-1)\m}\le l \le s+1+\frac{2(s+1)}{\N-2}-\frac{2}{\m}-1.
		\end{equation}
		Because the right half-edge is illegal, we have $s\equiv l$ mod 2. By the first inequality in \eqref{eqestiillegal} we get $s+2\le l$. Then by the second inequality in \eqref{eqestiillegal} we have
		$$
		2\le \frac{2(s+1)}{\N-2}-\frac{2}{\m}<\frac{2(l-1)}{\N-2}=\frac{2(\N-k-1)}{\N-2}\le 2,
		$$
		which is a contraction.
	\end{enumerate}
\end{proof}
\begin{rmk}
	As a special case for $\h=0$, Lemma \ref{lemnopointinsertion} works for BCT moduli spaces $\overline{\mathcal M}_{0,B^r,I}^{\frac{1}{r}}$ for $B^r=\{\underbrace{r-2,r-2,\dots,r-2}_{r+1}\}$ and $I=\emptyset$. For simplicity we denote this BCT moduli space by $\overline{\mathcal M}_{0,r+1,\emptyset}^{\frac{1}{r}}.$
\end{rmk}  
\begin{prop}\label{propfemptyset}
We have
	$$
	\left\langle
	\begin{array}{c}
	\hfill  \emptyset\hfill\null\\
	\hfill\smash{\underbrace{ \m\quad \hfill \m\quad\hfill\dots\hfill\quad \m}_{\N}}\quad\hfill \b\hfill\null\\
	\end{array}
	\right\rangle_0^{\frac{1}{r},\text{o},\h}=\left\langle
	\begin{array}{c}
	\hfill  \emptyset\hfill\null\\
	\hfill \smash{\underbrace{\N-2\quad \hfill \N-2\quad\hfill\dots\hfill\quad \N-2\quad \hfill \N-2}_{\N+1}}\hfill\null\\
	\end{array}
	\right\rangle_0^{\frac{1}{\N},\text{o,BCT}}.
	$$
\end{prop}
\begin{rmk}
    According to \cite{BCT2}, we have 
    $$
    \left\langle
	\begin{array}{c}
	\hfill  \emptyset\hfill\null\\
	\hfill \smash{\underbrace{\N-2\quad \hfill \N-2\quad\hfill\dots\hfill\quad \N-2\quad \hfill \N-2}_{\N+1}}\hfill\null\\
	\end{array}
	\right\rangle_0^{\frac{1}{\N},\text{o,BCT}}=\pm\N!.$$
 The $\pm$ sign comes from the orientation conventions, see Remark \ref{rmk: orientation issue}.
\end{rmk}
\begin{proof}[Proof of Proposition \ref{propfemptyset}]
	
	We write $B=\{\underbrace{ \m, \hfill \m,\hfill\dots\hfill, \m}_{\N},\b\}$. By lemma \ref{lemnopointinsertion}, no point insertion is carried out at any boundary of  $\overline{\mathcal M}_{0,B,\emptyset}^{\frac{1}{r}}$, therefore $\overline{\mathcal M}_{0,B,\emptyset}^{\frac{1}{r},\h}=\overline{\mathcal M}_{0,B,\emptyset}^{\frac{1}{r}}$.

	 The smooth locus $\mathcal M_{0,B,\emptyset}^{\frac{1}{r},\h}$ is isomorphic to $\mathcal M_{0,\N+1,0}$ via the forgetful morphism. The smooth locus $\mathcal M_{0,\N+1,\emptyset}^{\frac{1}{\N}}$ of the BCT moduli space is also isomorphic to $\mathcal M_{0,\N+1,0}$ via the forgetful morphism. Moreover, the fibers of both Witten bundles over  $\mathcal M_{0,B,\emptyset}^{\frac{1}{r},\h}$ and  $\mathcal M_{0,\N+1,\emptyset}^{\frac{1}{\N}}$ over $[C]\in \mathcal M_{0,\N+1,0} $ are characterized by the positions of zeros of the sections of corresponding spin bundles $\lvert J \rvert$ over $\lvert C\rvert$, up to scaling by a real scalar. We can construct, up to a real scalar, a map $\Theta$ between the sections of the two Witten bundles: $\Theta$ takes a section to another section with the same zeros. 
    Moreover, since (in both ${\mathcal M}_{0,B,\emptyset}^{\frac{1}{r},\h}$ and ${\mathcal M}_{0,\N+1,\emptyset}^{\frac{1}{\N}}$) all boundary markings are legal, according to Observation \ref{obs zero on boundary}, if the positivity condition is preserved at a single boundary point on $\partial \Sigma\subset \lvert C\rvert$ by $\Theta$, then the positivity conditions at all non-special boundary points are preserved; this allows us to choose the scalar to make $\Theta$ preserve the positivity condition at every non-special boundary point on $\partial \Sigma$. 
    
	Note that $\overline{\mathcal M}_{0,B,\emptyset}^{\frac{1}{r},\h}$ and $\overline{\mathcal M}_{0,\N+1,\emptyset}^{\frac{1}{\N}}$ contain no stratum which parameterizes disks with internal nodes. By Lemma \ref{lemnopointinsertion}, the canonical sections of Witten bundles are given by sections on their smooth loci satisfying positivity boundary condition. By construction $\Theta$ preserves the positivity boundary conditions, hence it sends canonical sections to canonical sections. By definition the open $r$-spin correlators are the numbers of zeros of the canonical sections, the proposition is proven since $\Theta$ preserves the number of zeros. 
\end{proof}

\subsection{Computation of $
	\left\langle
	\tau_0^{a_1}\dots\tau_0^{a_l}\sigma^{\m}\sigma^{\m}\dots\sigma^{\m}
	\right\rangle_0^{\frac{1}{r},\text{o},\h}
	$}
We turn to consider the correlators of the form
$
\left\langle
\begin{array}{c}
	\hfill  a_1\quad \hfill\dots\quad \hfill a_l\hfill\null\\
	\hfill\smash{\underbrace{ \m\quad \hfill \m\quad\hfill\dots\hfill\quad \m}_{K+1}}\hfill\null\\
\end{array}
\right\rangle_0^{\frac{1}{r},\text{o},\h}.
$
When $l=1$, the correlators are zero unless 
$$2a_1+(K+1)\m-r+2=Kr,$$
which means $$K+1=\frac{2a_1+2}{r-\m}=\frac{a_1+1}{\h+1}.$$
\begin{lem}\label{leml1xm}
	For $0\le K \le \N-1$ we have
	$$
	\left\langle
	\begin{array}{c}
	\hfill (K+1)(\h+1)-1\hfill\null\\
	\hfill\smash{\underbrace{ \m\quad \hfill \m\quad\hfill\dots\hfill\quad \m}_{K+1}}\hfill\null\\
	\end{array}
	\right\rangle_0^{\frac{1}{r},\text{o},\h}=K! \vphantom{\left\langle
		\begin{array}{c}
		\hfill K(h+1)-1\hfill\null\\
		\hfill\underbrace{ m\quad \hfill m\quad\hfill\dots\hfill\quad m}_{K}\hfill\null\\
		\end{array}
		\right\rangle_0^{\frac{1}{r},\text{o},\h}}.
	$$
\end{lem}
\begin{proof}
	 We induct on $K$. For $K=0$, the moduli space is zero-dimensional hence the lemma holds. For $K=1$, the moduli space is one-dimensional, and the lemma can be easily checked from definitions as in \cite[Example 3.25]{BCT2}. Assuming the lemma holds for $K_0\le N-2$. We apply TRRs to the following correlator which contains a single $\psi$ class
	 $$\left\langle
	 \begin{array}{c}
	 \hfill (\h+1)\psi\quad\hfill (K_0+1)(\h+1)-1\hfill\null\\
	 \hfill\smash{\underbrace{ \m\quad \hfill \m\quad\hfill\dots\hfill\quad \m}_{ K_0+2}}\hfill\null\\
	 \end{array}
	 \right\rangle_0^{\frac{1}{r},\text{o},\h}\vphantom{\left\langle
	 	\begin{array}{c}
	 	\hfill (h+1)\psi\quad\hfill K(h+1)-1\hfill\null\\
	 	\hfill\underbrace{ \m\quad \hfill \m\quad\hfill\dots\hfill\quad \m}_{ K+1}\hfill\null\\
	 	\end{array}
	 	\right\rangle_0^{\frac{1}{r},\text{o},\h}}.$$
 	If $K_0+2< N$, there are two possibly non-zero terms contributing in TRRs: the main term
 	$$
 	\alpha=\left\langle
 	\begin{array}{c}
 	\hfill \h+1 \quad\hfill (K_0+1)(\h+1)-1\quad\hfill r-1-(K_0+2)(\h+1)  \hfill\null\\
 	\end{array}
 	\right\rangle_0^{\text{ext}}\left\langle
 	\begin{array}{c}
 	\hfill (K_0+2)(\h+1)-1\hfill\null\\
 	\hfill\smash{\underbrace{ \m\quad \hfill \m\quad\hfill\dots\hfill\quad \m}_{K_0+2}}\hfill\null\\
 	\end{array}
 	\right\rangle_0^{\frac{1}{r},\text{o},\h}\vphantom{\left\langle
 		\begin{array}{c}
 		\hfill (K+1)(h+1)-1\hfill\null\\
 		\hfill\underbrace{ m\quad \hfill m\quad\hfill\dots\hfill\quad m}_{K+1}\hfill\null\\
 		\end{array}
 		\right\rangle_0^{\frac{1}{r},\text{o},\h}}
 	$$
 	and
 	$$
 	\beta=\left\langle
 	\begin{array}{c}
 	\hfill (K_0+1)(\h+1)-1\hfill\null\\
 	\hfill\smash{\underbrace{ \m\quad \hfill \m\quad\hfill\dots\hfill\quad \m}_{K_0+1}}\hfill\null\\
 	\end{array}
 	\right\rangle_0^{\frac{1}{r},\text{o},\h}\vphantom{\left\langle
 		\begin{array}{c}
 		\hfill K(\h+1)-1\hfill\null\\
 		\hfill\underbrace{ \m\quad \hfill \m\quad\hfill\dots\hfill\quad \m}_{K}\hfill\null\\
 		\end{array}
 		\right\rangle_0^{\frac{1}{r},\text{o},\h}}
 	\left\langle
 	\begin{array}{c}
 	\hfill \h \quad\hfill \h+1\quad\hfill r-3-2\h  \hfill\null\\
 	\end{array}
 	\right\rangle_0^{\text{ext}}\left\langle
 	\begin{array}{c}
 	\hfill 2\h+1\hfill\null\\
 	\hfill \m \quad \hfill \m\hfill\null\\
 	\end{array}
 	\right\rangle_0^{\frac{1}{r},\text{o},\h}.
 	$$
 	If $K_0+2=\N$, there is a third possibly non-zero contribution
 	$$
 	\gamma=\left\langle
	\begin{array}{c}
	\hfill  \emptyset\hfill\null\\
	\hfill\smash{\underbrace{ \m\quad \hfill \m\quad\hfill\dots\hfill\quad \m}_{\N}}\quad\hfill \b\hfill\null\\
	\end{array}
	\right\rangle_0^{\frac{1}{r},\text{o},\h}\left\langle
 	\begin{array}{c}
 	\hfill \frac{r-2-\b}{2} \quad\hfill \h+1\quad\hfill \frac{r-4+\b-2\h}{2}  \hfill\null\\
 	\end{array}
 	\right\rangle_0^{\text{ext}}\left\langle
 	\begin{array}{c}
 	\hfill \frac{r-\b+2\h}{2}\quad (\N-1)(\h+1)-1\hfill\null\\
 	\hfill \emptyset\hfill\null\\
 	\end{array}
 	\right\rangle_0^{\frac{1}{r},\text{o},\h}.
 	$$
 	The last contribution always vanishes, since we can calculate directly from the definitions as in \cite[Example 3.24]{BCT2} that 
 	$$\left\langle
 	\begin{array}{c}
 	\hfill \frac{r-\b+2\h}{2}\quad (\N-1)(\h+1)-1\hfill\null\\
 	\hfill \emptyset\hfill\null\\
 	\end{array}
 	\right\rangle_0^{\frac{1}{r},\text{o},\h}=0.$$
	Therefore in any case, the TRR with respect to a boundary marking with twist $\m$ gives
	$$
	 \alpha-(K_0+1)\beta;
	$$
	the TRR with respect to the internal marking with twist $ (K_0+1)(\h+1)-1$ gives $0$. Since all TRRs calculate the same number, it must hold that
	$$
	\alpha=(K_0+1)\beta,
	$$
	which means
	$$
	\left\langle
	\begin{array}{c}
	\hfill ((K_0+1)+1)(\h+1)-1\hfill\null\\
	\hfill\smash{\underbrace{ \m\quad \hfill \m\quad\hfill\dots\hfill\quad \m}_{(K_0+1)+1}}\hfill\null\\
	\end{array}
	\right\rangle_0^{\frac{1}{r},\text{o},\h}=(K_0+1)
	\left\langle
	\begin{array}{c}
	\hfill (K_0+1)(\h+1)-1\hfill\null\\
	\hfill\smash{\underbrace{ \m\quad \hfill \m\quad\hfill\dots\hfill\quad \m}_{K_0+1}}\hfill\null\\
	\end{array}
	\right\rangle_0^{\frac{1}{r},\text{o},\h};\vphantom{\left\langle
	\begin{array}{c}
	\hfill (K_0+1)(\h+1)-1\hfill\null\\
	\hfill{\underbrace{ \m\quad \hfill \m\quad\hfill\dots\hfill\quad \m}_{K_0+1}}\hfill\null\\
	\end{array}
	\right\rangle_0^{\frac{1}{r},\text{o},\h}}, 
	$$
	and the induction step follows.	
\end{proof}

\begin{dfn}
	For a set of internal twists $I=\{a_1,\dots,a_l\colon \h+1\le a_i\le r-1\}$, we define $K(I)$ and $x(I)$ to be the unique integers satisfying
	$$
	2\sum_{a_i\in I} a_i+K(I)\m+x(I)-r+2=(2\lvert I\rvert+K(I)-2)r
	$$
	and
	$$
	\m\le x(I)\le r-2.
	$$
	When $K(I)\ge 0$, we define $$f(I):=
	\left\langle
	\begin{array}{c}
	\hfill  a_1\quad \hfill\dots\quad \hfill a_l\hfill\null\\
	\hfill\smash{\underbrace{ \m\quad \hfill \m\quad\hfill\dots\hfill\quad \m}_{K(I)}\quad\hfill x(I)}\hfill\null\\
	\end{array}
	\right\rangle_0^{\frac{1}{r},\text{o},\h}\vphantom{\left\langle
		\begin{array}{c}
		\hfill  a_1\quad \hfill\dots\quad \hfill a_l\hfill\null\\
		\hfill\underbrace{ m\quad \hfill m\quad\hfill\dots\hfill\quad m}_{K(I)}\quad\hfill x(I)\hfill\null\\
		\end{array}
		\right\rangle_0^{\frac{1}{r},\text{o},\h}}.
	$$
	When $K(I)=-1$ and $x(I)=\m$, we define $$f(I):=
	\left\langle
	\begin{array}{c}
	\hfill  a_1\quad \hfill\dots\quad \hfill a_l\hfill\null\\
	\hfill\emptyset\hfill\null\\
	\end{array}
	\right\rangle_0^{\frac{1}{r},\text{o},\h}.
	$$
	We also write $$\mathcal I:=\{I\colon K(I)\ge 0\}\cup\{I\colon K(I)=-1,x(I)=\m\}.$$
	For $I\in \mathcal I$, we define
	$$
	D(I):=2\lvert I\rvert+K(I)-2.
	$$
	Note that $\N=K(\emptyset)$ and $\b=x(\emptyset)$, we also denote by $\c$ the integer such that $\b+2\c=r-2$ as in the type-II method.
\end{dfn}
\begin{rmk}\label{rmk bound of K}
	Observe that it always holds that $K(I)\le \N=K(\emptyset)$. If $x(I)=\m$, then $K(I)\le \N-1$.
\end{rmk}
\begin{rmk}
    For all $I\in \mathcal I$, the set 
    $
    B:=\{\underbrace{\m,\dots,\m}_{K(I)},x(I)\}
    $
    is the unique set satisfying $(B,I)\in \mathcal E$. Moreover, $(B,I)\in \mathcal E_2$ means $x(I)=\m$.
\end{rmk}

\begin{lem}\label{lemxm}
	If $\lvert I\rvert\ge 2$, $x(I)= \m$ and $-1\le K(I) \le \N-1$, then $f(I)$ can be calculated from the following given data: $f(\emptyset)$, the closed extended correlators (which were all calculated in \cite{BCT_Closed_Extended}), and correlators of the form $f(I')$ with $\lvert I'\rvert<\lvert I\rvert$.
\end{lem}
\begin{proof}
	 To calculate $f(I)$, we use the type-II method and apply TRRs to
	$$
	\left\langle
	\begin{array}{c}
	\hfill  (a_1-\c)\psi\quad\hfill a_2\quad\hfill\dots\quad \hfill a_l\hfill\null\\
	\hfill\smash{\underbrace{ \m\quad \hfill \m\quad\hfill\dots\hfill\quad \m}_{K(I)+\N}\quad\hfill \m}\hfill\null\\
	\end{array}
	\right\rangle_0^{\frac{1}{r},\text{o},\h}\vphantom{\left\langle
		\begin{array}{c}
		\hfill  a_1\quad \hfill\dots\quad \hfill a_l\hfill\null\\
		\hfill\underbrace{ m\quad \hfill m\quad\hfill\dots\hfill\quad m}_{K(I)}\quad\hfill x(I)\hfill\null\\
		\end{array}
		\right\rangle_0^{\frac{1}{r},\text{o},\h}}.
	$$
	One possibly non-zero term which appears in the TRRs is the main term
	$$
	\alpha:=f(\emptyset)f(I)\left\langle
	\begin{array}{c}
	\hfill \c \quad\hfill a_1-\c\quad\hfill r-2-a_1 \hfill\null\\
	\end{array}
	\right\rangle_0^{\text{ext}}.
	$$
	According to Remark \ref{rmk determine central side contribution}, all other terms contributing to TRRs are products of a closed extended correlator and correlators of the form $f(I')$ with $\lvert I'\rvert<\lvert I\rvert$.
	\par
	On the other hand, in the TRR expression with respect to a boundary marking, the coefficient multiplying $\alpha$ is $-\binom{K(I)+\N}{\N}$; in the TRR expression with respect to an arbitrary internal marking, the coefficient multiplying $\alpha$ is $-\binom{K(I)+\N+1}{\N}$. 
 
 The lemma is proven by the type-II method: we subtract the two different TRR expressions. This gives
 \[\left(\binom{K(I)+\N+1}{\N}-\binom{K(I)+\N}{\N}\right)\alpha+\ldots=0,\]
 where the terms in $\ldots$ are products of a closed extended correlator and correlators of the form $f(I')$ with $\lvert I'\rvert<\lvert I\rvert$. Thus we can express $\alpha,$ hence also $f(I)$ by using such terms together with $f(\emptyset)$. \end{proof}

\begin{lem}\label{lemknottoolargexnm}
	If $I\ne \emptyset$, $x(I)\neq \m$ and $0\le K(I) \le \N-2$, then $f(I)$ can be calculated from the following given data: $f(\emptyset)$, closed extended correlators, correlators of the form $f(I')$ with $\lvert I'\rvert<\lvert I\rvert$ and, in the case $\lvert I\rvert=1$,  correlators of the form $f(I')$ with $\lvert I'\rvert=1$ and $x(I')=\m$.
	
\end{lem}

\begin{proof}
	To calculate $f(I)$, we use the type-I method and apply TRRs to
	$$
	\left\langle
	\begin{array}{c}
		\hfill \frac{x(I)-m}{2}\psi \quad\hfill a_1\quad \hfill\dots\quad \hfill a_l\hfill\null\\
		\hfill\smash{\underbrace{ \m\quad \hfill \m\quad\hfill\dots\hfill\quad \m}_{K(I)}\quad\hfill \m}\hfill\null\\
	\end{array}
	\right\rangle_0^{\frac{1}{r},\text{o},\h}\vphantom{\left\langle
		\begin{array}{c}
			\hfill  a_1\quad \hfill\dots\quad \hfill a_l\hfill\null\\
			\hfill\underbrace{ m\quad \hfill m\quad\hfill\dots\hfill\quad m}_{\frac{K(I)}{2}}\quad\hfill x(I)\hfill\null\\
		\end{array}
		\right\rangle_0^{\frac{1}{r},\text{o},\h}}.
	$$
	One possibly non-zero term which appears in the TRRs is the main term
	\begin{equation}\label{eq main term knottoolargexnm}
	\alpha:=f(I)\left\langle
	\begin{array}{c}
	\hfill \frac{r-2-x(I)}{2} \quad\hfill \frac{x(I)-\m}{2}\quad\hfill \frac{r-2+\m}{2} \hfill\null\\
	\end{array}
	\right\rangle_0^{\text{ext}}
	\left\langle
	\begin{array}{c}
	\hfill   \frac{r-2-\m}{2}\hfill\null\\
	\hfill \m\hfill\null\\
	\end{array}
	\right\rangle_0^{\frac{1}{r},\text{o},\h}.
	\end{equation}
	Another type of terms which are possibly non-zero and may appear in the TRRs is
	\begin{equation}\label{eq another term knottoolargexnm}
	\beta_i:=\left\langle
	\begin{array}{c}
	\hfill a_i \quad\hfill\frac{x(I)-\m}{2}\quad\hfill r-2-a_i-\frac{x(I)-\m}{2}\hfill\null\\
	\end{array}
	\right\rangle_0^{\text{ext}}\left\langle
	\begin{array}{c}
	\hfill \frac{x(I)-\m}{2}+a_i \quad\hfill \smash{\overbrace{a_1\quad \hfill\dots\quad \hfill a_l}^{\text{without }a_i}}\hfill\null\\
	\hfill\smash{\underbrace{ \m\quad \hfill \m\quad\hfill\dots\hfill\quad \m}_{K(I)}\quad\hfill \m}\hfill\null\\
	\end{array}
	\right\rangle_0^{\frac{1}{r},\text{o},\h}\vphantom{\left\langle
		\begin{array}{c}
		\hfill  a_1\quad \hfill\dots\quad \hfill a_l\hfill\null\\
		\hfill\underbrace{ m\quad \hfill m\quad\hfill\dots\hfill\quad m}_{K(I)}\quad\hfill x(I)\hfill\null\\
		\end{array}
		\right\rangle_0^{\frac{1}{r},\text{o},\h}}.	
	\end{equation}
	If $\lvert I\rvert=1$, then $\beta_i$ is a product of a closed extended correlator and a correlator of the form $f(I')$ with $\lvert I'\rvert=1$ and $x(I')=\m$; if $\lvert I\rvert\ge2$, by lemma \ref{lemxm}, all $\beta_i$ can be calculated from $f(\emptyset)$, closed extended correlators and correlators of the form $f(I')$ with $\lvert I'\rvert<\lvert I\rvert$.
	\par
	In the case  $\lvert I\rvert\ge2$, there is an additional type of possibly non-zero terms
	$$
	\gamma_i:=f(\{a_i\})\left\langle
	\begin{array}{c}
	\hfill \frac{r-2-x(\{a_i\})}{2} \quad\hfill\frac{x(I)-\m}{2}\quad\hfill \frac{r-2+\m+x(\{a_i\})-x(I)}{2}\hfill\null\\
	\end{array}
	\right\rangle_0^{\text{ext}}\left\langle
	\begin{array}{c}
	\hfill \frac{r-2-\m-x(\{a_i\})+x(I)}{2} \quad\hfill \smash{\overbrace{a_1\quad \hfill\dots\quad \hfill a_l}^{\text{without }a_i}}\hfill\null\\
	\hfill\smash{\underbrace{ \m\quad \hfill \m\quad\hfill\dots\hfill\quad \m}_{K(I)-K(\{a_i\})}\quad\hfill \m}\hfill\null\\
	\end{array}
	\right\rangle_0^{\frac{1}{r},\text{o},\h}\vphantom{\left\langle
		\begin{array}{c}
		\hfill  a_1\quad \hfill\dots\quad \hfill a_l\hfill\null\\
		\hfill\underbrace{ m\quad \hfill m\quad\hfill\dots\hfill\quad m}_{K(I)-K(\{a_i\})}\quad\hfill x(I)\hfill\null\\
		\end{array}
		\right\rangle_0^{\frac{1}{r},\text{o},\h}}.	
	$$
	These terms can be calculated from the data given in the statement of the lemma, by applying lemma \ref{lemxm}.
	\par
	Since $K(I)<\N-1$, according to Remark \ref{rmk determine central side contribution}, all other terms contributing in TRRs are products of a closed extended correlator and correlators of the form $f(I')$ with $\lvert I'\rvert<\lvert I\rvert$.
	\par
	On the other hand, in the TRR with respect to a boundary marking $\m$, the coefficient of the main term $\alpha$ is $-1$; in the TRR with respect to an arbitrary internal marking, the coefficient of $\alpha$ is $0$. Thus $f(I)$ can be calculated from the data given in the statement of the lemma by the type-I method and the lemma is proven.
\end{proof}

The above lemmas do not cover the cases where $I\ne \emptyset$, $x(I)\neq \m$ and $K(I)=\N\text{ or }\N-1$. We treat them separately for different $\N$ in the following subsections.
\subsubsection{Case $\N\ge 4$}
In this subsection we assume  $\N\ge 4$. This implies $\m\ge \frac{r}{2}$. 

\begin{lem}\label{lemmnxlargek}
	If $\N\ge 4$, $I\ne \emptyset$, $x(I)\neq \m$ and $K(I)=\N\text{ or }\N-1$, then $f(I)$ can be calculated from the following given data: $f(\emptyset)$, closed extended correlators, correlators of the form $f(I')$ with $\lvert I'\rvert<\lvert I\rvert,$ and, in the case $\lvert I\rvert=1$, correlators of the form $f(I')$ with $\lvert I'\rvert=1$ and $x(I')=\m$.
\end{lem}
\begin{proof}
	Similarly to the proof of Lemma \ref{lemknottoolargexnm}, we use the type-I method to calculate $f(I)$: we apply different TRRs to
	$$
	\left\langle
	\begin{array}{c}
	\hfill \frac{x(I)-\m}{2}\psi \quad\hfill a_1\quad \hfill\dots\quad \hfill a_l\hfill\null\\
	\hfill\smash{\underbrace{ \m\quad \hfill \m\quad\hfill\dots\hfill\quad \m}_{K(I)}\quad\hfill \m}\hfill\null\\
	\end{array}
	\right\rangle_0^{\frac{1}{r},\text{o},\h}\vphantom{\left\langle
		\begin{array}{c}
		\hfill  a_1\quad \hfill\dots\quad \hfill a_l\hfill\null\\
		\hfill\underbrace{ m\quad \hfill m\quad\hfill\dots\hfill\quad m}_{K(I)}\quad\hfill x(I)\hfill\null\\
		\end{array}
		\right\rangle_0^{\frac{1}{r},\text{o},\h}}.
	$$
	As in the proof of Lemma \ref{lemknottoolargexnm}, we still have the main term $\alpha$ as \eqref{eq main term knottoolargexnm}, terms $\beta_i$ as \eqref{eq another term knottoolargexnm}, and terms which can be written as products of a closed extended correlator and correlators of the form $f(I')$ with $\lvert I'\rvert<\lvert I\rvert$. However, there is an additional possibly non-zero term $\delta$ that may appear in a TRR in this case:
	$$
	\delta=f(\emptyset)\left\langle
	\begin{array}{c}
	\hfill \c \quad\hfill \frac{x(I)-\m}{2}\quad\hfill r-2-\c-\frac{x(I)-\m}{2} \hfill\null\\
	\end{array}
	\right\rangle_0^{\text{ext}}
	\left\langle
	\begin{array}{c}
		\hfill \frac{x(I)-\m}{2}+\c \quad\hfill a_1\quad \hfill\dots\quad \hfill a_l\hfill\null\\
		\hfill\smash{\underbrace{ \hfill \m\quad\hfill\dots\hfill\quad \m}_{K(I)+1-\N}}\hfill\null\\
	\end{array}
	\right\rangle_0^{\frac{1}{r},\text{o},\h}\vphantom{\left\langle
		\begin{array}{c}
			\hfill  a_1\quad \hfill\dots\quad \hfill a_l\hfill\null\\
			\hfill\underbrace{ m\quad \hfill m\quad\hfill\dots\hfill\quad m}_{K(I)}\quad\hfill x(I)\hfill\null\\
		\end{array}
		\right\rangle_0^{\frac{1}{r},\text{o},\h}}.
	$$
	Assuming $\frac{x(I)-\m}{2}+\c\le r-1$ (otherwise this term is zero), we need to prove that
	$$
	f\left(I\cup\left\{\frac{x(I)-\m}{2}+\c\right\}\right)=\left\langle
	\begin{array}{c}
	\hfill \frac{x(I)-\m}{2}+\c \quad\hfill a_1\quad \hfill\dots\quad \hfill a_l\hfill\null\\
	\hfill\smash{\underbrace{ \hfill \m\quad\hfill\dots\hfill\quad \m}_{K(I)+1-\N}}\hfill\null\\
	\end{array}
	\right\rangle_0^{\frac{1}{r},\text{o},\h}\vphantom{\left\langle
		\begin{array}{c}
		\hfill  a_1\quad \hfill\dots\quad \hfill a_l\hfill\null\\
		\hfill\underbrace{ m\quad \hfill m\quad\hfill\dots\hfill\quad m}_{K(I)}\quad\hfill x(I)\hfill\null\\
		\end{array}
		\right\rangle_0^{\frac{1}{r},\text{o},\h}}
	$$
	can be calculated from the data given in the statement of the lemma. We use an additional type-II method to do this: we apply TRRs to
	$$
	\left\langle
	\begin{array}{c}
	\hfill \frac{x(I)-\m}{2}+\c \quad\hfill (a_1-\c)\psi\quad \hfill a_2 \quad \hfill\dots\quad \hfill a_l\hfill\null\\
	\hfill\smash{\underbrace{ \m\quad \hfill \m\quad\hfill\dots\hfill\quad \m}_{K(I)}\quad\hfill \m}\hfill\null\\
	\end{array}
	\right\rangle_0^{\frac{1}{r},\text{o},\h}\vphantom{\left\langle
		\begin{array}{c}
		\hfill  a_1\quad \hfill\dots\quad \hfill a_l\hfill\null\\
		\hfill\underbrace{ m\quad \hfill m\quad\hfill\dots\hfill\quad m}_{K(I)}\quad\hfill x(I)\hfill\null\\
		\end{array}
		\right\rangle_0^{\frac{1}{r},\text{o},\h}}.
	$$
	According to Remark \ref{rmk determine central side contribution}, all the possibly non-zero terms contributing here are listed below.
	\begin{itemize}
		\item 
        Terms with a central contribution $f(I')$ such that $\lvert I' \rvert>\lvert I \rvert$. The only possibly non-zero one among them is
  $$		f(\emptyset)\left\langle
		\begin{array}{c}
		\hfill \c \quad\hfill a_1-\c\quad\hfill r-2-a_1 \hfill\null\\
		\end{array}
		\right\rangle_0^{\text{ext}}f\left(I\cup\left\{\frac{x(I)-\m}{2}+\c\right\}\right).
		$$ 
		This is the main term in the additional type-II method. Note that its coefficient in the TRR with respect to a boundary marking $\m$ is $-\binom{K(I)}{\N}$, and its coefficient in the TRR with respect to an arbitrary internal marking is $-\binom{K(I)+1}{\N}$. Therefore, we need to show that all other terms can be calculated from the data given in the statement of the lemma.
		\item Terms which can be written as products of a closed extended correlator and correlators of the form $f(I')$ with $\lvert I'\rvert<\lvert I\rvert$.
		\item Terms with a central contribution $f(I')$ such that $\lvert I' \rvert=\lvert I \rvert$. Since $f(I')$ is a central contribution, we have $x(I')=\m$; 
        by Lemma \ref{lemxm} it can be calculated from the data given in the statement of the lemma. Moreover, all side contributions of these terms are of the form $f(I'')$ for $\lvert I''\rvert\le 1$, they are covered by the data given in the statement of the lemma when $\lvert I \rvert \ge 2$; the only exceptional case $\lvert I \rvert=\lvert I'' \rvert=1$ is also a special case of the next item.
		\item 
        Terms with a side contribution $f(I')$ such that $\lvert I' \rvert\ge \lvert I \rvert$. The only possibility is that $I'=J:=\left\{\frac{x(I)-\m}{2}+\c,a_2,\dots,a_l\right\}$ and the term is
		$$
			f(J)\left\langle
			\begin{array}{c}
			\hfill \frac{r-2-x(J)}{2} \quad\hfill a_1-\c\quad\hfill\frac{r-2+x(J)}{2}-a_1+\c \hfill\null\\
			\end{array}
			\right\rangle_0^{\text{ext}}f\left(\left\{\frac{r-2+x(J)}{2}-a_1+\c\right\}\right).
		$$
	\end{itemize}
	We will show that $K(J)\le \N-2$, then the lemma will follow from Lemma \ref{lemknottoolargexnm}. From 
	$$
		\N\m+\b-r+2=(\N-2)r
	$$
	we get
	$$
	2\N(\h+1)=r+\b+2.
	$$
	Similarly we have 
	$$
	2\sum_{i=1}^l (r-a_i)+2K(I)(\h+1)=r+x(I)+2
	$$
	and
	$$
	2\sum_{i=1}^l (r-a_i)+2K(J)(\h+1)=r+x(J)+2-2a_1+x(I)-\m+2\c.
	$$
	Since 
	$$
	2(r-a_1)\le2\sum_{i=1}^l (r-a_i)=x(I)-\b+2(\N-K(I))(\h+1),
	$$
	we have
	\begin{equation*}
	\begin{split}
	2(\N-K(J))(\h+1)&=2(\N-K(I))(\h+1)+2(K(I)-K(J))(\h+1)\\
	&\ge 2(r-a_1)-x(I)+\b-x(J)+2a_1+\m-2\c\\
	&=r+2\b+2+\m-x(I)-x(J)\\
	\end{split}
	\end{equation*}
	Note that $\m\le \b,x(I),x(J)\le r-2$ and $\m\ge \frac{r}{2},\h\le \frac{r-4}{4}$ (since $\N\ge 4$), we get
	$$
	2(\N-K(J))(\h+1)\ge 3\m-r+6 =2r-6\h>2\h+2
	$$
	which implies 
	$$
	\N-K(J)\ge 2
	$$
\end{proof}
\begin{rmk}
	The above proof works for all $\m$ satisfying $\m>\frac{r-3}{2}$. When $ \frac{r}{3}\le \m <\frac{r}{2}$  we have $\N=3$, so this proof covers some case of $\N=3$.
\end{rmk}
\begin{prop}\label{propnge4}
	For $\N\ge 4$, we can calculate $f(I)$ for all $I\in \mathcal I$.
\end{prop}
\begin{proof}
	The case $I=\emptyset$ is calculated in Proposition \ref{propfemptyset}. The case $\lvert I\rvert=1$, $x(I)=\m$ is calculated in Lemma \ref{leml1xm}. The case $\lvert I\rvert=1$, $x(I)\ne \m$ is covered by Lemma \ref{lemknottoolargexnm} and Lemma \ref{lemmnxlargek}. Then we can induct on $\lvert I\rvert$ by using Lemma \ref{lemxm}, Lemma \ref{lemknottoolargexnm} and Lemma \ref{lemmnxlargek} repeatedly.
\end{proof}

\subsubsection{Case $\N=3$}
In this case we have 
$$
\frac{r}{3}\le \m\le \frac{r-2}{2}.
$$
\begin{lem}\label{lemn3k3}
	If $\N=3$, $I\ne \emptyset$ and $K(I)=3$ (then $x(I)\ne \m$), then $f(I)$ can be calculated from the following given data: closed extended correlators and correlators of the form $f(I')$ with $D(I')\le D(I)-1$.
\end{lem}
\begin{proof}
    To calculate
    $$f(I)=
	\left\langle
	\begin{array}{c}
	\hfill  a_1\quad \hfill\dots\quad \hfill a_l\hfill\null\\
	\hfill \m\quad \hfill \m\quad\hfill \m\hfill\quad x(I)\hfill\null\\
	\end{array}
	\right\rangle_0^{\frac{1}{r},\text{o},\h},
	$$
    we use the type-I method and apply TRRs to
	$$
	\left\langle
	\begin{array}{c}
	\hfill \frac{x(I)-\m}{2}\psi \quad\hfill a_1\quad \hfill\dots\quad \hfill a_l\hfill\null\\
	\hfill \m\quad \hfill \m\quad\hfill \m\hfill\quad \m\hfill\null\\
	\end{array}
	\right\rangle_0^{\frac{1}{r},\text{o},\h}.
	$$
	We analyze all possibly non-zero terms appearing in the TRRs. 
    Each non-zero term in the TRRs is of the form $$f^{ext}\cdot f({I^{ctr}})\cdot\prod_{i=1}^s f(I^{sd}_{i}),$$ where $f^{ext}$ is the closed extended correlator. According to Remark \ref{rmk determine central side contribution} we have \begin{equation}\label{eq sum of I}\lvert I^{ctr}\rvert \le \lvert I\rvert+1,~\lvert I^{sd}_i\rvert \le \lvert I\rvert,~\lvert I^{ctr}\rvert+\sum_{i=1}^s \lvert I^{sd}_i\rvert\le \lvert I\rvert+1,
    \end{equation}
    \begin{equation}\label{eq sum of K}
    K(I^{ctr})+\sum_{i=1}^s K(I^{sd}_i)=K(I)
    \end{equation}
    and
    \begin{equation}\label{eq sum of D}
     D(I^{ctr})+\sum_{i=1}^s D(I^{sd}_i)+2s-2 \le D(I).
    \end{equation}
	\begin{itemize}
		\item The main term is $$\alpha=
		f(I)\left\langle
		\begin{array}{c}
		\hfill \frac{r-2-x(I)}{2} \quad\hfill \frac{x(I)-\m}{2}\quad\hfill \frac{r-2+\m}{2} \hfill\null\\
		\end{array}
		\right\rangle_0^{\text{ext}}\left\langle
    	\begin{array}{c}
    	\hfill \frac{r-2-\m}{2}\hfill\null\\
    	\hfill \m\hfill\null\\
    	\end{array}
    	\right\rangle_0^{\frac{1}{r},\text{o},\h}.
		$$
		\par
		In the TRR with respect to a boundary marking $\m$, the coefficient of $\alpha$ is $-1$; in the TRR with respect to the internal marking $a_1$, the coefficient of $\alpha$ is $0$.
		\item Note that for $\lvert I^{sd}\rvert<\lvert I\rvert$ we have (note that $K(I^{sd})\le \N$ by Remark \ref{rmk bound of K})
		$$
		D(I^{sd})=2\lvert I^{sd}\rvert+K(I^{sd})-2\le 2\lvert I\rvert-2+\N-2=D(I)-2.
		$$
		Therefore for all possibly non-zero terms except $\alpha$, the side contributions are of the form $f(I^{sd})$ with $D(I^{sd})\le D(I)-2$.
		\item If $f(I^{ctr})$ is the central contribution, we can also show that $$D(I^{ctr})\le D(I)-1.$$ 
  Indeed, every non-zero term has at least one side contribution, since otherwise we have $K(I^{ctr})=3=\N,$ which contradicts  Remark \ref{rmk bound of K} (note that $x(I^{ctr})=\m$). In the case where this term has a side contribution $f(\emptyset)$, by \eqref{eq sum of K} we have $K(I^{ctr})\le K(I)-\N=0$, hence
		$$
		D(I^{ctr})=2\lvert I^{ctr}\rvert+K(I^{ctr})-2\le 2\left(\lvert I\rvert+1\right)+0-2=D(I)-1.
		$$
		In the case this term has a side contribution $f(J)$ for $J\ne \emptyset$, by \eqref{eq sum of I} we have $\lvert I^{ctr}\rvert\le\lvert I\rvert$ and
		$$
		D(I^{ctr})=2\lvert I^{ctr}\rvert+K(I^{ctr})-2\le 2\lvert I\rvert+(\N-1)-2=D(I)-1.
		$$
		\end{itemize}
        Thus, all non-zero terms other than $\alpha,$ can be written as products of a closed extended correlator and correlators of the form $f(I')$ with $D(I')\le D(I)-1$. We can thus calculate $f(I)$ from the data given in the statement of the lemma using the type-I method.
\end{proof}

\begin{lem}\label{lemn3k2xm}
	If $\N=3$, $\lvert I\rvert\ge 2$, $K(I)=2$ and $x(I)=\m$, then $f(I)$ can be calculated from the following given data: closed extended correlators and correlators of the form $f(I')$ with $D(I')\le D(I)-2$.
\end{lem}
\begin{proof}
    To calculate
    $$f(I)=
	\left\langle
	\begin{array}{c}
	\hfill  a_1\quad \hfill\dots\quad \hfill a_l\hfill\null\\
	\hfill \m\quad \hfill \m\quad\hfill \m\hfill\null\\
	\end{array}
	\right\rangle_0^{\frac{1}{r},\text{o},\h},
	$$
    we use the type-II method and apply TRRs to
	$$
	\left\langle
	\begin{array}{c}
	\hfill  (a_1-\c)\psi\quad \hfill a_2\quad \hfill\dots\quad \hfill a_l\hfill\null\\
	\hfill \m\quad \hfill \m\quad\hfill \m\hfill\quad \m\hfill\quad \m\hfill\quad \m\hfill\null\\
	\end{array}
	\right\rangle_0^{\frac{1}{r},\text{o},\h}.
	$$
        We analyze all possibly non-zero terms appearing in the TRRs. Again, each non-zero term in the TRRs is of the form $$f^{ext}\cdot f({I^{ctr}})\cdot\prod_{i=1}^s f(I^{sd}_{i}).$$ 
        According to Remark \ref{rmk determine central side contribution} we have \begin{equation}\label{eq sum of I type 2}\lvert I^{ctr}\rvert \le \lvert I\rvert,~\lvert I^{sd}_i\rvert \le \lvert I\rvert-1,~\lvert I^{ctr}\rvert+\sum_{i=1}^s \lvert I^{sd}_i\rvert\le \lvert I\rvert,
        \end{equation}
        \begin{equation}\label{eq sum of K type 2}
    K(I^{ctr})+\sum_{i=1}^s K(I^{sd}_i)=K(I)+3
    \end{equation}
    and
    \begin{equation}\label{eq sum of D type 2}
     D(I^{ctr})+\sum_{i=1}^s D(I^{sd}_i)+2s-2 \le D(I)+1.
    \end{equation}
    
	\begin{itemize}
		\item The main term is 
        $$\alpha=
		f(\emptyset)\left\langle
		\begin{array}{c}
		\hfill \c \quad\hfill a_1-\c\quad\hfill r-2-a_1 \hfill\null\\
		\end{array}
		\right\rangle_0^{\text{ext}}f(I),
		$$
		\par
		In the TRR with respect to a boundary marking $\m$, the coefficient of $\alpha$ is $-\binom{5}{3}$; in the TRR with respect to the internal marking $a_2$, the coefficient of $\alpha$ is $-\binom{6}{3}$. Note also that the factor $f(\emptyset)$ is covered by the given date since $D(\emptyset)=1$.
		\item Note that for $\lvert I^{sd}\rvert\le\lvert I\rvert-1$ we have (note that $K(I^{sd})\le \N$ by Remark \ref{rmk bound of K})
		$$
		D(I^{sd})=2\lvert I^{sd}\rvert+K(I^{sd})-2\le 2\lvert I\rvert-2+\N-2=D(I)-1.
		$$
		Therefore for all possibly non-zero terms, the side contributions are of the form $f(I^{sd})$ with $D(I^{sd})\le D(I)-1$. The equality holds only if $K(I^{sd})=\N=3$.
	    \item Each possibly non-zero term has at least one side contribution, since otherwise its central contribution $f(I^{ctr})$ satisfies $K(I^{ctr})=5>\N-1$, which contradicts Remark \ref{rmk bound of K}.  In the case where this term only has one $f(\emptyset)$ as side contributions, then the equality in \eqref{eq sum of I type 2}, \textit{i.e.} $\lvert I^{ctr} \rvert=\lvert I \rvert$ holds only if this term is just $\alpha$; when this term is not $\alpha$, by \eqref{eq sum of K type 2} we have $K(I^{ctr})=K(I)$ and
     $$
     D(I^{ctr})=2\lvert I^{ctr}\rvert+K(I^{ctr})-2\le 2(\lvert I\rvert-1)+K(I)-2=D(I)-2.
     $$
     In the case this term has at least two $f(\emptyset)$ contributions, by \eqref{eq sum of D type 2} its central contribution $f(I^{ctr})$  $$D(I^{ctr})\le D(I)+1+2-2\cdot 2-2D(\emptyset)=D(I)-3.$$ 
     In the case this term has a side contribution $f(J)$ for $J\ne \emptyset$, then by \eqref{eq sum of I type 2} its central contribution $f(I^{ctr})$ satisfies $\lvert I^{ctr}\rvert\le\lvert I\rvert-1$, hence by Remark \ref{rmk bound of K}$$D(I^{ctr})=2\lvert I^{ctr}\rvert+K(I^{ctr})-2\le 2\lvert I\rvert-2+\N-1-2=D(I)-2.$$
		\end{itemize}
        Therefore all possibly non-zero terms other than $\alpha$ can be written as products of a closed extended correlator and correlators of the form $f(I')$ with either $$D(I')\le D(I)-2$$ or $$K(I')=3\text{ and } D(I')=D(I)-1.$$
        After applying Lemma \ref{lemn3k3} to the latter case, we can calculate $f(I)$ from the data given in the statement of the lemma by the type-II method, and the lemma is proven.
\end{proof}

\begin{lem}\label{lemn3k2xnm}
	If $\N=3$, $I\ne \emptyset$, $K(I)=2$ and $x(I)\ne \m$, then $f(I)$ can be calculated from the following given data: closed extended correlators and correlators of the form $f(I')$ with $D(I')\le D(I)-1$.
\end{lem}
\begin{proof}
    Just like the proof of Lemma \ref{lemn3k3}, we use the type-I method to calculate $f(I)$ and apply TRRs to
	$$
	\left\langle
	\begin{array}{c}
	\hfill \frac{x(I)-\m}{2}\psi \quad\hfill a_1\quad \hfill\dots\quad \hfill a_l\hfill\null\\
	\hfill \m\quad \hfill \m\quad\hfill  \m\hfill\null\\
	\end{array}
	\right\rangle_0^{\frac{1}{r},\text{o},\h}.
	$$
	The equations \eqref{eq sum of I}, \eqref{eq sum of K} and \eqref{eq sum of D} still hold for non-zero terms in this case. We consider all the contributions $f(I')$ in a possibly non-zero term satisfying $D(I')\ge D(I)$: 
	\begin{itemize}
	    \item 
		If a side contribution $f(I^{sd})$ satisfies $D(I^{sd})\ge D(I)$, then by \eqref{eq sum of D} the only possibility is that $f(I^{sd})$ is the only side contribution in this term and $I^{sd}=I$, so this term is the main term $\alpha$. Similarly to Lemma \ref{lemn3k3}, the coefficients of $\alpha$ in TRRs with respect to the boundary marking $\m$ and the internal marking $a_1$ are different.
		\item
		If a central contribution $f(I^{ctr})$ satisfies $D(I^{ctr})\ge D(I)$, then by \eqref{eq sum of D}, either there is no side contribution, or there is only one zero-dimensional side contribution. However, there is no possible  zero-dimensional side contribution; this is because $D(\emptyset)=1$, and $D(I^{sd})=2\lvert I^{sd}\rvert+K(I^{sd})-2=0$ hold for a side contribution $f(I^{sd})$ with $I^{sd}\ne \emptyset$ only if $\lvert I^{sd}\rvert=1,~K(I^{sd})=0$, which is not possible since $a_i\ge \h+1$ and $K(\{a_i\})\ge 1$ for all $a_i\in I$.
  
        So the only possibility is that there is no side contribution. In this case we have $\lvert I^{ctr}\rvert=\lvert I\rvert$, $K(I^{ctr})=2$, $x(I^{ctr})=\m$ and $D(I^{ctr})=D(I)$. We can calculate this contribution from the data given in the statement of the lemma by Lemma \ref{lemn3k2xm} (when $\lvert I \rvert\ge 2$)  or Lemma  \ref{leml1xm} (when $\lvert I \rvert=1$).
	\end{itemize}
 Therefore we can use the type-I method to calculate $f(I)$ from the data given in the statement of the lemma, and the lemma is proven.
\end{proof}

\begin{lem}\label{lemn3k1xm}
	If $\N=3$, $\lvert I\rvert\ge 2$, $K(I)=1$ and $x(I)=\m$, then $f(I)$ can be calculated from the following given data: closed extended correlators and correlators of the form $f(I')$ with $D(I')\le D(I)-1$.
\end{lem}

\begin{proof}
    Similarly to the proof of Lemma \ref{lemn3k2xm}, we use the type-II method to calculate $f(I)$ and apply TRRs to
	$$
	\left\langle
	\begin{array}{c}
	\hfill  (a_1-\c)\psi\quad \hfill a_2\quad \hfill\dots\quad \hfill a_l\hfill\null\\
	\hfill \m\quad \hfill \m\quad\hfill \m\hfill\quad \m\hfill\quad \m\hfill\null\\
	\end{array}
	\right\rangle_0^{\frac{1}{r},\text{o},\h}.
	$$
	The equations \eqref{eq sum of I type 2}, \eqref{eq sum of K type 2} and \eqref{eq sum of D type 2} still hold for non-zero terms in this case. All the contributions $f(I')$ in a possibly non-zero term satisfying $D(I')\ge D(I)$ are:
	\begin{itemize}
	    \item $f(I)$ as a central contribution in the main term $\alpha$. Similarly to Lemma \ref{lemn3k2xm}, the coefficients of $\alpha$ in TRRs with respect to the boundary marking $\m$ and the internal marking $a_2$ are different.

        There are no other central contributions $f(I^{ctr})$ because by Remark \ref{rmk bound of K}, equations \eqref{eq sum of K type 2} and \eqref{eq sum of D type 2} there must be exactly one side contribution $f(I^{sd})$ in the same term, and it satisfies $D(I^{sd})\le 1$ and $K(I^{sd})\ge 2$,  which is impossible for $I^{sd}\ne \emptyset$.
        
	    \item 
	    $f(I^{sd})$ with $\lvert I^{sd}\rvert=\lvert I\rvert-1$, $K(I^{sd})=3$ and $D(I^{sd})=D(I)$ as side contributions. We can calculate these contributions from the data given in the statement of the lemma by Lemma \ref{lemn3k3}.

        It can be easily seen from \eqref{eq sum of I type 2} and Remark \ref{rmk bound of K} that there are no other possible side contributions.
	\end{itemize}
	Therefore we can calculate $f(I)$ from the data given in the statement of the lemma by the type-II method, and the lemma is proven.
\end{proof}

\begin{lem}\label{lemn3k1xnm}
	If $\N=3$, $I\ne \emptyset$, $K(I)=1$ and $x(I)\ne \m$, then $f(I)$ can be calculated from the following given data: closed extended correlators and correlators of the form $f(I')$ with $D(I')\le \max\{ D(I)-1,1\}$.
\end{lem}
\begin{proof}
   Similarly to the proof of Lemma \ref{lemn3k3}, we use the type-I method to calculate $f(I)$ and apply TRRs to
	$$
	\left\langle
	\begin{array}{c}
	\hfill \frac{x(I)-\m}{2}\psi \quad\hfill a_1\quad \hfill\dots\quad \hfill a_l\hfill\null\\
	\hfill \m\quad \hfill  \m\hfill\null\\
	\end{array}
	\right\rangle_0^{\frac{1}{r},\text{o},\h}.
	$$
	The equations \eqref{eq sum of I}, \eqref{eq sum of K} and \eqref{eq sum of D} still hold for non-zero terms in this case. All the contributions $f(I')$ in a possibly non-zero term satisfying $D(I')\ge D(I)$ are:
	\begin{itemize}
	    \item 
		$f(I)$ as a side contribution in the main term $\alpha$. Similarly to Lemma \ref{lemn3k3}, the coefficients of $\alpha$ in TRRs with respect to the boundary marking $\m$ and the internal marking $a_1$ are different.

        There are no other side contribution $f(I^{sd})$ because if $I^{sd}\ne I$ we have $\lvert I^{sd}\rvert\le \lvert I\rvert-1$; on the other hand, by \eqref{eq sum of K} we have $K(I^{sd})\le K(I)+1$ which means $D(I^{sd})\le D(I)-1$.
		\item
		$f(I^{ctr})$ with $\lvert I^{ctr}\rvert=\lvert I\rvert$, $K(I')=1$, $x(I^{ctr})=\m$ and $D(I^{ctr})=D(I)$ as central contributions. We can calculate these contributions from the data given in the statement of the lemma by Lemma \ref{lemn3k1xm} (when $\lvert I \rvert\ge 2$) or Lemma  \ref{leml1xm} (when $\lvert I \rvert=1$).

        By the same argument as in Lemma \ref{lemn3k2xnm}, there are no other possible central contributions.
	\end{itemize}
 Therefore we can calculate $f(I)$ from the data given in the statement of the lemma by the type-I method, and the lemma is proven
\end{proof}

\begin{lem}\label{lemn3aismall}
    Let $\N=3$, $I\in 
    \mathcal I$ and  $x(I)=\m$. If 
    $$
    a_i < \m \quad\forall a_i \in I,
    $$
    then
    $$
    D(I)\le 1.
    $$
\end{lem}
\begin{proof}
    Since $I\in \mathcal I$ we have 
    $$
    2\sum_{a_i\in I}a_i+(K(I)+1)\m-r+2=\left(2\lvert I\rvert+K(I)-2\right)r,
    $$
    then
    $$
    2\sum_{a_i\in I}(r-a_i)=(1-K(I))r+(1+K(I))\m+2 > 2\lvert I\rvert (r-\m).
    $$
    Assuming the lemma does not hold, that is, 
    $$
    D(I)=2\lvert I\rvert+K(I)-2\ge 2,
    $$
    then
    $$
    (1-K(I))r+(1+K(I))\m+2> (4-K(I)) (r-\m),
    $$
    hence
    $$
    5\m > 3r-2.
    $$
    But this is impossible when $\N=3$, since then $\m\le \frac{r-2}{2}$.
\end{proof}

\begin{lem}\label{lemn3k0xm}
	If $\N=3$, $\lvert I\rvert\ge 2$, $a_1\ge \m$, $K(I)=0$ and $x(I)=\m$, then $f(I)$ can be calculated from the following given data: closed extended correlators and correlators of the form $f(I')$ with $D(I')\le D(I)-1$.
\end{lem}

\begin{proof}
    Similarly to the proof of Lemma \ref{lemn3k2xm}, we use the type-II method to calculate $f(I)$ and apply TRRs to
	$$
	\left\langle
	\begin{array}{c}
	\hfill  (a_1-\c)\psi\quad \hfill a_2\quad \hfill\dots\quad \hfill a_l\hfill\null\\
	\hfill \m\quad \hfill \m\quad\hfill \m\hfill\quad \m\hfill\null\\
	\end{array}
	\right\rangle_0^{\frac{1}{r},\text{o},\h}.
	$$
	The equations \eqref{eq sum of I type 2}, \eqref{eq sum of K type 2} and \eqref{eq sum of D type 2} still hold for non-zero terms in this case. All the contributions $f(I')$ in a possibly non-zero term satisfying $D(I')\ge D(I)$ are:
	\begin{itemize}
	    \item $f(I)$ as a central contribution in the main term $\alpha$. Again, similarly to Lemma \ref{lemn3k2xm}, the coefficients of $\alpha$ in different TRRs are different.
	    \item 
	    $f(I^{sd})$ with $\lvert I^{sd}\rvert=\lvert I\rvert-1$, $K(I^{sd})=2$ and $D(I^{sd})=D(I)$ as side contributions. We can calculate these contributions from the data given in the statement of the lemma by Lemma \ref{lemn3k2xnm} (when $x(I')\ne \m$) or Lemma \ref{lemn3k2xm} (when $x(I^{sd})= \m$, $\lvert I \rvert\ge 2$) or Lemma  \ref{leml1xm} (when $x(I^{sd})= \m$, $\lvert I \rvert=1$).
	    \item 
	    $f(I^{ctr})$ with $\lvert I^{ctr}\rvert=\lvert I\rvert-1$, $K(I^{ctr})=2$, $x(I^{ctr})=\m$ and $D(I^{ctr})=D(I)$ as central contributions. We can calculate these contributions from the data given in the statement of the lemma by Lemma \ref{lemn3k2xm} (when  $\lvert I \rvert\ge 2$) or Lemma  \ref{leml1xm} (when  $\lvert I \rvert=1$).
	    \item Note that there exist no side contributions $f(I^{sd})$ with $\lvert I'\rvert=\lvert I\rvert-1$, $K(I^{sd})=3$ and $D(I^{sd})=D(I)+1$ because $a_1\ge \m$. In fact, in this case $I^{sd}=I\setminus\{a_1\}$. From
	    $$
	   2\sum_{i=1}^{l}a_i+\m-r+2=(2l-2)r
	    $$
	    we get
	    $$
	   (2l-1)r-\left(3\m+2\sum_{i=2}^{l}a_i-r+2\right)=r-2\m+2a_1\ge r>r-2;
	    $$
	    this means $K(I^{sd})\ne 3$.
	\end{itemize}
	Therefore every contribution to the TRRs, except $f(I),$ can be calculated from closed extended correlators and correlators of the form $f(I')$ with $D(I')\le D(I)-1$. Then we can calculate $f(I)$ by combining two different TRRs.
\end{proof}

\begin{lem}\label{lemn3k0xnm}
	If $\N=3$, $I\ne \emptyset$, $K(I)=0$ and $x(I)\ne \m$, then $f(I)$ can be calculated from the following given data: closed extended correlators and correlators of the form $f(I')$ with $D(I')\le \max\{ D(I)-1,1\}$.
\end{lem}
\begin{proof}
   Similarly to the proof of Lemma \ref{lemn3k3}, we use the type-I method to calculate $f(I)$ and apply TRRs to
	$$
	\left\langle
	\begin{array}{c}
	\hfill \frac{x(I)-\m}{2}\psi \quad\hfill a_1\quad \hfill\dots\quad \hfill a_l\hfill\null\\
	\hfill   \m\hfill\null\\
	\end{array}
	\right\rangle_0^{\frac{1}{r},\text{o},\h}.
	$$
	The equations \eqref{eq sum of I}, \eqref{eq sum of K} and \eqref{eq sum of D} still hold for non-zero terms in this case. Similarly to Lemma \ref{lemn3k1xnm}, all the contributions $f(I')$ in a possibly non-zero term satisfying $D(I')\ge D(I)$ are:
	\begin{itemize}
	    \item 
		$f(I)$ as a side contribution in the main term $\alpha$. Again, similarly to Lemma \ref{lemn3k3}, the coefficients of $\alpha$ in different TRRs are different.
		\item
		$f(I^{ctr})$ with $\lvert I^{ctr}\rvert=\lvert I\rvert$, $K(I^{ctr})=0$, $x(I^{ctr})=\m$ and $D(I^{ctr})=D(I)$ as central contributions. We can calculate these contributions from the data given in the statement of the lemma by Lemma \ref{lemn3k0xm} (when $\lvert I \rvert\ge 2$) or Lemma  \ref{leml1xm} (when $\lvert I \rvert=1$).
	\end{itemize}
 Therefore we can calculate $f(I)$ from the data given in the statement of the lemma by the type-I method, and the lemma is proven.
\end{proof}

\begin{lem}\label{lemn3k-1xm}
	If $\N=3$, $\lvert I\rvert\ge 2$, $a_1\ge \m$, $K(I)=-1$ and $x(I)=\m$, then $f(I)$ can be calculated from the following given data: closed extended correlators and correlators of the form $f(I')$ with $D(I')\le \max \{D(I)-1,1\}$.
\end{lem}

\begin{proof}
    Similarly to the proof of Lemma \ref{lemn3k2xm}, we use the type-II method to calculate $f(I)$ and apply TRRs to
	$$
	\left\langle
	\begin{array}{c}
	\hfill  (a_1-\c)\psi\quad \hfill a_2\quad \hfill\dots\quad \hfill a_l\hfill\null\\
	\hfill \m\quad \hfill  \m\hfill\quad \m\hfill\null\\
	\end{array}
	\right\rangle_0^{\frac{1}{r},\text{o},\h}.
	$$
	The equations \eqref{eq sum of I type 2}, \eqref{eq sum of K type 2} and \eqref{eq sum of D type 2} still hold for non-zero terms in this case. All the contributions $f(I')$ in a possibly non-zero term satisfying $D(I')\ge D(I)$ are:
	\begin{itemize}
	    \item $f(I)$ as a central contribution in the main term $\alpha$. Again, similarly to Lemma \ref{lemn3k2xm}, the coefficients of $\alpha$ in two different TRRs are different.
	    \item 
	    $f(I^{sd})$ with $\lvert I^{sd}\rvert=\lvert I\rvert-1$, $K(I^{sd})=1$ and $D(I^{sd})=D(I)$ as side contributions. We can calculate these contributions from the data given in the statement of the lemma by Lemma \ref{lemn3k1xnm} (when $x(I^{sd})\ne \m$) or Lemma \ref{lemn3k1xm} (when $x(I^{sd})= \m$, $\lvert I \rvert\ge 2$) or Lemma  \ref{leml1xm} (when $x(I^{sd})= \m$, $\lvert I \rvert=1$).
	    \item 
	    $f(I^{ctr})$ with $\lvert I^{ctr}\rvert=\lvert I\rvert-1$, $K(I^{ctr})=1$, $x(I^{ctr})=\m$ and $D(I^{ctr})=D(I)$ as central contributions. We can calculate these contributions from the data given in the statement of the lemma by Lemma \ref{lemn3k1xm} (when  $\lvert I \rvert\ge 2$) or Lemma  \ref{leml1xm} (when  $\lvert I \rvert=1$).
	    \item $f(I^{ctr})$ with $\lvert I^{ctr}\rvert=\lvert I\rvert-1$, $K(I^{ctr})=2$, $x(I^{ctr})=\m$ and $D(I^{ctr})=D(I)+1$ as central contributions. We can calculate these contributions from the data given in the statement of the lemma by Lemma \ref{lemn3k2xm} (when  $\lvert I \rvert\ge 2$) or Lemma  \ref{leml1xm} (when  $\lvert I \rvert=1$).
	    \item Note that there exist no side contributions $f(I^{sd})$ with $\lvert I'\rvert=\lvert I\rvert-1$, $K(I^{sd})=2$ and $D(I^{sd})=D(I)+1$ since $a_1\ge \m$. In fact, in this case $I^{sd}=I\setminus\{a_1\}$. From
	    $$
	   2\sum_{i=1}^{l}a_i-r+2=(2l-3)r
	    $$
	    we get
	    $$
	   (2l-2)r-\left(2\m+2\sum_{i=2}^{l}a_i-r+2\right)=r-2\m+2a_1\ge r>r-2;
	    $$
	    this means $K(I')\ne 2$.
	\end{itemize}
 Therefore, we can calculate $f(I)$ from the data given in the statement of the lemma by the type-II method, and the lemma is proven.
\end{proof}

\begin{prop}\label{propn3}
	For $\N=3$, we can calculate $f(I)$ for all $I\in \mathcal I$.
\end{prop}
\begin{proof}
	The case $D(I)=0$ or $1$ can be calculated directly by definition. Then we can induct on $D(I)$ by using the lemmas \ref{lemn3k3}--\ref{lemn3k-1xm} repeatedly.
	
\end{proof}

\subsubsection{Case $\N=2$}
In this case we have 
$$
\m<\frac{r}{3},~\b=r-2-2\m,~\c=\m.
$$
The strategy in this subsection is slightly different from the previous two. Instead of calculating $f(I)$ for all $I\in \mathcal I$ independently, we enhance Proposition \ref{prop general computation} by showing that we can calculate $F(B,I)$ for $(B,I)\in \mathcal E_1\cup \mathcal E_2$ from lower dimension correlators.
\begin{lem}\label{lem n2 ai small}
If $\N=2$, $I\in \mathcal I$, $I\ne \emptyset$, $x(I)=\m$ and $$\max_{a\in I}a\le\frac{r-2+\m}{2},$$ then $$D(I)\le 1.$$
\end{lem}
\begin{proof}
Since $I\in \mathcal I$, we have that 
\begin{equation*}
    \begin{split}
        (2\lvert I\rvert+K(I)-2)r&=2\sum_{a\in I}a+(K(I)+1)\m-r+2\\
        &\le(r-2+\m)\lvert I\rvert+(K(I)+1)\m-r+2,
    \end{split}
\end{equation*}
which means
$$
(r-\m)(\lvert I\rvert+K(I)+1)\le 2r+2-2\lvert I\rvert.
$$
Since $\lvert I\rvert\ge 1$ and $\m<\frac{r}{3}$ we have
$$
\lvert I\rvert+K(I)+1<3.
$$
By the integrality and $K(I)\ge -1$ we have
$$
D(I)=2\lvert I\rvert+K(I)-2=2(\lvert I\rvert+K(I)+1)-K(I)-4\le 2\cdot 2 +1-4=1.
$$
\end{proof}

\begin{dfn}
Define the set $\mathcal J\subseteq \mathcal I$ by
$$
\mathcal J:=\left\{I\in \mathcal I\colon x(I)=\m,~\max_{a\in I}a>\frac{r-2+\m}{2}\right\}.
$$
\end{dfn}
\begin{lem}\label{lem n2 ai large}
If $\N=2$, $I\in \mathcal J$, $\lvert I\rvert\ge 2$, then $f(I)$ can be calculated from the following given data: closed extended correlators,  correlators of the form $f(I')$ with $D(I')\le \max\{ D(I)-1,1\}$, correlators of the form $f(I')$ with $I'\in \mathcal J$, $D(I')=D(I)$ and $\max_{a\in I'}a>\max_{a\in I}a$.
\end{lem}
\begin{proof}
Write
$$
I=\{a_1,\dots,a_l\}
$$
and assume 
$$
a_1=\max_{i=1,\ldots,l}a_i.
$$
To calculate 
	$$
	f(I)=\left\langle
	\begin{array}{c}
	\hfill a_1 \quad\hfill a_2 \quad \hfill\dots\quad \hfill a_l\hfill\null\\
	\hfill\smash{\underbrace{ \m\quad \hfill \m\quad\hfill\dots\hfill\quad \m}_{K(I)}\quad\hfill \m}\hfill\null\\
	\end{array}
	\right\rangle_0^{\frac{1}{r},\text{o},\h}\vphantom{\left\langle
		\begin{array}{c}
		\hfill  a_1\quad \hfill\dots\quad \hfill a_l\hfill\null\\
		\hfill\underbrace{ m\quad \hfill m\quad\hfill\dots\hfill\quad m}_{K(I)}\quad\hfill x(I)\hfill\null\\
		\end{array}
		\right\rangle_0^{\frac{1}{r},\text{o},\h}},
	$$
	we use the type-II method and apply TRRs to 
	$$
	\left\langle
	\begin{array}{c}
	\hfill (a_1-\m)\psi \quad\hfill a_2 \quad \hfill\dots\quad \hfill a_l\hfill\null\\
	\hfill\smash{\underbrace{ \m\quad \hfill \m\quad\hfill\dots\hfill\quad \m}_{K(I)}\quad\hfill \m\quad\hfill \m\quad\hfill \m}\hfill\null\\
	\end{array}
	\right\rangle_0^{\frac{1}{r},\text{o},\h}\vphantom{\left\langle
		\begin{array}{c}
		\hfill  a_1\quad \hfill\dots\quad \hfill a_l\hfill\null\\
		\hfill\underbrace{ m\quad \hfill m\quad\hfill\dots\hfill\quad m}_{K(I)}\quad\hfill x(I)\hfill\null\\
		\end{array}
		\right\rangle_0^{\frac{1}{r},\text{o},\h}}.
	$$
	According to Remark \ref{rmk determine central side contribution}, all the contributions $f(I')$ in a possibly non-zero term satisfying $D(I')\ge D(I)$ are:
	\begin{itemize}
	    \item $f(I)$ as a central contribution. The term to which it contributes is the main term $\alpha$. The coefficient of $\alpha$ in the TRR with respect to the internal marking $a_2$ is $-\binom{K(I)+3}{2}$,
     while the coefficient of $\alpha$ in the TRR with respect to the internal marking $\m$ is $-\binom{K(I)+2}{2}$.
	    \item $f(I')$ with $\lvert I'\rvert=\lvert I\rvert-1$, $K(I')=K(I)+2$, $x(I')=\m$ and $D(I')=D(I)$ as a central contribution. There are two cases.
	    \begin{itemize}
	        \item $f(I')$ contributes in a term with no side contributions. In this case we have 
	        $$
	        I'=\{a_1-\m+a_i,a_2,\dots,a_{i-1},a_{i+1},\dots,a_l)\}
	        $$
	        for some $i\in \{2,3,\dots,l\}$. Since $a_i\ge \h+1$ and $\m< \frac{r}{3}$ we have 
	        $$
	        a_1-\m+a_i\ge a_1-\m+\frac{r-\m}{2}>a_1,
	        $$
	        hence $I'\in \mathcal J$ and
	        $$
	        \max_{a\in I'}a>\max_{a\in I}a.
	        $$
	        \item
	        $f(I')$ contributes in a term with one side contribution of the form $f(\{a_i\})$ for some $i\in \{2,3,\dots,l\}$ with $K(\{a_i\})=0$; however $a_i\ge \h+1$ implies $K(\{a_i\})\ne 0$, so this case is impossible.
	    \end{itemize}
	    \item Note that there is no side contribution of the form $f(I')$ with $\lvert I'\rvert=\lvert I\rvert-1$, $K(I')=K(I)+2$ and $D(I')=D(I)$. In fact, assuming $f(I')$ appears in some term, then the central contribution of this term is $
        \left\langle
        \begin{array}{c}
        \hfill  \h\hfill\null\\
        \hfill \m\hfill\null\\
        \end{array}
        \right\rangle_0^{\frac{1}{r},\text{o},\h}
        $ and the closed contribution of this term is $\left\langle
        \begin{array}{c}
        \hfill \frac{r-2-x(I')}{2} \quad\hfill a_1-\m\quad\hfill r-2-\h  \hfill\null\\
        \end{array}
        \right\rangle_0^{\text{ext}}$. However, since $x(I')\le r-2$ and $a_1>\frac{r-2+\m}{2}$, it must hold that
        \[
        \begin{split}
        &\frac{r-2-x(I')}{2} +(a_1-\m)+(r-2-\h)\\&\quad\qquad\quad\qquad> 0+\frac{r-2+\m}{2}-\m+r-2-\frac{r-2-\m}{2}=r-2. \end{split}\]
        This means the closed contribution in this term is zero.
\end{itemize}
Therefore we can calculate $f(I)$ from the data given in the statement of the lemma using the type-II method, and the lemma is proven.
\end{proof}
\begin{prop}\label{propn2}
	For $\N=2$, we can calculate all the correlators $F(s)$ for $s\in \mathcal S$.
\end{prop}
\begin{proof}
    We prove the proposition by induction on $D(s)$. For all $s$ with $D(s)=0$ or $D(s)=1$, we can calculate $F(s)$ by definition.
    \par
    We assume that we have calculated all $F(s)$ with $D(s)\le D_0-1$, we need to calculate all $F(s)$ with $D(s)=D_0$. By Proposition \ref{prop general computation}, it is enough to calculate all $F(s)$ with $s\in \mathcal E_1\cup \mathcal E_2$ and $D(s)=D_0$. 
    \begin{itemize}
    \item
    For the unique $s\in \mathcal E_1$, the correlator $F(s)=f(\emptyset)$ is calculated in Proposition \ref{propfemptyset}.
    \item
    For $s=(B,I)\in \mathcal E_2$ and $I\notin \mathcal J$, Lemma \ref{lem n2 ai small} shows that $D(s)\le 1$, so we can calculate $F(s)$ by definition.
    \item
    It remains to calculate $F(s)=f(I)$ for $s=(B,I)\in \mathcal E_2$ and $I\in \mathcal J$. Write
    $$
    \mathcal J_{D_0}:=\{I\in \mathcal J\colon D(s)=D_0\}. 
    $$
     We can calculate $f(I)$ for all $I\in \mathcal J_{D_0}$ in the following way. The case $\lvert I\rvert=1$ is covered by Lemma \ref{leml1xm}. For $\lvert I\rvert\ge2$, by our induction hypothesis (\textit{i.e.} we have calculated all $F(s)$ with $D(s)\le D_0-1$) and Lemma \ref{lem n2 ai large}, $f(I)$ can be calculated from $f(I')$ with $I'\in \mathcal J_{D_0}$ and $\max_{a\in I'}a>\max_{a\in I}a$. Since $\mathcal J_{D_0}$ is finite, we can calculate all $f(I)$ for all $I\in \mathcal J_{D_0}$ by an additional induction on $\max_{a\in I}a$.
    \end{itemize}

\end{proof}

\subsection{Conclusion}
We summarize this section by the following theorem.
\begin{thm}\label{thm calculate}
    All primary open $(r,\h)$-spin correlators are determined by the topological recursion relation Theorem \ref{thm TRR}, the vanishing result of Proposition \ref{prop:vanish small internal}, closed extend $r$-spin correlators defined and calculated in \cite{BCT_Closed_Extended}, and the initial values
    \begin{itemize}
        \item $f(\emptyset)$, which is calculated in Section \ref{subsec computation minimal};
        
        \item $F(B,I)$ for $(B,I)\in \mathcal S$ with $D(B,I)\le 1$, which is calculated directly from the definition;
        
        \end{itemize}
        Moreover, as a consequence, all $(r,\h)$-spin correlators (including non-primary ones) are determined, again by the topological recursion relation Theorem \ref{thm TRR}.
\end{thm}

\section{Point insertion in open FJRW theories and a Fermat quintic example}
\label{sec:open_fjrw}
Fan--Jarvis--Ruan--Witten (FJRW) theory of quantum singularity \cite{FJR} is an important generalization of Witten's $r$-spin theory. To describe this theory in a nutshell, let $W$ be the quasi-homogeneous polynomial
\begin{equation}\label{eq:W}W(x_1,\ldots, x_a)= \sum_{i=1}^a\prod_{j=1}^ax_i^{c_{ij}},\end{equation} and let $G_W$ be the maximal group of diagonal symmetries of $W,$
\[G_W=\{(\lambda_1,\ldots,\lambda_a)\in({\mathbb{C}}^*)^a|~W(\lambda_1x_1,\ldots,\lambda_ax_a)=W(x_1,\ldots,x_a)\}.\]
In \cite{FJR,FJR2,FJR3}, a moduli space ${\mathcal{W}}_{W,G}$ and a virtual fundamental class $c_{(W,G)}$ are associated to a \emph{non-degenerate} quasi homogeneous polynomial $W$ and an \emph{admissible} group $G<G_W$.
The moduli and the class $c_{(W,G)}$ give rise to the FJRW cohomological field theory, and in genus $0$ to the FJRW Frobenius algebra.
For example, $r$-spin theory corresponds to the special case where $W=x^r$ and $G=\mathbb Z/r\mathbb Z$. FJRW theory was proven to be related to integrable hierarchies \cite{FSZ10}, to mirror symmetry \cite{FJR,HeLiShenWebb} and to other geometric models \cite{ChiodoRuan, FJRGLSM}. 
We will now briefly describe one important family of FJRW theory, referring the reader to \cite{FJR,FJR2,FJR3} for a detailed description and introduction of FJRW theory. 

Consider the pair 
\[(x_1^r+\ldots+x_{m}^r,\mathbb{Z}/r\mathbb{Z}),\]of a Fermat polynomial and its \emph{minimal} symmetry group. 
Roughly speaking, in this theory we consider the usual moduli space of $r$-spin curves, where the Witten class $c_W$ is replaced by $c_W^{m},$ the cupping of $m$ copies of the Witten class, and intersection numbers are as usual integrals of the form
\[\langle\tau_{d_1}^{(a_1,\ldots,a_1)}\cdots\tau_{d_n}^{(a_n,\ldots,a_n)}\rangle^{\frac{1}{r},c,m}_g:=r^{(1-g)}\int_{\overline{\mathcal{M}}_{g,\{a_1,\ldots,a_n\}}}c_W^{ m}\cup\psi_1^{d_1}\cdots\psi_n^{d_n}.\]
This theory can also be thought of as an intersection theory on the moduli of curves together with $m$ $r$-spin structures, where the choice of the minimal symmetry group $\mathbb{Z}/r\mathbb{Z}$ implies that for each point the $m$ twists are the same, so that the vector of twists at a point has the form of the form $(a,\ldots,a),$ explaining the above notation.

An important special case of the above is the Fermat quintic with minimal symmetry group 
\[(x_1^5+x_2^5+x_3^5+x_4^5+x_5^5,\mathbb{Z}/5\mathbb{Z}).\]
In the Fermat quintic example this corresponds to working with the FJRW class $c_{\mathcal W}^{5},$ the cupping of five copies of Witten's $5$-spin class, and considering only twists of the form $(a,a,a,a,a)$ for $a\in\{0,1,2,3,4\}.$
In the (closed) FJRW setting this theory was shown to satisfy a beautiful correspondence with the Calabi-Yau (CY) Fermat quintic \cite{ChiodoRuan}, 
the well-known \emph{LG/CY correspondence} of Chiodo and Ruan.

The study of open FJRW theory beyond $r$-spin has been recently initiated in \cite{GKT2}. In
\cite{GKT2}, an open FJRW construction based on BCT boundary conditions was considered. In this section we will concentrate on defining open FJRW theories using the point insertion boundary conditions. It turns out that this type of boundary conditions is perfectly suited for the Fermat case with the minimal symmetry group.

We will now construct the open FJRW intersection theory for $(x_1^r+\ldots+x_{m}^r,\mathbb{Z}/r\mathbb{Z}).$ In this theory we consider correlators of the form
\[
\left\langle \prod_{i\in I}\tau^{a_i}_{d_i}\prod_{i\in B}\sigma^{b_i}\right\rangle^{1/r,o,\h,m}_0 := \int_{\oPMh_{0,B,I}}e({\mathcal{W}}^{\oplus m}\oplus\bigoplus_{i\in I} {\mathbb{L}}_i^{\oplus d_i} ; \mathbf{s})
\]
as in Definition \ref{dfn correlator point insertion}, involving direct sums of $m$ copies of the $r$-spin Witten bundle and the relative cotangent line bundles $\mathbb{L}_i.$
Note that by \cite[Theorem 4.8]{TZ1}, for odd $m=2d+1,$ the bundle $\widetilde{{\mathcal{W}}}^{\oplus(2d+1)}\to\widetilde{\mathcal{M}}_{0,k,l}^{\frac{1}{r},\h}$ over the glued moduli space is canonically relatively oriented. Also ${{\mathcal{W}}}^{\oplus(2d+1)}\to\overline{\mathcal{M}}_{0,k,l}^{\frac{1}{r},\h},$ the sum of an odd number of copies of the Witten bundle over the non-glued moduli, is therefore canonically relatively oriented. 

\begin{rmk}\label{rmk:BCT_quintic}
    One can also define a BCT-like open FJRW theory for this setting, by using positivity boundary conditions for all boundary nodes but those which have an illegal half-node of twist $0.$ For the latter nodes forgetful boundary conditions are used as in \cite{BCT2}. Similar arguments show that these numbers are also well-defined, but as we shall see below, in certain cases of interest, which include the Fermat quintic with minimal symmetry group, their correlators vanish.
\end{rmk}

\begin{prop}\label{prop:even_vanishing_quintic}
Fix $r\geq 2,\h\in\{0,1,\ldots,\lfloor r/2\rfloor-1\}$ and odd $m>1.$
Then if $|B|$ is even, every correlator of the form
\[\left\langle \prod_{i\in I}\tau^{a_i}_{d_i}\prod_{i\in B}\sigma^{b_i}\right\rangle^{1/r,o,\h,m}_0\]vanishes.
\end{prop}
\begin{proof}
Write 
\[E = {\mathcal{W}}^{\oplus m}\oplus\bigoplus_{i\in I}{\mathbb{L}}_i^{\oplus d_i}.\] We assume \[\text{rank}(E)=\dim\oPMh_{0,B,I}=|B|+2|I|-3,\] otherwise the correlators are defined to be $0.$
Since $|B|$ is even, $\text{rank}(E)$ is odd. 
Let $s$ be a smooth transverse canonical multisection of $E.$ 
Then the correlator of interest equals $z=\#Z(s),$ the weighted number of zeros of $s.$ As in Remarks \ref{rmk glued section canonical} and \ref{rmk intersection numbers on glued mod}, we may consider $s$ as a section on the glued moduli space $\widetilde{\mathcal{PM}}^{\frac{1}{r},\h}_{0,B,I},$ and the $z$ is the number of zeros of $s$ there.
Observe that $\#Z(-s)$ equals $-z,$ since the zeros of $-s$ are the same, but because of odd rank and dimension, the degree of each zero is altered to minus the degree.

Denote by $U_+\subseteq \widetilde{\mathcal{PM}}^{\frac{1}{r},\h}_{0,B,I}$ a neighbourhood of the removed boundaries of $\widetilde{\mathcal{PM}}^{\frac{1}{r},\h}_{0,B,I}$ in which $s$ is required to be positive in the sense of Section \ref{sec sections}. Write \[R=\partial\widetilde{\mathcal{PM}}^{\frac{1}{r},\h}_{0,B,I}\cup U_+.\]
We will now homotope between $s|_R$ and $-s|_R,$ 
and show that throughout the homotopy no zero is added. 
We may perform the homotopy for the $m$ different copies of the Witten bundle, one by one, and then perform an arbitrary homotopy to the components of $s$ corresponding to the lines ${\mathbb{L}}_i$. The homotopy performed to the $i$th copy changes its boundary conditions from positivity to negativity on $R.$  However, when we homotope the component of the section corresponding to the $i$th copy, each of the remaining $m-1$ components has a definite sign, hence does not vanish on $R.$ Then, throughout the remaining homotopies, all the $m$ components of $s$ are negative at $R.$ 
Thus, the homotopy between $s|_R$ to $-s|_R$ does not produce new zeros in $R.$ Hence
\[z=\#Z(s)=\#Z(-s)=-\#Z(s)=-z,\]where the middle equality follows from Lemma \ref{lem zero difference as homotopy}. This shows that $z=0.$ \end{proof}

We now consider the case of the quintic more carefully. The possible twists are of the form $(a,a,a,a,a)$ for $a=0,\ldots, 4.$

When one tries to define the open version of the FJRW theory for the quintic with minimal symmetry group, with \emph{BCT boundary conditions} as in Remark \ref{rmk:BCT_quintic}, one easily observes that the result is a trivial theory, in the sense that the only non-zero primary invariants are those which correspond to a moduli space of dimension $0$. The reason is that the BCT version constrains the boundary twists to be $(3,3,3,3,3),$ but it is easy to see that in this case for almost all choices of twists $(a,a,a,a,a)$ for internal points, the rank of the direct sum of five Witten bundles is higher than the dimension of the moduli, giving rise only to vanishing intersection numbers. The only cases in which the dimension of the moduli equals the rank of the direct sum of the Witten bundle, is when either there are no boundary points at all, or there is at least one boundary point, but also at least one internal insertion with twist $(0,0,0,0,0).$ In the former case there is a nowhere vanishing global section for each Witten bundle separately \cite[Proposition 3.19]{BCT2}, hence the intersection number vanishes. In the latter case, using the fact that the Witten bundle is pulled back from the moduli without $0$-twisted points, one easily verifies that the only non-zero number is when there is a single boundary marking of twist $(3,3,3,3,3)$ and a single internal marking of twist $(0,0,0,0,0),$ and this number is $1.$

Consider now the (maximal, \textit{i.e.} $\h=
\lfloor \frac{5}{2}\rfloor-1=1$) point insertion theory for the open quintic with minimal symmetry group. The main difference is that now we also allow $(1,1,1,1,1)$ boundary twists, in addition to $(3,3,3,3,3)$ boundary twists. While the latter twists still give rise to trivial intersection numbers, we believe that primary correlators involving only $(1,1,1,1,1)$ twists may be non-trivial. Note that an intersection number having $l$ internal markings and $k$ boundary markings with twists $(1,1,1,1,1)$ corresponds to a non-empty moduli space precisely when
\begin{equation}\label{eq:deg_quint}5|2l+k-3.\end{equation}In this case the rank and dimension of the moduli space agree. Now, with the point insertion boundary conditions, Proposition \ref{prop:vanish small internal}
shows that all primary intersection numbers with $l\geq 1$ are $0.$ 

So we are left with considering the case $l=0.$ In this case, by \eqref{eq:deg_quint} non-zero intersection numbers may appear only for $k=5d-2,$ for a natural number $d.$
When $d$ is even, Proposition \ref{prop:even_vanishing_quintic} has the following corollary.
\begin{cor}\label{cor:vanishing_quintic_even}
For even $d,$\[\langle(\sigma^{1})^{5d-2}\rangle^{\frac{1}{r}=\frac{1}{5},\text{o},\h=1,m=5}_0=0.\]
\end{cor}
We currently do not know how to calculate the open intersection numbers for $k=5d-2,$ in case $d$ is odd, and in particular we do not know if they vanish or not.

In \cite{PSW}, 
Pandharipande, Solomon and Walcher study the open GW theory of the quintic CY $3$-fold. They calculate, for each odd $d,$ a non-zero invariant. 
Recently, following Walcher \cite{Walcher}, Aleshkin and Liu conjectured \cite{Melissa} that there should exist an open FJRW for the Fermat quintic with minimal symmetry group whose genus-zero potential satisfies an open version of LG/CY correspondence \cite{ChiodoRuan}, relating its correlators to the enumerative invariants of \cite{PSW}.
\\\textbf{Conjecture.}
The open quintic theory described above satisfies the open LG/CY correspondence conjectured in \cite{Melissa}.

A (weak) evidence for this conjecture is that, also in \cite{Melissa}, the invariants are labelled by natural numbers $d,$ and those which correspond to even $d$ vanish.

\section{Extensions, generalizations and applications of the point insertion technique}\label{sec:generalization}

In this section we briefly describe how the point insertion procedure is applied to other intersection theories and to higher genus. We present only the general idea, leaving the details, including an axiomatization of the point insertion idea, to future works.

We will describe scenarios in which one can perform point insertion to obtain a moduli space which effectively has only spurious boundaries.
Then, on this moduli space one can define intersection numbers which may involve open $\psi$-classes,  additional classes like Hodge classes, and, in the case of Gromov--Witten theories, also evaluation maps (which are equivalent to constraining the marked points to lie on homology cycles in the target pair $(M,L)$). It can be shown, just like in Section \ref{sec sections}, that such correlators will be well-defined. 

\subsection{Point insertion in the presence of a target pair of spaces}\label{sec:target}
The point insertion method can also be used to define Open Gromov--Witten (OGW) theories for target manifolds in the presence of open $r$-spin structures, or, under certain assumptions, even without. We will sketch the constructions, leaving the details to future works. In what follows we ignore orientation and smoothness questions.

Let $(M,L)$ be a pair of a $2n$-dimensional symplectic manifold $M$ 
with a compatible or tame almost complex structure, and $L$ a Lagrangian submanifold. 

\subsubsection{OGW theories with $r$-spin}\label{subsub:OGW_r_spin} 
We begin with OGW theories which are enriched with $r$-spin structures. We will require, the additional data of a $2n$-dimensional submanifold $S\subset L\times M$ which intersects $L\times L\subset L\times M$ transversally, and the intersection is the diagonal $\Delta\subset L\times L.$ The theory will be $S$-dependant.
\begin{rmk}\label{rmk:homological_S}
There is a cohomological interpretation of $S.$ We will work with $\mathbb{Q}$ coefficients.
Let \[\iota: H^*(M)\to H^*(L)\] be the natural map induced from the inclusion $L\hookrightarrow M,$ and let 
\[[\Delta]=\sum_{i=1}^N\pi_i\otimes\zeta_i\]be the K\"unneth decomposition of the diagonal $\Delta\hookrightarrow L\times L$.
The additional data $S$ which satisfies the above conditions can be seen to have a K\"unneth decomposition of the form \[[S]=\sum_{i=1}^N\pi_i\otimes\rho_i,~~\text{where 
 }\iota(\rho_i)=\zeta_i.\]
 Thus, since $\{\zeta_i\}_i$ is a basis for $[H^*(L)],$ the data of $S$ cohomologically corresponds to the choice of a section
 \[\varsigma:H^*(L)\to H^*(M)\]for $\iota.$
 The existence of such a section, which is equivalent to $\iota$ being surjective, is a cohomological obstruction. When $\iota$ surjects, different choices of sections give rise to different theories.
\end{rmk}
With this data, we can define, for every $r\geq 2,\h\in\{0,\ldots,\lfloor\frac{r}{2}\rfloor-1\},$ and odd $m,$ an enumerative theory as follows by effectively cancelling every codimension-$1$ boundary.
Denote by $\overline{\mathcal{M}}_{0,B,I}^{1/r}(M,L;\gamma)$ the moduli space of stable maps from a graded marked $r$-spin disk into $(M,L)$, with boundary markings labelled by $B$ and internal markings labelled by $I$, and of degree $\gamma\in H_2(M,L)$.  We will assume niceness assumptions: that the moduli space is a smooth orbifold with corners of the expected dimension, and that boundary strata can be written as fibered products of orbifolds with corners. We denote the evaluation maps by
\[\text{ev}_{M,i}:\overline{\mathcal{M}}_{0,B,I}^{1/r}(M,L;\gamma)\to M,\] for the internal marking $i\in I,$ and\[\text{ev}_{L,b}:\overline{\mathcal{M}}_{0,B,I}^{1/r}(M,L;\gamma)\to L,\] for the boundary marking $b\in B.$

The only possible codimension-$1$ boundaries we may need to cancel are those which parameterize objects with a single boundary node, whose illegal side has twist at most $2\h.$ These are the spurious boundary strata that in this paper we treated by applying the point insertion method. The choice of $S$ allows us to lift the point insertion method to the current setting.

Consider a boundary stratum
\[\overline{\mathcal{M}}^{1/r}_{0,B_1\cup\{\star_1\},I_1}(M,L;\gamma_1)\times_{\Delta}\overline{\mathcal{M}}^{1/r}_{0,B_2\cup\{\star_2\},I_2}(M,L;\gamma_2),\]
where $(B_1,B_2)$ is a partition of the boundary markings, $(I_1,I_2)$ a partition of the internal markings, $\star_1,\star_2$ are the half-nodes, $\gamma=\gamma_1+\gamma_2,$ and the fiber product is with respect to the evaluation maps at $\star_1,\star_2:$ they are required to map to the same point in $L,$ that is, we intersect the Cartesian product of moduli space with the condition\footnote{We use  $(\text{ev}_{L,\star_1},\text{ev}_{L,\star_2})\in\Delta\subset L\times L$ as a shorthand notation for $(\text{ev}_{L,\star_1}(u),\text{ev}_{L,\star_2}(v))\in\Delta\subset L\times L,$ for $u\in\overline{\mathcal{M}}^{1/r}_{0,B_1\cup\{\star_1\},I_1}(M,L;\gamma_1),~v\in\overline{\mathcal{M}}^{1/r}_{0,B_2\cup\{\star_2\},I_2}(M,L;\gamma_2)$ }
\[(\text{ev}_{L,\star_1},\text{ev}_{L,\star_2})\in\Delta\subset L\times L.\]
Now, suppose, without loss of generality, that $\star_2$ is the illegal side, and consider the moduli space $\overline{\mathcal{M}}^{1/r}_{0,B_2,I_2\cup\{\star\}}(M,L;\gamma_2).$ This moduli space has a real codimension-$1$
boundary\[\overline{\mathcal{M}}^{1/r}_{0,\{\star'_1\},\{\star\}}(M,L;0)\times_{\Delta}\overline{\mathcal{M}}^{1/r}_{0,B_2\cup\{\star_2\},I_2}(M,L;\gamma_2)\] obtained by letting its new marking $\star$ collide with the boundary, thus forming a degree-$0$ bubble. Observe that
\[\overline{\mathcal{M}}^{1/r}_{0,\{\star'_1\},\{\star\}}(M,L;0)\simeq \text{point}\times L\simeq L\] canonically.
Thus,
\begin{equation}\label{eq:point_insertion_OGW}
\overline{\mathcal{M}}^{1/r}_{0,\{\star'_1\},\{\star\}}(M,L;0)\times_{\Delta}\overline{\mathcal{M}}^{1/r}_{0,B_2\cup\{\star_2\},I_2}(M,L;\gamma_2)\simeq\overline{\mathcal{M}}^{1/r}_{0,B_2\cup\{\star_2\},I_2}(M,L;\gamma_2) 
\end{equation}canonically.
Now, consider the moduli space
\[\overline{\mathcal{M}}^{1/r}_{0,B_1\cup\{\star_1\},I_1}(M,L;\gamma_1)\times_{S}\overline{\mathcal{M}}^{1/r}_{0,B_2,I_2\cup\{\star\}}(M,L;\gamma_2),\]
where this time the product is fibered on $S,$ that is, we have the constraint\[(\text{ev}_{L,\star_1},\text{ev}_{M,\star})\in S\subset L\times M.\]
It has a real codimension-$1$ boundary \begin{align*}\overline{\mathcal{M}}^{1/r}_{0,B_1\cup\{\star_1\},I_1}(M,L;\gamma_1)&\times_{S}\left(\overline{\mathcal{M}}^{1/r}_{0,\{\star'_1\},\{\star\}}(M,L;0)\times_{\Delta}\overline{\mathcal{M}}^{1/r}_{0,B_2\cup\{\star_2\},I_2}(M,L;\gamma_2)\right)\\&\simeq\overline{\mathcal{M}}^{1/r}_{0,B_1\cup\{\star_1\},I_1}(M,L;\gamma_1)\times_{\Delta}\overline{\mathcal{M}}^{1/r}_{0,B_2\cup\{\star_2\},I_2}(M,L;\gamma_2),\end{align*}where the last passage used \eqref{eq:point_insertion_OGW} and the fact that in this boundary stratum the fibered product $\overline{\mathcal{M}}^{(1/r,\h)}_{0,\{\star'_1\},\{\star\}}(M,L;0)\times_\Delta\ldots$ amounts to requiring $\text{ev}_{M,\star}\in L,$ and the constraint $(\text{ev}_{L,\star_1},\text{ev}_{M,\star})\in S$ is replaced by
\[(\text{ev}_{L,\star_1},\text{ev}_{M,\star})\in S\cap L\times L\Leftrightarrow
(\text{ev}_{L,\star_1},\text{ev}_{M,\star})\in \Delta,\]by our assumptions.
We can thus glue the two moduli spaces $\overline{\mathcal{M}}^{1/r}_{0,B_1\cup\{\star_1\},I_1}(M,L;\gamma_1)\times_{S}\overline{\mathcal{M}}^{1/r}_{0,B_2,I_2\cup\{\star\}}(M,L;\gamma_2)$ and $\overline{\mathcal{M}}_{0,B,I}^{1/r}(M,L;\gamma)$ along this common boundary.

We can do it consistently for as many nodes as possible, as we did when we defined $\widetilde{\mathcal M}^{\frac{1}{r},\h}_{0,B,I}$ in Section \ref{subsec PI} and \cite[Section 4]{TZ1} to obtain a glued moduli space $ \widetilde{\mathcal M}^{\frac{1}{r},\h}_{0,B,I}(M,L,S;\gamma).$
Note that we have a natural forgetful map
\[\text{For}_{\text{target}}:
\widetilde{\mathcal M}^{\frac{1}{r},\h}_{0,B,I}(M,L,S;\gamma)\to\widetilde{\mathcal M}^{\frac{1}{r},\h}_{0,B,I},\]which forget the map to the target pair $(M,L)$ and possibly contracts some components. The Witten bundle $\widetilde{\mathcal{W}}\to\widetilde{\mathcal M}^{\frac{1}{r},\h}_{0,B,I}(M,L,S;\gamma)$ is defined either directly or by pullback from
$\widetilde{\mathcal M}^{\frac{1}{r},\h}_{0,B,I}.$

The glued moduli space parameterizes stable maps from nodal, not necessarily connected, marked graded $r$-spin disks, where a subset of boundary (internal, respective) special points are labelled by $B$ ($I$, respectively), and certain pairs of (unlabelled) special points are required to lie on $S,$ under the assumptions that the total degree is $\gamma,$ and that the graph, whose vertices are the connected stable disk components and there is an edge between two vertices if they contain a pair of points as above, is connected and genus-zero. We combinatorially refer to a pair of points which are constrained to lie on $S$ as connected by a dashed line. 

The only codimension-$1$ boundaries of $\widetilde{\mathcal M}^{\frac{1}{r},\h}_{0,B,I}(M,L,S;\gamma)$ correspond to strata parameterizing objects with either a contracted boundary component (type-CB), or a boundary node whose illegal side has a twist greater than $2\h$ (type-R or type-NS+). These are precisely the boundaries that are handled by the positivity boundary conditions of the Witten bundle.

Thus, one can define correlators which involve the open Witten class, open $\psi$-classes and classes coming from the target pair: 
for $A_1,\ldots,A_l\in H^*(M),~B_1,\ldots,B_k\in H^*(L),\beta\in H_2(M,L),$ denote by 
\begin{align*}
        &\left\langle\tau^{a_1}_{d_1}(A_1)\dots\tau^{a_l}_{d_l}(A_l)\sigma^{b_1}(B_1)\dots\sigma^{b_k}(B_k)
		\right\rangle_{0,\beta}^{\frac{1}{r},\h,\text{o},m,M,L,S}\\
  &\quad\qquad=\int_{[\widetilde{\mathcal M}^{\frac{1}{r},\h}_{0,B,I}(M,L,S;\beta)]}e\left(\widetilde{{\mathcal{W}}}^{\oplus m}\oplus\bigoplus_{i\in I}\widetilde{\mathbb L}_i^{\oplus d_i},\widetilde{\mathbf{s}}_{total}\right)\prod_{i=1}^l \text{ev}_{M,i}^*(A_i)
  \prod_{j=1}^k\text{ev}_{L,i}^*(B_i),
\end{align*}
where
$e(\widetilde{{\mathcal{W}}}^{\oplus m}\oplus\bigoplus_{i\in I}\widetilde{\mathbb L}_i^{\oplus d_i},\widetilde{\mathbf{s}}_{total})$
is the relative Euler class of the glued bundle with the boundary conditions of this paper as in Remark \ref{rmk intersection numbers on glued mod}. 
It can be shown that such correlators are well-defined.

An almost spinless example is obtained by considering the $r=2$ case, where all markings are NS. In this case there is no Witten class, and we obtain, in $g=0,$ a usual OGW theory.

\subsubsection{OGWs with a selection rule}\label{subsub:even}

The ideas of the previous subsection generalize also to theories without spin structures, as long as we have the following additional data:
\begin{enumerate}
\item A $2n$-dimensional submanifold $S\subset L\times M$ as in the previous subsection. 
\item A \emph{selection rule} for boundary half-nodes: there is a consistent way to choose a half-node for every boundary node of every boundary stratum which appears in the theory's virtual fundamental chain. 
This rule generalizes the `legal'/`illegal' decorations of half-nodes, and is responsible for the choice of which half-node to replace by an internal marking in the glued component.
By `consistent' we mean that if two strata have a common smoothing, and under this smoothing a given node is identified with another node, then the two ways to choose the half-nodes of these nodes agree.
The virtual fundamental chain may be the moduli itself, a virtual analogue of it, or, for theories with additional data, such as $r$-spin theories with target manifolds, the pullback of the open Witten chain, that is, the zero locus of a canonical section of the open Witten bundle.  
\item If contracted boundary nodes may appear in the theory's virtual fundamental chain, we will also require the choice of a $(n+1)$-dimensional submanifold $C\subset M$ with boundary satisfying
\[\partial C=L,\]
and the theory will also be $C$-dependant.
\end{enumerate}
The third item does not appear in the $r$-spin setting, since there contracted boundaries are treated by the positivity boundary conditions.

With this data we can again define an enumerative theory by effectively cancelling every codimension-$1$ boundary which may appear in the virtual cycle.

Boundary nodes are treated as in the previous subsection, only that the selection rule replaces the graded spin structure in choosing the distinguished half-node to which we apply the point insertion procedure. 

The treatment in boundary strata which parameterizes objects with a contracted boundary node is more standard, and can be found, for example in \cite[Section 7]{FOOO_ii} (in a different formalism). Note that a moduli space $\overline{\mathcal{M}}_{0,B,I}(M,L;\gamma)$ (which may or may not have inserted nodes of as $\star$ above) may have a boundary component which parameterizes stable disks with a contracted boundary node
only if \[B=\emptyset~\text{and }\partial\gamma=0,\] where $\partial \colon H_2(M,L)\to H_1(L)$ is the natural map. Every such $\gamma$ is in the image of the natural map $q:H_2(M)\to H_2(M,L)$ from the relative homology exact sequence. For such $\gamma,$ every boundary stratum of  $\overline{\mathcal{M}}_{0,\emptyset,I}(M,L;\gamma)$ which parameterizes disks with a contracted boundary is canonically diffeomorphic to a closed moduli space
\[\overline{\mathcal{M}}_{0,I\cup\{\hat{\star}\}}(M;\hat{\gamma})\cap\left(\text{ev}_{M,\hat\star}\in L\right),\]
where $\hat{\gamma}\in q^{-1}(\gamma)\subseteq H_2(M),$
and the new special internal point $\hat\star$ is an \emph{inserted contracted boundary}. 
This space is naturally identified as a boundary of the moduli space
\[\overline{\mathcal{M}}_{0,I\cup\{\hat{\star}\}}(M;\hat{\gamma})\cap\left(\text{ev}_{M,\hat\star}\in C\right).\]We glue the latter space to $\overline{\mathcal{M}}_{0,\emptyset,I}(M,L;\gamma)$ along this common boundary. We can combine this gluing procedure with the previous gluing procedure described above (based on the point insertion scheme), in a consistent way, and the final result is a compact moduli space $$
    \widetilde{\mathcal M}_{0,B,I}(M,L,S,C;\gamma)
    $$ which is without boundary. This space can also be considered as a moduli space of stable maps from disconnected disks with a constrain that, as in the previous subsection, certain pairs of points are required to lie on $S$, together with an additional constraint that certain (unlabelled) internal special points on connected component not meeting $\text{Fix}(\phi)$ are required to lie on $C$. 

Thus, one can define correlators which involve the open Witten class, open $\psi$ classes and classes coming from the target pair: 
For $A_1,\ldots,A_l\in H^*(M),~B_1,\ldots,B_k\in H^*(L),\beta\in H_2(M,L),$ denote by \[\left\langle\tau^{a_1}_{d_1}(A_1)\dots\tau^{a_l}_{d_l}(A_l)\sigma^{b_1}(B_1)\dots\sigma^{b_k}(B_k)
		\right\rangle_{0,\beta}^{\text{o},M,L,S}
  =\int_{[\widetilde{\mathcal M}_{0,B,I}(M,L,S,C;\beta)]}\prod_{i=1}^l \text{ev}_{M,i}^*(A_i)\psi_i^{d_i}\prod_{j=1}^k\text{ev}_{L,i}^*(B_i),\]where the notations are as in the previous subsection, and $\psi_i=c_1(\widetilde{\mathbb{L}}_i)$.  Again, it can be shown that such correlators are well-defined.

\subsubsection{OGWs without a selection rule}\label{subsub:odd}

We now consider an important variant of the previous scenario. This time we require: 
\begin{itemize}
\item a smooth cobordism $cob\subseteq L\times L$ between $\Delta$ and $\Delta_1\cup \Delta_2,$ where $\Delta_1,\Delta_2=\Delta_1^{\text{op}}\subseteq L\times L,$\footnote{by $\Delta_1^{\text{op}}$ (and $S_1^{\text{op}}$ below) we mean the image of $\Delta_1$ (respectively, $S_1$) inside $L\times L$ (respectively, $M\times L$) under the map $L\times L\to L\times L$ (respectively, $L\times M\to M\times L$) which switches the order of factors.} are unions of $n$ dimensional smooth submanifolds which intersect transversally, and any two of $\Delta,\Delta_1,\Delta_2$ intersect transversally. Suppose that we have $S_1\subseteq L\times M,~S_2=S_1^{\text{op}}\subseteq M\times L,$ which satisfy
\[S_i\pitchfork L\times L=\Delta_i,~i=1,2.\]
\item If necessary, we have $C$ as above, satisfying $\partial C=L.$
\end{itemize}
We act similarly to the previous subsections.
First, observe that
\[\overline{\mathcal{M}}_{0,B_1\cup\{\star_1\},I_1}(M,L;\gamma_1)\times_{cob}\overline{\mathcal{M}}_{0,B_2\cup\{\star_2\},I_2}(M,L;\gamma_2),\] has three distinguished boundary components, induced from the boundary of $cob$, \textit{i.e.}
\begin{equation}\label{eq:cob_bdry_1}
\overline{\mathcal{M}}_{0,B_1\cup\{\star_1\},I_1}(M,L;\gamma_1)\times_{\Delta}\overline{\mathcal{M}}_{0,B_2\cup\{\star_2\},I_2}(M,L;\gamma_2),
\end{equation}
and
\begin{equation}\label{eq:cob_bdry_i}
\overline{\mathcal{M}}_{0,B_1\cup\{\star_1\},I_1}(M,L;\gamma_1)\times_{\Delta_i}\overline{\mathcal{M}}_{0,B_2\cup\{\star_2\},I_2}(M,L;\gamma_2),~i=1,2.
\end{equation}
The other boundaries of it are of the form
\begin{align*}\label{eq:cob_bdry_non_top}
\left(\partial\overline{\mathcal{M}}_{0,B_1\cup\{\star_1\},I_1}(M,L;\gamma_1)\right)&\times_{cob}\left(\overline{\mathcal{M}}_{0,B_2\cup\{\star_2\},I_2}(M,L;\gamma_2)\right)\\&\cup \left(\overline{\mathcal{M}}_{0,B_1\cup\{\star_1\},I_1}(M,L;\gamma_1)\right)\times_{cob}\left(\partial\overline{\mathcal{M}}_{0,B_2\cup\{\star_2\},I_2}(M,L;\gamma_2)\right).
\end{align*}
This time we will glue, to the boundary  
\[\overline{\mathcal{M}}_{0,B_1\cup\{\star_1\},I_1}(M,L;\gamma_1)\times_{\Delta}\overline{\mathcal{M}}_{0,B_2\cup\{\star_2\},I_2}(M,L;\gamma_2)\]
of $\overline{\mathcal{M}}_{0,B,I}(M,L;\gamma)$, the boundary \eqref{eq:cob_bdry_1} of $\overline{\mathcal{M}}_{0,B_1\cup\{\star_1\},I_1}(M,L;\gamma_1)\times_{cob}\overline{\mathcal{M}}_{0,B_2\cup\{\star_2\},I_2}(M,L;\gamma_2).
$
To the boundary \eqref{eq:cob_bdry_i} of $\overline{\mathcal{M}}_{0,B_1\cup\{\star_1\},I_1}(M,L,\gamma_1)\times_{cob}\overline{\mathcal{M}}_{0,B_2\cup\{\star_2\},I_2}(M,L;\gamma_2)
$, for $i=1,$ we glue, in a complete analogy way to the previous subsection, the boundary \begin{align*}\overline{\mathcal{M}}_{0,B_1\cup\{\star_1\},I_1}(M,L;\gamma_1)&\times_{S_1}\left(\overline{\mathcal{M}}_{0,\{\star''_1\},\{\star'_2\}}(M,L;0)\times_{\Delta}\overline{\mathcal{M}}_{0,B_2\cup\{\star_2\},I_2}(M,L;\gamma_2)\right)\\&\quad\qquad\simeq\overline{\mathcal{M}}_{0,B_1\cup\{\star_1\},I_1}(M,L;\gamma_1)\times_{\Delta_1}\overline{\mathcal{M}}_{0,B_2\cup\{\star_2\},I_2}(M,L;\gamma_2)\end{align*} of \[\overline{\mathcal{M}}_{0,B_1\cup\{\star_1\},I_1}(M,L;\gamma_1)\times_{S_1}\overline{\mathcal{M}}_{0,B_2,I_2\cup\{\star'_2\}}(M,L;\gamma_2),\]
which, as above, is canonically identified with it. We similarly treat the boundary \eqref{eq:cob_bdry_i} for $i=2,$ only that then we make $\star_1$ internal instead of $\star_2.$

The main purpose of the cobordism is to interpolate between the boundaries of the original moduli, and the boundaries coming from the point insertions.

We deal with boundary strata which parameterize disks with contracted boundary nodes using $C,$ as in the previous subsection, and we simultaneously and iteratively perform the cobordism and point insertion gluings to obtain a compact and boundaryless moduli space, denoted by $\widetilde{\mathcal M}_{0,B,I}(M,L,cob,S_1,C;\gamma).$
Again this space parameterizes stable maps from disconnected disks with additional constraints that certain pairs of points are required to lie on $S_1,S_2$ or $cob$, and that certain internal points on connected components not meeting $\text{Fix}(\phi)$ are required to lie on $C$. Graphically we can draw for this scenario two types of dashed lines, the usual one which connects an internal point and a boundary point which are constrained to lie on $S_i,$ and a \emph{boundary-dashed line} which connects a pair of boundary points that are constrained to lie on $cob.$ 
\begin{rmk}\label{rmk:homological_S_odd}
There is a cohomological interpretation of $S_1,S_2$. We use the same notations as in Remark \ref{rmk:homological_S}. Write
\begin{equation}\label{eq:Delta12}[\Delta_1]=\sum_{i=1}^N\pi_i\otimes\zeta_i,~[\Delta_2]=\sum_{i=1}^N\zeta_i\otimes\pi_i \in H^*(L)\otimes H^*(L)\simeq H^*(L\times L).\end{equation} We have \begin{equation}\label{eq:Delta=1+2}[\Delta]=[\Delta_1]+[\Delta_2].\end{equation}
The existence of $S_1$ (or $S_2$) implies the existence of elements $\rho_i\in H^*(M)$ for $i=1,\ldots, N,$ which satisfy
\[\iota(\rho_i)=\zeta_i,\]
such that 
\[[S_1]=\sum_{i=1}^N\pi_i\otimes\rho_i,~[S_2]=\sum_{i=1}^N\rho_i\otimes\pi_i\]
in $H^*(L\times M),~H^*(M\times L)$ respectively.
 The existence of $\rho_i,~i=1,\ldots, N$ requires weaker assumptions than $\iota$ being surjective: 
 it just needs to be surjective on $W=\text{Span}(\zeta_i)$ which may be strictly smaller than the whole $H^*(L).$
 The data of $S_1,S_2$ is cohomologically equivalent to the choice of a section
 \[\varsigma:W\to H^*(M)\]for $\iota.$
 Different choices of $W$ which allow decompositions \eqref{eq:Delta12},~\eqref{eq:Delta=1+2} for which $\iota$ surjects on $W,$ or of a section $\varsigma,$ give rise to different theories.
\end{rmk}

\begin{ex}\label{ex:sphere}
Suppose that $L$ is a Lagrangian sphere, and suppose that there exists a smooth cobordism $cob$ between $\Delta$ and $\Delta_1\cup\Delta_2:=\text{pt}\times L\cup L\times\text{pt}.$ Write \[S_1=\text{pt}\times M,~~S_2=M\times\text{pt}.\]
Then $S_i\pitchfork L\times L =\Delta_i,$ and in this case we can construct a compact boundaryless moduli space, and thus, if it is orientable and nice enough, also an intersection theory.

One such target pair is $(\mathbb{CP}^1,\mathbb{RP}^1).$ In \cite{BPTZ} the OGW theory of this space was studied, using another scheme of boundary conditions - a forgetful scheme. The relation between this scheme and the one described above is to some extent analogous to the relation between the open $r$-spin theory of \cite{BCT2}, and the point insertion open $(r,\h=0)$-spin theory presented in this paper. A notable difference is that in the \cite{BPTZ} theory the point constraints come from the \emph{relative homology} $H_*(\mathbb{CP}^1,\mathbb{RP}^1),$ while in the point insertion theory the constraints are not relative, that is, come from $H_*(\mathbb{CP}^1).$ 
\end{ex}
\begin{ex}\label{ex:odd}
A second example, which generalizes the former one, is the following. Suppose $n$ is odd, and that $H_*(L)$ has a basis represented by submanifolds $P_1,\ldots, P_{2N},$ where $P_1,\ldots,P_{N}$ are odd-dimensional basis elements, and the other elements are even-dimensional. Assume that the basis is self-dual with respect to Poincar\'e duality, and that the dual of $P_i$ is $P_{2N-i}.$ Thus, by K\"unneth formula
\[[\Delta]=\sum_{i=1}^{2N} [P_i]\otimes [P_{2N-i}].\]
Assume that $cob$ is a cobordism between $\Delta$ and
\[\Delta_1\cup\Delta_2:=
\left(\bigcup_{i=1}^N P_i\times P_{2N-i}\right)\cup\left( \bigcup_{i=1}^N P_{2N-i}\times P_i\right).
\]
Suppose that there exist submanifolds $R_1,\ldots, R_N\subseteq M$ with
\[R_i\pitchfork L = P_i,\]
then we can take \[S_1 = \bigcup_{i=1}^N R_i\times P_{2N-i},~~S_2 = \bigcup_{i=1}^N P_{2N-i}\times R_i.\]
\end{ex}
In light of Remark \ref{rmk:homological_S_odd}, the point insertion scheme presented in the last example, requires $\iota$ to surject only on the even part of $H^*(L)$. In the case of $L$ being a homology sphere (the first example), this requirement always holds.
\begin{rmk}[Dependence of choices]
Different choices of the cobordism $cob$ or cycle $C$ give rise to different enumerative theories, related by a wall-crossing formula that can be written in terms of closed and open GW invariants. 
\end{rmk}
\subsubsection{Equivariant theory}
Suppose that a real torus $T=(\mathbb R^\times)^m$ acts on $M,$ preserving $L.$ If, in the context of Section \ref{subsub:OGW_r_spin} $S$ is preserved, or in the context of Section \ref{subsub:even} $S,C$ are also preserved, or, in the context of Section \ref{subsub:odd}, $cob,S_1,S_2,C$ are also preserved, then the torus action lifts to the glued moduli spaces, and under niceness assumptions one can define an equivariant OGW theory.
\subsection{High genus}\label{sec:high_genus}
The construction presented in this work, as well as the constructions which were sketched in Section \ref{sec:target}, could have been extended to higher genus, if one could have constructed virtual fundamental chains for the higher genus open theories, which satisfy some natural assumptions on the boundary behaviour.
If such constructions were to be found, 
it is expected that the gluing procedure described here would extend to higher genus without a change. Then, under certain orientability assumptions, one would obtain higher genus extensions of the $g=0$ point insertion theories described here.
We have taken basic steps towards achieving this in the open $r$-spin case, in $g=1,$ and have a prediction for a genus-$1$ topological recursion.
It is likely to expect that higher genus theories will be related to integrable hierarchies.

\subsection{Additional classes and universal recursions}
\subsubsection{The open Hodge bundle}\label{subsec:hodge}
The point insertion idea can be implemented in other stringy-like intersection open theories. As an example, we briefly describe how one can apply the point insertion idea to the Hodge bundle, but the same idea applies to many other bundles of similar flavour, and variants of the Hodge bundle. 
In the closed setting the Hodge bundle $\mathcal{H}_{g,n}\to\Mbar_{g,n}$ is the rank $g$ complex vector bundle whose fiber over a moduli point $C$ is the vector space of Abelian differentials on $C.$ Abelian differentials are meromorphic $1$-forms whose only possible poles are at the nodes, and they are simple poles satisfying that the sum of residues at the two half-nodes is $0.$

In order to define the open Hodge bundle, we act in analogy to the definition of the open Witten bundle:
the natural conjugation on real curves, which include the surfaces of the form $\Sigma\sqcup_{\partial\Sigma}\overline{\Sigma}$ for a surface with boundary $\Sigma$, lifts to their cotangent bundles, and induce a conjugation on the pullback of the complex Hodge bundle.

We start without point insertions. The \emph{open Hodge bundle} (before point insertion) is the rank $g$ \emph{real}  bundle $\mathcal{H}_{g,k,l}\to\Mbar_{g,k,l}$ obtained as the conjugation fixed part of the complex $\mathcal{H}_{g,k+2l}.$

Fibers of this bundle are meromorphic $1$-forms on the doubled surface, which are invariant under the (natural lift of the) conjugation map. Observe that the invariance under conjugation implies that conjugate half-nodes have conjugate residues,  boundary nodes have real residues, and contracted boundaries have imaginary residues.
The open Hodge bundle is also defined on moduli space of stable maps or FJRW theories by pullback. We can similarly define twisted Hodge bundles by allowing poles of certain orders in some of the marked points.

Suppose that we are given an open intersection theory with a point insertion scheme, such as the $r$-spin theories, FJRW theories considered in this paper, or OGW theories briefly described in this section. 
Let $\widetilde{\mathcal{M}}$ be the corresponding (glued) moduli space, parameterizing disconnected surfaces with dashed lines and additional structure, as in the cases described above, and $\overline{\mathcal{M}} $ be its unglued version, made from the same connected components before gluing identified boundaries.

We define $\mathcal{H}\to\overline{\mathcal{M}}$ as the vector bundle whose fiber at a disconnected surface $\Sigma$ is a conjugation-invariant meromorphic $1$-form $\alpha$ satisfying
\begin{itemize}
\item all the poles are simple, and can appear at half-nodes or points coming from point insertion (inserted contracted boundary nodes and points paired by (boundary-)dashed lines);
\item the sum of residues at the two half-nodes of a given node is $0;$
\item the residue at an inserted contracted boundary node is purely imaginary;
\item if $\star$ is an inserted point, connected by a dashed line to the boundary point $\star_b$, then $\text{Res}_{\star}(\alpha),~\text{Res}_{\star_b}(\alpha)$ are real and satisfy

\[2\text{Res}_{\star}(\alpha)+\text{Res}_{\star_b}(\alpha)=0.\]
\item if $\star_1,\star_2$ are boundary points connected by a boundary-dashed line as in Subsection \ref{subsub:odd}, then $\text{Res}_{\star_1}(\alpha),~\text{Res}_{\star_2}(\alpha)$ are real and satisfy
\[\text{Res}_{\star_1}(\alpha)+\text{Res}_{\star_2}(\alpha)=0.\]
\end{itemize}
Note that the sum of residues of $\alpha$ on each irreducible connected component of the surface is $0.$
It is easy to see that the rank of the vector bundle is the genus $g.$ 

The extra conditions allow identifying the Hodge bundles of two boundary strata which are identified under the gluing.  
For example, as mentioned above, the residue of a conjugation invariant meromorphic $1$-form at a contracted boundary node is imaginary, explaining the third item above.
When an inserted point $\star$ with real residue $\text{Res}_{\star}(\alpha)$ reaches the boundary as in the right-hand side of Figure \ref{fig rh surface}, the residue at the left boundary half-node of $D_1^{AI}$ is $\text{Res}_{\star}(\alpha)+\overline{\text{Res}_{\star}(\alpha)}=2\text{Res}_{\star}(\alpha)$; on the other hand, in order for the gluing of the moduli spaces along this boundary stratum to lift to a gluing of the Hodge bundles, this residue must be minus of the residue at the boundary point $\star_b$ of $D_2$ that connects to $\star$ via the dashed line, 
which must therefore be $-2\text{Res}_{\star}(\alpha)$; this explains the condition of the fourth item above.

We should comment that in general the open Hodge bundle is not necessarily oriented. Therefore intersection theories based on it should either take an even number of copies of it, its complexification, or any other class compensating for the absence of orientation.
It can be shown that twisted Hodge bundles in which boundary markings are allowed to have poles of odd orders, are canonically relatively oriented.

Another comment is that one can also define positivity boundary conditions for open Hodge bundles, by requiring positivity of residues at certain half-nodes, similarly to the positivity boundary conditions we defined in open $r$-spin theories.

We leave the treatment of Hodge-like theories to future works.

\subsubsection{$\psi_i$-classes and universal TRRs}
In the contexts of Sections \ref{sec:open_fjrw},~\ref{sec:target},~\ref{sec:high_genus},~\ref{subsec:hodge} the relative cotangent line bundles ${\mathbb L}_i$ are defined, just as in Section \ref{sec mod and bundle}, and they are easily seen to be complex line bundles whose fibers are naturally identified along glued strata.

In genus $0,$ by studying the zero locus of TRR sections, as in Section \ref{sec:trr}, one can show that the zero locus is the same as the one found in the proof of Theorem \ref{thm TRR}.
The $\psi_i$-class induces a cohomology class.
When all boundaries of the virtual cycle are cancelled, as in the open $r$-spin theories of this paper, or the scenarios described above, any collection of $\psi$-classes with the glued virtual cycle gives rise to a well-defined homology class in the expected dimension, or a number in case this expected dimension is $0.$

This can be written, in $g=0,$ as universal topological recursion relations for such theories, which have the pictorial form\[[\psi_i]=\sum\left[
\begin{tikzpicture}[scale=0.45, every node/.style={scale=0.5},baseline={([yshift=-.5ex]current bounding box.center)}]
\draw (0,0) circle (1.5);
\draw (-1.5,0) arc (180:360:1.5 and 0.5);
\draw[dashed](1.5,0) arc (0:180:1.5 and 0.5);

\draw (-1.5,-3) arc (180:360:1.5 and 0.5);
\draw[dashed](1.5,-3) arc (0:180:1.5 and 0.5);
\draw(1.5,-3) arc (0:180:1.5);

\node at (0,-1.5) [circle,fill,inner sep=1pt]{};

\draw (-6.5,-2) arc (180:360:1.5 and 0.5);
\draw[dashed](-3.5,-2) arc (0:180:1.5 and 0.5);
\draw(-3.5,-2) arc (0:180:1.5);

\node at (-3.5,-2) [circle,fill,inner sep=1pt]{};

\draw (-5.5,2) arc (180:360:1.5 and 0.5);
\draw[dashed](-2.5,2) arc (0:180:1.5 and 0.5);
\draw(-2.5,2) arc (0:180:1.5);

\node at (-2.5,2) [circle,fill,inner sep=1pt]{};

\draw (2.5,2) arc (180:360:1.5 and 0.5);
\draw[dashed](5.5,2) arc (0:180:1.5 and 0.5);
\draw(5.5,2) arc (0:180:1.5);

\node at (2.5,2) [circle,fill,inner sep=1pt]{};

\draw (3.5,-2) arc (180:360:1.5 and 0.5);
\draw[dashed](6.5,-2) arc (0:180:1.5 and 0.5);
\draw(6.5,-2) arc (0:180:1.5);

\node at (3.5,-2) [circle,fill,inner sep=1pt]{};

\node at (-0.6,0.8){*};
\draw[dashed] (-2.5,2) -- (-0.6,0.85);

\node at (0.6,0.8){*};
\draw[dashed] (2.5,2) -- (0.6,0.85);

\node at (-0.6,-0.8){*};
\draw[dashed] (-3.5,-2) -- (-0.6,-0.75);

\node at (0.6,-0.8){*};
\draw[dashed] (3.5,-2) -- (0.6,-0.75);

\node at (0,1)[circle,fill,inner sep=1pt]{};
 \node at (0.15,0.7) {\Large $a_1$};

\node at (0,-2.1)[circle,fill,inner sep=1pt]{};
 \node at (0.28,-2.3) {\Large $a_2$};

\node at (0,2.5) {$\dots\dots$};
\end{tikzpicture} \right]=\sum\left[
\begin{tikzpicture}[scale=0.45, every node/.style={scale=0.5},baseline={([yshift=-.5ex]current bounding box.center)}]
\draw (0,0) circle (1.5);
\draw (-1.5,0) arc (180:360:1.5 and 0.5);
\draw[dashed](1.5,0) arc (0:180:1.5 and 0.5);

\draw (-1.5,-3) arc (180:360:1.5 and 0.5);
\draw[dashed](1.5,-3) arc (0:180:1.5 and 0.5);
\draw(1.5,-3) arc (0:180:1.5);

\node at (0,-1.5) [circle,fill,inner sep=1pt]{};

\draw (-6.5,-2) arc (180:360:1.5 and 0.5);
\draw[dashed](-3.5,-2) arc (0:180:1.5 and 0.5);
\draw(-3.5,-2) arc (0:180:1.5);

\node at (-3.5,-2) [circle,fill,inner sep=1pt]{};

\draw (-5.5,2) arc (180:360:1.5 and 0.5);
\draw[dashed](-2.5,2) arc (0:180:1.5 and 0.5);
\draw(-2.5,2) arc (0:180:1.5);

\node at (-2.5,2) [circle,fill,inner sep=1pt]{};

\draw (2.5,2) arc (180:360:1.5 and 0.5);
\draw[dashed](5.5,2) arc (0:180:1.5 and 0.5);
\draw(5.5,2) arc (0:180:1.5);

\node at (2.5,2) [circle,fill,inner sep=1pt]{};

\draw (3.5,-2) arc (180:360:1.5 and 0.5);
\draw[dashed](6.5,-2) arc (0:180:1.5 and 0.5);
\draw(6.5,-2) arc (0:180:1.5);

\node at (3.5,-2) [circle,fill,inner sep=1pt]{};

\node at (-0.6,0.8){*};
\draw[dashed] (-2.5,2) -- (-0.6,0.85);

\node at (0.6,0.8){*};
\draw[dashed] (2.5,2) -- (0.6,0.85);

\node at (-0.6,-0.8){*};
\draw[dashed] (-3.5,-2) -- (-0.6,-0.75);

\node at (0.6,-0.8){*};
\draw[dashed] (3.5,-2) -- (0.6,-0.75);

\node at (0,1)[circle,fill,inner sep=1pt]{};
 \node at (0.15,0.7) {\Large $a_1$};

\node at (0,-3.5)[circle,fill,inner sep=1pt]{};
 \node at (0.28,-3.75) {\Large $b_1$};

\node at (0,2.5) {$\dots\dots$};
\end{tikzpicture} \right].\]
We can write the recursion more explicitly in the scenarios of Section \ref{sec:target}. We will focus on the second type of topological recursion (which assumes $l\geq 2$) of Theorem \ref{thm TRR}. The formula for the first type TRR (analogous to the first item of Theorem \ref{thm TRR}, which assumes $k,l\geq 1$) is similar and will be omitted.
Let $\sum_{i=1}^K\nu_i\otimes\mu_i$ be the (cohomological) K\"unneth decomposition of the diagonal of $M\times M.$

We start with the scenario of Section \ref{subsub:OGW_r_spin}, which considers OGW theories coupled with open $r$-spin theories. 
We use the notations of Section \ref{subsub:OGW_r_spin} and Remark \ref{rmk:homological_S}, but for shortness we omit all superscripts, except for $\text{o},\text{c}$ which indicate whether the intersection number corresponds to an open or closed topology (for closed $r$-spin we also allow closed extended insertions).

The universal TRR in the context of Section \ref{subsub:OGW_r_spin}
reads (assuming $l\geq 2$)
\begin{equation}
		\begin{split}
		&\left\langle
		\tau^{a_1}_{d_1+1}(A_1)\tau^{a_2}_{d_2}(A_2)\dots\tau^{a_l}_{d_l}(A_l)\sigma^{b_1}(B_1)\sigma^{b_2}(B_2)\dots\sigma^{b_k}(B_k)
		\right\rangle_{0,\beta}^{\text{o}}
		\\=&\sum_{\substack{s\ge 0\\-1\le a \le r-2\\
  e\in[K]\\0 \le t_1,\ldots,t_s \le \h\\f_1,\ldots,f_s\in[N]
  }}\sum_{\substack{\coprod_{j=-1}^{s}U_j=\{3,4,\dots,l\}\\\coprod_{j=0}^{s}T_j=\{1,2,\dots,k\}\\\beta_{-1}\in H_2(M),\beta_0,\ldots,\beta_s\in H_2(M,L)\\\sum \beta_j=\beta\\ \{(U_j,T_j,t_j,\beta_j)\}_{1\le j \le s}\text{unordered}}}(-1)^s\left\langle
		\tau^{a}_{0}(\nu_e)\tau^{a_1}_{d_1}(A_1)\prod_{i\in U_{-1}}\tau^{a_i}_{d_i}(A_i)\prod_{m=1}^{s}\tau^{t_m}_{0}(\rho_{f_m})
		\right\rangle_{0,\beta_{-1}}^{\text{c}}\\
		&\cdot\left\langle
		\tau^{r-2-a}_{0}(\mu_e)\tau^{a_2}_{d_2}(A_2)\prod_{i\in U_0}\tau^{a_i}_{d_i}(A_i)\prod_{i\in T_0}\sigma^{b_i}(B_i)
		\right\rangle_{0,\beta_0}^{\text{o}}
		\prod_{m=1}^{s}\left\langle
		\sigma^{r-2-2t_m}(\pi_{f_m})\prod_{i\in U_m}\tau^{a_i}_{d_i}(A_i)\prod_{i\in T_m}\sigma^{b_i}(B_i)
		\right\rangle_{0,\beta_m}^{\text{o}}.
		\end{split}
		\end{equation}
Note that by $\sum\beta_j$ we mean $\beta_{-1}+\sum_{i\geq 0}\beta_j, $ where $\sum_{j\geq 0}\beta_j$ is taken in the group $H_2(M,L)$ and the first plus sign should be interpreted as the usual additive action of $H_2(M)$ on $H_2(M,L).$

In the context of Sections \ref{subsub:even},~\ref{subsub:odd} and Remarks \ref{rmk:homological_S},~\ref{rmk:homological_S_odd}, we have the same formula, but without the summation over twists, that is
\begin{equation}
		\begin{split}
		&\left\langle
		\tau_{d_1+1}(A_1)\tau_{d_2}(A_2)\dots\tau_{d_l}(A_l)\sigma(B_1)\sigma(B_2)\dots\sigma(B_k)
		\right\rangle_{0,\beta}^{\text{o}}
		\\=&\sum_{\substack{s\ge 0\\
  e\in[K]\\f_1,\ldots,f_s\in[N]
  }}\sum_{\substack{\coprod_{j=-1}^{s}U_j=\{3,4,\dots,l\}\\\coprod_{j=0}^{s}T_j=\{1,2,\dots,k\}\\\beta_{-1}\in H_2(M),\beta_0,\ldots,\beta_s\in H_2(M,L)\\\sum \beta_j=\beta\\ \{(U_j,T_j,\beta_j)\}_{1\le j \le s}\text{unordered}}}^{-} (-1)^s\left\langle
		\tau_{0}(\nu_e)\tau_{d_1}(A_1)\prod_{i\in U_{-1}}\tau_{d_i}(A_i)\prod_{m=1}^{s}\tau_{0}(\rho_{f_m})
		\right\rangle_{0,\beta_{-1}}^{\text{c}}\\
		&\cdot\left\langle
		\tau_{0}(\mu_e)\tau_{d_2}(A_2)\prod_{i\in U_0}\tau_{d_i}(A_i)\prod_{i\in T_0}\sigma(B_i)
		\right\rangle_{0,\beta_0}^{\text{o}}
		\prod_{m=1}^{s}\left\langle
		\sigma(\pi_{f_m})\prod_{i\in U_m}\tau_{d_i}(A_i)\prod_{i\in T_m}\sigma(B_i)
		\right\rangle_{0,\beta_m}^{\text{o}},
		\end{split}
		\end{equation}
where $\bar\Sigma$ means that we only sum over collections of $U_j, T_j,\beta_j$ which correspond to geometries which may appear in the glued moduli space, under the selection rule for the point insertion.

Also the vanishing result of Proposition \ref{prop:vanish small internal} extends to a more general setting (again we omit most superscripts).
In the scenario of Subsection \ref{subsub:OGW_r_spin}
\[\left\langle\tau_0^{a_1}(A_1)\cdots\right\rangle^{o}=0,\]where $a_1\leq\h$ and $A_1\in\text{Span}\{\rho_i\},$~$\cdots$ represents the remaining insertions, and the dimension of the moduli space on which the intersection numbers is calculated is positive.
In the context of Subsection \ref{subsub:even}, suppose that the selection rule satisfies the following additional conditions:
\begin{itemize}
    \item 
    Whenever an internal marked point $z_i$ collides with the boundary, resulting in a degree-$0$ bubble containing only the internal point $z_i$ and a boundary half-node $h$, the selection rule always chooses the half-node $\sigma_0(h)$ (\textit{i.e.} the half-node corresponding to same node as $h$ but not in the same bubble as $z_i$) to be the inserted point $z^*_i$.
    \item The selection rules for any node in a component containing $z_i,$ and in the same component only with $z^*_i$ replaced by $z_i,$ are the same.
\end{itemize}
Then, with the same assumptions on the dimension of the moduli space and on $A_1,$ and the same interpretation of $\cdots,$ we have
\[\left\langle\tau_0(A_1)\cdots\right\rangle^{o}=0.\]
The proof is again by construction a bijection $\nu$ as in the proof of Proposition \ref{prop:vanish small internal}.
We also expect this vanishing result to extend to higher genus (with the same proof), once a rigorous higher genus definition is found.

\subsection{Allowing a $-1$-boundary twist}
In closed $r$-spin theory \cite{Witten93} intersection numbers with Ramond insertions (twist $r-1$) vanish. The moral reason for this vanishing is the fact that the evaluation map at a Ramond marking is surjective, as degree reasons show (e.g. Lemma \ref{lem surjection of total evaluation map}), and that the fiber of the spin line at a Ramond marking is rationally trivial.

In \cite{BCT_Closed_Extended}, following \cite{JKV2}, it was noted that in $g=0$ one can still define intersection theories with one marked point of twist $-1,$ and the Witten bundle in this case is still an orbifold vector bundle. In this case the evaluation map at a Ramond marking is no longer surjective. This \emph{closed extended} theory plays an important role in open $r$-spin theory. An interesting feature of this theory is its rule of twists at the nodes. In usual closed $r$-spin, if the two branches of the node have twists $a,b$ respectively, then 
$a+b=r-2\mod~r$. In the NS case it restricts to $a+b=r-2,$ but in the Ramond case we have $a+b=2r-2$. In the closed extended theory the rule is effectively modified to 
\[a+b=r-2\] always, where in the Ramond case the half-node in the component which contains the $-1$ point gets twist $r-1,$ and the other side gets twist $-1.$

It turns out that when $r$ is odd, one can also extend the open theory \emph{with maximal point insertions} to allow a single $-1$-twisted boundary marking, which we shall refer to as the \emph{anchor}. We call this theory the \emph{open extended theory}.

The reasoning of \cite{JKV2,BCT_Closed_Extended}
shows that in this case there is no dimension jump in the open Witten fibers, and that they form an orbifold vector bundle.

Just like the surjectivity of the evaluation map to Ramond markings fails at the closed extended theory, in the open extended theory, the positivity argument of Proposition \ref{prop positivity pointwise} fails when adding a $-1$ boundary anchor. 

The way to overcome this failure again uses point insertions:
the point insertion at an NS boundary node is as usual. However, when we reach a Ramond node, we perform point insertion to the half-node on the side closer to the anchor. By ``closer" we mean that if we normalize the Ramond node (before we do the point insertion), then the closer half-node is the one which belongs to the connected component containing the anchor, and we consider different disk components connected by a dashed line as one connected component for this purpose. From the point of view of gluing moduli spaces, this point insertion gives rise to a moduli space of disconnected disks as before, only that we also allow $h=\frac{r-1}{2}$, and the selection rule for which side to make internal in this case, is according to which half-node is closer to the anchor. This selection rule guarantees at most one point of twist $-1$ for each disk.

The proof that this construction works is similar to the other cases studied in this paper, and also the proof of the following topological recursion relations is similar to that of Theorem \ref{thm TRR}.
\begin{thm}\label{thm TRR_open_ext}
\begin{itemize}
\item[(a)] (Boundary marked point TRR) If $l,k\ge 1$, then
\begin{equation*}
		\begin{split}
		&\left\langle
		\tau^{a_1}_{d_1+1}\tau^{a_2}_{d_2}\dots\tau^{a_l}_{d_l}\sigma^{b_1}\sigma^{b_2}\dots\sigma^{b_k}\sigma^{-1}
		\right\rangle_0^{\frac{1}{r},\text{o},\text{ext}}
		\\=&\sum_{\substack{s\ge 0\\-1\le a \le r-2\\ 0 \le t_j \le \frac{r-1}{2}}}\sum_{\substack{\coprod_{j=-1}^{s}R_j=\{2,3,\dots,l\}\\\coprod_{j=0}^{s}T_j=\{2,3,\dots,k\}\\ \{(R_j,T_j,t_j)\}_{1\le j\le s}\text{unordered}}}(-1)^s\left\langle
		\tau^{a}_{0}\tau^{a_1}_{d_1}\prod_{i\in R_{-1}}\tau^{a_i}_{d_i}\prod_{j=1}^{s}\tau^{t_j}_{0}
		\right\rangle_0^{\frac{1}{r},\text{ext}}\\
		&\cdot\left\langle
		\tau^{r-2-a}_{0}\sigma^{b_1}\prod_{i\in R_0}\tau^{a_i}_{d_i}\prod_{i\in T_0}\sigma^{b_i}\sigma^{-1}
		\right\rangle_0^{\frac{1}{r},\text{o},\text{ext}}
		\prod_{j=1}^{s}\left\langle
		\sigma^{r-2-2t_j}\prod_{i\in R_j}\tau^{a_i}_{d_i}\prod_{i\in T_j}\sigma^{b_i}
		\right\rangle_0^{\frac{1}{r},\text{o},\text{ext}}.
		\end{split}
		\end{equation*}

\item[(b)] (Internal marked point TRR) If $l\ge 2$, then
\begin{equation*}
		\begin{split}
		&\left\langle
		\tau^{a_1}_{d_1+1}\tau^{a_2}_{d_2}\dots\tau^{a_l}_{d_l}\sigma^{b_1}\sigma^{b_2}\dots\sigma^{b_k}
		\sigma^{-1}\right\rangle_0^{\frac{1}{r},\text{o},\text{ext}}
		\\=&\sum_{\substack{s\ge 0\\-1\le a \le r-2\\ 0 \le t_j \le \frac{r-1}{2}}}\sum_{\substack{\coprod_{j=-1}^{s}R_j=\{3,4,\dots,l\}\\\coprod_{j=0}^{s}T_j=\{1,2,\dots,k\}\\ \{(R_j,T_j,t_j)\}_{1\le j\le s}\text{unordered}}}(-1)^s\left\langle
		\tau^{a}_{0}\tau^{a_1}_{d_1}\prod_{i\in R_{-1}}\tau^{a_i}_{d_i}\prod_{j=1}^{s}\tau^{t_j}_{0}
		\right\rangle_0^{\frac{1}{r},\text{ext}}\\
		&\cdot\left\langle
		\tau^{r-2-a}_{0}\tau^{a_2}_{d_2}\prod_{i\in R_0}\tau^{a_i}_{d_i}\prod_{i\in T_0}\sigma^{b_i}\sigma^{-1}
		\right\rangle_0^{\frac{1}{r},\text{o},\text{ext}}
		\prod_{j=1}^{s}\left\langle
		\sigma^{r-2-2t_j}\prod_{i\in R_j}\tau^{a_i}_{d_i}\prod_{i\in T_j}\sigma^{b_i}
		\right\rangle_0^{\frac{1}{r},\text{o},\text{ext}}.
		\end{split}
		\end{equation*}

\end{itemize}
\end{thm}

\bibliographystyle{abbrv}
\bibliography{OpenBiblio}

\end{document}